\documentclass[]{amsart}

\usepackage[Symbol]{upgreek}

\usepackage{graphicx}

\usepackage{amssymb}
\usepackage{epic}
\usepackage{mathrsfs}
\usepackage{accents}
\usepackage{array}
\usepackage{stmaryrd}
\usepackage{enumitem}
\usepackage{longtable}

\usepackage[hidelinks]{hyperref}

\usepackage{etoolbox} 

\usepackage{tikz}

\input{xy}
\xyoption{matrix}
\xyoption{arrow}

\numberwithin{equation}{section}

\makeatletter
\renewcommand{\tocsection}[3]
 { \indentlabel{\@ifnotempty{#2}{\parbox{2.5em}{\ignorespaces#1 #2.}\quad}}#3}

\makeindex

\newcommand{\bbold}{\mathbb}

\def\R { {\bbold R} }
\def\Q { {\bbold Q} }
\def\Z { {\bbold Z} }
\def\C { {\bbold C} }
\def\N { {\bbold N} }
\def\T { {\bbold T} }

\def\c {\mathcal{C}}

\def\g {\operatorname{g}}
\def \I{\operatorname{I}}

\def \Dx{\operatorname{D}}

\def \order{\operatorname{order}}

\def \exc {{\mathscr E}}
\def \ex{\operatorname{e}}

\def \Frac {\operatorname{Frac}}

\def \Univ {{\operatorname{U}}}

\renewcommand\epsilon{\varepsilon}

\def \d{\operatorname{d}}

\def \ev{\operatorname{e}}
\def \bar {\overline}
\def \<{\langle}
\def \>{\rangle}

\def \tilde {\widetilde}

\def \hat {\widehat}

\def \((  {(\!(}
\def \)) {)\!)}

\def \k {{{\boldsymbol{k}}}}

\DeclareMathSymbol{\precequ}{\mathrel}{symbols}{"16}
\DeclareMathSymbol{\succequ}{\mathrel}{symbols}{"17}

\renewcommand{\Re}{\operatorname{Re}}
\renewcommand{\Im}{\operatorname{Im}}

\newtheorem{theorem}{Theorem}[section]

\newtheorem{lemma}[theorem]{Lemma}
\newtheorem{prop}[theorem]{Proposition}
\newtheorem{cor}[theorem]{Corollary}

\newtheorem{corintro}{Corollary}

\newtheorem*{theoremA}{Theorem A}
\newtheorem*{theoremB}{Theorem B}

\theoremstyle{definition}

\newtheorem{definition}[theorem]{Definition}

\theoremstyle{remark}

\newtheorem*{remark}{Remark}
\newtheorem*{question}{Question}

\newtheorem{remarkNumbered}[theorem]{Remark}

\newcommand{\abs}[1]{\lvert#1\rvert}
\newcommand{\dabs}[1]{\lVert#1\rVert}

\def \fm {{\mathfrak m}}

\def \fn {{\mathfrak n}}
\def \fv {{\mathfrak v}}
\def \fw {{\mathfrak w}}

\let\oldi\i
\let\oldj\j
\renewcommand\i{\relax\ifmmode{\boldsymbol{i}}\else\oldi\fi}
\renewcommand\j{\relax\ifmmode{\boldsymbol{j}}\else\oldj\fi}

\renewcommand\leq{\leqslant}
\renewcommand\geq{\geqslant}
\renewcommand\preceq{\preccurlyeq}
\renewcommand\succeq{\succcurlyeq}
\renewcommand\le{\leq}
\renewcommand\ge{\geq}
\renewcommand\frak{\mathfrak}

\DeclareMathAlphabet{\mathbf}{OML}{cmm}{b}{it}

\DeclareFontFamily{U}{fsy}{}
\DeclareFontShape{U}{fsy}{m}{n}{<->s*[.9]psyr}{}
\DeclareSymbolFont{der@m}{U}{fsy}{m}{n}
\DeclareMathSymbol{\der}{\mathord}{der@m}{182}

\DeclareSymbolFont{der@m}{U}{fsy}{m}{n}
\DeclareMathSymbol{\derdelta}{\mathord}{der@m}{100}

\newcommand\wt{\operatorname{wt}}

\newcommand\dwt{\operatorname{dwt}}

\newcommand\ndeg{\operatorname{ndeg}}

\DeclareSymbolFont{imag@m}{OT1}{cmr}{m}{ui}
\DeclareMathSymbol{\imag}{\mathord}{imag@m}{105}

\DeclareFontFamily{OMS}{smallo}{}
\DeclareFontShape{OMS}{smallo}{m}{n}{<->s*[.65]cmsy10}{}
\DeclareSymbolFont{smallo@m}{OMS}{smallo}{m}{n}
\DeclareMathSymbol{\smallo}{\mathord}{smallo@m}{79}

\DeclareFontFamily{OMS}{largerdot}{}
\DeclareFontShape{OMS}{largerdot}{m}{n}{<->s*[.8]cmsy10}{}
\DeclareSymbolFont{largerdot@m}{OMS}{largerdot}{m}{n}
\DeclareMathSymbol{\largerdot}{\mathord}{largerdot@m}{15}

\DeclareMathSymbol{\llambda}{\mathord}{der@m}{108}
\DeclareMathSymbol{\rrho}{\mathord}{der@m}{114}

\def \upl{\uplambda}
\def \Upl{\Uplambda}
\def \upo{\upomega}
\def \Upo{\Upomega}

\def \w{\uptau}

\def\HLO{\Upl\Upo}

\newcommand{\equationqed}[1]{\[\pushQED{\qed}#1 \qedhere\popQED\]\let\qed\relax}
\newcommand{\alignqed}[1]{\begin{align*}\pushQED{\qed} #1 \qedhere\popQED\end{align*}\let\qed\relax}

\makeatletter
\newcommand{\dminus}{\mathbin{\text{\@dminus}}}

\newcommand{\@dminus}{%
  \ooalign{\hidewidth\raise1ex\hbox{\bf.}\hidewidth\cr$\m@th-$\cr}%
}
\makeatother

\def\ddeg{\operatorname{ddeg}}
\def\dwm{\operatorname{dwm}}

\def \Car{\mathcal{C}^r_a}
\def \Caz{\mathcal{C}^0_a}
\def \Cao{\mathcal{C}^1_a}

\def \Gr{\mathcal{C}^r}
\def \Gn{\mathcal{C}^n}

\def \Gi{\mathcal{C}^{<\infty}}
\def \Ginf{\mathcal{C}^{\infty}}
\def \Gom{\mathcal{C}^{\omega}}
\def \inv{\operatorname{inv}}

\def \Caj{\mathcal{C}^j_a}
\def \Cajl{\mathcal{C}^{j-1}_a}

\def \Car{\mathcal{C}^r_a}
\def \Carl{\mathcal{C}^{r-1}_a}
\def \Carm{\mathcal{C}^{r+1}_a}
\def \Carmi{\mathcal{C}^{r-i}_a}
\def \Caz{\mathcal{C}^0_a}
\def \Cazr{\mathcal{C}^r_{a_0}}
\def \Cazrl{\mathcal{C}^{r-1}_{a_0}}

\def \Coo{\mathcal{C}^1_0}
\def \Cao{\mathcal{C}^1_a}

\def \Can{\mathcal{C}^n_a}

\def \Calinf{\mathcal{C}^{<\infty}}
\def \Calr{\mathcal{C}^{r}}
\def \Caln{\mathcal{C}^{n}}

\def\b{\operatorname{b}}
\def \inte{\operatorname{int}}

\makeatletter
\renewcommand\part{\@startsection{part}{0}%
  \z@{\linespacing\@plus\linespacing}{.5\linespacing}%
  {\normalfont\bfseries\centering}}
\makeatother

\makeatletter
\renewcommand\theindex{\@restonecoltrue\if@twocolumn\@restonecolfalse\fi
  \columnseprule\z@ \columnsep 35\p@
  \twocolumn[\@xp\part\@xp*\@xp{\bf Index}\bigskip]%
  \let\item\@idxitem
  \parindent\z@  \parskip\z@\@plus.3\p@\relax
  \small}

\makeatletter
\newcommand{\smallbullet}{} 
\DeclareRobustCommand\smallbullet{%
  \mathord{\mathpalette\smallbullet@{0.6}}%
}
\newcommand{\smallbullet@}[2]{%
  \vcenter{\hbox{\scalebox{#2}{$\m@th#1\bullet$}}}%
}
\makeatother

\begin{document}

\title{The theory of maximal Hardy fields}
\author[Aschenbrenner]{Matthias Aschenbrenner}
\address{Kurt G\"odel Research Center for Mathematical Logic\\
Universit\"at Wien\\
1090 Wien\\ Austria}
\email{matthias.aschenbrenner@univie.ac.at}

\author[van den Dries]{Lou van den Dries}
\address{Department of Mathematics\\
University of Illinois at Urbana-Cham\-paign\\
Urbana, IL 61801\\
U.S.A.}
\email{vddries@math.uiuc.edu}

\author[van der Hoeven]{Joris van der Hoeven}
\address{CNRS, LIX (UMR 7161)\\ 
Campus de l'\'Ecole Polytechnique\\  91120 Palaiseau \\ France}
\email{vdhoeven@lix.polytechnique.fr}

\date{November, 2026}

\begin{abstract} 
We show that  all maximal Hardy fields are elementarily equivalent as differential fields to the differential field $\T$ of transseries, and give various applications of this result and its proof.\end{abstract}

\pagestyle{plain}
 
\maketitle


\tableofcontents

\section*{Introduction}

\noindent
Hardy~\cite{Ha} made sense of Du Bois-Reymond's ``orders of infinity'' \cite{dBR71}--\cite{dBR77}. This led to the notion of a Hardy field (Bourbaki~\cite{Bou}). A {\it Hardy field}\/ is a field~$H$ of germs at $+\infty$ of differentiable real-valued functions
on intervals~$(a,+\infty)$ such that
for any differentiable function whose germ is in~$H$ the germ of its derivative is also in~$H$. 
(See Section~\ref{sec:hardy} for more precision.) 
Every Hardy field is naturally a differential field, and also
an ordered field with the germ of $f$ being~$>0$ iff~$f(t)>0$, eventually. 
{\it Hardy fields are the natural domain of asymptotic analysis, where all rules hold, without qualifying conditions}\/ \cite[p.~297]{Ros}.
Many basic facts about Hardy fields can be found in Bourbaki~\cite{Bou}, Boshernitzan~\cite{Boshernitzan81}--\cite{Boshernitzan87}, and Rosen\-licht~\cite{Ros}--\cite{Rosenlicht95}.

\medskip
\noindent
Hardy~\cite{Har12a} focused on the Hardy field consisting of the germs of logarithmico-expo\-nen\-tial functions (LE-functions, for short): these
functions are the real-valued functions obtainable from real constants and the identity function~$x$ using addition, multiplication,  division,  taking logarithms, and exponentiating. Examples include  the germs of 
the function given for large positive $x$ by
$x^r$ ($r\in\R$), $\ex^{x^2}$, and $\log\log x$. 
Besides the  germs of LE-functions, the germs of many other naturally occurring non-oscillating differentially algebraic functions lie in Hardy fields. This includes in particular several special functions like the error function $\operatorname{erf}$, the exponential integral $\operatorname{Ei}$, the Airy functions $\operatorname{Ai}$ and $\operatorname{Bi}$, etc. 
 There are also Hardy fields which contain (germs of)  differentially transcendental functions, such as the Riemann $\zeta$-function and Euler's $\Gamma$-function~\cite{Rosenlicht83}, and
even functions ultimately growing faster than each LE-function~\cite{Boshernitzan86}. 

\medskip
\noindent 
Germs of functions in Hardy fields are non-oscillating in a strong sense.
 In certain applications, this kind of tameness alone is crucial: for example, \'Ecalle's proof~\cite{Ecalle} of Dulac's Conjecture (a weakened version of Hilbert's 16th Problem) essentially amounts to showing that the germ of the Poincar\'e return map at a cross section of a limit cycle lies in a Hardy field (at $0+$ instead of $+\infty$). 
A stronger form of tameness is o-minimality: indeed, every o-minimal structure on the real field naturally gives rise to 
a Hardy field (of germs of definable functions).  
This yields a wealth of  examples such as those obtained from quasi-analytic Denjoy-Carleman classes~\cite{RSW},
or containing certain transition maps of plane analytic vector fields~\cite{KRS},
and explains the role of Hardy fields in model theory and its applications to real analytic geometry and
 dynamical systems~\cite{AvdD4, BMS, Miller}. 
 
\medskip
\noindent
Hardy fields have also   found other applications: for
effective counterparts to Hardy's theory of LE-functions, see~\cite{GeddesGonnet, Gruntz, vdH98, RSSvdH, Shackell90}.
Hardy fields have provided an analytic setting for extensions of this work beyond LE-functions~\cite{vdH96, SS98, SS99, Shackell96, Shackell}. They have also been useful in ergodic theory (see, e.g., \cite{BMR, BoshernitzanWierdl, FW, KoMue}), and other areas of  mathematics~\cite{BB, CS, CMR, Elliott, FishmanSimmons, GvdH}. 

\medskip
\noindent
In the remainder of this introduction, $H$ is a Hardy field.
Then   $H(\R)$ (obtained by adjoining the germs of the constant functions) is also a Hardy field,
and for any~${h\in H}$, the germ $\ex^h$ generates a Hardy field~$H(\ex^h)$ over $H$, and so does any differentiable germ with derivative
$h$. Moreover,~$H$ has a unique Hardy field extension
that is algebraic over $H$ and real closed.  (See Section~\ref{sec:hardy} for these facts, especially Proposition~\ref{prop:Li(H(R))}.)  
Our main result is Theorem~\ref{thm:char d-max}, and it yields what appears to be the ultimate fact about differentially algebraic Hardy field extensions:

\begin{theoremA}\label{thm:A} Let  $P(Y)$ be a differential polynomial in a single
differential indeterminate $Y$ over $H$,  and let $f< g$ in $H$ be such that $P(f) <0 <  P(g)$. Then there is a $y$ in a Hardy field extension of $H$ such that $f < y < g$ and $P(y)=0$.
\end{theoremA} 

\noindent
By Zorn, every Hardy field extends to a maximal Hardy field, 
so by the theorem above, maximal Hardy fields have the intermediate value property for differential polynomials. 
(In \cite{ADHfgh} we show there are very many maximal Hardy fields, namely~$2^{\frak{c}}$ many, where $\frak{c}$ is the cardinality of the continuum.) By the results mentioned earlier, maximal Hardy fields are also Liouville closed $H$-fields in the 
sense of \cite{AvdD2}; thus they contain  the germs of all LE-functions. Hiding behind the
intermediate value property of Theorem~A are two more fundamental properties, {\it $\upo$-freeness}\/ and {\it newtonianity,}\/ which are central in our book [ADH]. (Roughly speaking, $\upo$-freeness concerns second-order homogeneous differential equations, and newtonianity is a strong version of differential-henselianity.) 
In~\cite{ADH5} we show that any Hardy field has an $\upo$-free Hardy field extension, and in this paper we prove 
the much harder result that any $\upo$-free Hardy field
extends to a newtonian $\upo$-free Hardy field: Theorem~\ref{thm:char d-max}, 
which is really the main result of this paper. It follows that every maximal Hardy field is, in the terminology of \cite{ADH2}, an 
{\it $H$-closed field}\/ with small derivation.\index{closed!$H$-closed}\index{H-closed@$H$-closed}\index{Hardy field!H-closed@$H$-closed} Now 
the elementary theory $T_H$ of $H$-closed fields with small derivation (denoted by $T^{\text{nl}}_{\text{small}}$ in [ADH]) is 
{\em complete}, by~[ADH, 16.6.3]. This means in particular that,  as advertised in the abstract of this paper,  any two maximal Hardy fields are indistinguishable by their elementary 
properties:

\begin{corintro} \label{cor:elem equiv} If $H_1$ and $H_2$ are maximal Hardy fields, then
$H_1$ and $H_2$ are elementarily equivalent as ordered differential fields. 
\end{corintro}

\noindent
To derive Theorem~A we use also key results from the book
 \cite{JvdH} to the effect that~$\T_{\text{g}}$, the ordered differential field of grid-based transseries,\index{transseries!grid-based}
 is $H$-closed with small derivation and the intermediate value property for differential polynomials. 
In particular, it is a model of the complete
theory $T_H$. Thus maximal Hardy fields have the intermediate value property for differential polynomials as well,
and this amounts to Theorem~A, obtained here as a byproduct of more fundamental results.  (A detailed account of the differential intermediate value property for $H$-fields is in \cite{ADHip}.) 
We sketch the proof of our main result (Theorem~\ref{thm:char d-max})  
later in this introduction, after describing further consequences.

\subsection*{Further consequences of  our main result}
 In [ADH] we prove more than completeness of $T_H$:   a certain natural extension by definitions of $T_H$ has quantifier elimination.
This leads to a strengthening of Corollary~\ref{cor:elem equiv} by allowing parameters from a common Hardy subfield of~$H_1$ and~$H_2$.
To fully appreciate this statement requires more knowledge of model theory, as in~[ADH, Appendix~B], which we do not assume for this introduction. However,
we can explain a special case in a direct way,
in terms of  solvability of systems of algebraic differential
equations, inequalities, and asymptotic inequalities.
Here we find it convenient to use  the notation
for asymptotic relations introduced by du~Bois-Reymond and Hardy  instead of Bachmann-Landau's  $O$-notation:  
  for germs~$f$,~$g$ in a Hardy field set
\begin{align*}  
f\preceq g	&\quad:\Longleftrightarrow\quad f=O(g) &\hskip-2em& :\Longleftrightarrow\quad \abs{f}\leq c\abs{g}\text{ for some real $c>0$,}\\
f\prec g	&\quad:\Longleftrightarrow\quad f=o(g) &\hskip-2em& :\Longleftrightarrow\quad \abs{f} < c\abs{g}\text{ for all real  $c>0$. }
\end{align*}  
Let now~$Y=(Y_1,\dots,Y_n)$ be a tuple of distinct (differential) indeterminates, and consider    a system of the following form:
\begin{equation}\tag{$\ast$}\label{eq:system}\begin{cases} &
\begin{matrix}
P_1(Y) & \varrho_1 & Q_1(Y) \\
\vdots & \vdots    & \vdots \\
P_k(Y) & \varrho_k & Q_k(Y)
\end{matrix}\end{cases}
\end{equation}
Here the $P_i$ and $Q_i$ are differential polynomials in $Y$  (that is, 
polynomials in the indeterminates $Y_j$ and their formal derivatives $Y_j',Y_j'',\dots$) with coefficients in our Hardy field~$H$, and each
$\varrho_i$ is one of the symbols ${=}$, ${\neq}$, ${\leq}$, ${<}$, ${\preccurlyeq}$,  ${\prec}$. Given a Hardy field $E\supseteq H$,
a {\it solution}\/\index{solution!system} of~\eqref{eq:system} in $E$ is an $n$-tuple $y=(y_1,\dots,y_n)\in E^n$ 
such that for~$i=1,\dots,k$, the relation
$P_i(y) \, \varrho_i\, Q_i(y)$ holds in $E$.
Here is a Hardy field analogue of the ``Tarski Principle''
of real algebraic geometry [ADH, B.12.14]: 

\begin{corintro}\label{cor:systems, 1}
If the system \eqref{eq:system} has a solution in \emph{some}  Hardy field extension of $H$, then \eqref{eq:system} has a solution in \emph{every} maximal
Hardy field extension of $H$. 
\end{corintro}

\noindent
(The symbols $\neq$, $\leq$, $<$, $\preceq$ in \eqref{eq:system} are for convenience only:
their occurrences  can be eliminated at the cost of increasing $k$, $n$.
But $\prec$ is essential; see [ADH, 16.2.6].) Besides the quantifier elimination alluded to, Corollary~\ref{cor:systems, 1} depends on Lemma~\ref{lem:canonical HLO}, which says that for any Hardy field  $H$ all maximal Hardy field extensions of $H$ induce the same $\HLO$-cut on $H$, as defined in [ADH, 16.3]. 

\medskip
\noindent
In particular, taking for $H$ the smallest Hardy field $\Q$, we see that a system \eqref{eq:system} with a solution in some  Hardy field  
has a solution in \emph{every} maximal
Hardy field, thus recovering a special case of our Corollary~\ref{cor:elem equiv}. Call such a system~\eqref{eq:system} over~$\Q$ {\it consistent.}\/ For example, with $X$, $Y$, $Z$ denoting here single distinct differential indeterminates, the system
$$Z'\ \preceq\ Y'Z,\qquad Y \preceq 1, \qquad 1\prec Z$$
is inconsistent, whereas for any $Q\in\Q\{Y\}$ and $n\geq 2$ the system
$$X^n Y'\ =\ Q(Y),\qquad X' = 1,\quad Y\prec 1$$
is consistent.
The completeness of~$T_H$ has an interesting consequence: 

\begin{corintro}
There exists an algorithm for deciding whether a given system~\eqref{eq:system} over~$\Q$ is consistent.
\end{corintro}

\noindent
Albeit the resulting algorithm is very impractical, this opens up the possibility of automating a substantial part of
asymptotic analysis in Hardy fields. 
We  remark that Singer~\cite{Singer} proved the existence of an algorithm for
deciding whether a given system \eqref{eq:system} over $\Q$ without occurrences of
${\preccurlyeq}$ or ${\prec}$ has a solution in {\it some}\/ ordered differential field (and then it will have a solution 
in the ordered differential field of germs of real meromorphic functions at $0$); but  there are such systems, like
$$X'\ =\ 1,\qquad XY^2\ =\ 1-X,$$
which are solvable in an ordered differential field, but not in a Hardy field.
Also, algorithmically deciding the solvability of  a system~\eqref{eq:system} over~$\Q$ in a {\it given}\/ Hardy field
$H$ may be impossible when $H$ is ``too small'': e.g., if~$H=\R(x)$, by~\cite{Denef}.

\medskip
\noindent
As these results suggest, the aforementioned quantifier elimination for $T_H$ yields a kind of  ``resultant'' for systems~\eqref{eq:system} that allows one to make explicit within~$H$ itself for which choices of coefficients of the 
differential polynomials~$P_i$,~$Q_i$  the system~\eqref{eq:system} has a solution in a Hardy field
extension of $H$. Without going into details,
 we only mention here some attractive  consequences   for systems~\eqref{eq:system} depending on  parameters.
For this, let $X_1,\dots, X_m,Y_1,\dots, Y_n$ be distinct indeterminates and~$X=(X_1,\dots,X_m)$, $Y=(Y_1,\dots,Y_n)$,
and consider a system
\begin{equation}\tag{$\ast \ast$}\label{eq:system param, Y}\begin{cases} &
\begin{matrix}
P_1(X,Y) & \varrho_1 & Q_1(X,Y) \\
\vdots & \vdots    & \vdots \\
P_k(X,Y) & \varrho_k & Q_k(X,Y)
\end{matrix}\end{cases}
\end{equation}
where $P_i$, $Q_i$ are now differential polynomials in $(X,Y)$ over $H$, and the $\varrho_i$ are as before.
Specializing $X$ to $c\in \R^m$ then yields a system
\begin{equation}\tag{$\ast c$}\label{eq:system param}\begin{cases} &
\begin{matrix}
P_1(c,Y) & \varrho_1 & Q_1(c,Y) \\
\vdots & \vdots    & \vdots \\
P_k(c,Y) & \varrho_k & Q_k(c,Y)
\end{matrix}\end{cases}
\end{equation}
where $P_i(c,Y)$, $Q_i(c,Y)$ are  differential polynomials in $Y$ with coefficients in the Hardy field~$H(\R)$.
(We only substitute real constants, so may assume that the~$P_i$,~$Q_i$ are {\it polynomial}\/ in $X$, that is, none of the derivatives~$X_j',X_j'',\dots$ occur in the $P_i, Q_i$.)
Using~[ADH, 16.0.2(ii)] we obtain:

\begin{corintro}\label{cor:parametric systems}
The  set of all $c\in\R^m$ such that 
the system \eqref{eq:system param} has a solution in some Hardy field extension of $H$ is semialgebraic.
\end{corintro}

\noindent
Recall: a subset of $\R^m$ is said to be {\it semialgebraic}\/ if it is a finite union of  sets 
$$\big\{ c\in\R^m:\  p(c)=0,\ q_1(c)>0,\dots,q_l(c)>0 \big\}$$
where $p,q_1,\dots,q_l\in\R[X]$  are ordinary polynomials.\index{semialgebraic}
(The topological and geometric properties of semialgebraic sets have been studied extensively~\cite{BCR}. For example, it is well-known that a semialgebraic set can   only have finitely many connected components, and that each such component is itself semialgebraic.) 

\medskip
\noindent
In connection with  Corollary~\ref{cor:parametric systems} we mention that the asymptotics of Hardy field solutions to algebraic differential equations~${Q(Y)=0}$,
where~$Q$ is a differential polynomial with constant real coefficients, has been investigated by Hardy~\cite{Har12}  and   Fowler~\cite{Fowler} in  cases where $\order Q\leq 2$ (see \cite[Chapter~5]{Bellman}), and later by Shackell~\cite{SS95,Shackell93,Shackell95} in general.
Special case of  our corollary: for any differential polynomial~$P(X,Y)$ with constant real coefficients, the set of parameters~$c\in\R^m$ such that the differential
equation~${P(c,Y)=0}$ has a solution~$y$  in some Hardy field, in addition possibly also satisfying  given asymptotic  side conditions (such as~$y\prec 1$),  is semialgebraic. 
Example: the set of real parameters~$(c_1,\dots,c_{m})\in\R^{m}$ for which the homogeneous linear differential equation
$$y^{(m)} + c_{1}y^{(m-1)}+\cdots+ c_m y\  =\ 0$$
has a nonzero solution $y\prec 1$ in a Hardy field is semialgebraic; in fact, it is the
set of all  $(c_1,\dots,c_{m})\in\R^{m}$ such that the polynomial
$Y^m + c_{1}Y^{m-1}+ \cdots+ c_m \in \R[Y]$ has a negative real zero. 
Nonlinear example:
for $g_2, g_3\in \R$ the differential equation
$$(Y')^2\ =\ 4Y^3-g_2Y-g_3  $$
has a nonconstant solution in a Hardy field iff
 $g_2^3=27g_3^2$   and~$g_3 \leq 0$.
In both cases, the Hardy field solutions are  germs of logarithmico-exponential functions. But the
class of differentially algebraic germs in Hardy fields is much more extensive; for example, the antiderivatives of $\ex^{x^2}$  are not logarithmico-exponential (Liouville).

\medskip
\noindent
Instead of $c\in \R^m$, substitute $h\in H^m$ for $X$ in   \eqref{eq:system param, Y}, resulting  in a system 
\begin{equation}\tag{$\ast h$}\label{eq:system param, h}\begin{cases} &
\begin{matrix}
P_1(h,Y) & \varrho_1 & Q_1(h,Y) \\
\vdots & \vdots    & \vdots \\
P_k(h,Y) & \varrho_k & Q_k(h,Y)
\end{matrix}\end{cases}
\end{equation}
where $P_i(h,Y)$, $Q_i(h,Y)$ are now differential polynomials in $Y$ with coefficients in~$H$.
It is   well-known  that for any semialgebraic set $S\subseteq \R^{m+1}$  
there is a natural number~$B=B(S)$ such that for every $c\in\R^m$, if the section
$\big\{y\in \R: (c,y)\in S\big\}$
 has~$> B$ elements, then this section has nonempty interior in $\R$. In contrast,
the set of solutions of \eqref{eq:system param, h} for $n=1$ in a maximal $H$ can be simultaneously infinite and discrete in the order topology of $H$: this happens precisely if some nonzero one-variable differential polynomial over $H$ vanishes on  this solution set [ADH, 16.6.11].
(Consider the example of the single algebraic differential equation $Y'=0$, which has solution set $\R$ in each maximal Hardy field.)
Nevertheless, we have the  following
uniform finiteness principle for solutions of~\eqref{eq:system param, h};
its proof is
considerably deeper than Corollary~\ref{cor:parametric systems} and
 also draws on results from~\cite{ADHdim}.

\begin{corintro}\label{cor:uniform finiteness}
There is a natural number $B=B\eqref{eq:system param, Y}$ such that for all $h\in H^m$:
if the system~\eqref{eq:system param, h} has $> B$ solutions in some Hardy field extension of $H$,
then~\eqref{eq:system param, h} has continuum many solutions in every maximal Hardy field extension of~$H$.
\end{corintro}

\noindent
Next we turn to issues of smoothness and analyticity in Corollary~\ref{cor:systems, 1}. 
By definition, a Hardy field is a differential subfield of the  differential ring $\Calinf$ consisting of the germs of functions $(a,+\infty)\to\R$ ($a\in\R$)  which are, for each~$n$, eventually $n$-times
continuously differentiable. Now $\Calinf$ has the differential subring $\Ginf$ whose elements are
the germs that are eventually $\c^\infty$. A {\it $\Ginf$-Hardy field}\/ is a Hardy field~$H\subseteq\Ginf$. (See~\cite{Gokhman} for an example
of a Hardy field~${H\not\subseteq\Ginf}$.)
A $\Ginf$-Hardy field  is said to be {\it $\Ginf$-maximal}\/ if it has no proper $\Ginf$-Hardy field
extension. Now $\Ginf$ in turn has 
the differential subring $\Gom$ whose elements are  the germs that are eventually
real analytic, and so we define likewise  $\Gom$-Hardy fields
($\Gom$-maximal Hardy fields, respectively). Our main theorems go through in the $\Ginf$- and $\Gom$-settings;
combined
with model completeness of~$T_H$ shown in~[ADH, 16.2] this ensures the existence of solutions with appropriate smoothness
in Corollary~\ref{cor:systems, 1}:

\begin{corintro}\label{cor:systems, 2}
If $H\subseteq\Ginf$ and the system \eqref{eq:system} has a solution in some Hardy field extension of $H$, then \eqref{eq:system} has a solution in every $\Ginf$-maximal Hardy field extension of $H$. In particular, if $H$ is   $\Ginf$-maximal  and
\eqref{eq:system} has a solution in a Hardy field extension of $H$, then it has a solution in $H$.
\textup{(}Likewise with $\Gom$ in place of $\Ginf$.\textup{)}
\end{corintro}

\noindent
We already mentioned  $\T_{\text{g}}$ as a quintessential example of an $H$-closed field. Its cousin~$\T$, the ordered differential field
of transseries,   extends $\T_{\text{g}}$ and is also $H$-closed with constant field~$\R$~[ADH, 15.0.2].\index{transseries} The elements of $\T$ are certain generalized  
series (in the sense of Hahn) in an indeterminate $x>\R$ with real coefficients,  involving exponential and logarithmic terms, such as
$$ f\  =\  \ex^{\frac{1}{2}\ex^x}-5\ex^{x^2}+\ex^{x^{-1}+2x^{-2}+\cdots}+\sqrt[3]{2}\log x-x^{-1}+\ex^{-x}+\ex^{-2x}+\cdots+5\ex^{-x^{3/2}}.$$
Mathematically  significant examples are the more simply structured trans\-series
\begin{multline*}
\operatorname{Ai}\ =\  \frac{\ex^{-\xi}}{2\sqrt{\pi}x^{1/4}}\sum_n (-1)^n   c_n \xi^{-n},
\quad
\operatorname{Bi}\ =\ \frac{\ex^{\xi}}{\sqrt{\pi}x^{-1/4}}\sum_n c_n \xi^{-n}, \\
\qquad\text{where $\ c_n\ =\ \frac{(2n+1)(2n+3)\cdots (6n-1)}{(216)^n n!}\ $ and $\ \xi\ =\ \frac{2}{3}x^{3/2}$,}
\end{multline*}
which are $\R$-linearly independent solutions of  the Airy equation $Y''=xY$ {\cite[Chap\-ter~11, (1.07)]{Olver}}.
For   information about~$\T$, see [ADH, Appendix~A] or~\cite{Edgar,JvdH}. We just mention here that
like each $H$-field, $\T$ comes equipped with its
own versions of the asymptotic relations $\preceq$, $\prec$, defined as for $H$ above. 
The asymptotic rules valid in all Hardy fields, such as
$$f\preceq 1\ \Rightarrow\ f'\prec 1,\qquad f\preceq g\prec 1\ \Rightarrow\ f'\preceq g', \qquad
f'=f\neq 0\ \Rightarrow\ f\succ x^n $$
also hold in $\T$. 
Here $x$ denotes, depending on the context, the germ of the identity function on $\R$, as well as the element   $x\in\T$.
Section~\ref{sec:embeddings into T} gives a finite axiomatization of these rules, thus completing
an investigation initiated by A.~Robinson~\cite{Robinson73}. 

\medskip
\noindent
Now suppose that we are given an embedding~$\iota\colon H\to\T$   of ordered differential fields.
We may view such an embedding as a {\it formal expansion operator}\/ and its inverse as a {\it summation operator.}\/
(See Section~\ref{sec:embeddings into T} below  for an example of a Hardy field, arising from a fairly rich o-minimal structure, which admits such an embedding.)
From~\eqref{eq:system} we obtain a system
\begin{equation}\tag{$\iota\ast$}\label{eq:iota(system)}\begin{cases} &
\begin{matrix}
\iota(P_1)(Y) & \varrho_1 & \iota(Q_1)(Y) \\
\vdots & \vdots    & \vdots \\
\iota(P_k)(Y) & \varrho_k & \iota(Q_k)(Y)
\end{matrix}\end{cases}
\end{equation}
of algebraic differential equations and (asymptotic) inequalities over $\T$,
where~$\iota(P_i)$, $\iota(Q_i)$   denote the differential polynomials over $\T$ obtained by applying $\iota$ to the coefficients
of $P_i$, $Q_i$, respectively. A {\it solution}\/\index{solution!system} of \eqref{eq:iota(system)} is a tuple $y=(y_1,\dots,y_n)\in\T^n$
such that $\iota(P_i)(y) \, \varrho_i\, \iota(Q_i)(y)$ holds in $\T$, for $i=1,\dots,k$.
Differential-difference equations in $\T$ are sometimes amenable to functional-analytic
techniques like fixed point theorems or small (compact-like) operators~\cite{vdH:noeth}, and
the formal nature of transseries also makes it possible to solve
algebraic differential equations in $\T$ by quasi-algorithmic methods~\cite{vdH:PhD,JvdH}. The simple example of the Euler equation
$$Y'+Y\ =\ x^{-1}$$
is instructive: its solutions in $\Calinf$
are given by the germs of 
$$t\mapsto \ex^{-t}\int_1^t \frac{\ex^{s}}{s}\,ds+c\ex^{-t}\colon (1,+\infty)\to\R\qquad (c\in\R),$$
all contained in a common Hardy field extension of $\R(x)$.
The solutions of this differential equation in $\T$ are 
$$\sum_n n!\, x^{-(n+1)}+c\ex^{-x}\qquad (c\in\R),$$
where the particular solution $\sum_n n!\, x^{-(n+1)}$ is obtained as the unique fixed point of the   operator $f\mapsto x^{-1}-f'$ on
the differential subfield
$\R(\!(x^{-1})\!)$ of $\T$ (cf.~[ADH, 2.2.13]).
(Note:  $\sum_n n!\,t^{-(n+1)}$ diverges for each $t>0$.)
In general, the existence of a solution of
\eqref{eq:iota(system)} in $\T$  entails the existence of a solution of~\eqref{eq:system}  in some  Hardy field extension of~$H$ and vice versa; more precisely:

\begin{corintro}\label{cor:systems, 3}
The system \eqref{eq:iota(system)} has a solution in $\T$ iff \eqref{eq:system} has a solution in some Hardy field extension of $H$.
In this case, we can choose a solution  of \eqref{eq:system} in a Hardy field extension $E$ of $H$ for which $\iota$ extends
to an embedding of ordered differential fields~$E\to\T$.
\end{corintro}

\noindent
In particular, a system \eqref{eq:system} over $\Q$ is consistent if and only if it has a solution in~$\T$.
(The ``if'' direction  already follows from  [ADH, Chapter~16] and~\cite{vdH:hfsol};  the latter constructs a summation operator  on the ordered differential subfield~$\T^{\operatorname{da}}\subseteq\T$   of   differentially algebraic transseries.)

\medskip
\noindent
It is worth noting that a  result about differential equations in one variable, like
Theorem~A (or Theorem~B below), yields similar facts about \emph{systems}\/ of algebraic differential equations and asymptotic inequalities in several variables over Hardy fields as in the corollaries above; we owe this to the strength of the model-theoretic methods employed in~[ADH]. But our theorem in combination with  [ADH] already has interesting consequences 
for one-variable differential polynomials over $H$ and over its ``complexification''~${K:=H[\imag]}$ (where~$\imag^2=-1$),
which is a differential subfield of the differential ring $\Calinf[\imag]$. Some of these facts are analogous to familiar properties of
ordinary one-variable polynomials over the real or complex numbers.
First, by Theorem~A, every differential polynomial in one variable over $H$ of  odd degree has a zero in a Hardy field extension of $H$. (See Corollary~\ref{cor:odd degree}.)   For example,  a differential polynomial like
$$ (Y'')^5+\sqrt{2}\ex^x (Y'')^4 Y'''-x^{-1}\log x \, Y^2 Y''+ YY'-\Gamma$$
has a zero in every maximal Hardy field extension of the Hardy field $\R\langle \ex^x,\log x,\Gamma\rangle$.
Passing to $K=H[\imag]$ we have:

\begin{corintro}\label{cor:zeros in complexified Hardy field extensions}
For each  differential polynomial~$P\notin K$ in a single differential indeterminate  with coefficients in $K$   there are~$f$, $g$ in a Hardy field extension of $H$  such that~$P(f+g\imag)=0$. 
\end{corintro}

\noindent
In particular, each   linear differential equation
$$y^{(n)}+a_{1}y^{(n-1)}+\cdots+a_n y\ =\  b\qquad (a_1,\dots,a_{n},b\in K)$$
has a solution $y=f+g\imag$ where $f$, $g$ lie  in some Hardy field extension  of $H$. (Of course, if~$b=0$, then we may take here the trivial solution $y=0$.) Although this special case of Corollary~\ref{cor:zeros in complexified Hardy field extensions} concerns
differential polynomials of degree~$1$, the only proof we have uses our more general extension theorems: a solution $y$
of a linear differential equation of order $n$ over $K$ as above may simultaneously be a zero of a non-linear differential polynomial $P$ over $K$ of    order~$< n$, and the structure of the differential field extension   of $K$ generated by $y$ is governed by $P$ (when taken of minimal complexity in the sense of [ADH, 4.3]).

\medskip
\noindent
The main theorem of this paper  has new consequences even for linear differential equations over Hardy fields.
These are the subject of the follow-up paper~\cite{ADH6}, 
where we also prove   Conjecture~4 of Boshernitzan~\cite{Boshernitzan82}.
Here
we just mention one immediate consequence of Corollary~\ref{cor:zeros in complexified Hardy field extensions} for linear differential operators, as it helps to
 motivate our interest in the ``universal exponential extension'' of $K$,   explained later in this introduction.  For this, 
  consider  the  ring~$R[\der]$   of linear differential operators over a differential ring~$R$:
this ring  is a free left $R$-module with  basis~$\der^n$~($n\in\N$)
such that~$\der^0=1$ and~$\der \cdot f=f\der+f'$ for~$f\in R$, where~${\der:=\der^1}$. (See~[ADH, 5.1] or~\cite[2.1]{vdPS}.)
Any~$A\in R[\der]$ gives rise to an additive map~${y\mapsto A(y)\colon R\to R}$, with~${\der^n(y)=y^{(n)}}$ (the $n$th derivative of $y$ in~$R$) and~$r(y)=ry$ for~$r=r\cdot 1\in {R\subseteq R[\der]}$. 
The elements of~$\der^n+R\der^{n-1}+\cdots+R\subseteq R[\der]$ are said to be  {\it monic}\/ of order~$n$.
It is well-known~\cite{CamporesiDiScala,Mammana26,Mammana31} that for $R=\Calinf[\imag]$, each monic~$A\in R[\der]$   factors as a product of monic operators of order~$1$ in $R[\der]$; if~$A\in K[\der]$, then such a factorization already happens over the complexification of some Hardy field extension of $H$:

\begin{corintro}\label{cor:factorization intro}
If $H$ is maximal, then each monic operator in $K[\der]$  is a product of  monic operators of order $1$ in $K[\der]$.
\end{corintro}

\noindent
This follows quite easily from Corollary~\ref{cor:zeros in complexified Hardy field extensions} using the Riccati transform [ADH, 5.8].
Let now $A\in K[\der]$ be monic of order~$n$, and   fix a maximal Hardy field extension~$E$ of~$H$.
The factorization result in Corollary~\ref{cor:factorization intro} gives rise to a description of a fundamental system of solutions for the  homogeneous linear differential equation~$A(y)=0$ in terms of Hardy field germs. Here, of course, complex exponential terms    naturally appear, but only in a controlled way: the $\C$-linear space consisting of all~$y\in\Calinf[\imag]$ with~$A(y)=0$ has a basis of the form
$$f_1\ex^{\phi_1\imag},\ \dots,\ f_n\ex^{\phi_n\imag}$$
where $f_j\in E[\imag]$ and $\phi_j\in E$ with $\phi_j=0$ or~$\abs{\phi_j}>\R$ for $j=1,\dots,n$. 
(This will be refined in \cite{ADH6}; e.g., the   basis elements $f_i\ex^{\phi_i\imag}$  for distinct frequencies~$\phi_i$ can be arranged
to be pairwise orthogonal in a certain sense.)
A special case: if~${y\in\Calinf[\imag]}$ is {\it holonomic,}\/ that is, $L(y)=0$ for some monic $L\in\C(x)[\der]$, then $y$
is a $\C$-linear combination of germs $f\ex^{\phi\imag}$ where $f\in E[\imag]$, $\phi\in E$, and~$\phi=0$ or~$\abs{\phi}>\R$.
Here, more information about the~$f$,~$\phi$ is available (see, e.g.,~\cite[VIII.7]{FS}, \cite[\S{}19.1]{Wasow}).
Many special functions 
are holonomic~\cite[B.4]{FS}.

\subsection*{Synopsis of the proof of our main theorem} 
In the rest of the paper we assume familiarity with the terminology and concepts of asymptotic differential algebra from our book~[ADH].
We refer to the section {\it Concepts and Results from~\textup{[ADH]}}\/  in the introduction of~\cite{ADH4} for a concise exposition of the 
 background from~[ADH] required to read \cite{ADH4} and the present paper.
The differential-algebraic and valuation-theoretic tools from~[ADH] were further developed in \cite{ADH4}, and we review some of this
in Section~\ref{sec:prelims}.
However, the proof of our main result also requires analytic arguments in an essential way. Some of this analysis adapts \cite{vdH:hfsol} to a more general setting. 
As mentioned earlier, Theorem~A is a consequence of the following extension theorem: 

\begin{theoremB}\label{thm:B}
Every $\upo$-free Hardy field has a newtonian  Hardy field extension.
\end{theoremB}

\noindent
The proof of this is long and intertwined with the normalization procedures in \cite{ADH4}, so it may be useful to outline the strategy behind it. 

\subsubsection*{Holes and slots}
For now, let $K$ be an $H$-asymptotic field with rational asymptotic integration. (Cases to keep in mind are $K=H$ and
$K=H[\imag]$ where $H$ is a Liouville closed Hardy field.) 
In \cite{ADH4}  we introduced the apparatus of {\it holes}\/ in~$K$ as a means to systematize the study of   solutions of  
algebraic differential equations over~$K$ in immediate asymptotic extensions of~$K$: such a hole in $K$ is a triple~$(P,\fm,\hat f)$
where~$P$ is a nonzero differential polynomial in a single differential indeterminate~$Y$ with coefficients in~$K$,   $\fm\in K^\times$, and $\hat f\notin K$ lies an immediate asymptotic extension  of~$K$
with~${P(\hat f)=0}$ and~$\hat f\prec\fm$. 
It is sometimes technically convenient to work with the more flexible concept of a {\it slot}\/  in $K$, where instead of~$P(\hat f)=0$  we only require  
$P$ to vanish at $(K,\hat f)$ in the sense of~[ADH, 11.4]. The {\it complexity}\/ of a slot $(P,\fm,\hat f)$ is the complexity of the differential polynomial~$P$ as  in~[ADH, p.~216]. (See also Section~\ref{sec:prelims} below.) 
Now if $K$ is $\upo$-free, then by \cite[Lem\-ma~3.2.1]{ADH4},
$$\text{$K$ is newtonian}\quad\Longleftrightarrow\quad\text{$K$ has no hole.}$$
This equivalence suggests an opening move for proving Theorem~B by induction on complexity as follows:
Let $H\supseteq\R$ be an $\upo$-free Liouville closed Hardy field, and suppose $H$ is not newtonian; it is enough to show that then $H$ has a proper Hardy field extension. By the above equivalence, $H$ has a hole
$(P, \fm,\hat f)$, and we can take here $(P,\fm,\hat f)$ to be of minimal complexity among holes in $H$.
This minimality has consequences that are important for us; for example $r:=\order P\geq 1$, $P$ is a minimal annihilator of $\hat f$ over $H$, and $H$ is $(r-1)$-newtonian as defined in~[ADH, 14.2]. We arrange $\fm=1$  by
replacing $(P,\fm,\hat f)$ with the hole~$(P_{\times\fm},1,\hat f/\fm)$ in $H$.

\subsubsection*{Solving algebraic differential equations over Hardy fields}
For Theorem~B it is enough to show that under these conditions $P$ is a minimal annihilator
of some germ~${f\in \Calinf}$ that generates a (necessarily proper) Hardy field extension~$H\langle f\rangle$ of~$H$.  So at a minimum, we need to  find    a  solution in $\Calinf$ to the algebraic differential equation~$P(Y)=0$. For this, it is natural to use fixed point techniques as in~\cite{vdH:hfsol}.
Notation: for~${a\in\R}$, $\c^n_a$ is the $\R$-linear space of $n$-times continuously differentiable functions~$[a,+\infty)\to\R$; each $f\in\Calinf$ has a representative in~$\c^n_a$. 

\subsubsection*{A fixed point theorem}
Let $L:=L_P\in H[\der]$ be the linear part  of $P$.  We can replace~$(P,1,\hat f)$ with another minimal hole in~$H$ to arrange ${\order L=r}$.
Representing the coefficients of  $P$ (and thus of $L$) by functions in~$\c^0_a$ we obtain  an $\R$-linear operator $y\mapsto L(y)\colon \c^r_a \to\c^0_a$.
For now we make the  bold assumption that   $L\in H[\der]$ splits over  $H$. Using such a splitting and increasing $a$ if necessary, $r$-fold integration yields an $\R$-linear operator $L^{-1}\colon \c^0_a\to\c^r_a$ which is a {\it right-inverse}\/ of~$L\colon \c^r_a \to\c^0_a$, that is, $L\big(L^{-1}(y)\big)=y$ for all $y\in\c^0_a$.
Consider    the (generally non-linear) operator
$$f\mapsto \Phi(f):=L^{-1}\big(R(f)\big)$$ on $\c^r_a$; here~$P=P_1-R$ where~$P_1$ is the homogeneous part of degree~$1$ of~$P$.
We try to show that $\Phi$ restricts to a contractive operator
 on a closed ball of an appropriate subspace of~$\c^r_a$ equipped with a suitable complete norm,  
 whose fixed points are then solutions to~$P(Y)=0$; this may also involve increasing $a$ again and replacing the coefficient functions of $P$ by their corresponding restrictions. 
We can obtain such contractivity if $R$ is asymptotically small compared to~$P_1$ in a certain sense.
The latter can indeed be achieved by transforming~$(P,1,\hat f)$ into a certain normal form through successive {\it refinements}\/ and ({\it additive}\/, {\it multiplicative}\/, and {\it compositional}\/) {\it conjugations}\/ of the hole~$(P,1,\hat f)$. This normalization
is done under more general algebraic assumptions in \cite[Section~3.3]{ADH4}.  The analytic arguments leading to fixed points are in Sections~\ref{sec:IHF}--\ref{sec:smoothness}.  Developments below involve the algebraic closure~$K:=H[\imag]$ of $H$ and we work more generally with a decomposition~$P=\tilde{P}_1-R$ where~${\tilde{P}_1\in K\{Y\}}$ is   homogeneous of degree $1$, not necessarily~$\tilde{P}_1=P_1$, such that~$L_{\tilde{P}_1}$ splits over $K$ and $R$ is ``small'' compared to $\tilde{P}_1$.

\subsubsection*{Passing to the complex realm} 
In general we are not so lucky that~$L$ splits over~$H$. The minimality of our hole $(P,1,\hat f)$ does not even ensure that $L$ splits over~$K$. At this point we recall from [ADH, 11.7.23] that $K$ is $\upo$-free because $H$ is.  
We can also draw hope from the fact that every nonzero linear differential operator over $K$ would split over $K$ if
$H$ were newtonian [ADH, 14.5.8].  Although $H$ is not newtonian, it is $(r-1)$-newtonian, and $L$ is only of order $r$, so
we optimistically restart our attempt, and instead of a hole of minimal complexity in $H$, we now let~$(P,\fm,\hat f)$ be a hole of minimal complexity in $K$. Again it follows that $r:=\order P\geq 1$,  $P$ is a minimal annihilator of $\hat f$ over $K$, and $K$ is $(r-1)$-newtonian.
As before we arrange that~$\fm=1$ and the linear part~$L_P\in K[\der]$ of $P$ has order~$r$. We can also arrange~$\hat f\in \hat K = \hat H[\imag]$ where~$\hat H$ is an immediate asymptotic extension of~$H$. So~$\hat f=\hat g+\hat h\imag$ where $\hat g,\hat h\in\hat H$ satisfy~$\hat g,\hat h\prec 1$, and~$\hat g\notin H$ or~$\hat h\notin H$, say~$\hat g\notin H$.
 Now minimality of $(P,1,\hat f)$ and algebraic closedness of $K$ give that~$K$ is $r$-linearly closed, that is, every nonzero~$A\in K[\der]$ of order~$\leq r$ splits over~$K$~\cite[Co\-rol\-la\-ry~3.2.4]{ADH4}. Then~$L_P$   splits over~$K$ as desired, and a   version of the above fixed point construction with $\c^r_a[\imag]$ in place of~$\c^r_a$ can be carried out successfully to solve~$P(Y)=0$ in the differential ring extension $\Calinf[\imag]$ of $\Calinf$.

\subsubsection*{Return to the real world} 
But at this point we face another obstacle: even once we have our hands on a zero~$f\in\Calinf[\imag]$ of $P$, it is not clear why~$g:=\Re f$ should generate a proper Hardy field extension of $H$: Let $Q$ be a minimal annihilator of~$\hat g$ over~$H$; we cannot expect that $Q(g)=0$. 
If $L_Q\in H[\der]$ splits over~$K$, then 
we can try to apply  fixed point arguments like the ones above, with~$(P,1,\hat f)$ replaced by the hole $(Q,1,\hat g)$ in~$H$, 
to  find a zero~$y\in\Calinf$ of~$Q$. (We do need to take care that constructed zero is real.)
Unfortunately we can only ascertain  that~$1\leq s\leq 2r$ for~$s:=\order Q$, and  since we may have~${s>r}$,  we cannot leverage the    minimality of~$(P,1,\hat f)$ anymore to ensure that $L_Q$ splits over~$K$, or to normalize~$(Q,1,\hat g)$ in the same way as 
indicated above for~$(P,1,\hat f)$. 
This situation seems hopeless, but 
now a purely differential-algebraic observation  comes to the rescue: although  the linear part~$L_{Q_{+\hat g}}\in\hat H[\der]$ of the differential polynomial~$Q_{+\hat g}\in\hat H\{Y\}$ also has order~$s$ (which may be $>r$),   {\it if $\hat K$ is $r$-linearly closed, then  $L_{Q_{+\hat g}}$ does split over~$\hat K$}; see~[ADH, 5.1.37]. If moreover~$g\in H$ is sufficiently close to~$\hat g$, then the linear part~$L_{Q_{+g}}\in H[\der]$ of~$Q_{+g}\in H\{Y\}$ is close to an operator in $H[\der]$ that does split
over~$K=H[\imag]$, and so
using~$(Q_{+g},1,\hat g-g)$ instead of~$(Q,1,\hat g)$  may offer a way out of this impasse.

\subsubsection*{Approximating $\hat g$}  Suppose for a moment that $H$ is (valuation) dense in $\hat H$. 
Then by extending $\hat H$ we arrange that $\hat H$ is the completion of $H$, and $\hat K$ of $K$ (as in~[ADH, 4.4]). In this case $\hat K$ inherits from $K$ the property of being $r$-linearly closed, by results in~\cite[Section~1.7]{ADH4}, and the desired approximation of $\hat g$ by $g\in H$ can be achieved. 
We cannot in general expect $H$ to be dense in $\hat H$. But we are saved by \cite[Section~1.5]{ADH4}  to the effect that $\hat g$ can be made {\it special}\/ over~$H$ in the sense of~[ADH, 3.4], that is, some nontrivial convex subgroup $\Delta$ of
the value group of~$H$ is cofinal in $v(\hat g-H)$. Then passing to the $\Delta$-specializations of the various
valued differential fields encountered above (see [ADH, 9.4]) we regain density and this allows us to implement the desired approximation. The technical details are  carried out in the first three sections in Part~4 of  \cite{ADH4}.  
A minor  obstacle to obtain the necessary specialness of~$\hat g$ is that the hole~$(Q,1,\hat g)$ in~$H$ may not be of minimal complexity. This can be  ameliorated by using a differential polynomial of minimal complexity vanishing at~$(H,\hat g)$ instead of~$Q$,   in the process replacing the hole~$(Q,1,\hat g)$ in $H$ by a  slot in~$H$, which we then aim to approximate   by  a {\it strongly split-normal}\/ slot in~$H$; see~\cite[Definition~4.3.21]{ADH4} (or Section~\ref{sec:prelims} below). 
Another caveat:  to carry out our approximation scheme  we require $\deg P>1$.  
Fortunately,  if $\deg P=1$, then necessarily $r=\order P=1$, and this case
can be dealt with through separate (non-trivial) arguments: see Section~\ref{sec:d-alg extensions} where we
finish the proof of Theorem~B.

\subsubsection*{Enlarging the Hardy field}
Now suppose we have finally arranged things so that our Fixed Point Theorem applies: it delivers $g\in\Calinf$ such that~$Q(g)=0$ and~$g\prec 1$.
(Notation: for a germ~$\phi\in\Calinf[\imag]$ and~$0\neq \fn\in H$  we write $\phi\prec \fn$ if $\phi(t)/\fn(t)\to 0$ as~$t\to+\infty$.)
However, in order that~$g$ generates a proper Hardy field extension~$H\langle g\rangle$ of $H$
isomorphic to~$H\langle\hat g\rangle$ by an isomorphism over~$H$ sending $g$ to $\hat g$ requires that $g$ and $\hat g$ have similar asymptotic properties with respect
to the elements of~$H$. For example, suppose~$h,\fn\in H$  and~$\hat g-h\prec\fn\preceq 1$; then we must show~$g-h\prec\fn$. 
(Of course, we need to show much more about the asymptotic behavior of $g$, and this is expressed using the notion of {\it asymptotic similarity}\/: see
Sections~\ref{sec:asymptotic similarity} and~\ref{sec:d-alg extensions}.)
Now the germ $(g-h)/\fn\in\Calinf$ is a zero of the conjugated differential polynomial~$Q_{+h,\times\fn}\in H\{Y\}$, as is the element~$(\hat g-h)/\fn\prec 1$ of~$\hat H$. The Fixed Point Theorem can also be used to  produce a zero~$y\prec 1$ of $Q_{+h,\times \fn}$ in $\Calinf$. Set $g_1:=y\fn+h$; then~$Q(g)=Q(g_1)=0$ and~$g,g_1\prec 1$. We are thus naturally led to consider the difference $g-g_1$ between
the solutions~$g,g_1\in\Calinf$ of the differential equation (with asymptotic side condition)
\begin{equation}\label{eq:Q=0}\tag{$\operatorname{E}$}
Q(Y)\ =\ 0,\qquad Y\prec 1.
\end{equation} 
If we manage to show~$g-g_1\prec\fn$, then $g-h=(g-g_1)-y\fn\prec\fn$ as required. Simple estimates coming out of the proof of the Fixed Point Theorem are not good enough for this (cf.~Lemma~\ref{lem:close}). 
We need a generalization of the Fixed Point Theorem
for {\it weighted norms}\/ with weight function given by a representative of $\fn$; this is what Section~\ref{sec:weights} is about.
To render this generalized version useful, we also have to construct a right-inverse~$A^{-1}$ of the linear differential operator~${A\in H[\der]}$ that depends in some sense uniformly on
 $\fn$.
This is also carried out in Section~\ref{sec:weights}, by refining
our approximation arguments, in the process improving strong split-normality to {\it strong repulsive-normality}\/ as defined in~\cite[4.5.32]{ADH4} (see also Section~\ref{sec:prelims}).

\subsubsection*{Exponential sums}
Just for this discussion, call  $\phi\in\Calinf[\imag]$ {\it small}\/ if~$\phi\prec\fn$ for all~${\fn\in H}$ with
$v\fn\in v(\hat g-H)$. Thus our aim is to show that differences between solutions of \eqref{eq:Q=0} in $\Calinf$ are small in this sense. 
We 
show that each such difference gives rise to a zero $z\in\Calinf[\imag]$  of~$A$ with $z\prec 1$
whose smallness would imply  the smallness of the difference under consideration. To ensure that every zero $z\prec 1$ of $A$
is indeed small requires us to have performed beforehand yet another (rather unproblematic) normalization procedure on our  slot, transforming it into {\it ultimate}\/ shape.
(See \cite[Section~4.4]{ADH4} or Section~\ref{sec:prelims} below.) 
Recall the special fundamental systems of solutions to linear differential equations over
maximal Hardy fields explained after Corollary~\ref{cor:factorization intro}: since $A$ splits over~$K$, our zero $z$ of $A$ is a $\C$-linear combination of exponential terms. In \cite{ADH4} we introduced a formalism to deal with such exponential sums over $K$, based on the concept of the {\it universal exponential extension}\/ of a differential field. This is interpreted analytically in Section~\ref{sec:ueeh}. 
From an asymptotic condition like~$z\prec 1$ we need to obtain asymptotic information about the summands of $z$
when expressed as an exponential sum.
For this we  exploit facts about uniform distribution mod~$1$ for germs in Hardy fields  due to Boshernitzan~\cite{BoshernitzanUniform}; see Sections~\ref{sec:udmod1} and \ref{sec:ueeh}.

\medskip
\noindent
This finishes the sketch of the proof of our main result in Sections~\ref{sec:prelims}--\ref{sec:d-alg extensions}. 
The remaining Sections~\ref{sec:transfer} and \ref{sec:embeddings into T} contain the applications of this theorem discussed earlier in this introduction:
In Section~\ref{sec:transfer}  we show how to transfer first-order logical properties of the differential field $\mathbb T$ of transseries to maximal Hardy fields, proving in particular Theorem~A and Corollaries~\ref{cor:elem equiv}--\ref{cor:systems, 2}, \ref{cor:zeros in complexified Hardy field extensions}, and~\ref{cor:factorization intro}, as well as
the first part of Corollary~\ref{cor:systems, 3}.
In Section~\ref{sec:embeddings into T} we then investigate embeddings
of Hardy fields into~$\mathbb T$,   finish the proof of Corollary~\ref{cor:systems, 3}, and determine the universal theory of Hardy fields, 
in the course establishing a conjecture from \cite{AvdD2}.

\subsection*{Previous work} This paper depends essentially on \cite{ADH4, ADH5} (and on  our book [ADH]).  Here are some earlier special cases of our results: 
Theorem~A for~$P$ of order~$1$ is in~\cite{D}. By \cite{vdH:hfsol} there exists a Hardy field $H\supseteq\R$ isomorphic as an ordered differential field to $\mathbb T_{\text{g}}$, so by \cite{JvdH} this~$H$ has the intermediate value property for all differential polynomials over it.
Boshernitzan~\cite[Remark on p.~117]{Boshernitzan87} states a consequence of Corollary~\ref{cor:zeros in complexified Hardy field extensions}: if $a_1$, $a_2$, $b$ are elements of a Hardy field $H$, then some~$y$ in a Hardy field extension of $H$ satisfies~$y''+a_1y'+a_2y=b$.

\subsection*{Acknowledgment} We thank the referee for helpful suggestions and corrections.

\subsection*{Notations and conventions}
We follow the conventions from [ADH]. Thus  $m$,~$n$   range over the set~$\N=\{0,1,2,\dots\}$ of natural numbers. Given an additively written abelian group $A$ we set $A^{\ne}:=A\setminus\{0\}$, and given a commutative ring $R$ (always with identity $1$)
we let $R^\times$ be the multiplicative group of units of $R$. (So if $K$ is a field, then~$K^{\neq}=K^\times$.)
If $R$ is a differential ring (by convention containing $\Q$ as a subring) and $y\in R^\times$, then $y^\dagger=y'/y$ denotes the logarithmic derivative of $y$, and $R^\dagger:=\{y^\dagger:\ y\in R^\times\}$, an additive subgroup of $R$. 
 The prefix ``$\d$'' abbreviates ``differentially''; for example, ``$\d$-algebraic'' means ``differentially algebraic''.
A differential polynomial~$P(Y)\in R\{Y\}=R[Y, Y',Y'',\dots]$ of order $\le r\in \N$ is often expressed as
$$ P\ =\  \sum_{\i}P_{\i}Y^{\i}, \qquad (Y^{\i}\ :=\ Y^{i_0}(Y')^{i_1}\cdots (Y^{(r)})^{i_r})$$
where $\i=(i_0,\dots, i_r)$ ranges over $\N^{1+r}$, with coefficients $P_{\i}\in R$, and $P_{\i}\ne 0$ for only finitely many $\i$.
Also $|\i|:=i_0+\cdots + i_r$, $\| \i\|:=i_1 + 2i_2+\cdots + ri_r$ for such $\i$. 
 
\section{Preliminaries}\label{sec:prelims}

\noindent
For ease of reference and  convenience of the reader we summarize in this section much of the asymptotic differential algebra from \cite{ADH4} that we need later.  

\subsection*{The universal exponential extension}
{\it In this subsection $K$ is a differential field with algebraically closed constant field $C$ and divisible group $K^\dagger$
 of logarithmic derivatives.}\/
(These conditions are satisfied if $K$ is an algebraically closed differential field.)
An {\it exponential extension}\/ of $K$ is   a differential ring extension~$R$ of~$K$ such that~$R=K[E]$ for some $E\subseteq R^\times$ with~${E^\dagger\subseteq K}$.
By~\cite[Section 2.2, especially Corollary~2.2.11]{ADH4}:

\begin{prop}
There is an exponential extension $\Univ$ of $K$ with $C_{\Univ}=C$ such that every exponential extension $R$ of $K$ with $C_R=C$
embeds into $\Univ$ over $K$; any two such exponential extensions of $K$ are isomorphic over $K$.
\end{prop}

\noindent
We call $\Univ$ the {\it universal exponential extension}\/ of $K$,   denoted by $\Univ_K$ if we want to stress the dependence on $K$.
Here is how we constructed $\Univ$ in \cite[Section~2.2]{ADH4}:
First, take a {\it complement\/} $\Lambda$ of $K^\dagger$,\label{p:complement} that is, a $\Q$-linear subspace of $K$ such that~$K=K^\dagger\oplus \Lambda$ (internal direct sum of $\Q$-linear subspaces of~$K$).  Below~$\lambda$ ranges over~$\Lambda$. Next, let~$\ex(\Lambda)$ be a multiplicatively written abelian group,
isomorphic to the additive subgroup $\Lambda$ of~$K$, 
with isomorphism $\lambda\mapsto \ex(\lambda)\colon \Lambda\to \ex(\Lambda)$, and 
let~$\Univ := K\big[\!\ex(\Lambda)\big]$ be
the group ring of $\ex(\Lambda)$ over $K$, an integral domain.   
As $K$-linear space,  
$\Univ = \bigoplus_\lambda K\ex(\lambda)$ (an internal direct sum of $K$-linear subspaces), 
so for every $f\in \Univ$ we have a unique family $(f_\lambda)$ in $K$ with $f_\lambda=0$ for all but finitely many $\lambda$ and
\begin{equation}\label{eq:flambda}
f \ =\  \sum_{\lambda} f_{\lambda}  \ex(\lambda). 
\end{equation}
We turn $\Univ$ into a differential ring extension of $K$
such that 
$\ex(\lambda)' = \lambda\ex(\lambda)$ for all $\lambda$.
Then $\Univ$ is the universal exponential extension of $K$; in fact, by \cite[Lemma~2.2.9]{ADH4}:

\begin{lemma}\label{lem:embed into U}
Let $R$ be an exponential extension of $K$ with $C_R=C$, and set~$\Lambda_0 := \Lambda \cap R^\dagger$, a subgroup of $\Lambda$. Then there is a morphism $K\big[\!\ex(\Lambda_0)\big]\to R$ of differential rings over $K$, and any such morphism is an isomorphism.
\end{lemma}

\noindent
We denote the differential fraction field of $\Univ$ by $\Omega=\Omega_K$; then $C_\Omega=C$ by \cite[remark before Lemma~2.2.7]{ADH4}. 
In particular, $C_{\Univ}=C$; moreover,  
$\Univ^\times=K^\times\ex(\Lambda)$ and thus~$\Univ^\dagger = K^\dagger+\Lambda = K$, by \cite[remark before Example~2.2.4]{ADH4}.
These properties also characterize~$\Univ$ up to isomorphism over $K$, by \cite[Corollary~2.2.10]{ADH4}:

\begin{cor}\label{corcharexp}
Every exponential extension $U$ of $K$ with $C_U = C$ and $K \subseteq U^\dagger$ is isomorphic to $\Univ$ over $K$.
\end{cor}

\noindent
Our interest in $\Univ$ has to do with factoring  linear differential operators in $K[\der]$: Let~$A\in K[\der]^{\neq}$ and~$r:=\order(A)$, and consider 
$\ker_{\Univ}A=\big\{f\in\Univ:A(f)=0\big\}$, a  $C$-linear subspace of $\Univ$.
By \cite[Lemma~2.4.1]{ADH4}, $\ker_{\Univ}A$ has a basis contained in~$\Univ^\times$. Moreover, $\dim_C\ker_\Univ A\leq r$,
and assuming that $K$ is $1$-linearly surjective when~$r\geq 2$,  we have  by~\cite[Corollary~2.4.8]{ADH4}:
\begin{equation}\label{eq:2.4.8}
\dim_C\ker_\Univ A=r\quad\Longleftrightarrow\quad \text{$A$ splits over $K$.}
\end{equation}
Let $v\colon K^\times\to\Gamma$  be a valuation on $K$.  Call the va\-lu\-a\-tion~$v_{\g}\colon \Univ^{\neq}\to\Gamma$ on~$\Univ$ such that~$v_{\g}f = \min_\lambda vf_\lambda$ for $f\in\Univ^{\neq}$ as in \eqref{eq:flambda}  the {\it gaussian extension} of~$v$   to~$\Univ$; cf.~\cite[Proposition~2.1.3]{ADH4}.
We
denote by~$\preceq_{\g}$ the dominance relation on $\Omega$ associated to the extension of~$v_{\g}$ to a valuation on $\Omega$, with corresponding asymptotic relations~$\asymp_{\g}$ and~$\prec_{\g}$;
 cf.~[ADH, (3.1.1)].
The valued differential field $K$ may be asymptotic with small derivation, while $\Omega$ with the above  valuation   has neither one of these properties~\cite[example before Lemma~2.5.1]{ADH4}.
If $K$ is $\d$-valued of $H$-type with small derivation and asymptotic integration, then $\abs{v_{\g}(\ker^{\neq}_{\Univ}A)}\leq \dim_C\ker_{\Univ} A$ by \cite[Lemma~2.5.11]{ADH4}, where $|S|$ is the cardinality of a set $S$. Next a variant of
 \eqref{eq:2.4.8} which follows from \cite[Corollary 2.4.18]{ADH4}:
 
 \medskip\noindent
\begin{lemma}~\label{lemvar12} If $H$ is a Liouville closed $H$-field, $r=2$, and $A\in H[\der]$ splits over~$K:=H[\imag]$, $\imag^2=-1$, then
$\dim_{C} \ker_{\Univ} A = r$.
\end{lemma}

\subsection*{The span of a linear differential operator}
Let $K$ be a valued differential field with small derivation, and
$A = a_0+a_1\der+\cdots+a_r\der^r\in K[\der]$ where $a_0,\dots,a_r\in K$,  $a_r\neq 0$.
The {\it span}\/ $\fv(A)$ of $A$  is defined as
$$\fv(A)\ :=\ a_r/a_m \in\ K^\times\qquad\text{where $m:=\dwt(A)$.}$$  
Note that $\fv(A)\preceq 1$. The next result is \cite[Lemma~3.1.1]{ADH4} and says that~$\fv(A)$ does not change much under small additive pertubations of $A$:

\begin{lemma}\label{lem:fv of perturbed op}
If $B\in K[\der]$, $\order(B)\leq r$ and $B\prec \fv(A) A$, then: \begin{enumerate}
\item[\textup{(i)}] $A+B\sim A$, $\dwm(A+B)=\dwm(A)$, and $\dwt(A+B)=\dwt(A)$;
\item[\textup{(ii)}] $\order(A+B)=r$ and $\fv(A+B) \sim \fv(A)$. 
\end{enumerate}
\end{lemma}

\noindent
The span of a linear differential operator serves as a ``yardstick'' for approximation arguments in \cite{ADH4}.
Another important use of $\fv(A)$ is in bounding the factors in a splitting of $A$ \cite[Corollary~3.1.6]{ADH4}.
To state this, let  $(g_1,\dots,g_r)\in K^r$ be a splitting of $A$ over $K$, that is,
$A=f(\der-g_1)\cdots(\der-g_r)$ where $f\in K^\times$.
Then
\begin{equation} \label{bound on linear factors}
g_1,\dots, g_r\ \preceq\ \fv(A)^{-1}.
\end{equation}  
Suppose $K=H[\imag]$ where $H$ is a real closed $H$-field and $\imag^2=-1$. Then we call 
the splitting $(g_1,\dots,g_r)$ of $A$ over $K$   {\it strong}\/ 
if $\Re g_1,\dots,\Re g_r\succeq \fv(A)^\dagger$.
We say that~$A$ {\it splits strongly}\/ over $K$ if it has a strong splitting   over~$K$;
see \cite[Section~4.2]{ADH4}.

\subsection*{Holes and slots}
{\it In this subsection $K$ is an $H$-asymptotic field with small derivation and rational asymptotic integration, and $\Gamma:=v(K^\times)$.}\/ We let $a$, $b$ range over~$K$ and $\fm$, $\fn$  over~$K^\times$.
A {\it hole}\/ in~$K$ is a triple $(P,\fm,\hat a)$  with $P\in K\{Y\}\setminus K$ and~$\hat a\in \hat{K}\setminus K$ for some immediate asymptotic extension
$\hat{K}$ of $K$, such that $\hat a \prec \fm$ and $P(\hat a)=0$.
A {\it slot} in~$K$ is a triple~$(P,\fm,\hat a)$ where~${P\in K\{Y\}\setminus K}$ and~$\hat a$ is an element of~$\hat K\setminus K$, for some immediate asymptotic extension~$\hat K$ of $K$, such that~$\hat a\prec\fm$ and~$P\in Z(K,\hat a)$. (See~[ADH, 11.4] for $Z(K,\hat a)$.) The {\it order}, {\it degree}, and {\it complexity} of a slot $(P,\fm,\hat a)$ in~$K$ are defined to be the order, degree, and complexity of the differential polynomial~$P$, respectively, and its {\it linear part}\/ is the linear part~$L_{P_{\times\fm}}\in K[\der]$ of $P_{\times\fm}$. Every hole in $K$ is a slot in $K$, by \cite[Lemma~3.2.2]{ADH4}. If $\phi$ is active in $K$ and $(P,\fm,\hat a)$ is a slot in $K$ (respectively, a hole in $K$), then $(P^\phi,\fm,\hat a)$ is a slot in~$K^\phi$ (respectively a hole in $K^\phi$) of the same complexity as $(P,\fm,\hat a)$.

A hole  in~$K$ is {\it minimal}\/ if no hole in $K$ has smaller complexity. 
A slot $(P,\fm,\hat a)$ in $K$ is {\it $Z$-minimal} if~$P$ is of minimal complexity among elements of~$Z(K,\hat a)$.
By~\cite[Lem\-ma~3.2.2]{ADH4}, minimal holes in $K$ are $Z$-minimal slots in $K$. 
Slots  $(P,\fm,\hat a)$ and~$(Q,\fn,\hat b)$ in~$K$ are said to be {\it equivalent}   if~$P=Q$, $\fm=\fn$, and~$v(\hat a-a)=v(\hat b-a)$ for all $a$. This is an equivalence relation on the class of slots in~$K$. Each  $Z$-minimal slot in $K$ is equivalent to a $Z$-minimal hole in $K$,  by~\cite[Lem\-ma~3.2.14]{ADH4}.

{\em In the rest of this subsection~$(P,\fm,\hat a)$ is a slot in $K$ of order $r\ge 1$}.
For $a$,~$\fn$ such that $\hat a-a\prec\fn\preceq\fm$ 
we obtain a slot $(P_{+a},\fn,{\hat a-a})$ in~$K$
of the same complexity as~$(P,\fm,\hat a)$, and
slots of this form are said to {\it refine}\/ $(P, \fm, \hat a)$ and are called {\it refinements}\/ of~$(P,\fm,\hat a)$.
Here is~\cite[Corollary~3.2.29]{ADH4}:

\begin{lemma}\label{find zero of P} 
Suppose  $K$ is  $\d$-valued and $\upo$-free, $\Gamma$ is divisible, $L$ is a newtonian
$H$-asymptotic extension of $K$, and $(P,\fm,\hat a)$ is  $Z$-minimal.
Then there exists~${\hat b\in L}$ such that $K\<\hat b\>$ is an immediate extension of $K$ and~$(P,\fm,\hat b)$ is a hole in $K$ equivalent to~$(P,\fm,\hat a)$. If  $(P,\fm,\hat a)$ is also a  hole in $K$,
then there is an embedding~$K\langle \hat a\rangle\to L$ of valued differential fields over $K$.
\end{lemma}

\begin{remark}
Dropping the assumptions on $K$, $\Gamma$ in the beginning of Lemma~\ref{find zero of P}, let $L$ be a newtonian
$H$-asymptotic extension of $K$ and let $(P,\fm,\hat a)$ be  $Z$-minimal of order~$1$ and degree~$1$.
Then the conclusions of Lemma~\ref{find zero of P} hold, by~\cite[Corollary~3.2.30]{ADH4} and the fact that $\deg P=1$ yields quasilinearity of $(P,\fm, \hat a)$ as defined below. 
\end{remark}

\noindent
Set $w:=\wt(P)$, and
if $\order(L_{P_{\times\fm}})=r$, set~$\fv:=\fv(L_{P_{\times\fm}})$.
We call $(P,\fm,\hat a)$
\begin{enumerate}
\item {\it quasilinear} if $\ndeg P_{\times\fm}=1$;
\item {\it special} if  some nontrivial convex subgroup   of $\Gamma$ is cofinal in~$v\big(\frac{\hat{a}}{\fm}-K\big)$;
\item {\it steep} if  $\order(L_{P_{\times \fm}})=r$ and $\fv \prec^\flat 1$; and
\item {\it deep} if it is steep and for all active $\phi\preceq 1$ in $K$, we have $\ddeg S_{P^\phi_{\times\fm}}=0$ (hence $\ndeg S_{P_{\times\fm}}=0$)  and  $\ddeg P^\phi_{\times\fm}=1$ (hence $\ndeg P_{\times \fm}=1$).
\end{enumerate}
(Here, 
$S_{Q}\in K\{Y\}$ denotes the separant of a differential polynomial $Q\in K\{Y\}$.)
From~\cite[Lem\-ma~3.2.36]{ADH4} we recall a way to obtain special holes in $K$:

{\sloppy
\begin{lemma}\label{lem:special dents} 
Suppose $K$ is $r$-linearly newtonian, and $\upo$-free if~$r>1$. If~$(P,\fm, \hat a)$ is  quasilinear, and  $Z$-minimal or a hole in $K$, then  $(P,\fm, \hat a)$ is special.
\end{lemma}}

\noindent
Next some important approximation properties of special $Z$-minimal  slots in $K$ (cf.~\cite[Lemma~3.2.37 and Corollary~3.3.15]{ADH4}):

\begin{prop}\label{small P(a)}
With $\fm=1$, if $(P,1,\hat a)$ is special and $Z$-minimal, and $\hat a-a\preceq\fn\prec 1$ for some $a$,
then $\hat a-b\prec\fn^{r+1}$ for some $b$, and $P(b)\prec\fn P$ for any such~$b$.
\end{prop}

\begin{prop}\label{specialvariant} If
  $(P,\fm,\hat a)$ is deep,  special, and $Z$-minimal, then for all~$n\ge 1$ there is an~$a$ with $\hat a-a\prec \fv^n\fm$, where $\fv:=\fv(L_{P_{\times\fm}})$.
\end{prop} 

\noindent
For  $Q\in K\{Y\}$ we denote by $Q_d$  the homogeneous part of degree $d\in\N$ of~$Q$,
and we set~$Q_{>1}:=\sum_{d>1}Q_d$, and likewise with $\ge$ or $\ne$ in place of $>$.   
We  say that~$(P,\fm,\hat a)$ is {\it normal} if it is steep and $(P_{\times\fm})_{> 1}\prec_{\Delta(\fv)} \fv^{w+1} (P_{\times\fm})_1$, and
that~$(P,\fm,\hat a)$ is {\it strictly normal}\/ if  it is steep and 
$(P_{\times\fm})_{\neq 1}\prec_{\Delta(\fv)} \fv^{w+1} (P_{\times\fm})_1$.  
Here and below,  $\Delta(\fv):=\big\{\gamma\in\Gamma:\gamma=o(v\fv)\big\}$, a convex subgroup of $\Gamma$ (cf.~\cite[Part~3]{ADH4}).

\subsection*{Normalizing minimal holes}
{\it In the rest of this section \textup{(}thus including later subsections\textup{)}
$H$ is a real closed $H$-field with small derivation and  asymptotic integration. Let  $C:= C_H$ be its constant field and
$\Gamma:=v(H^\times)$ its value group, with~$\gamma$ ranging over $\Gamma$}.
We also let~$\hat H$ be an immediate asymptotic extension of~$H$ and~$\imag$ an element  of an asymptotic extension of $\hat H$ with $\imag^2=-1$. 
Then~$\hat H$ is also an $H$-field (by [ADH, 10.5.8])  and $\imag\notin\hat H$. Moreover, the $\d$-valued field~$K:=H[\imag]$ is an algebraic closure of~$H$
with~$C_K=C[\imag]$,
and~$\hat K:=\hat H[\imag]$ is an immediate asymptotic extension of~$K$.

Lemma~4.2.15 in \cite{ADH4} complements Lemma~\ref{find zero of P}: if~$H$ is $\upo$-free, then every $Z$-minimal slot in $K$ of positive order  is equivalent to a hole
$(P,\fm,\hat a)$  in $K$ such that~${\hat a\in\hat K\setminus K}$ for a suitable choice of $\upo$-free~$\hat H$ as above.

For the next two results (Lemmas~4.3.31 and 4.3.32 in~\cite{ADH4})  we assume
that~$H$ is $\upo$-free and $(P,\fm,\hat a)$ is a minimal hole in 
$K$ of positive order, with $\fm\in H^\times$ and~$\hat a \in\hat K\setminus K$.
We let $a$ range over $K$ and $\fn$ over $H^\times$.

\begin{lemma}\label{lem:achieve strong splitting}  
For some refinement $(P_{+a},\fn,\hat a-a)$  of $(P,\fm,\hat a)$ and active $\phi$ in~$H$ with $0<\phi\preceq 1$,
the hole $(P^\phi_{+a},\fn,\hat a-a)$ in $K^\phi$ is deep and  normal, its  linear part splits 
strongly over $K^\phi$, and it is moreover strictly normal if $\deg P>1$. 
\end{lemma}

\noindent
With further hypotheses on $K$ we can also  achieve strict normality  when $\deg P=1$: 

\begin{lemma}\label{lem:achieve strong splitting, d=1}
If $\der(K)=K$ and $\I(K)\subseteq K^\dagger$, then some refinement ${(P_{+a},\fn,\hat a-a)}$  of $(P,\fm,\hat a)$ and active $\phi$ in~$H$ with~$0<\phi\preceq 1$ are such that
the hole~$(P^\phi_{+a},\fn,\hat a-a)$ in $K^\phi$ is deep and strictly normal, and its  linear part splits 
strongly over $K^\phi$. 
\end{lemma}

\begin{remarkNumbered}\label{rem:achieve strong splitting, d=1} Instead of assuming that 
$H$ is $\upo$-free and $(P,\fm,\hat a)$ is a minimal hole in 
$K$ of positive order, assume that $H$ is $\upl$-free and $(P, \fm, \hat a)$ is a slot in $K$ of order and degree~$1$.
Then Lemma~\ref{lem:achieve strong splitting, d=1} goes through 
with ``hole'' replaced by ``slot''; cf.~\cite[remark after Lemma~4.3.32]{ADH4}.
\end{remarkNumbered}

\subsection*{Ultimate exceptional values, ultimate slots}
{\it Let $H$ be a Liouville closed $H$-field with small derivation.}\/
By the discussion at the beginning of Section~4.4 of \cite{ADH4}, 
the subspace $K^\dagger$ of the $\Q$-linear space $K$ 
has a complement~$\Lambda$  such that  $\Lambda\subseteq H\imag$. We  fix
such $\Lambda$ and let~$\Univ:= K\big[\!\ex(\Lambda)\big]$ be the universal exponential extension of $K$ as above, with differential fraction field~$\Omega:=\Frac(\Univ)$. If $\I(K)\subseteq K^\dagger$, then we may additionally
 choose $\Lambda=\Lambda_H\imag$ where~$\Lambda_H$ is a complement of the subspace~$\I(H)$ of the
$C$-linear space $H$; see \cite[remarks before Lemma~4.4.5]{ADH4}.
Let $A\in K[\der]^{\neq}$ and~$r:=\order(A)$. 
For each~$\lambda$ the operator $$A_\lambda\ :=\ A_{\ltimes\ex(\lambda)}\ =\ \ex(-\lambda)A\ex(\lambda)\in \Omega[\der]$$ has coefficients in $K$, by~[ADH, 5.8.8]. 
We call the elements of the set  
$$\exc^{\operatorname{u}}(A)\ =\ \exc^{\operatorname{u}}_{K}(A)\ :=\ \bigcup_\lambda\, \exc^{\ev}(A_\lambda) \ \subseteq\ \Gamma$$
the {\it ultimate exceptional values of $A$}\/ with respect to $\Lambda$. Thus $\exc^{\ev}(A)\subseteq \exc^{\operatorname{u}}(A)$, and by \cite[(2.5.2)]{ADH4}, we have
$v_{\g}(\ker_{\Univ}^{\neq} A) \subseteq \exc^{\operatorname{u}}(A)$.

{\it In the rest of this subsection we assume $\I(K)\subseteq K^\dagger$.}\/ 
(By \cite[Proposition~1.7.28]{ADH4}, this holds if $K$ is $1$-linearly newtonian.)
Then $\exc^{\operatorname{u}}(A)$ does not depend on our choice of $\Lambda$ \cite[Corollary~4.4.1]{ADH4}. 
Moreover, by \cite[Lem\-ma~4.4.4]{ADH4} we have $\abs{\exc^{\operatorname{u}}(A)} \leq r$, and 
\begin{equation}\label{eq:vkerunivA vs excuA}
\dim_{C[\imag]} \ker_{\Univ} A=r\ \Longrightarrow\ 
v_{\g}(\ker_{\Univ}^{\neq} A)=\exc^{\operatorname{u}}(A).
\end{equation}
Let $(P,\fm,\hat a)$ be a slot in $H$  of order~$r\geq 1$ in $H$, where $\hat a\in \hat H\setminus H$. We call $(P, \fm,\hat a)$ {\it ultimate}\/ if for all $a\prec\fm$ in $H$, 
$$\order(L_{P_{+a}})=r\ \text{ and }\ \exc^{\operatorname{u}}(L_{P_{+a}}) \cap v(\hat a-H)\ <\  v(\hat a-a).$$
Sometimes, this ultimate condition can be simplified: by \cite[Lemmas~4.4.12, 4.4.13]{ADH4}, if $(P,\fm,\hat a)$ is normal or
$\deg P =1$, then
\begin{equation}\label{eq:ultimate normal linear} \text{$(P,\fm,\hat a)$  is ultimate} \quad\Longleftrightarrow\quad
\exc^{\operatorname{u}}(L_P) \cap v(\hat a-H) \leq v\fm.
\end{equation}
Similarly, we call a slot $(P,\fm,\hat a)$ of order $r\geq 1$ in $K$, where  $\hat a\in \hat K\setminus K$,
{\it ultimate}\/ if for all $a\prec\fm$ in $K$ we have
$$\order(L_{P_{+a}})=r\ \text{ and }\ \exc^{\operatorname{u}}(L_{P_{+a}}) \cap v(\hat a-K)\ <\  v(\hat a-a).$$
Every refinement of an ultimate slot in $H$ remains ultimate, and likewise with $K$ in place of $H$
 \cite[Lemma~4.4.10 and remarks after Lemma~4.4.17]{ADH4}.
By \cite[Propositions~4.4.14, 4.4.18, and Remarks~4.4.15, 4.4.19]{ADH4} we have:

\begin{prop}\label{prop:achieve ultimate} 
Let $(P,\fm,\hat a)$ with $\hat a\in \hat H\setminus H$ be a slot in $H$ of positive order. If $(P, \fm, \hat a)$ is
normal or $\deg P=1$, then $(P, \fm, \hat a)$  has an ultimate refinement. Likewise, for any slot $(P,\fm,\hat a)$ in $K$ of positive order with ${\hat a\in \hat K\setminus K}$. 
\end{prop}

\subsection*{Repulsive splitting}
In this subsection $f$ ranges over $K$ and $\fm$   over $H^\times$.  

\begin{definition}\label{def:repulsive}
As in \cite[Section 4.5]{ADH4} we say that $f$ is {\it attractive}\/ if~$\Re f\succeq 1$ and~$\Re f<0$, and {\it repulsive}\/ if~$\Re f\succeq 1$ and $\Re f>0$.
Given $\gamma>0$ we also say that $f$ is {\it $\gamma$-repulsive}\/ if $\Re f>0$ or $\Re f\succ\fm^\dagger$ for all
$\fm$ with $\gamma=v\fm$.
\end{definition}

\noindent
The following is \cite[Corollary~4.5.5]{ADH4}:

\begin{lemma}\label{lem:repulsive}
Suppose $f$ is $\gamma$-repulsive where $\gamma=v\fm>0$,  and $\Re f\succeq 1$. Then~$f$ is repulsive iff $f-\fm^\dagger$ is repulsive, and $f$ is attractive iff $f-\fm^\dagger$ is attractive.
\end{lemma}

\noindent
Given $\hat a\in \hat H\setminus H$, we also say that $f$ is {\it $\hat a$-repulsive}\/ if it is $\gamma$-repulsive for each~$\gamma\in v(\hat a-H)\cap\Gamma^>$, that is,   $\Re f>0$ or $\Re f \succ \fm^\dagger$ for all~$a\in H$  and $\fm$
with~$\fm\asymp\hat a-a\prec 1$. 
Let also $A\in K[\der]^{\neq}$ have order $r\ge 1$.  A  splitting 
 $(g_1,\dots,g_r)$  of $A$ over $K$ is
{\it $\hat a$-repulsive}\/ if $g_1,\dots,g_r$ are $\hat a$-repulsive. If there is an  $\hat a$-repulsive splitting of $A$ over~$K$, then $A$ is said to
  {\it split  $\hat a$-repulsively}\/ over $K$.
See \cite[Section~4.5]{ADH4} for more about these notions,
and Sections~\ref{sec:IHF} and~\ref{sec:weights} below for their analytic significance.

\subsection*{Split-normal and repulsive-normal slots}
Let $(P, \fm, \hat a)$ be a steep slot in $H$ of order $r\ge 1$  with~$\hat a\in \hat H\setminus H$. Set~$L:=L_{P_{\times \fm}}$,  $\fv:=\fv(L)$, and $w:=\wt(P)$. Note that for~$Q\in K\{Y\}$ we have $Q_{\ge 1}=Q-Q(0)$. 
We say that~$(P,\fm,\hat a)$ is
\begin{enumerate}
\item {\it split-normal} if  $(P_{\times\fm})_{\geq 1}=Q+R$ where $Q, R\in H\{Y\}$, $Q$ is homogeneous of degree~$1$ and order~$r$,  $L_Q$ splits over $K$, and $R\prec_{\Delta(\fv)} \fv^{w+1} (P_{\times\fm})_1$;
\item
 {\it strongly split-normal}
 if   
$P_{\times\fm}=Q+R$, $Q,R\in H\{Y\}$, $Q$ homogeneous of degree~$1$ and order~$r$, 
$L_Q$   splits strongly over $K$, and $R\prec_{\Delta(\fv)} \fv^{w+1} (P_{\times\fm})_1$;
\item {\it repulsive-normal} if (1) holds with
``splits $\hat a/\fm$-repulsively over $K$'' in place of ``splits over $K$''; and
\item  {\it strongly repulsive-normal} if  (2) holds with ``has a strong $\hat a/\fm$-repulsive splitting over $K$'' in place of
``splits strongly over $K$''.
\end{enumerate}
The following diagram summarizes some dependencies among these properties:
$$
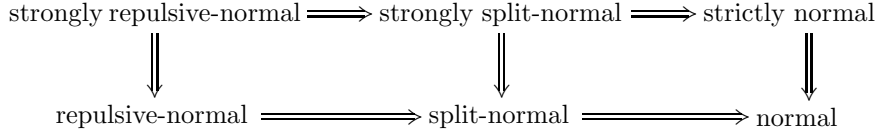
\xymatrix@L=6pt{	\parbox{11em}{strongly repulsive-normal} \ar@{=>}[r] \ar@{=>}[d]&      \parbox{9.25em}{strongly split-normal} \ar@{=>}[d]   \ar@{=>}[r] &  \parbox{8em}{strictly normal} \ar@{=>}[d] \\
\parbox{7.5em}{repulsive-normal} \ar@{=>}[r] &   \parbox{5.5em}{split-normal} \ar@{=>}[r]  & 
\parbox{4em}{normal}  }  $$
Here is the Normalization Theorem from \cite{ADH4} that is crucial in Section~\ref{sec:d-alg extensions}:

\begin{theorem} \label{c4543} 
Suppose $H$ is
$\upo$-free,  Liouville closed, with archimedean  ordered constant field $C$, and $1$-linearly newtonian algebraic closure $K= H [\imag]$. If $H$ is not newtonian, then
for some $Z$-minimal special hole $(Q, 1, \hat b)$ in~$H$ with $\order Q\geq 1$ and some active~$\phi > 0$ in~$H$ with $\phi \preceq 1$, the hole
$(Q^{\phi}, 1, \hat b)$ in $H^{\phi}$ is deep, strongly repulsive-normal,
and ultimate.
\end{theorem}

\section{Hardy Fields}\label{sec:hardy}

\noindent
In the rest of this paper we  freely use the notations about  functions and germs introduced in~\cite[Sections~2 and~3]{ADH5}. We recall some of it in this section, before stating   basic extension theorems about Hardy fields needed later.
We finish this section with a general fact about bounding  the derivatives of solutions to linear differential equations,
based on~\cite{Esc, Landau}.

\subsection*{Functions and germs}
Let $a$ range over $\R$ and $r$ over $\N\cup\{\omega,\infty\}$. Then
 $\c_a^r$ denotes the  $\R$-algebra  of functions $[a,+\infty)\to\R$ which extend to a $\c^r$-function~$U\to\R$ for some open $U\supseteq [a,+\infty)$. Here, as usual,~$\c^\omega$  means ``analytic''. Hence
 $$ \c_a\ :=\ \c_a^0 \ \supseteq \c_a^1 \ \supseteq\ \c_a^2  \ \supseteq\  \cdots  \ \supseteq\ \c_a^\infty  \ \supseteq\ \c_a^\omega.$$
Each subring $\c_a^r$ of $\c$ has its complexification $\c_a^r[\imag]=\c_a^r+\c_a^r\imag$, a subalgebra of the $\C$-algebra~$\c_a[\imag]$
of continuous functions $[a,+\infty)\to\C$. For $f\in \c_a[\imag]$ we have 
$$\bar{f}\ :=\ \Re f - \imag \Im f\ \in\  \c_a[\imag],\qquad |f|\ :=\ \sqrt{(\Re f)^2+(\Im f)^2}\ \in\ \c_a.$$  Let $\c^r$ be the $\R$-algebra of germs at $+\infty$ of functions in $\bigcup_a\c_a^r$. Thus~$\c:=\c^0$ consists of the germs
at $+\infty$ of continuous $\R$-valued functions on in\-ter\-vals~$[a,+\infty)$, $a\in\R$, and
 $$ \c\ =\ \c^0 \ \supseteq \c^1 \ \supseteq\ \c^2  \ \supseteq\  \cdots  \ \supseteq\ \c^\infty  \ \supseteq\ \c^\omega.$$
For each $r$ we also have the  $\C$-subalgebra  $\Gr[\imag]=\Gr+\Gr\imag$ of $\c[\imag]$.
For~$n\ge 1$ we have a derivation $g\mapsto g'\colon \Gn[\imag]\to\c^{n-1}[\imag]$
such that $\text{(germ of $f$)}'=\text{(germ of $f'$)}$ for $f\in\bigcup_a \Can[\imag]$, and
this derivation restricts to a derivation  $\Gn\to\c^{n-1}$.  
Therefore~$\Gi[\imag]  := \bigcap_{n}\, \Gn[\imag]$ 
is naturally a differential ring with ring of constants $\C$, and
$\Gi  := \bigcap_{n}\, \Gn$
is a differential subring of $\Gi[\imag]$ with ring of constants $\R$.
Note that~$\Gi[\imag]$ has~$\Ginf[\imag]$ as a differential subring, $\Gi$ has $\Ginf$ as a differential subring,
and the differential ring $\Ginf$ has in turn
the differential subring~$\Gom$.

\subsection*{Asymptotic relations} We often use the same notation for a $\C$-valued function 
on a subset of~$\R$ containing an interval $(a, +\infty)$, $a\in \R$,  as for its germ if the resulting ambiguity is harmless.
With this convention, given a property~$P$ of complex numbers
and $g\in \c[\imag]$ we say that {\em $P\big(g(t)\big)$ holds eventually\/} if~$P\big(g(t)\big)$ holds for all sufficiently large real $t$.
We equip $\c$ with the partial ordering given by~$f\leq g:\Leftrightarrow f(t)\leq g(t)$, eventually,
and equip $\c[\imag]$ with the asymptotic relations~$\preceq$,~$\prec$,~$\sim$ defined as follows: for $f,g\in \c[\imag]$,
\begin{align*} f\preceq g\quad &:\Longleftrightarrow\quad \text{there exists $c\in \R^{>}$ such that $|f|\le c|g|$,}\\
f\prec g\quad &:\Longleftrightarrow\quad \text{$g\in \c[\imag]^\times$ and $\lim_{t\to \infty} f(t)/g(t)=0$} \\
 &\phantom{:} \Longleftrightarrow\quad \text{$g\in \c[\imag]^\times$ and $\abs{f}\leq c\abs{g}$ for all $c\in\R^>$},\\
f\sim g\quad &:\Longleftrightarrow\quad \text{$g\in \c[\imag]^\times$ and
$\lim_{t\to \infty} f(t)/g(t)=1$}\\ 
\quad&\phantom{:} \Longleftrightarrow\quad f-g\prec g.
\end{align*}
We also use these notations for   functions in $\c_a$ ($a\in\R$); for example, for $f\in\c_a[\imag]$ and $g\in \c_b[\imag]$~($a,b\in \R$), $f\preceq g$ means: $\text{(germ of $f$)}\preceq \text{(germ of $g$)}$. 

Let now $H$ be a {\it Hausdorff field}\/, that is,
a subring of $\mathcal C$ that happens to be a field. 
Then the restriction of the partial ordering of $\c$ to $H$ is total and makes $H$ into an ordered field.
The ordered field $H$ has a convex subring 
$\mathcal{O}:=\{f\in H:f\preceq 1\}$,
which is a valuation ring of $H$, and we
consider $H$ accordingly as a valued ordered field. 
Moreover, $H[\imag]$ is a subfield of $\c[\imag]$, and
$\mathcal{O}+\mathcal{O}\imag = \big\{f\in H[\imag]: f\preceq 1\big\}$ is the unique valuation ring of $H[\imag]$ whose intersection with $H$ is $\mathcal{O}$. In this way we consider
$H[\imag]$ as a valued field extension of $H$. 
The asymptotic relations~$\preceq$,~$\prec$,~$\sim$ on $\c[\imag]$
restricted to~$H[\imag]$ are exactly the asymptotic relations $\preceq$, $\prec$, $\sim$ on $H[\imag]$ that~$H[\imag]$ has as a valued field 
(cf.~[ADH, (3.1.1)]); likewise with $H$ in place of $H[\imag]$.

\subsection*{Hardy fields}
In this subsection $H$ is a   {\it Hardy field}\/:  a differential subfield of~$\Gi$. 
A germ $y\in\c$ is said to be {\it hardian} if it lies in a Hardy field, and {\it $H$-har\-dian} (or {\it hardian over $H$}) if it lies in a Hardy field extension
of~$H$. (Thus $y$ is hardian iff~$y$ is $\Q$-hardian.)
Every Hardy field is a Hausdorff field, and
we consider~$H$  as an ordered valued differential field with the ordering and valuation on~$H$ as above.
Hardy fields are pre-$H$-fields, and $H$-fields if they contain~$\R$.
We also equip the differential subfield $H[\imag]$  of~$\Calinf[\imag]$ with the unique valuation ring 
lying over that of~$H$. Then $H[\imag]$ is a pre-$\d$-valued field of $H$-type with small derivation, and
if~$H\supseteq\R$, then $H[\imag]$ is $\d$-valued   with constant field $\C$.

Recall that $H$ is said to be {\it maximal}\/ if it has no proper Hardy field extension, and that every Hardy field has a 
maximal Hardy field extension. Due to our focus on $\d$-algebraic Hardy field extensions
in this paper, a weaker condition is often more natural for our purposes: we say that $H$ is
{\it $\d$-maximal}\/ if 
it has no proper $\d$-algebraic Hardy field extension.   Zorn yields a
$\d$-maximal $\d$-algebraic Hardy field extension of $H$, hence the intersection $\Dx(H)$ of all $\d$-maximal Hardy fields containing $H$  is a $\d$-algebraic Hardy field extension of~$H$, called the {\it $\d$-perfect hull}\/ of $H$. 
If $\Dx(H)=H$, then we  say that $H$ is  {\it $\d$-perfect}.
So  
$$\text{maximal}\quad \Longrightarrow\quad\text{$\d$-maximal}\quad\Longrightarrow\quad\text{$\d$-perfect.}$$
The following fact  summarizes some well-known  extension theorems from \cite{Boshernitzan81,Bou,Robinson72,Ros,Sj} (see also \cite[Proposition~4.2]{ADH5}):

\begin{prop}\label{prop:Li(H(R))}
If $H$ is $\d$-perfect, then $H$   is real closed with $H\supseteq\R$,
and   for each $f\in H$ we have $\exp f\in H$ and $g'=f$ for some $g\in H$.
\end{prop}

\noindent
Hence every $\d$-perfect Hardy field  is a Liouville closed $H$-field with constant field~$\R$.  For the following result, see \cite[Corollary~6.12]{ADH5}:

\begin{prop}\label{prop:I(K)}
If $H$ is $\d$-perfect, then  $K=H[\imag]$ satisfies $\I(K)\subseteq K^\dagger$.
\end{prop}

\noindent
In  \cite{ADH6} we    give an example of a $\d$-perfect $H$ which is not $\upo$-free. 
However, if~$H$ is $\d$-maximal, then $H$ is $\upo$-free; this is a consequence of the following   result from~\cite{ADH5}:

\begin{theorem}\label{upo} Every Hardy field has a $\d$-algebraic $\upo$-free Hardy field extension.
\end{theorem} 

\noindent
Recall from the introduction that $H$ is called a {\it $\Ginf$-Hardy field,}\/  if $H\subseteq\Ginf$, and that  a $\Ginf$-Hardy field is called {\it $\Ginf$-maximal}\/ if it has no proper $\Ginf$-Hardy field extension; likewise we defined $\Gom$-Hardy fields and $\Gom$-maximal Hardy fields.
Theorem~\ref{upo} also holds  with ``Hardy field'' replaced by ``$\Ginf$-Hardy field'',  as well as by
``$\Gom$-Hardy field''. This follows from its proof in \cite{ADH5}, and also from general results about
smoothness of solutions to algebraic differential equations over $H$ in 
Section~\ref{sec:smoothness} below.  By these results
every $\Ginf$-maximal Hardy field  is $\d$-maximal, so
if~${H\subseteq\Ginf}$, then~$\Dx(H)\subseteq\Ginf$;
likewise with $\Gom$ in place of $\Ginf$.  (See Corollary~\ref{cor:Hardy field ext smooth}.)

\subsection*{Bounding solutions of linear differential equations} Let   $r\in\N^{\geq 1}$, and with~$\i$ ranging over~$\N^{r}$, let  
$$P\  =\ P(Y,Y',\dots,Y^{(r-1)})\ =\ \sum_{\dabs{\i}<r} P_{\i}Y^{\i}\ \in\ \c[\imag]\big[Y,Y',\dots,Y^{(r-1)}\big]$$ 
with $P_{\i}\in\c[\imag]$ for all $\i$ with $\dabs{\i}<r$, and $P_{\i}\ne 0$ for only finitely many such $\i$.
Then~$P$ gives rise to an   evaluation map
$$y  \mapsto P\big(y,y',\dots,y^{(r-1)}\big)\ :\  \mathcal{C}^{r-1}[\imag] \to \c[\imag].$$ 
Let $y\in\c^r[\imag]$ satisfy the differential equation
\begin{equation}\label{eq:Landau}
y^{(r)}\ =\ P\big(y,y',\dots,y^{(r-1)}\big).
\end{equation}
In addition, $\fm$ with $0<\fm \preceq 1$ is a hardian germ, and
$\eta\in\c$ is eventually increasing with~$\eta(t)>0$ eventually, and $n\geq r$.

\begin{prop}\label{prop:EL}
Suppose $P_{\i}\preceq\eta$ for all $\i$, $P(0)\preceq \eta\,\fm^n$, and $y\preceq \fm^n$. 
Then
$$y^{(j)}\ \preceq\ \eta^j\fm^{n-j(1+\varepsilon)}\qquad\text{  for $j=0,\dots,r$ and  all $\varepsilon\in\R^>$,}$$
with $\prec$ in place of $\preceq$ if $y\prec \fm^n$ and $P(0)\prec\eta\,\fm^n$.
\end{prop}

\noindent
Proposition~\ref{prop:EL} for $\fm\asymp 1$ is covered by the following result:

\begin{theorem}[{Landau \cite{Landau}}]\label{thm:Landau}
Suppose $y\preceq 1$ and $P_{\i}\preceq\eta$ for all $\i$. Then $y^{(j)}\preceq \eta^j$ for $j=0,\dots,r$.
Moreover, if   $y\prec 1$, then
$y^{(j)} \prec \eta^j$ for $j=0,\dots,r-1$, and
if in addition $P(0)\prec\eta$, then  also
$y^{(r)} \prec \eta^r$.
\end{theorem}

\noindent
For convenience of the reader we include a proof of this theorem using the following result of Hardy-Littlewood  \cite[Theorem~3]{HaLi}. (A short proof of the case $n=1$ is in~\cite{Mordell},
and \cite[pp.~422--423]{HaLi} shows how the general case follows from this case.)

\begin{lemma}\label{lem:HL, n}
Let $a\in\R$ and
$\phi,\psi\colon [a,+\infty)\to (0,+\infty)$ be  continuous and increasing, and
let  $f\in\c^n_a[\imag]$, $n\geq 1$, be
such that $f\preceq\phi$, $f^{(n)}\preceq\psi$. Then for~$j=0,\dots,n$ we have~$f^{(j)}\preceq \phi^{1-j/n}\psi^{j/n}$.
If additionally $f\prec\phi$ or $f^{(n)}\prec\psi$, then $f^{(j)}\prec  \phi^{1-j/n}\psi^{j/n}$  for~$j=1,\dots,n-1$.
\end{lemma}

\begin{proof}[Proof of Theorem~\ref{thm:Landau}]
Take $a\in\R$ such that $\eta$ is represented by an increasing continuous function~$\eta\colon [a,+\infty)\to (0,+\infty)$, and
$y$ by a function $y\in \c_a^r[\imag]$. Then
$$t\mapsto \psi(t)\  :=\  \max\left( 1, \max_{a\leq s\leq t} \abs{y^{(r)}(s)} \right)\ \colon\  [a,+\infty)\to [1,+\infty)$$ 
is continuous and increasing with $\abs{y^{(r)}} \leq \psi$.
By Lemma~\ref{lem:HL, n} we have 
$y^{(j)}\preceq\psi^{j/r}$ for $j=0,\dots,r-1$ and thus
$P_{\i}y^{\i} \preceq \eta \psi^{\dabs{\i}/r}\preceq\eta \psi^{1-1/r}$ if $\dabs{\i}<r$.
So~$y^{(r)}=P\big(y,\dots,y^{(r-1)}\big) \preceq  \eta\psi^{1-1/r}$.
Take $C\in\R^{>}$ such that
$$\abs{y^{(r)}(t)}\  \leq\  C\eta(t)\psi(t)^{1-1/r}\qquad\text{for  all  $t\geq a$.}$$
Increasing $C$ we arrange $C\eta(a)\psi(a)^{1-1/r}\geq 1$.
As  $\eta\psi^{1-1/r}$ is   increasing, 
$$\psi(t)\  \leq\  
\max\left( 1, \max_{a\leq s\leq t} C\eta(s)\psi(s)^{1-1/r}  \right)\
 \leq\  C\eta(t)\psi(t)^{1-1/r}\qquad\text{for   $t\geq a$.}$$
Hence~$\abs{y^{(r)}(t)} \leq \psi(t)\leq C^r\eta^r(t)$ for $t\geq a$, so $y^{(r)}\preceq\eta^r$.
By  Lemma~\ref{lem:HL, n} again this yields~$y^{(j)}\preceq \eta^j$ for $j=0,\dots,r$. Assume now that~$y\prec 1$. 
Then by that same lemma,  $y^{(j)}\prec \eta^j$ for $j<r$. We have  $\eta\succeq 1$, so if~$0<\dabs{\i}<r$, then~$y^{\i}\prec\eta^{\dabs{\i}}\preceq\eta^{r-1}$. 
Hence if additionally $P(0)\prec\eta$, then~$y^{(r)}=P\big(y,\dots,y^{(r-1)}\big) \prec \eta^r$.
\end{proof}

\noindent
To deduce from Theorem~\ref{thm:Landau} the general case of Proposition~\ref{prop:EL}  we  use: 
 
\begin{lemma}\label{lem:bd mult conj} Suppose that $\fm\prec 1$ and  $z\in\c^r[\imag]$. If
$z^{(j)} \preceq \eta^j$ for~$j=0,\dots,r$, then~$(z\fm^n)^{(j)}\preceq  \eta^j\fm^{n-j}$ for $j=0,\dots,r$,
and likewise with $\prec$ instead of $\preceq$.
\end{lemma} 
\begin{proof}
Corollary~1.1.12 in \cite{ADH4}  yields $(\fm^{n})^{(m)} \preceq \fm^{n-m}$ for   $m\leq n$.
Thus  if $z^{(j)} \preceq \eta^j$ for~$j=0,\dots,r$, then
$$ z^{(k)}  (\fm^{n})^{(j-k)} \preceq \eta^k \fm^{n-(j-k)}   \preceq \eta^j\fm^{n-j} \qquad (0\leq k\leq j\leq r),$$
so $(z\fm^n)^{(j)}\preceq  \eta^j\fm^{n-j}$ for $j=0,\dots,r$, by the Product Rule. The argument with $\prec$ instead of $\preceq$ is similar.
\end{proof}

\begin{proof}[Proof of Proposition~\ref{prop:EL}] We assume $\fm\prec 1$ (as we may).  Proposition~\ref{prop:Li(H(R))} yields a Liouville closed Hardy field $H\supseteq \R$ with $\fm\in H$. For $i=0,\dots,r$ set
$$Y_i\ :=\  \sum_{j=0}^i {i\choose j} Y^{(i-j)} (\fm^n)^{(j)}\in H\big[Y,Y',\dots,Y^{(i)}\big]\ \subseteq\ \c[\imag]\big[Y, Y',\dots, Y^{(r)}\big].$$
Then for  $z:=y\,\fm^{-n}\preceq 1$ in $\c^r[\imag]$ the Product Rule gives
$$Y_i(z,z',\dots,z^{(i)})\ =\ (z\,\fm^n)^{(i)}\ =\ y^{(i)}    \qquad(i=0,\dots,r),$$
so with
$$Q\ :=\ Y^{(r)}-\fm^{-n}\big(Y_r-P(Y_0,\dots,Y_{r-1})\big)\in\c[\imag]\big[Y,Y',\dots,Y^{(r-1)}\big]$$
we have by substitution of $z,\dots, z^{(r)}$ for $Y, Y',\dots, Y^{(r)}$, 
\begin{align*}
z^{(r)}\ &= \ Q\big(z,z',\dots,z^{(r-1)}\big) + \fm^{-n}\big( y^{(r)} - P(y,y',\dots,y^{(r-1)}) \big) \\
&=\  Q\big(z,z',\dots,z^{(r-1)}\big).  
\end{align*}
In $H\{Y\}$ we have
$(Y^{\j})_{\times\fm^n} = Y_0^{j_0}\cdots Y_r^{j_r}$ for $\j=(j_0,\dots,j_r)\in\N^{1+r}$. 
Let
 $\varepsilon\in\R^>$; then $\fm^{-\varepsilon}\in H$. We equip~$H\{Y\}$ with the gaussian extension of the valuation of~$H$; see [ADH, 4.5]. Then~${\fm^{-n} (Y^{\j})_{\times\fm^n}  \preceq  \fm^{-\varepsilon}}$ for~${\j\in\N^{1+r}\setminus\{ 0 \}}$, by~[ADH, 6.1.4].
Take $Q_{\i}\in \c[\imag]$
for $\dabs{\i}<r$ such that
$$Q\ =\ \sum_{\dabs{\i}<r}Q_{\i}Y^{\i} \qquad(Q_{\i}\ne 0 \text{ for only finitely many }\i).$$
Together with $P_{\i}\preceq\eta$ for all $\i$ and $P(0)\preceq \eta\,\fm^{n}$, the remarks above
 yield $Q_{\i}\preceq\eta\,\fm^{-\varepsilon}$ for all~$\i$.
By Theorem~\ref{thm:Landau} applied to $P$, $y$, $\eta$ replaced by~$Q$,~$z$,~$\eta\,\fm^{-\varepsilon}$,    we now obtain~$z^{(j)}\preceq (\eta\,\fm^{-\varepsilon})^j$, with~$\prec$ in place of $\preceq$ if~$y\prec\fm^n$ and~$P(0)\prec\eta\,\fm^n$.
Using Lemma~\ref{lem:bd mult conj} with~$\eta\,\fm^{-\varepsilon}$   in place of $\eta$  finishes the proof of Proposition~\ref{prop:EL}. \end{proof} 

\medskip
\noindent
The following immediate consequence of Proposition~\ref{prop:EL} is used in Section~\ref{sec:ueeh}.  (The case $\fm=1$   is due to Esclangon~\cite{Esc}.)

\begin{cor}\label{cor:EL}
Suppose $f_1,\dots,f_r\in\c[\imag]$ and $y\in\c^r[\imag]$ satisfy
$$y^{(r)}+f_1y^{(r-1)}+\cdots+f_ry\ =\ 0,\qquad f_1,\dots,f_r\preceq\eta,\quad y\ \preceq\ \fm^n.$$ 
Then $y^{(j)}\preceq \eta^j\fm^{n-j(1+\varepsilon)}$ for $j=0,\dots,r$ and all $\varepsilon\in\R^>$, with $\prec$ in place of $\preceq$ if~$y\prec \fm^n$. 
\end{cor}

\section{Hardy Fields and Uniform Distribution}\label{sec:udmod1}

\noindent
Section~\ref{sec:ueeh} gives an analytic description of the universal exponential extension~$\Univ$  of the algebraic closure~$K$ of a Liouville closed Hardy field extending~$\R$. The elements of~$\Univ$ are exponential sums with coefficients and exponents in~$K$.
To extract asymptotic information about the summands in such a sum we  refine here results of Boshernitzan~\cite{BoshernitzanUniform} about uniform distribution mod~$1$ for
functions whose germs are in a Hardy field. 
Our  reference for uniform distribution mod $1$  is \cite[Ch.~1, \S{}9]{KuipersNiederreiter}.
We also need some facts about
 trigonometric polynomials, almost periodic functions, and their mean values. These are treated in~\cite{Bohr,Corduneanu} mainly  in the one-variable case; adaptations to the multivariable case, required here, generally are straightforward.

 {\it Throughout this section we assume $n\geq 1$.} We equip
$\R^n$  with its usual Le\-bes\-gue measure $\mu_n$, and {\em measurable\/}   means measurable with respect to $\mu_n$.
Identifying $\C$ with $\R^2$ in the usual way, this also equips $\C$ with its Lebesgue measure.  
  {\it For vectors~$r=(r_1,\dots,r_n)$ and $s=(s_1,\dots,s_n)$ in~$\R^n$ we let~${r\cdot s}:= {r_1s_1+\cdots + r_ns_n}\in \R$ be the usual dot product of $r$ and~$s$. We also set 
$rs:=(r_1s_1,\dots, r_ns_n)\in\R^n$, not to be confused with $r\cdot s\in\R$. Moreover,  we let~${v,w\colon\R^n\to\mathbb C}$ be complex-valued functions on $\R^n$, and let $s=(s_1,\dots,s_n)$ range over $\R^n$, and $T$ over $\R^{>}$.}\/
We set~$|s|=|s|_\infty:=\max\big\{\abs{s_1},\dots,\abs{s_n}\big\}$ and~$\dabs{w} := \sup_{s}\, \abs{w(s)}\in [0,+\infty]$.
We shall also have occasion to consider various functions~$\R^n\to\mathbb C$ obtained from~$w$: $\overline{w}$, $\abs{w}$, as well as $w_{+r}$ and~$w_{\times r}$  (for $r\in\R^n$), defined
by
$$\overline{w}(s)\ :=\ \overline{w(s)}, \quad 
\abs{w}(s)\ :=\ \abs{w(s)}, \quad
w_{+r}(s)\ :=\ w(r+s),\quad
w_{\times r}(s)\ :=\ w(rs).$$

\subsection*{Almost periodic functions}
Let $\alpha$ range over $\R^n$.
Call~$w$ a {\it trigonometric polynomial}\/ if there are~$w_\alpha\in\mathbb C$, with $w_\alpha\ne 0$ for only finitely many~$\alpha$,
such that 
\begin{equation}\label{eq:trig poly}
w(s)\ =\ \sum_\alpha w_\alpha \ex^{(\alpha\cdot s)\imag}\qquad\text{for all $s$.}
\end{equation}
The coefficients~$w_\alpha$ in~\eqref{eq:trig poly} are uniquely determined by $w$. (See   Corollary~\ref{cor:id thm} below.)
 If $w$ is a trigonometric polynomial, then~$\overline{w}$
is a trigonometric polynomial, and for $r\in\R^n$, so are   $w_{+r}$ and~$w_{\times r}$.
The trigonometric polynomials form a subalgebra of the $\mathbb C$-algebra of
uniformly continuous bounded functions $\R^n\to\mathbb C$. 
The latter is a Banach algebra
with respect to $\|\cdot\|$, and the   elements of the closure of
its subalgebra of trigonometric polynomials with respect to this norm 
are the {\it almost periodic}\/ functions (in the sense of Bohr), which therefore form a Banach subalgebra.
In particular, if $v$, $w$ are almost periodic, so are
$v+w$ and $vw$. Moreover, if $w$ is   almost periodic, then so are $\overline{w}$ and $w_{+r}$, $w_{\times r}$ for  $r\in\R^n$.

We say that $v$ is  {\it $1$-periodic} if $v_{+k}=v$ for all $k\in\Z^n$.  
If $v$ is continuous and $1$-periodic, then $v$ is almost periodic: indeed, by ``Stone-Weierstrass'' there is for every~$\varepsilon\in\R^>$ a $1$-periodic trigonometric polynomial~$w$ with~${\dabs{v-w}<\varepsilon}$. 
(See  \cite[(7.4.2)]{Dieudonne} for the case $n=1$.) 

\subsection*{Mean value} The function $w$ is said to have a {\em mean value\/}  if $w$ is bounded and measurable, and the limit
\begin{equation}\label{eq:mean value}
\lim_{T\to\infty} \frac{1}{T^n} \int_{[0,T]^n} w(s)\,ds
\end{equation}
exists \textup{(}in $\mathbb C$\textup{)}; in that case we
call  the quantity~\eqref{eq:mean value} the {\it mean value}  of $w$ and denote it by $M(w)$.
One verifies easily that if $v$ and~$w$ have a mean value, then so do the functions $v+w$, $cw$ ($c\in\mathbb C$),
and $\overline{w}$, with
$$M(v+w)\ =\ M(v)+M(w),\quad M(cw)\ =\ cM(w),\quad\text{ and }\quad M(\overline{w})\ =\ \overline{M(w)}.$$ 
If $w$ has a mean value and $w(\R^n)\subseteq \R$, then $M(w)\in\R$.

\begin{lemma}\label{lem:mean value add conj}
Let $d\in\R^n$. Then $w$ has a mean value iff $w_{+d}$ has a mean value,
in which case $M(w)=M(w_{+d})$.
\end{lemma}
\begin{proof}
It suffices to treat the case $d=(d_1,0,\dots,0)$, $d_1\in\R^{>}$.
For~$T > d_1$ we have
\begin{multline*}
\left|\int_{[0,T]^n} w_{+d}(s)\,ds - 
 \int_{[0,T]^n} w(s)\,ds \right|\ =\ \\
  \left| \int_{[T,d_1+T]\times [0,T]^{n-1}} w(s)\,ds - \int_{[0,d_1]\times [0,T]^{n-1}} w(s)\,ds\right|   \ \leq\ 2d_1\dabs{w}T^{n-1},
\end{multline*}
and this yields the claim.
\end{proof}

\begin{cor}
Suppose $w$ has a mean value and $T_0\in\R^>$.
If $w(s)=0$ for all~$s\in (\R^{\geq})^n$ with $\abs{s}\geq T_0$, then $M(w)=0$.
If $w(\R^n)\subseteq\R$ and $w(s)\geq 0$ for all~$s\in (\R^{\geq})^n$ with $\abs{s}\geq T_0$, then $M(w)\geq 0$.
\end{cor}

\begin{lemma}\label{lem:mv and sup}
Suppose $w$ has a mean value and $w(\R^n)\subseteq\R$. Then
$$ \liminf_{\abs{s}\to \infty} w(s)\ \leq\ 
 M(w)\ \leq\ \limsup_{\abs{s}\to \infty} w(s)\qquad \text{ with $s$ ranging over $(\R^{\ge})^n$}.$$
\end{lemma}
\begin{proof}
Assume~$L:=\limsup_{\abs{s}\to \infty} w(s)<M(w)$. Let~$\varepsilon:=\frac{1}{2}\big(M(w)-L\big)$, and
take~$T_0\in\R^>$ such that~$w(s)\leq M(w)-\varepsilon$ for all $s$ with $\abs{s}\geq T_0$.
The previous corollary then yields $M(w)\leq M(w)-\varepsilon$, a contradiction.
This shows the second inequality; the first inequality is proved in a similar way.
\end{proof}

\noindent
The following is routine:

\begin{lemma}\label{lem:mv sequence}
Let $(v_m)$ be a sequence of bounded measurable functions $\R^n\to \mathbb C$ with a mean value, such that $\lim\limits_{m\to \infty} \dabs{v_m-w}=0$, and let $w$ be bounded and measurable. Then
$w$ has a mean value, and~$\lim\limits_{m\to \infty} M(v_m)=M(w)$.
\end{lemma}

\noindent
If $\alpha\in\R^n$ and
$w(s)=\ex^{\imag (\alpha\cdot s)}$ for all $s$, then $w$ has a mean value, with~$M(w)=1$ if $\alpha=0$ and $M(w)=0$ otherwise.
Together with Lemma~\ref{lem:mv sequence}, this yields the  well-known fact that if $w$ is almost periodic, then $w$ has a mean value. (See \cite[Theorem~1.12]{Corduneanu} for the case $n=1$.) 
Moreover, using Lemma~\ref{lem:mv sequence} it is also easy to show that if  $w$ is almost periodic and $r\in (\R^\times)^n$, then
the almost periodic function~$w_{\times r}$ has the same mean value as $w$.
If $w$ is continuous and $1$-periodic, then $w$ has a mean value, namely $M(w)=\int_{[0,1]^n} w(s)\,ds$.
For a proof of the following result in the case $n=1$, see \cite[Theorem~1.19]{Corduneanu}:

\begin{prop}[Bohr]\label{prop:Bohr} If $w$ is almost periodic,
$w(\R^n)\subseteq\R^{\geq}$, and $M(w)=0$, then $w=0$.
\end{prop}
 
\noindent
By Proposition~\ref{prop:Bohr}, the map $(v,w)\mapsto \langle v,w\rangle:=M(v\overline{w})$ is  a positive definite hermitian form on the $\mathbb C$-linear space of almost
periodic functions~$\R^n\to\mathbb C$.
For a trigonometric polynomial $w$ as in \eqref{eq:trig poly} we have~$w_\alpha=\langle w,\ex^{(\alpha\cdot s)\imag}\rangle$, and thus:

\begin{cor}\label{cor:id thm} For $w$ as in  \eqref{eq:trig poly}, 
if $w=0$, then $w_\alpha=0$ for all~$\alpha$.  
\end{cor}

\subsection*{Uniform distribution}
{\sloppy\it In this subsection $f_1,\dots, f_n\colon \R \to \R$ are measurable, and~$f : = (f_1,\dots,f_n) \colon \R\to\R^n$. }
   
\begin{theorem}[Weyl] \label{thm:Weyl} 
The following conditions on $f$ are equivalent:
\begin{enumerate}
\item[\rm{(i)}] $\displaystyle\lim_{T\to \infty} \frac{1}{T}\int_0^T \ex^{2\pi\imag (k\cdot f(t))}\,dt\ =\ 0$ for all $k\in (\Z^n)^{\neq}$; 
\item[\rm{(ii)}] $\displaystyle\lim_{T\to\infty} \frac{1}{T} \int_0^{T} (w\circ f)(t)\,dt\ =\ \int_{[0,1]^n} w(s)\,ds$ for every continuous $1$-periodic~$w$;
\item[\rm{(iii)}]  for every continuous $1$-periodic
$w$,  the function  ${w\circ f}\colon\R\to\mathbb C$ has mean value  $M(w\circ f)=M(w)$.
\end{enumerate}
\end{theorem}

\noindent
We say that 
$f$ is {\it uniformly distributed~$\operatorname{mod}\,1$} (abbreviated: u.d.~$\operatorname{mod}\,1$)
if $f$ satisfies condition (i) in Theorem~\ref{thm:Weyl}.
This is not the usual definition
but is equivalent to it by \cite[Theorem~9.9]{KuipersNiederreiter}.
The latter,  in combination with~\cite[Exercise~9.26]{KuipersNiederreiter} 
also yields the implication (i)~$\Rightarrow$~(ii) in Theorem~\ref{thm:Weyl}.
For (ii)~$\Rightarrow$~(i), apply~(ii) to the $1$-periodic trigonometric polynomial $w$ given by $w(s)=\ex^{2\pi\imag (k\cdot s)}$. 
The equivalence~(ii)~$\Leftrightarrow$~(iii) is clear by earlier remarks.

In the next lemma we consider a strengthening of  u.d.~$\operatorname{mod}\,1$. For~$\alpha\in\R^n$  
set~$\alpha f:= (\alpha_1f_1,\dots, \alpha_nf_n)\colon\R\to\R^n$. 

{\sloppy
\begin{lemma}\label{lem:ud map fn}
The following conditions on $f$ are equivalent:
\begin{enumerate}
\item[\rm{(i)}] $\alpha f$ is  u.d.~$\operatorname{mod}\,1$ for all $\alpha\in (\R^{\times})^n$;
\item[\rm{(ii)}] $\displaystyle\lim_{T\to\infty} \frac{1}{T}\int_0^T \ex^{2\pi\imag (\beta\cdot f(t))}\,dt\ =\ 0$ for all $\beta\in (\R^n)^{\neq}$;
\item[\rm{(iii)}] for every almost periodic
$w$,   $w\circ f$ has mean value~$M(w\circ f)=M(w)$.
\end{enumerate}
\end{lemma}}
\begin{proof}
Assume (i); let $\beta\in (\R^n)^{\neq}$. For $i=1,\dots,n$ set~$\alpha_i:=1$, $k_i:=0$ if~$\beta_i=0$ and ~$\alpha_i:=\beta_i$, $k_i:=1$ if $\beta_i\neq 0$. Then $k=(k_1,\dots,k_n)\in (\Z^n)^{\neq}$, 
$\alpha=(\alpha_1,\dots,\alpha_n)$ is in~$(\R^\times)^n$,  and $\beta\cdot f(t) = k\cdot (\alpha f)(t)$ for all $t\in\R$.
Now (ii) follows from the above definition of ``u.d.~$\operatorname{mod}\,1$'' applied to~$\alpha f$ in place of $f$.
For (ii)~$\Rightarrow$~(iii), assume (ii). Then   each trigonometric polynomial $v$  gives a function $v\circ f$ with mean
value~${M(v\circ f)} = M(v)$.
Let $w$ be almost periodic, and
take a sequence~$(v_m)$ of trigonometric
polynomials~${\R^n\to \mathbb C}$ such that~${\dabs{v_m-w}\to 0}$ as~$m\to\infty$.
So~$M(v_m)\to M(w)$ as~$m\to\infty$, by Lemma~\ref{lem:mv sequence}.
Also $\dabs{(v_m\circ f)-(w\circ f)}\to 0$ as~$m\to \infty$, 
hence by Lemma~\ref{lem:mv sequence} again, 
$w\circ f$ has a mean value and
$M(v_m)=M(v_m\circ f)\to M(w\circ f)$ as $m\to \infty$. Therefore~$M(w\circ f)=M(w)$.
Finally, assume~(iii), and let~$\alpha\in (\R^{\times})^n$; to show that $\alpha f$ is
u.d.~$\operatorname{mod}\,1$ we verify that condition~(iii) in Theorem~\ref{thm:Weyl}
holds for $\alpha f$ in place of~$f$. Thus suppose $w$ is continuous and $1$-periodic. By (iii) applied to the almost periodic function~$w_{\times\alpha}$ in place of $w$,   the function~$w_{\times\alpha}\circ f=w\circ (\alpha f)$ has a mean value and~${M(w_{\times\alpha}\circ  f)}=M(w_{\times\alpha})$;
 now use that $M(w_{\times\alpha})=M(w)$ by a remark after Lemma~\ref{lem:mv sequence}.
\end{proof}
 
\noindent
We say that $f$ is {\it uniformly distributed} (abbreviated: u.d.) if it satisfies one of the equivalent conditions in Lemma~\ref{lem:ud map fn}.

\begin{cor}\label{cor:limsup vs sup}
Suppose $w$ is almost periodic, $w(\R^n)\subseteq\R^{\geq}$, and $f$ is  u.d.   Then
$$\limsup_{t\to +\infty} w\big(f(t)\big) = 0\ \Longleftrightarrow\   w=0.$$
\end{cor}
\begin{proof}
Lemmas~\ref{lem:mv and sup} and~\ref{lem:ud map fn} give $0\le M(w)=M(w\circ f)\leq \limsup\limits_{t\to+\infty} w\big(f(t)\big)$.  Now use Proposition~\ref{prop:Bohr}.
\end{proof}

\subsection*{Application to Hardy fields}

\noindent
{\it In  this subsection $f_1,\dots, f_n\colon \R\to \R$  are continuous, their germs, denoted also by $f_1,\dots, f_n$, lie in a common Hardy field, and as above~$f:=(f_1,\dots,f_n)\colon \R\to \R^n$.}\/ 

\begin{theorem}[{Boshernitzan \cite[Theorem~1.12]{BoshernitzanUniform}}]\label{thm:Bosh} 
$$\text{$f$ is u.d.~$\operatorname{mod}\,1$}\ \Longleftrightarrow\ \text{$k_1f_1+\cdots+k_nf_n \succ \log x$ for all $(k_1,\dots,k_n)\in (\Z^n)^{\neq}$.}$$
\end{theorem}

\noindent
For example, given $\lambda_1,\dots,\lambda_n\in\R$, the map $t\mapsto (\lambda_1 t,\dots,\lambda_n t) \colon \R \to \R^n$ is u.d.~$\operatorname{mod}\,1$
iff $\lambda_1,\dots,\lambda_n$ are $\Q$-linearly independent (Weyl~\cite{Weyl}).

\begin{cor}\label{nbos}
We have the following equivalence:
$$\text{$f$ is u.d.}\ \Longleftrightarrow\ \text{$\beta_1f_1+\cdots+\beta_nf_n \succ \log x$ for all $(\beta_1,\dots,\beta_n)\in (\R^n)^{\neq}$.}$$
In particular, if $\log x\prec f_1\prec\cdots\prec f_n$, then $f$ is u.d. 
\end{cor}
{\sloppy
\begin{proof}
Let $\beta=(\beta_1,\dots,\beta_n)\in (\R^n)^{\neq}$ and define $\alpha_i$, $k_i$ as in the proof of~(i)~$\Rightarrow$~(ii) in Lem\-ma~\ref{lem:ud map fn}.
Then~${\beta_1f_1+\cdots+\beta_nf_n}= k_1 (\alpha_1 f_1) +\cdots + k_n (\alpha_n f_n)$, and the germs~$\alpha_1 f_1,\dots,\alpha_n f_n$ lie in a common Hardy field by Proposition~\ref{prop:Li(H(R))}. Using this observation and Theorem~\ref{thm:Bosh} yields the forward direction. The backward direction is also an easy consequence of Theorem~\ref{thm:Bosh}.
\end{proof}}

\noindent
We now get to the result  that we actually need in Section~\ref{sec:ueeh}.

\begin{prop}\label{prop:limsup vs sup}
Suppose $w$ is almost periodic, $w(\R^n)\subseteq\R^{\geq}$, $1\prec f_1\prec\cdots\prec f_n$, and
 $\limsup\limits_{t\to+\infty} w\big(f(t)\big) = 0$. Then $w=0$.
\end{prop}
\begin{proof}
We first arrange $f_1>\R$, replacing $f_1,\dots,f_n$ and $w$ by $-f_1,\dots,-f_n$ and the function $s\mapsto w(-s)\colon\R^n\to\R^{\geq}$, if $f_1<\R$.
Pick $a$ such that the restriction of $f_1$ to $\R^{\geq a}$ is strictly increasing, set $b:=f_1(a)$, and let 
$f_1^{\inv}\colon\R^{\geq b}\to\R$ be the compositional inverse of this restriction.
Set $g_j(t):=(f_j \circ f_1^{\inv})(t)$ for $t\ge b$ and~$j=1,\dots,n$
and  consider the map $$g\ =\ (g_1,\dots, g_n)\ =\ f\circ f_1^{\inv}\  :\  \R^{\ge b} \to \R^n.$$
The germs of $g_1,\dots,g_n$, denoted by the same symbols, 
lie in a common Hardy field (see for example \cite[Section~4]{ADH5}) and satisfy $x=g_1\prec  g_2 \prec  \cdots \prec  g_n$.
Now $f_1^{\inv}$ is strictly increasing and moreover~$f_1^{\inv}(t)\to+\infty$ as $t\to+\infty$, so
$$\limsup_{t\to+\infty} w\big(f(t)\big)\ =\ \limsup_{t\to+\infty} w\big(f\big(f_1^{\inv}(t)\big)\big)\ =\ 
\limsup_{t\to+\infty} w\big(g(t)\big)\ =\ 0.$$
Thus replacing $f_1,\dots,f_n$ by continuous functions $\R\to\R$ with the same 
germs as~$g_1,\dots,g_n$, we arrange $x=f_1\prec f_2\prec\cdots\prec f_n$.
Then $f$ is u.d.~by  Corollary~\ref{nbos}. Now use Corollary~\ref{cor:limsup vs sup}.
\end{proof}

\section{Universal Exponential Extensions of Hardy Fields}\label{sec:ueeh}

 \noindent
{\it In this section $H\supseteq \R$ is a Liouville closed Hardy field and $K:=H[\imag]$. So $K$ is an algebraically closed $\d$-valued field with constant field $\C$.}\/ In order to give an analytic
  description of the universal exponential extension $\Univ=\Univ_K$ of $K$
 we consider the differential ring extension $\Calinf[\imag]$ of $K$
with   ring of constants $\C$. For any $f\in \Calinf[\imag]$ we have a $g\in \Calinf[\imag]$
with $g'=f$, and then $u=\ex^g\in \Calinf[\imag]^\times$ satisfies
$u^\dagger =f$.  

\begin{lemma}\label{purelyimag} Suppose $f\in \Calinf[\imag]$ is purely imaginary, that is,
$f\in \imag\Calinf$. Then there is a $u\in \Calinf[\imag]^\times$
such that $u^\dagger=f$ and $|u|=1$. 
\end{lemma}
\begin{proof} Taking $g\in \imag\Calinf$ with $g'=f$, the resulting
$u=\ex^g$ works. 
\end{proof}

\noindent 
We define the subgroup $\ex^{H\imag}$ of $\Calinf[\imag]^\times$ by 
$$\ex^{H\imag}\ :=\ \big\{\!\ex^{h\imag}:\ h\in H\big\}\ =\ \big\{u\in \Calinf[\imag]^\times:\ |u|=1,\ u^\dagger\in H\imag\big\}.$$ 
Then 
$(\ex^{H\imag})^\dagger=H\imag$ by Lemma~\ref{purelyimag}, so $(H^\times\cdot \ex^{H\imag})^\dagger=K$
and thus~$K[\ex^{H\imag}]$ is an exponential extension of $K$   with the same ring of constants $\mathbb C$ as~$K$.
Fix a complement~$\Lambda\subseteq H\imag$ of   the subspace $K^\dagger$ of the $\Q$-linear space $K$, and 
let $\lambda$ range over $\Lambda$.
The differential $K$-algebras~$\Univ:= K\big[\!\ex(\Lambda)\big]$ and
$K[\ex^{H\imag}]$ are isomorphic by Corollary~\ref{corcharexp}, but we need something better:

\begin{lemma}\label{exqe} There is an isomorphism $\Univ\to K[\ex^{H\imag}]$ 
of differential $K$-algebras that maps $\ex(\Lambda)$ into $\ex^{H\imag}$. 
\end{lemma} 
\begin{proof} We have a short exact sequence of commutative groups
$$1 \to S \xrightarrow{\ \subseteq\ } \ex^{H\imag} \xrightarrow{\ \ell\ } H\imag \to 0,$$
where $S=\big\{z\in \mathbb{C}^\times:\, |z|=1\big\}$ and
 $\ell(u):=u^\dagger$ for~$u\in \ex^{H\imag}$. Since the subgroup~$S$ of~$\mathbb{C}^\times$ is divisible, this sequence splits: we have a group embedding 
$e\colon H\imag\to \ex^{H\imag}$ such that $e(b)^\dagger=b$ for all $b\in H\imag$. (Note the italic $e$ versus the upright $\ex$.)  Then
the group embedding $$\ex(\lambda)\mapsto e(\lambda)\ :\ \ex(\Lambda)\to \ex^{H\imag}$$  extends uniquely to a $K$-algebra morphism $\Univ\to K[\ex^{H\imag}]$. Since
$\ex(\lambda)^\dagger=\lambda=e(\lambda)^\dagger$ for all $\lambda$, this is a
differential $K$-algebra morphism, and even an isomorphism by
Lemma~\ref{lem:embed into U} applied to $R=K[\ex^{H\imag}]$.   
\end{proof}

\noindent
Complex conjugation $f+g\imag\mapsto \overline{f+g\imag}=f-g\imag$ ($f,g\in \Calinf$) is an automorphism of the differential ring~$\Calinf[\imag]$ over $H$ and maps
$K[\ex^{H\imag}]$ onto itself, sending each~${u\in \ex^{H\imag}}$ to~$u^{-1}$. Thus any isomorphism
$\iota\colon \Univ\to K[\ex^{H\imag}]$ of differential $K$-algebras with~$\iota\big(\!\ex(\Lambda)\big)\subseteq \ex^{H\imag}$ as in  Lemma~\ref{exqe} also satisfies 
$\iota(\overline{f})=\overline{\iota(f)}$ for~$f\in \Univ$,
where $f\mapsto\overline{f}$ is the ring  automorphism of $\Univ$ which extends complex conjugation on~$K=H[\imag]$
and satisfies $\overline{\ex(\lambda)}=\ex(\overline{\lambda})=\ex({-\lambda})$ for all $\lambda$;
see \cite[Section~2.2, subsection ``The real case'']{ADH4}.) Hence  by \cite[Lemma 2.2.14, Corollary~2.2.17]{ADH4},
given such an isomorphism~$\iota$, any differential $K$-algebra isomorphism as in Lemma~\ref{exqe}   equals~$\iota\circ\sigma_\chi$ for a unique character~$\chi\colon\Lambda\to\C^\times$ with $\abs{\chi(\lambda)}=1$ for all $\lambda$, where~$\sigma_\chi$ is the differential $K$-algebra
automorphism of $\Univ$ with~$\sigma_\chi(\ex(\lambda))=\chi(\lambda)\ex(\lambda)$ for all~$\lambda$.
Fix such an isomorphism $\iota$ and identify $\Univ$ with its image $K[\ex^{H\imag}]$ via~$\iota$.

\begin{lemma}\label{lem:basis of kerUA} 
Let $A\in K[\der]^{\neq}$. Then the $\C$-linear space $\ker_{\Univ} A$  has a basis
$$f_1\ex^{\phi_1\imag},\dots,f_d\ex^{\phi_d\imag}\quad\text{ where $f_j\in K^\times$,  $\phi_j  \in H$ \textup{(}$j=1,\dots,d$\textup{)}.}$$
If $\I(K)\subseteq K^\dagger$, then we can choose the $f_j$, $\phi_j$ such that in addition, for each $j$, we have $\phi_j=0$ or $\phi_j\succ 1$.
\end{lemma}
{\sloppy
\begin{proof}
By the remarks before \eqref{eq:2.4.8}, $\ker_{\Univ} A$ has a basis contained in $\Univ^\times=K^\times\ex(\Lambda)$. Using~$\ex(\Lambda)\subseteq \ex^{H\imag}$, this yields the first part. For the second part, use that if~${\I(K)\subseteq K^\dagger}$ and $\phi\in H$, $\phi\preceq 1$, then $\ex^{\phi\imag}\in K^\times$ by \cite[Proposition~6.11]{ADH5}.
\end{proof}}

\begin{remark}
Let $A\in K[\der]^{\ne}$, and suppose $\I(K)\subseteq K^\dagger$. Then one can  choose the~$f_j$,~$\phi_j$ as in Lemma~\ref{lem:basis of kerUA} such that also for all $j\neq k$,
either
$\phi_j-\phi_k\succ 1$, or~${\phi_j=\phi_k}$ and~${vf_j\neq vf_k}$ in $v(K^\times)$. For such $f_j$, $\phi_j$ we have $\exc^{\operatorname{u}}(A) \supseteq \{ vf_1,\dots,vf_d\}$. These statements are proved in~\cite{ADH6} and not used later in this paper, but explain how~$\exc^{\operatorname{u}}(A)$ helps
to locate the~$y\in\Calinf[\imag]$ with $A(y)=0$.
\end{remark}

\noindent
We have the asymptotic relations
$\preceq_{\g}$ and~$\prec_{\g}$
on $\Univ$ coming from the gaussian extension~$v_{\g}$ of the valuation on $K$. (See Section~\ref{sec:prelims}.) But we also have the asymptotic relations induced on $\Univ=K[\ex^{H\imag}]$
by the relations $\preceq$ and $\prec$ defined on $\mathcal C[\imag]$. It is clear that for $f\in\Univ$:
\begin{align*}
f\preceq_{\g} 1 	&\quad \Longrightarrow \quad f\preceq 1 \quad \Longleftrightarrow \quad \text{for some $n$ we have $\abs{f(t)}\leq n$ eventually,} \\
f\prec_{\g} 1	&\quad \Longrightarrow \quad f\prec 1 \quad \Longleftrightarrow \quad \lim_{t\to+\infty} f(t)=0.
\end{align*}
As a tool for later use we derive a converse of the implication
$f\prec_{\g}1\Rightarrow f\prec 1$: Lemma~\ref{lem:gaussian ext dom} below, where we assume in addition
that~${\I(K)\subseteq K^\dagger}$ and $\Lambda$ is an $\R$-linear subspace
of $K$. 
This requires the material from Section~\ref{sec:udmod1} and some considerations about exponential sums treated in the next subsection.

\subsection*{Exponential sums over Hardy fields} In this subsection $n\geq 1$.
In the next lemma, $f=(f_1,\dots,f_m)\in H^m$ where $m\ge 1$ and
 $1\prec f_1\prec\cdots\prec f_m$. (In that lemma it doesn't matter which
 functions we use to represent the germs $f_1,\dots, f_m$.)
For~$r=(r_1,\dots,r_m)\in\R^m$ we set $r\cdot f:= r_1f_1+\cdots + r_mf_m\in H$.
 
\begin{lemma}\label{lem:limsup 1} 
Let $r^1,\dots,r^n\in\R^m$ be distinct and $c_1,\dots,c_n\in\mathbb C^\times$. Then
$$\limsup_{t\to \infty} \left|c_1 \ex^{(r^1\cdot f)(t)\imag}+ \cdots + c_n\ex^{(r^n\cdot f)(t)\imag}\right|\ >\ 0.$$ 
\end{lemma}
\begin{proof} Consider the trigonometric polynomial $w\colon\R^m\to \R^{\ge}$ given by
$$ w(s)\ :=\ \big|c_1 \ex^{(r^1\cdot s)\imag}+\cdots+c_n \ex^{(r^n\cdot s)\imag}\big|^2.$$  
By Corollary~\ref{cor:id thm} we have $w(s)>0$ for some $s\in\R^m$. 
Taking continuous representatives $\R \to \R$ of $f_1,\dots,f_m$, to be denoted also by $f_1,\dots,f_m$, the lemma now follows from Proposition~\ref{prop:limsup vs sup}. 
\end{proof}

\noindent
Next, let  $h_1,\dots,h_n\in H$ be distinct such that
$(\R h_1+\cdots+\R h_n)\cap \I(H)=\{0\}$. Since $H$ is Liouville closed we have
$\phi_1,\dots,\phi_n\in H$  such that $\phi_1'=h_1,\dots, \phi_n'=h_n$.

\begin{lemma}\label{lem:limsup 2}
Let  $c_1,\dots,c_n\in\mathbb C^\times$. Then for $\phi_1,\dots, \phi_n$ as above,
$$\limsup_{t\to \infty} \left| c_1 \ex^{\phi_1(t)\imag}+ \cdots + c_n \ex^{\phi_n(t)\imag}\right|\ >\ 0.$$
\end{lemma}
\begin{proof} The case $n=1$ is trivial, so let $n\ge 2$. Then $\phi_1,\dots, \phi_n$ are not all in~$\R$. 
Set~$V:= \R+\R\phi_1+\cdots+\R\phi_n \subseteq H$, so
$\der(V)=\R h_1+\cdots+\R h_n$.
We claim that~${V\cap \smallo_H=\{0\}}$ where $\smallo_H:=\{\phi\in H:\ \phi\prec 1\}$.
To see this, let $\phi\in V\cap \smallo_H$; then~$\phi'\in \der(V)\cap \I(H)=\{0\}$ and hence $\phi\in\R\cap \smallo_H=\{0\}$,
proving the claim.
Now $H$ is a Hahn space over $\R$ by~[ADH, p.~109], so by [ADH, 2.3.13] we have $f_1,\dots,f_m\in V$  ($1\le m\leq n$) such that~$V = \R + \R f_1 +\cdots + \R f_m$ and $1\prec f_1\prec \cdots \prec f_m$. For~$j=1,\dots,n$, $k=1,\dots,m$, 
take $t_j,r_{jk}\in\R$ such that $\phi_j=t_j+\sum_{k=1}^m r_{jk} f_k$
and set $r^j:=(r_{j1},\dots,r_{jm})\in\R^m$.
Since $\phi_{j_1}-\phi_{j_2}\notin\R$ for   $j_1\neq j_2$, we have
$r^{j_1}\neq r^{j_2}$ for   $j_1\neq j_2$.
Now apply Lemma~\ref{lem:limsup 1} to $c_1\ex^{t_1\imag},\dots, c_n\ex^{t_n\imag}$ in place of $c_1,\dots, c_n$. 
\end{proof}

\begin{cor}\label{corlimsupresult}
Let $f_1,\dots,f_n\in K$ and set $f:= f_1 \ex^{\phi_1\imag} + \cdots + f_n\ex^{\phi_n\imag}\in\Calinf[\imag]$, and suppose~$f\prec 1$. Then $f_{1},\dots,f_{n}\prec 1$.
\end{cor}
\begin{proof}
We may assume   $0 \neq f_1\preceq \cdots \preceq f_n$.
Towards a contradiction, suppose that~${f_n\succeq 1}$, and
take $m\leq n$ minimal such that $f_m\asymp f_n$. Then
 with~$g_j:=f_j/f_n\in K^\times$ and $g:=g_1 \ex^{\phi_1\imag} + \cdots + g_n\ex^{\phi_n\imag}$ we have $g\prec 1$ and
$g_1,\dots, g_n\preceq 1$, with $g_j\prec 1$ iff~$j<m$.
Replacing $f_1,\dots,f_n$ by $g_m,\dots,g_n$ and $\phi_1,\dots,\phi_n$ by $\phi_m,\dots,\phi_n$
we arrange $f_1\asymp\cdots\asymp f_n\asymp 1$. So
$$f_1=c_1+\varepsilon_1,\,\dots,\, f_n=c_n+\varepsilon_n\quad\text{ with
$c_1,\dots, c_n\in\mathbb C^\times$ and $\varepsilon_1,\dots, \varepsilon_n \prec 1$.}$$ 
Then $\varepsilon_1\ex^{\phi_1\imag} +\cdots + \varepsilon_n\ex^{\phi_n\imag}\prec 1$, hence
 $$c_1\ex^{\phi_1\imag} +\cdots + c_n\ex^{\phi_n\imag}\ =\ f-\big(\varepsilon_1\ex^{\phi_1\imag} +\cdots + \ex^{\phi_n\imag}\!\big)\ \prec\ 1.$$
Now Lemma~\ref{lem:limsup 2} yields the desired contradiction.
\end{proof}

\noindent
{\em In the rest of this section, $\I(K)\subseteq K^\dagger$}. As noted in Sec\-tion~\ref{sec:prelims} we can then take~${\Lambda=\Lambda_H\imag}$ where $\Lambda_H$ is an $\R$-linear complement
of $\I(H)$ in $H$. {\em We assume~$\Lambda$ has this form,}\/  and accordingly identify $\Univ$ with $K[\ex^{H\imag}]$ as explained before Lemma~\ref{lem:basis of kerUA}. 

\begin{lemma}\label{lem:gaussian ext dom} Let $f\in\Univ$ and $\fm\in H^\times$. Then $f\prec \fm\ \Leftrightarrow\ f\prec_{\g} \fm$. 
\end{lemma}
\begin{proof} The direction $\Leftarrow$ is clear. For $\Rightarrow$ we reduce to the case $\fm=1$, so let $f\prec 1$. We have $f=f_1\ex(h_1\imag)+\cdots + f_n\ex(h_n\imag)$ with
$f_1,\dots, f_n\in K$ and distinct $h_1,\dots, h_n$ in $\Lambda_H$, so $(\R h_1+\cdots+\R h_n)\cap \I(H)=\{0\}$. For~$h\in \Lambda_H$ we have~$\ex(h\imag)= \ex^{\phi\imag}$ with~$\phi\in H$ and $\phi'=h$.
Hence $f=f_1\ex^{\phi_1\imag} + \cdots + f_n\ex^{\phi_n\imag}$ where~$\phi_j\in H$ and~$\phi_j'=h_j$ for $j=1,\dots,n$. 
 Now  Corollary~\ref{corlimsupresult}  yields~${f\prec_{\g} 1}$.
\end{proof}

\subsection*{An application to slots in $H$} 
Until further notice $(P,1,\hat h)$ is  a slot in~$H$ of 
order~$r\geq 1$.
We also let $A\in K[\der]$ have order~$r$, and we let $\fm$ range over the elements of $H^\times$ such that
$v\fm\in v(\hat h -H)$. We begin with an important consequence of the bounds on solutions of linear differential
equations in
Section~\ref{sec:hardy}:  

{\sloppy
\begin{lemma}\label{lem:8.8 refined}
Suppose that $(P,1,\hat h)$ is $Z$-minimal, deep, and special,  and that~$\fv(L_P)\asymp\fv:=\fv(A)$. Let
$y\in\Calr[\imag]$ satisfy $A(y)=0$ and $y\prec \fm$ for all~$\fm$.  
Then $y',\dots,y^{(r)}\prec \fm$ for all $\fm$.
\end{lemma}}
\begin{proof} Proposition~\ref{specialvariant}  gives an $\fm\preceq \fv$, so it is enough to show $y',\dots,y^{(r)}\prec \fm$ for all $\fm\preceq \fv$. Accordingly we assume $0<\fm\preceq \fv$ below. 
As~$\hat h$ is special over $H$, we have~${2(r+1)v\fm}\in v(\hat h-H)$, so $y\prec\fm^{2(r+1)}$.
Then    Corollary~\ref{cor:EL}   with~$n=2(r+1)$,~$\eta=\abs{\fv}^{-1}$, $\varepsilon=1/r$ gives  for $j=0,\dots,r$:
\[y^{(j)}\ \prec\ \fv^{-j} \fm^{n-j(1+\varepsilon)}\ \preceq\ \fm^{n-j(2+\varepsilon)}\ \preceq\ \fm^{n-r(2+\varepsilon)}\  =\ \fm. \qedhere\]
\end{proof}

\noindent
Note that   if~$\dim_{\C}\ker_{\Univ} A =r$, then $\Univ=K[\ex^{H\imag}]\subseteq\Calinf[\imag]$ contains every $y\in\Calr[\imag]$ with~$A(y)=0$.  Lemma~\ref{lem:gaussian ext dom} is typically used in combination with the ultimate condition. Here is a first easy application:

\begin{lemma}\label{lem3.9 linear}
Suppose $\deg P=1$, $(P,1,\hat h)$ is ultimate, $\dim_{\C}\ker_{\Univ}L_P=r$, and $y$ in~$\Calr[\imag]$ satisfies $L_P(y)=0$ and~$y\prec 1$. Then~$y\prec \fm$ for all $\fm$.
\end{lemma}
\begin{proof}
We have~$y\in \Univ$, so $y\prec_{\g} 1$ by Lemma~\ref{lem:gaussian ext dom}. If $y=0$ we are done, so assume~${y\ne 0}$. 
Then \eqref{eq:vkerunivA vs excuA}
gives~$0 < v_{\g}y\in
v_{\g}(\ker^{\neq}_{\Univ}L_P) = \exc^{\operatorname{u}}(L_P)$, hence~$v_{\g}y > v(\hat h -H)$ by   \eqref{eq:ultimate normal linear},
so $y\prec_{\g}\fm$ for all $\fm$.  
Now Lemma~\ref{lem:gaussian ext dom} yields the   conclusion. 
\end{proof}

\begin{cor}\label{cor:3.6 linear}
Suppose $\deg P=1$, $(P,1,\hat h)$ is $Z$-minimal, deep,  special, and ultimate,  and $\dim_{\C}\ker_{\Univ}L_P=r$.
Let~$f,g\in \Calr[\imag]$ be such that $P(f)=P(g)=0$ and~$f,g\prec 1$.
Then~$({f-g})^{(j)}\prec \fm$ for~$j=0,\dots,r$ and all $\fm$.  
\end{cor}
\begin{proof}
Use Lemmas~\ref{lem:8.8 refined} and~\ref{lem3.9 linear} for $A=L_P$ and $y=f-g$.
\end{proof}

\noindent
{\it In the rest of this subsection we assume that
   $(P,1,\hat h)$ is normal and ultimate, $\dim_{\C} \ker_{\Univ}A=r$, and~$L_{P}=A+B$ where}
 $$B\prec_{\Delta(\fv)} \fv^{r+1}A, \qquad \fv:=\fv(A)\prec^\flat 1.$$
Then Lemma~\ref{lem:fv of perturbed op}   gives   $\fv(L_P)\sim\fv$,  and $v_{\operatorname{g}}(\ker^{\neq}_{\Univ} A)  = \exc^{\operatorname{u}}(A)  =  \exc^{\operatorname{u}}(L_P)$ by \eqref{eq:vkerunivA vs excuA}.
This yields a variant of Lemma~\ref{lem3.9 linear}, with a similar proof: 

\begin{prop}\label{lem3.9}  
If $y\in\Calr[\imag]$ and $A(y)=0$, $y\prec 1$,
then~$y\prec \fm$ for all~$\fm$.
\end{prop}

\noindent
The following result will be used in establishing a crucial non-linear version of Corollary~\ref{cor:3.6 linear}, namely Proposition~\ref{prop:notorious 3.6}.

\begin{cor}\label{cor:8.8 refined}
If $(P,1,\hat h)$ is $Z$-minimal, deep, and special, and
$y\in\Calr[\imag]$ is such that $A(y)=0$ and $y\prec 1$,  
then $y,y',\dots,y^{(r)}\prec \fm$ for all $\fm$.
\end{cor}
\begin{proof}
Use first  Proposition~\ref{lem3.9} and then Lemma~\ref{lem:8.8 refined}.
\end{proof}



\noindent
So far we didn't have to name an immediate asymptotic extension of $H$ where $\hat h$ is located, but for the ``complex'' version of the above we need to be more specific:
Let~$\hat H$ be an immediate asymptotic extension of $H$ and 
$\hat{K}=\hat{H}[\imag]\supseteq \hat{H}$ a corresponding immediate $\d$-valued extension of $K$.  The results in this subsection then go through  if instead of
$(P, 1,\hat h)$ being a slot in~$H$ of order $r\ge 1$  we assume that~$(P, 1,\hat h)$ is a slot in $K$ of order $r\ge 1$ with $\hat h\in \hat K \setminus K$, with $\fm$ now ranging over the elements of~$K^\times$ such that $v\fm\in v(\hat h - K)$. 

\section{Inverting Linear Differential Operators over Hardy Fields}\label{sec:IHF}

\noindent
Given a Hardy field $H$ and $A\in H[\der]$ we shall construe $A$ as
a $\C$-linear operator on various spaces of functions. 
We wish to construct right-inverses to  such operators.
A key assumption here is that $A$ splits over
$H[\imag]$. This reduces the
construction of such inverses mainly to the case of order $1$, and this case is handled in the first two subsections using suitable twisted integration operators.  In the third subsection we put things together and also show how to ``preserve reality''. In the last subsection we introduce damping factors.  Throughout we pay attention to the continuity of various operators with respect to various norms, for use in Section~\ref{sec:split-normal over Hardy fields}.    

\medskip
\noindent
We let $a$ range over $\R$ and $r$ over $\N\cup\{\infty,\omega\}$.  If $r\in \N$, then $r-1$ and $r+1$ have the usual meaning in $\Z$, while for $r\in \{ \infty,\omega\}$ we set $r-1= r+1:=r$. (This convention is just to avoid case distinctions.) We have the usual absolute value on~$\C$ given by~$|a+b\imag|=\sqrt{a^2+b^2}\in \R^{\ge}$ for $a,b\in \R$, so for
$f\in \c_a[\imag]$ we have $|f|\in \c_a$.

\subsection*{Integration and some useful norms}  
For $f\in \c_a[\imag]$ we define $\der_a^{-1}f \in \Cao[\imag]$ by
$$\der_a^{-1}f(t)\ :=\ \int_a^t f(s)\,ds\ :=\ \int_a^t \Re f(s)\, ds+ \imag\int_a^t \Im f(s)\, ds,$$
so $\der_a^{-1}f$  is the unique $g\in \Cao[\imag]$ such that $g'=f$ and $g(a)=0$.  
The integration operator $\der_a^{-1}\colon \c_a[\imag]\to  \Cao[\imag]$
is $\C$-linear and maps  $\Car[\imag]$ into $\Carm[\imag]$. 
For $f\in \c_a[\imag]$ we have
$$ 
\big|\der_a^{-1}f(t)\big|\le \big(\der_a^{-1}|f|\big)(t)\qquad\text{ for all $t\ge a$.}
$$ 
Let $f\in \c_a[\imag]$.  Call $f$ {\bf integrable at $\infty$\/} 
if $\lim_{t\to \infty} \int_a^t f(s)\,ds$ exists in $\C$. In that case we denote this
limit by $\int_a^{\infty} f(s)\,ds$ and put
$$ \int_{\infty}^a f(s)\,ds\ :=\  -\int_a^\infty f(s)\,ds,$$
and define $\der_{\infty}^{-1}f\in \Cao[\imag]$  by
$$\der_{\infty}^{-1}f(t)\ :=\ \int_{\infty}^t f(s)\,ds\ =\ \int_{\infty}^a f(s)\, ds +  \int_a^t  f(s)\, ds\ =\ \int_{\infty}^a f(s)\, ds +\der_a^{-1}f(t),$$
so  $\der_{\infty}^{-1}f$ is the unique $g\in \Cao[\imag]$  such that $g'=f$ and $\lim_{t\to \infty} g(t)=0$. Note that \label{p:Caint}
\begin{equation}\label{eq:integrable}
\c_a[\imag]^{\inte}\ :=\ \big\{f\in \c_a[\imag]:\ \text{$f$   is integrable at  $\infty$}\big\}
\end{equation}
is a $\C$-linear subspace of $\c_a[\imag]$ and that $\der_{\infty}^{-1}$ defines a $\C$-linear
operator from this subspace into $\Cao[\imag]$ which maps $\Car[\imag]\cap \c_a[\imag]^{\inte}$ into
$\Carm[\imag]$. If $f\in\c_a[\imag]$ and 
$g\in \c_a^{\inte}:=\c_a[\imag]^{\inte}\cap \c_a$ with $\abs{f}\leq g$ as germs in $\c$, then
$f\in \c_a[\imag]^{\inte}$; in particular,
if~$f\in\c_a[\imag]$ and $\abs{f}\in\c_a^{\inte}$, then 
$f\in \c_a[\imag]^{\inte}$. 

\medskip
\noindent
For $f\in \c_a[\imag]$ we set
$$ \|f\|_a\ :=\ \sup_{t\ge a} |f(t)|\ \in\ [0,\infty],$$
so (with $\b$ for ``bounded''): \label{p:Cab}
$$\c_a[\imag]^{\b}\ :=\ \big\{f\in \c_a[\imag]:\ \|f\|_a<\infty\big\}$$ is a $\C$-linear subspace of 
$\c_a[\imag]$, and $f\mapsto \|f\|_a$ is a norm on $\c_a[\imag]^{\b}$ making it a Banach space over $\C$.
It is also convenient to define for $t\ge a$ the seminorm \label{p:absa}
$$\|f\|_{[a,t]}\ :=\ \max_{a\le s\le t} |f(s)|$$ 
on $\c_a[\imag]$. More generally, let $r\in \N$. Then for $f\in \Car[\imag]$ we set
$$ \|f\|_{a;r}\ :=\ \max\big\{\|f\|_a, \dots, \|f^{(r)}\|_a\big\}\ \in\ [0,\infty],$$
so 
$$\Car[\imag]^{\b}:=\big\{f\in \Car[\imag]:\ \|f\|_{a;r}<\infty\big\}$$ 
is a $\C$-linear subspace of 
$\Car[\imag]$, and $f\mapsto \|f\|_{a;r}$ makes $\Car[\imag]^{\b}$ a normed vector space over $\C$. 
Note that for $f,g\in \Car[\imag]$ we have $\|fg\|_{a;r}\ \le\ 2^r\|f\|_{a;r}\|g\|_{a;r}$,
so $\Car[\imag]^{\b}$ is a subalgebra of the $\C$-algebra $\Car[\imag]$. 
If $f\in \Carm[\imag]$, then $f'\in \Car[\imag]$ with~$\|f'\|_{a;r}\le \|f\|_{a;r+1}$. 

\medskip
\noindent
With $\i=(i_0,\dots,i_r)$ ranging over $\N^{1+r}$,  let
$P=\sum_{\i} P_{\i} Y^{\i}$ (all $P_{\i}\in \c_a[\imag]$) be a polynomial in~$\c_a[\imag]\big[Y,Y',\dots,Y^{(r)}\big]$. 
For  $f\in \Car[\imag]$ we   set 
$$P(f)\ :=\ \sum_{\i} P_{\i} f^{\i}\in\c_a[\imag]\qquad\text{where $f^{\i}:=f^{i_0}(f')^{i_1}\cdots (f^{(r)})^{i_r}\in \c_a[\imag]$.}$$
We also let
$$\dabs{P}_a\ :=\ \max_{\i} \, \dabs{P_{\i}}_a \in [0,\infty].$$
Then  $\dabs{P}_a<\infty$ iff $P\in\c_a[\imag]^{\b}\big[Y,\dots,Y^{(r)}\big]$,
and $\dabs{\,\cdot\,}_a$ is a norm on the $\C$-linear space $\c_a[\imag]^{\b}\big[Y,\dots,Y^{(r)}\big]$.
In the following     assume $\dabs{P}_a<\infty$.
Then for $j=0,\dots,r$ such that $\partial P/\partial Y^{(j)}\neq 0$ we have
$$\dabs{\partial P/\partial Y^{(j)}}_a\  \leq\  (\deg_{Y^{(j)}}P)\cdot\dabs{P}_a.$$
Moreover:

\begin{lemma}\label{lem:bound on P(f)} 
If $P$ is homogeneous of degree $d\in\N$ and $f\in\Car[\imag]^{\b}$, then
$$\dabs{P(f)}_a\  \leq\  {d+r\choose r} \cdot \dabs{P}_a\cdot\dabs{f}_{a;r}^d.$$
\end{lemma}

\begin{cor}\label{cor:bound on P(f)}
Let $d\leq e$ in $\N$ be such that $P_{\i}=0$ whenever $\abs{\i}<d$ or $\abs{\i}>e$. Then
for~$f\in\Car[\imag]^{\b}$ we have
$$\dabs{P(f)}_a\ \leq\  D \cdot \dabs{P}_a\cdot \big(\dabs{f}_{a;r}^d+\cdots+\dabs{f}_{a;r}^e\big)$$
where $D=D(d,e,r):={e+r+1\choose r+1}-{d+r\choose r+1}\in\N^{\geq 1}$.  
\end{cor}

\noindent
Let $B\colon V \to \Car[\imag]^{\b}$ be a $\C$-linear map from a normed vector space $V$ over $\C$ into~$\Car[\imag]^{\b}$. Then we set 
$$\|B\|_{a;r}\ :=\  \sup\big\{\|B(f)\|_{a;r}:\ f\in V,\ \|f\| \le 1\big\}\ \in\ [0,\infty],$$
the {\bf operator norm of $B$}. 
Hence with the convention $\infty\cdot b:=b\cdot\infty:=\infty$ for~$b\in [0,\infty]$ we have 
$$\|B(f)\|_{a;r}\ \leq\ \|B\|_{a;r}\cdot \|f\|\qquad\text{for $f\in V$.}$$
Note that $B$ is continuous iff $\|B\|_{a;r}<\infty$. 
If the map $D\colon \Car[\imag]^{\b}\to \mathcal C_{a}^{s}[\imag]^{\b}$ ($s\in\N$) is also $\C$-linear, then
$$\| D\circ B \|_{a;s}\ \leq\ \| D \|_{a;s} \cdot \|B\|_{a;r}.$$
For $r=0$ we drop the subscript: $\|B\|_{a}:=\|B\|_{a;0}$.

\begin{lemma}\label{dermphi} Let $r\in \N^{\ge 1}$ and $\phi\in \Carl[\imag]^{\b}$. Then the $\C$-linear operator $$\der-\phi\ :\ \Car[\imag]\to \Carl[\imag], \quad f\mapsto f'-\phi f$$ maps
$\Car[\imag]^{\b}$ into $\Carl[\imag]^{\b}$, and its restriction
$\der-\phi\colon \Car[\imag]^{\b}\to \Carl[\imag]^{\b}$ is continuous with operator norm $\|\der-\phi\|_{a;r-1}\le 1+2^{r-1}\|\phi\|_{a;r-1}$.
\end{lemma}

\noindent
Let $r\in \N$, $a_0\in\R$, and let $a$ range over $[a_0,\infty)$. The 
$\C$-linear map  
$$f\mapsto  f|_{[a,+\infty)}\ \colon\ \mathcal C_{a_0}^r[\imag] \to \mathcal C_{a}^r[\imag]$$
satisfies $\|f|_{[a,+\infty)}\|_{a;r}\leq \|f\|_{a_0;r}$ for $f\in \mathcal C_{a_0}^r[\imag]$, so it maps
$\mathcal C_{a_0}^r[\imag]^{\b}$ into~$\Car[\imag]^{\b}$.
For~${f\in\mathcal C_{a_0}[\imag]}$ also denoting its germ at $+\infty$ and its restriction  $f|_{[a,+\infty)}$, we have: 
\begin{align*}
f\preceq 1 &\quad\Longleftrightarrow\quad \|f\|_a < \infty \ \text{for some $a$}  
\quad\Longleftrightarrow\quad  \|f\|_a < \infty \ \text{for all $a$,} \\
f\prec 1 & \quad\Longleftrightarrow\quad \|f\|_a\to 0\ \text{as $a\to\infty$.}
\end{align*}

\subsection*{Twisted integration} 
For  $f\in \c_a[\imag]$ we have the $\C$-linear operator $$g\mapsto fg\ :\ \c_a[\imag]\to \c_a[\imag],$$ also denoted by $f$. 
Fix a function $\phi\in \c_a[\imag]$. 

\begin{lemma}\label{lem:construct B}
Suppose for some $\varepsilon>0$ we have $\abs{\Re\phi(t)}\ge\varepsilon$ for all $t\ge a$.
Then there is a continuous operator~$B \colon \c_a[\imag]^{\b}\to \c_a[\imag]^{\b}$
with $\|B\|_a\le  \|1/\!\Re \phi \|_a$ such that $B\big(\c_a[\imag]^{\b}\big)\subseteq \Cao[\imag]$ and $B$
 is right-inverse to the operator $\der-\phi\colon\Cao[\imag]\to \c_a[\imag]$ in the sense that~$(\der-\phi)\circ B$ is the identity on~$\c_a[\imag]^{\b}$.
\end{lemma}
\begin{proof} Set $\Phi:= \der_a^{-1} \phi$, so
$\Phi\in \Cao[\imag]$, $\Phi(t)=\int_a^t \phi(s)\,ds$ for $t\ge a$, and $\Phi'=\phi$. Thus~$\ex^\Phi, \ex^{-\Phi}\in \Cao[\imag]$ with~$(\ex^\Phi)^\dagger = \phi$. Consider the $\C$-linear operator
$$B := \ex^{\Phi}\circ\ \der_a^{-1}\circ \ex^{-\Phi}\ \colon\ \c_a[\imag]\to \Cao[\imag],$$ so
$$ Bf(t)\ =\  \ex^{\Phi(t)}\int_a^t \ex^{-\Phi(s)}f(s)\, ds\quad \text{ for $f\in \c_a[\imag]$.}$$
It is easy to check that 
that~$(\der-\phi)\circ B$ is the identity on~$\c_a[\imag]$.
 Note that for
$f\in \c_a[\imag]$ we have $Bf(a)=0$, and thus $(Bf)'(a)=f(a)$, using $(Bf)'=f+\phi B(f)$.  
 Set $R:= \Re \Phi$ and $S:= \Im \Phi$, so $R, S\in \Cao$, $R'=\Re \phi$, $S'=\Im \phi$, and $R(a)=S(a)=0$. If~$\phi\in \Car[\imag]$, then $\ex^{\Phi}\in \Carm[\imag]$,  so $B$ maps
$\Car[\imag]$ into $\Carm[\imag]$. (So far we didn't use the hypothesis that for some $\varepsilon>0$ we have $\abs{\Re\phi(t)}\ge\varepsilon$ for all $t\ge a$.)

\medskip\noindent
Suppose $\epsilon>0$ and $\Re \phi(t)\le -\epsilon$ for all $t\ge a$. Then $-R$ has derivative $-R'(t)\ge \epsilon$ for all $t\ge a$, so $-R$ is strictly increasing with image
$[-R(a),\infty)=[0,\infty)$ and compositional inverse $(-R)^{\inv}\in \Coo$. Making 
the change of variables $-R(s)=u$ for~$s\ge a$, we obtain for $t\ge a$ and $f\in \c_a[\imag]$, and with $s:=(-R)^{\inv}(u)$,
\begin{align*} \int_a^t \ex^{-\Phi(s)}f(s) \,ds\ &=\ \int_{0}^{-R(t)} \ex^{-\Phi(s)}f(s)\frac{1}{-R'(s)}\, du,\ \text{ and thus}\\
 |Bf(t)|\ &\le\  \ex^{R(t)}\cdot \left(\int_{0}^{-R(t)}\ex^u\, du \cdot \|f\|_{[a,t]} \right)\cdot \left\|\frac{1}{\Re \phi}\right\|_{[a,t]} \\
 &=\ \big[1-\ex^{R(t)}\big]\cdot \|f\|_{[a,t]} \cdot \left\|\frac{1}{\Re \phi}\right\|_{[a,t]}\\   &\le\ \|f\|_{[a,t]} \cdot \left\|\frac{1}{\Re \phi}\right\|_{[a,t]}\
  \le\ \|f\|_a\cdot \left\|\frac{1}{\Re \phi}\right\|_a.
  \end{align*}
Thus $B$ maps $\c_a[\imag]^{\b}$ into $\c_a[\imag]^{\b}\cap \Cao[\imag]$ and its restriction to an operator on $\c_a[\imag]^{\b}$
(into $\c_a[\imag]^{\b}$) has operator norm $\le \big\|\frac{1}{\Re \phi}\big\|_a$.

\medskip\noindent
Next, suppose $\epsilon>0$ and $\Re \phi(t)\ge \epsilon$ for all $t\ge a$. Then $R'(t)\ge \epsilon$ for all $t\ge a$, so~$R(t) \ge \epsilon\cdot(t-a)$ for such $t$. Hence if
$f\in \c_a[\imag]^{\b}$, then $\ex^{-\Phi}f$ is integrable at~$\infty$.
Recall from \eqref{eq:integrable} that~$\c_a[\imag]^{\inte}$ is the $\C$-linear subspace of $\c_a[\imag]$ consisting of the~$g\in \c_a[\imag]$ that are integrable at $\infty$. Consider
the $\C$-linear maps $$f\mapsto \ex^{-\Phi}f \colon \c_a[\imag]^{\b} \to \c_a[\imag]^{\inte}, \qquad \der_{\infty}^{-1}\colon \c_a[\imag]^{\inte}\to \Cao[\imag],\ \quad f\mapsto \ex^{\Phi}f\colon \Cao[\imag]\to \Cao[\imag]. $$ Composition yields    
the $\C$-linear operator $B\colon \c_a[\imag]^{\b}  \to \Cao[\imag]$,  $$  Bf(t)\ :=\ \ex^{\Phi(t)}\int_\infty^t \ex^{-\Phi(s)}f(s)\, ds \qquad(f\in \c_a[\imag]^{\b}).$$
Then $(\der - \phi)\circ B$ is the identity on~$\c_a[\imag]^{\b}$. 
Note that $R$ is strictly increasing with image
$[0,\infty)$ and compositional inverse~${R^{\inv}\in \Coo}$. 
Making the change of variables $R(s)=u$ for $s\ge a$, we obtain for $t\ge a$ and~${f\in \c_a[\imag]^{\b}}$ with $s:=R^{\inv}(u)$,\begin{align*} \int_\infty^t \ex^{-\Phi(s)}f(s)\, ds\ &=\ -\int_{R(t)}^{\infty} \ex^{-\Phi(s)}f(s)\frac{1}{R'(s)}\, du,\ \text{ and thus}\\
 |Bf(t)|\ &\le\  \ex^{R(t)}\cdot \left( \int_{R(t)}^{\infty}\ex^{-u}\, du\right) \cdot \|f\|_t \cdot \left\|\frac{1}{\Re \phi}\right\|_t \\
  &\le\ \|f\|_t \cdot \left\|\frac{1}{\Re \phi}\right\|_t\ \le\ \|f\|_a \cdot \left\|\frac{1}{\Re \phi}\right\|_a.
  \end{align*}
Hence $B$ maps $\c_a[\imag]^{\b}$ into $\c_a[\imag]^{\b}\cap \Cao[\imag]$, and its restriction to an operator on
$\c_a[\imag]^{\b}$ into $\c_a[\imag]^{\b}$ has operator norm $\le \big\|\frac{1}{\Re \phi}\big\|_a$. 
If $\phi\in\Car[\imag]$, then $B$ maps~${\c_a[\imag]^{\b}\cap \Car[\imag]}$ into 
 $\c_a[\imag]^{\b}\cap\Carm[\imag]$. 
 \end{proof}

\noindent
The case that for some $\epsilon>0$ we have $\Re\phi(t)\le -\epsilon$ for all $t\ge a$ is called the {\em attractive case},   
and the case that for some $\epsilon>0$ we have $\Re\phi(t)\ge \epsilon$ for all~$t\ge a$ is called the {\em repulsive case}.  
In both cases the operator $B \colon \c_a[\imag]^{\b}\to \c_a[\imag]^{\b}$ provided by Lemma~\ref{lem:construct B} and its proof has the property that if $\phi\in\Car[\imag]$, then $B$ maps $\c_a[\imag]^{\b}\cap \Car[\imag]$ into   $\c_a[\imag]^{\b}\cap\Carm[\imag]$.  
We denote this operator $B$ by~$B_{\phi}$ if we need to indicate its dependence on $\phi$. Note also its dependence on $a$. 

\medskip
\noindent
Given a Hardy field $H$ and $f\in H[\imag]$ with  $\Re f \succeq 1$ we can choose $a$ and a representative of $f$ in  $\c_a[\imag]$, to be denoted also by $f$, such that $\Re f(t)\ne 0$ for all~$t\ge a$, and then $f\in \c_a[\imag]$ falls either under the attractive case or under the repulsive case. The original germ $f\in H[\imag]$ as well as the function $f\in \c_a[\imag]$ is accordingly said to be attractive, respectively repulsive.  (This agrees with the terminology of
\ref{def:repulsive}.)

\subsection*{Twists and right-inverses of linear operators over Hardy fields} 
Let $H$ be a Hardy field, $K:=H[\imag]$, and let $A\in K[\der]$ be a monic operator of
order $r\ge 1$, 
$$A\ =\ \der^r + f_1\der^{r-1}+\cdots + f_r, \qquad f_1,\dots, f_r\in K.$$ 
Take a real number $a_0$ and functions in $\c_{a_0}[\imag]$ that represent
the germs $f_1,\dots, f_r$ and to be denoted also by $f_1,\dots, f_r$. Whenever we increase below
the value of $a_0$, it is understood that we also update the functions $f_1,\dots, f_r$ accordingly, by restriction; the same holds for any function on $[a_0,\infty)$ that gets named. Throughout, $a$ ranges over $[a_0,\infty)$, and $f_1,\dots, f_r$ denote also the restrictions of these functions to $[a,\infty)$, and likewise for any function on $[a_0,\infty)$ that we name. Thus for any $a$ we have the $\C$-linear operator
$$A_a\ \colon\ \Car[\imag]\to\c_a[\imag], \quad y \mapsto y^{(r)} + f_1y^{(r-1)} + \cdots + f_ry.$$
Next, let $\fm\in H^\times$ be given. It gives rise to the twist 
$A_{\ltimes \fm}\in K[\der]$,  
$$A_{\ltimes \fm}\ :=\ \fm^{-1}A\fm\ =\ \der^r + g_1\der^{r-1}+\cdots + g_r, \qquad g_1,\dots, g_r\in K.$$
Now  [ADH, (5.1.1), (5.1.2), (5.1.3)] gives universal expressions for $g_1,\dots, g_r$ in terms of $f_1,\dots, f_r, \fm, \fm^{-1}$; for example, $g_1=f_1+r\fm^\dagger$. 
We represent the germ $\fm$ by a function 
in $(\Cazr)^\times$, also denoted by $\fm$. Let $\fm^{-1}$ likewise do double duty as the multiplicative inverse of $\fm$ in $\Cazr$. The expressions above can be used to show that the germs $g_1,\dots, g_r$ are represented by functions in $\c_{a_0}[\imag]$, to be denoted also by $g_1,\dots, g_r$, such that for all $a$ and all $y\in \Car[\imag]$ we have $$\fm^{-1} A_a(\fm y)\ =\ (A_{\ltimes \fm})_a(y), \text{ where }
  (A_{\ltimes \fm})_a(y)\ :=\  y^{(r)} + g_1y^{(r-1)} + \cdots + g_ry. $$
The operator $A_a\colon \Car[\imag]\to\c_a[\imag]$ is surjective:  see \cite[(10.6.3)]{Dieudonne} or \cite[\S{}19, I, II]{Walter}.  We aim to
construct a right-inverse of $A_a$ on the subspace $\c_a[\imag]^{\b}$ of $\c_a[\imag]$. 
For this, we assume given a splitting of $A$ over $K$, 
$$A\ =\ (\der-\phi_1)\cdots (\der-\phi_r), \qquad  \phi_1,\dots, \phi_r\in K.$$
Take functions in $\c_{a_0}[\imag]$, to be denoted also by $\phi_1,\dots, \phi_r$, that represent the germs $\phi_1,\dots, \phi_r$. We increase $a_0$ to arrange
$\phi_1,\dots, \phi_r\in \Cazrl[\imag]$. Note that for~$j=1,\dots,r$
the $\C$-linear map $\der-\phi_j\colon \Cao[\imag]\to \c_a[\imag]$ restricts to a
$\C$-linear map $A_j\colon \Caj[\imag]\to \Cajl[\imag]$, so that we obtain a map
$A_1\circ \cdots \circ A_r\colon \Car[\imag]\to\c_a[\imag]$. It is routine to verify  that for all sufficiently large $a$ we have
$$A_a\ =\ A_1\circ \cdots \circ A_r\ \colon\ \Car[\imag]\to\c_a[\imag].$$
We increase $a_0$ so that
$A_a=A_1\circ \cdots \circ A_r$ for all $a$. Note that $A_1,\dots, A_r$ depend on  $a$, but we prefer not to indicate this dependence notationally.

\medskip
\noindent
Now $\fm\in H^\times$ gives over $K$ the splitting
$$A_{\ltimes \fm}\ =\ (\der - \phi_1+\fm^\dagger)\cdots (\der-\phi_r+\fm^\dagger).$$
As before, the germ $\fm$ is represented by a function
$\fm\in (\Cazr)^\times$. With the usual notational conventions we 
have $\phi_j-\fm^\dagger\in \Cazrl[\imag]$, giving the $\C$-linear map 
$$\tilde A_j\  :=\  \der-(\phi_j-\fm^\dagger)\colon \Caj[\imag]\to \Cajl[\imag] \qquad (j=1,\dots,r),$$ which for all sufficiently large $a$ gives, just as for
$A_a$, a factorization
$$(A_{\ltimes \fm})_a\ =\ \tilde A_1\circ \cdots \circ \tilde A_r.$$
To construct a right-inverse of $A_a$ we now assume $\Re\phi_1,\dots, \Re\phi_r\succeq 1$.
Then we increase $a_0$ once more so that for all $t\ge a_0$,  $$\Re \phi_1(t),\dots, \Re \phi_r(t)\ne 0.$$
Recall that for $j=1,\dots,r$ we have
the continuous $\C$-linear operator 
$$B_j\ :=\ B_{\phi_j}\ \colon\ \c_a[\imag]^{\b}\to \c_a[\imag]^{\b}$$ from the previous subsection.   
The subsection on twisted integration now yields:
 
\begin{lemma}\label{cri} The continuous $\C$-linear operator
$$A_a^{-1}\ :=\  B_r\circ \cdots \circ B_1\ \colon\ \c_a[\imag]^{\b}\to \c_a[\imag]^{\b}$$ is a right-inverse of $A_a$: it maps $\c_a[\imag]^{\b}$ into $\c_a[\imag]^{\b}\cap \Car[\imag]$, and $A_a\circ A_a^{-1}$ is the identity on $\c_a[\imag]^{\b}$. For its operator norm we have $\|A_a^{-1}\|_a\ \le\  
\prod_{j=1}^r\big\|\frac{1}{\Re \phi_j}\big\|_a$.
\end{lemma}

\noindent 
Suppose $A$ is real in the sense that $A\in H[\der]$. Then by increasing $a_0$ we arrange that $f_1,\dots, f_{r}\in \c_{a_0}$.  Next, set 
$$\c_a^{\b}\ :=\ \c_a[\imag]^{\b}\cap\c_a \ =\  \big\{f\in \c_a:\, \|f\|_a<\infty\big\},$$
an $\R$-linear subspace of $\c_a$. Then the real part
$$\Re A_a^{-1}\ :\ \c_a^{\b} \to \c_a^{\b},\qquad (\Re A_a^{-1})(f)\ :=\ \Re\!\big(A_a^{-1}(f)\big)$$
is $\R$-linear and maps $\c_a^{\b}$ into $\Car$. Moreover, it is right-inverse to $A_a$ on $\c_a^{\b}$ in the sense that $A_a\circ \Re A_a^{-1}$ is the identity on $\c_a^{\b}$, and for $f\in \c_a^{\b}$,
$$\|(\Re A_a^{-1})(f)\|_a\ \le\ \|A_a^{-1}(f)\|_a.$$

\subsection*{Damping factors} Here $H$, $K$, $A$, $f_1,\dots, f_r$, $\phi_1,\dots, \phi_r$, $a_0$ are as in Lemma~\ref{cri}.
In particular, $r\in \N^{\ge 1}$, $\Re\phi_1,\dots, \Re \phi_r\succeq 1$, and $a$ ranges over $[a_0,\infty)$.
For later use we choose damping factors $u$ 
to make the operator $uA_a^{-1}$ more manageable than~$A_a^{-1}$.
For $j=0,\dots,r$ we set 
\begin{equation}\label{eq:Ajcirc}
A_j^{\circ}\ :=\ A_1\circ \cdots \circ A_j\ \colon\ \Caj[\imag]\to \c_a[\imag],
\end{equation} 
with $A_0^{\circ}$ the identity on $\c_a[\imag]$
and $A_r^\circ=A_a$, and
\begin{equation}\label{eq:Bjcirc}
B_j^\circ\ :=\ B_j\circ \cdots \circ B_1\ \colon\ \c_a[\imag]^{\b}  \to \c_a[\imag]^{\b},
\end{equation}
where $B_0^{\circ}$ is the identity
on $\c_a[\imag]^{\b}$ and $B_r^\circ=A_a^{-1}$. Then $B_j^\circ$ maps $\c_a[\imag]^{\b}$ in\-to~${{\c_a[\imag]^{\b}}\cap { \Caj[\imag]}}$ and $A_j^\circ \circ B_j^\circ$ is the identity on $\c_a[\imag]^{\b}$ by Lemma~\ref{cri}.

\begin{lemma}\label{teq} Let $u\in \Car[\imag]^\times$. Then for $i=0,\dots,r$ and $f\in \c_a[\imag]^{\b}$,
\begin{equation}\label{eq:derivatives of uA^-1}
\big[u\cdot A_a^{-1}(f)\big]^{(i)}\ =\ \sum_{j=r-i}^r u_{i,j} \cdot u\cdot B_j^{\circ}(f) \quad  \text{ in $\Carmi[\imag]$}
\end{equation}
with coefficient functions $u_{i,j}\in \Carmi[\imag]$ given by $u_{i,r-i}=1$, and for $0\le i< r$,
$$ u_{i+1,j}\ =\ \begin{cases}
u_{i,r}'+ u_{i,r}(u^\dagger + \phi_r)	& \text{if $j=r$,} \\
u_{i,j}'+ u_{i,j}(u^\dagger + \phi_j) + u_{i,j+1} & \text{if  $r-i\le j< r$.}\end{cases}$$
\end{lemma} 
\begin{proof} Recall that for $j=1,\dots,r$ and $f\in \c_a[\imag]^{\b}$ we have
$B_j(f)' = f + \phi_jB_j(f)$. 
It is obvious that \eqref{eq:derivatives of uA^-1} holds for $i=0$. Assuming \eqref{eq:derivatives of uA^-1} for a certain $i< r$ we get
$$\big[uA_a^{-1}(f)\big]{}^{(i+1)}\ =\  \sum_{j=r-i}^r u_{i,j}' \cdot uB_j^{\circ}(f) + \sum_{j=r-i}^r u_{i,j} \cdot \big[uB_j^{\circ}(f)\big]',$$ 
 and for $j=r-i,\dots, r$,
$$\big[uB_j^{\circ}(f)\big]'\ =\ u'B_j^{\circ}(f) + u\cdot \big[B_j^{\circ}(f)\big]'\
                    =\ u^\dagger\cdot uB_j^{\circ}(f) + 
                    uB_{j-1}^{\circ}(f) + \phi_j uB_j^{\circ}(f),$$
which gives the desired result. 
\end{proof} 

\noindent
Let $\fv\in \Cazr$ be such that $\fv(t)>0$ for all $t\ge a_0$, $\fv\in H$, $\fv\prec 1$.  Then we have the convex subgroup
 $$\Delta\ :=\ \big\{\gamma\in v(H^\times):\ \gamma=o(v\fv)\big\}$$ 
of $v(H^\times)$. {\em We assume that 
$\phi_1,\dots, \phi_r\preceq_{\Delta} \fv^{-1}$ in the asymptotic field $K$, where~$\phi_j$ and
 $\fv$ also denote their germs.}
For real~$\nu>0$ we have
$\fv^\nu\in (\Cazr)^\times$, so 
$$u\ :=\ \fv^\nu|_{[a,\infty)}\in (\Car)^\times, \qquad \|u\|_a<\infty.$$
In the next proposition $u$ has this meaning, a meaning which accordingly  varies with $a$.    
Recall that $A_a^{-1}$ maps $\c_a[\imag]^{\b}$ into $\c_a[\imag]^{\b}\cap \Car[\imag]$
with $\|A_a^{-1}\|_a<\infty$. 

{\samepage \begin{prop}\label{uban}  Assume $H$ is real closed and $\nu\in \Q$, $\nu > r$. Then:
\begin{enumerate}
\item[\rm(i)] the $\C$-linear operator $u A_a^{-1}\colon \c_a[\imag]^{\b} \to \c_a[\imag]^{\b}$ maps $\c_a[\imag]^{\b}$ into $\Car[\imag]^{\b}$; 
\item[\rm(ii)] $u A_a^{-1}\colon \c_a[\imag]^{\b} \to \Car[\imag]^{\b}$ is continuous;
\item[\rm(iii)] there is a real constant $c\ge 0$ such that $\|u A_a^{-1}\|_{a;r}\le c$ for all $a$;
\item[\rm(iv)] for all $f\in \c_a[\imag]^{\b}$ we have $uA_a^{-1}(f) \preceq \fv^{\nu}\prec 1$; 
\item[\rm(v)]  $\|u A_a^{-1}\|_{a;r}\to 0$ as $a\to \infty$.
\end{enumerate} 
\end{prop}}
\begin{proof} Note that
$\fv^\dagger\preceq_{\Delta} 1$ by [ADH, 9.2.10(iv)]. Denoting the germ of $u$ also by $u$ we have 
$u\in H$ and $u^\dagger=\nu\fv^\dagger \preceq_{\Delta} 1$, in particular, $u^\dagger\preceq \fv^{-1/2}$. Note that the
$u_{i,j}$ from Lemma~\ref{teq}---that is, their germs---lie in $K$. 
Induction on $i$  gives $u_{i,j}\preceq_{\Delta} \fv^{-i}$ for $r-i\le j\le r$. Hence $uu_{i,j}\preceq_{\Delta} \fv^{\nu-i}\prec_{\Delta} 1$ for $r-i\le j\le r$. Thus for $i=0,\dots,r$ we have a real constant
$$c_{i,a}\ :=\ \sum_{j=r-i}^r \|u\,u_{i,j}\|_a \cdot \|B_j\|_a\cdots \|B_1\|_a\in [0,\infty)  $$
with $\big\|\big[uA_a^{-1}(f)\big]{}^{(i)}\big\|_a\le c_{i,a}\|f\|_a$ for all $f\in \c_a[\imag]^{\b}$. 
Therefore $uA_a^{-1}$ maps $\c_a[\imag]^{\b}$ into $\Car[\imag]^{\b}$, and the operator $u A_a^{-1}\colon \c_a[\imag]^{\b} \to \Car[\imag]^{\b}$ is continuous with $$\|uA_a^{-1}\|_{a;r}\ \le\ c_a:=\max\{c_{0,a},\dots,c_{r,a}\}.$$
As to (iii), this is because for all $i$,~$j$, $\|u\,u_{ij}\|_a$ is decreasing as a function
of $a$, and $\|B_j\|_a\le \big\|\frac{1}{\Re \phi_j}\big\|_a$ for all $j$. For $f\in \c_a[\imag]^{\b}$ we have $A_a^{-1}(f)\in \c_a[\imag]^{\b}$, so  (iv) holds. As to (v),  $u\,u_{i,j}\prec 1$ gives $\|uu_{ij}\|_a\to 0$ as $a\to \infty$, for all $i$,~$j$. In view of~$\|B_j\|_a\le \big\|\frac{1}{\Re \phi_j}\big\|_a$ for all $j$, this gives
$c_{i,a}\to 0$ as $a\to \infty$ for $i=0,\dots,r$, so~$c_a\to 0$ as $a\to\infty$. 
\end{proof}


\section{Solving Split-Normal Equations over Hardy Fields}\label{sec:split-normal over Hardy fields}

\noindent
We construct here
solutions of suitable algebraic differential equations over Hardy fields.
These solutions lie in rings $\Car[\imag]^{\b}$ ($r\in \N^{\ge 1}$) and are obtained as
fixed points of certain contractive maps, as is common in solving differential equations. Here we use that~$\Car[\imag]^{\b}$ is a Banach space with respect to the norm $\|\cdot\|_{a;r}$. It will take some effort to
define the right contractions using the operators from Section~\ref{sec:IHF}.  

\medskip\noindent
In this section $H$, $K$, $A$, $f_1,\dots, f_r$, $\phi_1,\dots, \phi_r$, $a_0$ are as in Lemma~\ref{cri}.
In particular, $H$ is a Hardy field, $K=H[\imag]$, and
$$A=(\der-\phi_1)\cdots(\der-\phi_r)\qquad \text{where $r\in\N^{\ge 1}$, $\phi_1,\dots,\phi_r\in K$, $\Re\phi_1,\dots, \Re\phi_r\succeq 1$.}$$
Here $a_0$ is chosen so that we have representatives for $\phi_1,\dots, \phi_r$  in $\Cazrl[\imag]$, denoted also by $\phi_1,\dots,\phi_r$.
We let $a$ range over $[a_0,\infty)$.
In addition we assume that $H$ is real closed, and that
we are given a germ $\fv\in H^{>}$ such that~$\fv\prec 1$ and
$\phi_1,\dots, \phi_r\preceq_{\Delta} \fv^{-1}$ for the convex subgroup
$$\Delta\ :=\ \big\{\gamma\in v(H^\times):\ \gamma=o(v\fv)\big\}$$
of $v(H^\times)$. We increase $a_0$ so
that $\fv$ is represented by a function in $\Cazr$, also denoted by $\fv$, with $\fv(t)>0$ for all $t\ge a_0$.

\subsection*{Constructing fixed points} Consider a differential equation
\begin{equation}\label{eq:ADE}\tag{$\ast$}
A(y)\ =\ R(y),\qquad y\prec 1,
\end{equation}
where $R\in K\{Y\}$ has order $\le r$, degree $\le d\in \N^{\ge 1}$ and weight $\le w\in \N^{\ge r}$, with
$R\prec_{\Delta}\fv^w$. Now $R=\sum_{\j}R_{\j}Y^{\j}$ with $\j$ ranging here and below over the tuples~$(j_0,\dots, j_r)\in \N^{1+r}$  with $|\j|\le d$ and $\|\j\|\le w$; likewise for $\i$.
For each $\j$ we take a function in $\c_{a_0}[\imag]$
that  represents the germ $R_{\j}\in K$ and let $R_{\j}$ denote
this function as well as its restriction to any $[a,\infty)$.
Thus $R$ is represented on $[a,\infty)$ by a polynomial
$\sum_{\j}R_{\j}Y^{\j}\in \c_a[\imag]\big[Y, \dots, Y^{(r)}\big]$, to be denoted
also by $R$ for simplicity. This yields for each $a$ an evaluation map
$$f\mapsto R(f):=\sum_{\j}R_{\j}f^{\j}\ :\  \Car[\imag]\to \c_a[\imag].$$
As in [ADH, 4.2] we also have for every $\i$ the formal partial derivative
$$ R^{(\i)}\ :=\ \frac{\partial^{|\i|}R}{\partial^{i_0}Y\cdots \partial^{i_r}Y^{(r)}}\ \in\ \c_a[\imag]\big[Y,\dots, Y^{(r)}\big]$$ 
with $R^{(\i)}=\sum_{\j} R_{\j}^{(\i)}Y^{\j}$,  all $R_{\j}^{(\i)}\in \c_a[\imag]$ having their germs in $K$. 

\medskip
\noindent
A {\em solution of \eqref{eq:ADE} on $[a,\infty)$}\/\index{solution!split-normal equation} is a function $f\in \Car[\imag]^{\b}$ such that  $A_a(f)=R(f)$ and $f\prec 1$.  
One might try to obtain a solution of \eqref{eq:ADE} as a fixed point of the operator~$f\mapsto A_a^{-1}\big(R(f)\big)$, but this operator might fail to be contractive
on a useful space of functions. Therefore we twist $A$ and arrange things so that we can use Proposition~\ref{uban}. 
In the rest of this section we fix $\nu\in \Q$ with $\nu > w$ (so $\nu > r$) such that 
$R\prec_{\Delta}\fv^\nu$ and~$\nu\fv^\dagger\not\sim \Re \phi_j$ in $H$
for $j=1,\dots,r$. (Note that such $\nu$ exists.) Then the twist 
$\tilde A:=A_{\ltimes\fv^\nu}=\fv^{-\nu} A\fv^{\nu}\in K[\der]$ splits over $K$ as follows: 
\begin{align*} \tilde A\ &=\ (\der -\phi_1+\nu\fv^\dagger)\cdots (\der-\phi_r+\nu\fv^\dagger), \quad \text{ with }\\ 
\phi_j-\nu\fv^{\dagger}\ &\preceq_{\Delta}\ \fv^{-1}, \quad \Re\phi_j-\nu\fv^\dagger\ \succeq\ 1 \qquad(j=1,\dots,r).
\end{align*}
We also increase $a_0$ so that  
$\Re\phi_j(t)-\nu\fv^\dagger(t)\ne 0$ for all $t\ge a_0$
and such that for all $a$ and $u:=\fv^\nu|_{[a,\infty)}\in (\Car)^\times$ the operator $\tilde{A}_a\colon \Car[\imag]\to \c_a[\imag]$ satisfies $$\tilde{A}_a(y)\ =\ u^{-1}A_a(uy)\qquad(y\in \Car[\imag]).$$
(See the explanations before Lemma~\ref{cri} for definitions of
$A_a$ and $\tilde{A}_a$.) We now increase $a_0$ once more,  fixing it for the rest of the section except in the subsection ``Preserving reality'', so as to obtain as in Lemma~\ref{cri}, with $\tilde{A}$ in the role of $A$, a 
right-inverse $\tilde{A}_a^{-1}\colon \c_a[\imag]^{\b} \to \c_a[\imag]^{\b}$ for such~$\tilde{A}_a$. 

\begin{lemma} \label{bdua} We have a continuous operator \textup{(}not necessarily $\C$-linear\textup{)} $$\Xi_a\ :\ \Car[\imag]^{\b}\to \Car[\imag]^{\b},\quad
f\mapsto u\tilde{A}_a^{-1}\big(u^{-1}R(f)\big).$$
It has the property that $\Xi_a(f)\preceq \fv^ \nu\prec 1$ 
and $A_a\big(\Xi_a(f)\big)=R(f)$
for all $f\in \Car[\imag]^{\b}$. 
\end{lemma} 
\begin{proof} We have $\|u^{-1}R_\i\|_a<\infty$ for all $\i$, so
$u^{-1}R(f)=\sum_{\i}u^{-1}R_{\i}f^{\i}\in \c_a[\imag]^{\b}$ for all~$f\in \Car[\imag]^{\b}$, and thus $u\tilde{A}_a^{-1}\big(u^{-1}R(f)\big)\in \Car[\imag]^{\b}$  for such $f$, by Proposition~\ref{uban}(i).
Continuity of $\Xi_a$ follows from Proposition~\ref{uban}(ii) and continuity of~$f\mapsto u^{-1}R(f)\colon \Car[\imag]^{\b} \to \c_a[\imag]^{\b}$.
For $f\in \Car[\imag]^{\b}$ we have  $\Xi_a(f)\preceq \fv^\nu\prec 1$ by Proposition~\ref{uban}(iv),  and {\samepage
$$ u^{-1}A_a\big(\Xi_a(f)\big)\ =\ u^{-1}A_a\big[u\tilde{A}_a^{-1}\big(u^{-1}R(f)\big)\big]\ =\ \tilde{A}_a\big[\tilde{A}_a^{-1}\big(u^{-1}R(f)\big)\big]\ =\ u^{-1}R(f),$$
so $A_a\big(\Xi_a(f)\big)=R(f)$. } 
\end{proof}

\noindent
By Lemma~\ref{bdua},  each $f\in \Car[\imag]^{\b}$ with $\Xi_a(f)=f$  is a solution of \eqref{eq:ADE} on $[a,\infty)$.

\begin{lemma}\label{bdua, bds} 
There is a constant $C_a\in\R^{\geq}$ such that for all $f,g\in \Car[\imag]^{\b}$, 
$$ \|\Xi_a(f+g)-\Xi_a(f) \|_{a;r}\ \le\ C_a\cdot \max\!\big\{1, \|f\|_{a;r}^d\big\}\cdot\big(\|g\|_{a;r} + \cdots + \|g\|_{a;r}^d\big).$$
We can take these $C_a$ such that $C_a\to 0$ as $a\to \infty$, and we do so below. 
\end{lemma}
\begin{proof}
Let $f,g\in \Car[\imag]^{\b}$. We have the Taylor expansion
$$R(f+g)\ =\ \sum_{\i} \frac{1}{\i !}R^{(\i)}(f)g^{\i}\ =\ \sum_{\i}\frac{1}{\i !}\bigg[\sum_{\j}R_{\j}^{(\i)}f^\j\bigg]g^{\i}.$$
Now for all $\i$,~$\j$ we have $R^{(\i)}_{\j}\prec_{\Delta} \fv^\nu$ in $K$, so  
$u^{-1}R_{\j}^{(\i)}\prec 1$.  Hence
$$D_a\ :=\ \sum_{\i,\j} \big\|u^{-1}R^{(\i)}_{\j}\big\|_a\ \in\ [0,\infty)$$
has the property that $D_a\to 0$ as $a\to \infty$, and
$$\big\|u^{-1}\big(R(f+g)-R(f)\big)\big\|_a\ \le\ D_a\cdot\max\!\big\{1, \|f\|_{a;r}^d\big\}\cdot \big(\|g\|_{a;r} + \cdots + \|g\|_{a;r}^d\big).$$
So $h:= u^{-1}\big(R(f+g)-R(f)\big)\in \c_a[\imag]^{\b}$ gives $\Xi_a(f+g)-\Xi_a(f)=u\tilde{A}_a^{-1}(h)$, and 
$$ \|\Xi_a(f+g)-\Xi_a(f) \|_{a;r}= \|u\tilde{A}_a^{-1}(h) \|_{a;r} \le  \|u\tilde{A}_a^{-1} \|_{a;r}\cdot\|h\|_a.$$ 
Thus the lemma holds for $C_a:= \|u\tilde{A}_a^{-1} \|_{a;r}\cdot D_a$. 
\end{proof}


\noindent 
In the proof of the next theorem we  use the well-known fact that the normed vector space $\Car[\imag]^{\b}$ over $\C$ is actually a Banach space: it is the case $\w=1$ of the more general Lemma~\ref{lem:w-conv} below.
Thus if~$S\subseteq \Car[\imag]^{\b}$ is  nonempty and closed and $\Phi\colon S\to S$ is
contractive (that is, there is a $\lambda\in [0,1)$ such that~${\|\Phi(f)-\Phi(g)\|_{a;r}\leq \lambda \|f-g\|_{a;r}}$
for all~$f,g\in S$), then~$\Phi$ has a unique fixed point~$f_0$, and $\Phi^n(f)\to f_0$ as~$n\to\infty$,  for every $f\in S$. (See for example \cite[Ch.~II, \S{}5, IX]{Walter}.)

\begin{theorem} \label{thm:fix} For all sufficiently large $a$ the operator $\Xi_a$ maps the closed ball $$\big\{f\in \Car[\imag]:\ \|f\|_{a;r}\le 1/2\big\}$$
of the Banach space $\Car[\imag]^{\b}$ into itself and has a unique fixed point on this ball. 
\end{theorem}
\begin{proof} We have $\Xi_a(0)=u\tilde{A}_a^{-1}(u^{-1}R_0)$, so $\|\Xi_a(0)\|_{a;r}\le \|u\tilde{A}_a^{-1}\|_{a;r}\|u^{-1}R_0\|_{a}$. Now~$\|u^{-1}R_0\|_a\to 0$ as $a\to \infty$, hence by Proposition~\ref{uban}(iii) we can take $a$ 
so large that~$\|u\tilde{A}_a^{-1}\|_{a;r}\|u^{-1}R_0\|_{a}\le \frac{1}{4}$. 
For $f$,~$g$ in the closed ball above we have by Lemma~\ref{bdua, bds},
$$ \|\Xi_a(f)-\Xi_a(g) \|_{a;r}\ =\  \|\Xi_a(f+(g-f))-\Xi_a(f) \|_{a;r}\ \le \ C_a\cdot d\|f-g\|_{a;r}.$$
Take $a$ so  large that also $C_a d \le \frac{1}{2}$. Then $\|\Xi_a(f)-\Xi_a(g)\|_{a;r}\le \frac{1}{2}\|f-g\|_{a;r}$. 
Applying this to $g=0$ we see that $\Xi_a$ maps the closed ball above to itself. Thus~$\Xi_a$ has a unique
fixed point on this ball. 
\end{proof}

\noindent
If $\deg R\leq 0$ (so $R=R_0$), then $\Xi_a(f)=u\tilde{A}_a^{-1}(u^{-1}R_0)$ is independent of~$f\in \Car[\imag]^{\b}$, so for sufficiently large $a$, the fixed point~$f\in \Car[\imag]^{\b}$ of $\Xi_a$ with~${\dabs{f}_{a;r}\leq 1/2}$ is~$f=\Xi_a(0)=u\tilde{A}_a^{-1}(u^{-1}R_0)$. 

\medskip
\noindent
Next we investigate the difference between solutions of \eqref{eq:ADE} on $[a_0,\infty)$:

\begin{lemma}\label{lem:close} Suppose $f,g\in \Cazr[\imag]^{\b}$ and $A_{a_0}(f)=R(f)$, $A_{a_0}(g)=R(g)$. Then there are positive reals 
$E$,~$\epsilon$ such that for all $a$ 
there exists an $h_a\in \Car[\imag]^{\b}$ with the property that
for $\theta_a:=(f-g)|_{[a,\infty)}$, 
$$A_{a}(h_a)=0, \quad \theta_a-h_a\prec \fv^w, \quad
\|\theta_a-h_a\|_{a;r}\ \le\ E\cdot \|\fv^{\epsilon}\|_{a}\cdot \big(\|\theta_a\|_{a;r}+\cdots + \|\theta_a\|_{a;r}^d \big),$$
and thus $h_a\prec 1$ in case $f-g\prec 1$.   
\end{lemma}
\begin{proof} 
Set $\eta_a:= A_{a}(\theta_a)=R(f)-R(g)$, where $f$ and $g$ stand for their restrictions to $[a,\infty)$. From $R\prec \fv^{\nu}$ we get
$u^{-1}R(f)\in \c_a[\imag]^{\b}$ and $u^{-1}R(g)\in \c_a[\imag]^{\b}$, so
$u^{-1}\eta_a\in \c_a[\imag]^{\b}$. By Proposition~\ref{uban}(i),(iv) we have
$$\xi_a\ :=\ u\tilde{A}_a^{-1}(u^{-1}\eta_a)\in \Car[\imag]^{\b}, \qquad \xi_a\prec \fv^w.$$ 
Now $\tilde{A}_{a}(u^{-1}\xi_a)=u^{-1}\eta_a$, that is,
$A_{a}(\xi_a)=\eta_a$. Note that then $h_a:=\theta_a-\xi_a$ satisfies 
$A_{a}(h_a)=0$.
Now $\xi_a=\theta_a-h_a$ and $\xi_a=\Xi_a(g+\theta_a)-\Xi_a(g)$, hence   by  Lemma~\ref{bdua, bds} and its proof,
\begin{align*} \|\theta_a-h_a\|_{a;r}\ =\ \|\xi_a\|_{a;r}\ &\le\  C_a\cdot \max\!\big\{1,\|g\|_{a;r}^d\big\}\cdot \big(\|\theta\|_{a;r}+\cdots + \|\theta\|_{a;r}^d\big), \text{ with}\\
 C_a\ &:=\ \big\|u\tilde{A}_a^{-1}\big\|_{a;r}\cdot \sum_{\i,\j} \big\|u^{-1}R^{(\i)}_{\j}\big\|_a.
\end{align*}
Take a real $\epsilon>0$ such that $R\prec \fv^{\nu+\epsilon}$. This gives
a real $e>0$ such that $\sum_{\i,\j} \big\|u^{-1}R^{(\i)}_{\j}\big\|_a\le e\|\fv^\epsilon\|_a$ for all $a$. Now use Proposition~\ref{uban}(iii).
\end{proof}

\noindent
The situation we have in mind in the lemma above is that $f$ and $g$ are close at infinity, in the sense that $\|f-g\|_{a;r}\to 0$ as $a\to \infty$.
Then the lemma yields solutions of $A(y)=0$ that are {\em very\/} close to $f-g$ at infinity.  
However,  being very close at infinity as stated in Lemma~\ref{lem:close}, namely $\theta_a-h_a\prec\fv^w$ and the rest, is too weak for later use.
We take up this issue again in Section~\ref{sec:weights} below. 
(In Corollary~\ref{cor:h=0 => f=g} later in the present section we already show: if $f\neq g$ as germs, then~$h_a\neq 0$ for sufficiently large $a$.)

\subsection*{Preserving reality} We now assume in addition that $A$ and $R$ are real, that is,~${A\in H[\der]}$ and $R\in H\{Y\}$.  It is not clear that the fixed points constructed in the proof of Theorem~\ref{thm:fix} are then also real. Therefore we slightly modify this construction using real parts.  
We first apply the discussion following Lemma~\ref{cri} to $\tilde{A}$ as well as to $A$, increasing $a_0$ so  that
for all $a$ the  $\R$-linear real part~${\Re\tilde{A}_a^{-1}  \colon  \c_a^{\b} \to \c_a^{\b}}$ maps~$\c_a^{\b}$ into $\Car$ and is right-inverse to $\tilde{A}_a$ on $\c_a^{\b}$, with  
$$\big\|(\Re\tilde{A}_a^{-1})(f)\big\|_a\le \big\|\tilde{A}_a^{-1}(f)\big\|_a\qquad\text{ for all  $f\in \c_a^{\b}$.}$$ 
Next we set
$$(\Car)^{\b}\ :=\ \big\{f\in \Car:\ \|f\|_{a;r}<\infty\big\}\ =\ \Car[\imag]^{\b}\cap \Car,$$
which is a real Banach space with respect to $\|\cdot\|_{a;r}$. 
Finally, this increasing of $a_0$ is done so that the original $R_{\j}\in \c_{a_0}[\imag]$ restrict to updated functions $R_{\j}\in \c_{a_0}$. 
For all $a$, take $u$, $\Xi_a$  as in Lemma~\ref{bdua}. This lemma has the following real analogue as a consequence:

\begin{lemma} \label{realbdua} The operator 
$$\Re\, \Xi_a\ : \ (\Car)^{\b}\to (\Car)^{\b},\quad
f\mapsto \Re\!\big(\Xi_a(f)\big)$$
satisfies $(\Re\, \Xi_a)(f)\preceq\fv^{\nu}$ for $f\in (\Car)^{\b}$, and 
any fixed point  of $\Re\,\Xi_a$ is a solution of \eqref{eq:ADE} on $[a,\infty)$.
\end{lemma} 

\noindent
Below the constants $C_a$ are as in Lemma~\ref{bdua, bds}.

\begin{lemma} \label{realbdua, bds}
For $f,g\in (\Car)^{\b}$,
$$\big\|(\Re\, \Xi_a)(f+g)-(\Re \Xi_a)(f)\big\|_{a;r}\ \le\ C_a\cdot \max\!\big\{1, \|f\|_{a;r}^d\big\}\cdot \big(\|g\|_{a;r} + \cdots + \|g\|_{a;r}^d\big).$$
\end{lemma}


\noindent
The next corollary is derived from Lemma~\ref{realbdua, bds} in the same way as Theorem~\ref{thm:fix} from Lem\-ma~\ref{bdua, bds}:

\begin{cor}\label{cor:fix} For all sufficiently large $a$ the operator $\Re\,\Xi_a$ maps the closed ball $$\big\{f\in \Car:\ \|f\|_{a;r}\le 1/2\big\}$$
of the Banach space $(\Car)^{\b}$ into itself and has a unique fixed point on this ball. 
\end{cor}

\noindent
We also have a real analogue of Lemma~\ref{lem:close}:

\begin{cor} Suppose $f,g\in (\Cazr)^{\b}$ and $A_{a_0}(f)=R(f)$, $A_{a_0}(g)=R(g)$. Then there are positive reals $E$,~$\epsilon$ such that
for all $a$ there exists an $h_a\in (\Car)^{\b}$ with the property that
for $\theta_a:=(f-g)|_{[a,\infty)}$, 
$$A_{a}(h_a)=0, \quad \theta_a-h_a\prec \fv^{w},\quad
\|\theta_a-h_a\|_{a;r}\ \le\ E\cdot \|\fv^{\epsilon}\|_{a}\cdot \big(\|\theta_a\|_{a;r}+\cdots + \|\theta_a\|_{a;r}^d \big).$$
\end{cor}
\begin{proof}
Take $h_a$ to be the real part of an $h_a$ as in Lemma~\ref{lem:close}. 
\end{proof}

\subsection*{Some useful bounds}
To prepare for Section~\ref{sec:weights} we derive in this subsection some bounds from Lem\-mas~\ref{bdua, bds} and~\ref{lem:close}. Throughout we assume $d,r\in \N^{\ge 1}$. 
We begin with an easy inequality:

\begin{lemma}\label{lem:inequ power d}
Let  $(V,\dabs{\, \cdot\, })$ be a normed $\C$-linear space, and $f,g\in V$. Then
$$\dabs{f+g}^d\ 	\leq\  
2^d\cdot \max\!\big\{1,\dabs{f}^d\big\}\cdot\max\!\big\{1,\dabs{g}^d\big\}.$$
\end{lemma}
\begin{proof}
Use that 
$\dabs{f+g}   \leq \dabs{f}+\dabs{g} \leq 2 \max\!\big\{1,\dabs{f}\big\}\cdot\max\!\big\{1,\dabs{g}\big\}$. 
\end{proof}

\noindent
Now let  $u$, $\Xi_a$ be as in Lem\-ma~\ref{bdua}. By that lemma, the operator
$$\Phi_a\ \colon\ \Car[\imag]^{\b}\times\Car[\imag]^{\b} \to \Car[\imag]^{\b}, \quad (f,y)\mapsto \Xi_a(f+y)-\Xi_a(f)$$ 
is continuous. 
Furthermore $\Phi_a(f,0)=0$ for  $f\in \Car[\imag]^{\b}$ and
\begin{equation}\label{eq:Phia difference}
\Phi_a(f,g+y)-\Phi_a(f,g)\ =\ \Phi_a(f+g,y) \qquad\text{for $f,g,y\in \Car[\imag]^{\b}$.}
\end{equation}

\begin{lemma}\label{lem:2.1 summary}
There are $C_a,C_a^+\in\R^{\geq}$ such that for all $f,g,y\in\Car[\imag]^{\b}$,
\begin{equation}\label{eq:2.1, 1}
\dabs{\Phi_a(f,y)}_{a;r}  
 \ \le\ C_a\cdot \max\!\big\{1, \|f\|_{a;r}^d\big\}\cdot \big(\|y\|_{a;r} + \cdots + \|y\|_{a;r}^d\big),
\end{equation}
\vskip-1.5em
\begin{multline}\label{eq:2.1, 2}
\dabs{\Phi_a(f,g+y)-\Phi_a(f,g)}_{a;r}  \ \leq\ \\ C_a^+\cdot\max\!\big\{1,\dabs{f}_{a;r}^d\big\}\cdot\max\!\big\{1, \|g\|_{a;r}^d\big\}\cdot \big(\|y\|_{a;r} + \cdots + \|y\|_{a;r}^d\big).
\end{multline}
We can take these $C_a,C_a^+$ such that $C_a, C_a^+\to 0$ as $a\to\infty$, and do so below.
\end{lemma}

\begin{proof}
The $C_a$ as in Lemma~\ref{bdua, bds} satisfy the requirements on the $C_a$ here.
Now let~$f,g,y\in\Car[\imag]^{\b}$. Then
by \eqref{eq:Phia difference} and \eqref{eq:2.1, 1} we have
$$ \dabs{\Phi_a(f,g+y)-\Phi_a(f,g)}_{a;r}  \ \leq\ C_a \cdot\max\!\big\{1, \|f+g\|_{a;r}^d\big\}\cdot \big(\|y\|_{a;r} + \cdots + \|y\|_{a;r}^d\big).$$
Thus by Lemma~\ref{lem:inequ power d}, 
$C_a^+ := 2^d\cdot C_a$ has the required property.
\end{proof}

\noindent
Next, let $f$, $g$ be as in the hypothesis of Lemma~\ref{lem:close} and take $E$, $\epsilon$, and $h_a$ (for each~$a$) as in its conclusion. 
Thus for all $a$ and $\theta_a:=(f-g)|_{[a,\infty)}$, 
$$\|\theta_a-h_a\|_{a;r}\ \le\ E\cdot \|\fv^{\epsilon}\|_{a}\cdot \big(\|\theta_a\|_{a;r}+\cdots + \|\theta_a\|_{a;r}^d \big),$$
and if $f-g\prec 1$, then $h_a\prec 1$. So
$$\|\theta_a-h_a\|_{a;r}\ \le\ E\cdot \|\fv^{\epsilon}\|_{a}\cdot \big(\rho+\cdots+\rho^d\big),\quad \rho:=\dabs{f-g}_{a_0;r}.$$
We let
$$B_a\ :=\ \big\{ y\in\Car[\imag]^{\b}:\ \dabs{y-h_a}_{a;r} \leq 1/2\big\}$$
be the  closed ball  of radius $1/2$ around $h_a$ in $\Car[\imag]^{\b}$.
Using $\fv^\epsilon\prec 1$ we take $a_1\geq a_0$  so that  $\theta_a\in B_a$ for all $a\geq a_1$.
Then for $a\geq a_1$ we have $$\dabs{h_a}_{a;r}\  \leq\ \dabs{h_a-\theta_a}_{a;r} + \dabs{\theta_a}_{a;r}\ \leq\ \textstyle\frac{1}{2}+\rho,$$
and hence for $y\in B_a$, 
\begin{equation}\label{eq:dabs(y)}
\dabs{y}_{a;r}\ \leq\ \dabs{y-h_a}_{a;r} + \dabs{h_a}_{a;r}\ \leq\ \textstyle\frac{1}{2}+\big(\frac{1}{2}+\rho\big)\ =\ 1+\rho.
\end{equation}
Consider now the continuous operators
$$\Phi_a,\Psi_a\ :\ \Car[\imag]^{\b} \to \Car[\imag]^{\b},\qquad
\Phi_a(y):=\Xi_a(g+y)-\Xi_a(g), \quad \Psi_a(y):=\Phi_a(y)+h_a.$$
In the notation introduced above, $\Phi_a(y)=\Phi_a(g,y)$ for $y\in  \Car[\imag]^{\b}$. With $\xi_a$ as in the proof of Lemma~\ref{lem:close}
we also have~$\Phi_a(\theta_a)=\xi_a$ and  
$\Psi_a(\theta_a)=\xi_a+h_a=\theta_a$. Below we reconstruct the fixed point $\theta_a$ of $\Psi_a$ from $h_a$,  for sufficiently large $a$.   

\begin{lemma}\label{lem:Psin, b}
There exists $a_2\geq a_1$ such that for all $a\geq a_2$ we have
$\Psi_a(B_a)\subseteq B_a$, and $\dabs{{\Psi_a(y)-\Psi_a(z)}}_{a;r} \leq \frac{1}{2} \|y-z\|_{a;r}$ for all $y,z\in B_a$.
\end{lemma}
\begin{proof}  Take $C_a$ as in  Lemma~\ref{lem:2.1 summary}, and let $y\in B_a$. Then  by \eqref{eq:2.1, 1},
\begin{align*}  \|\Phi_a(y) \|_{a;r}\ &\le\ C_a\cdot \max\!\big\{1, \|g\|_{a;r}^d\big\}\cdot \big(\|y\|_{a;r} + \cdots + \|y\|_{a;r}^d\big),  \text{ so by}\ \eqref{eq:dabs(y)},  \\
\|\Psi_a(y)-h_a \|_{a;r}\  &\leq\ C_a M,\quad M:=\max\!\big\{1, \|g\|_{a_0;r}^d\big\}\cdot \big( (1+\rho)+\cdots+(1+\rho)^d\big).
\end{align*} 
Recall that $C_a\to 0$ as $a\to\infty$. Suppose $a\ge a_1$ is so large that $C_a M \le 1/2$. Then~$\Psi_a(B_a)\subseteq B_a$.
With $C_a^+$ as in Lemma~\ref{lem:2.1 summary}, \eqref{eq:2.1, 2} gives for $y,z\in\Car[\imag]^{\b}$,
\begin{multline*}
\dabs{\Phi_a(y)-\Phi_a(z)}_{a;r} \leq \\ C^+_a\cdot\max\!\big\{1, \|g\|_{a;r}^d\big\}\cdot\max\!\big\{1, \|z\|_{a;r}^d\big\}\cdot \big(\|y-z\|_{a;r} + \cdots + \|y-z\|_{a;r}^d\big).
\end{multline*}
{\samepage Hence with $N := \max\!\big\{1, \|g\|_{a_0;r}^d\big\}\cdot (1+\rho)^d\cdot d$ we obtain for $y,z\in B_a$ that
$$\dabs{\Psi_a(y)-\Psi_a(z)}_{a;r}\  \leq\ C^+_a N   \|y-z\|_{a;r}, $$
so
$\dabs{{\Psi_a(y)-\Psi_a(z)}}_{a;r} \leq \frac{1}{2} \|{y-z}\|_{a;r}$ if $C_a^+N\leq 1/2$. }
\end{proof}

\noindent
Below $a_2$ is as in Lemma~\ref{lem:Psin, b}. Thus $\theta_a\in B_a$ for $a\ge a_2$. 

\begin{cor}\label{cor:Psin, b}
If $a\geq a_2$, then  $\lim_{n\to\infty} \Psi_a^n(h_a)=\theta_a$ in $\Car[\imag]^{\b}$.
\end{cor}

\begin{proof} Let $a\ge a_2$. Then $\Psi_a$ has a unique fixed point on $B_a$. As $\Psi_a(\theta_a)=\theta_a\in B_a$, this fixed point is~$\theta_a$.
\end{proof}

\begin{cor}\label{cor:h=0 => f=g}
If $f\neq g$ as germs, then $h_a\neq 0$ for  $a\geq a_2$.
\end{cor}
\begin{proof}
Let $a\geq a_2$. Then $\lim_{n\to \infty} \Psi_a^n(h_a)=\theta_a$. If $h_a=0$, then $\Psi_a=\Phi_a$, and hence $\theta_a=0$, since $\Phi_a(0)=0$.
\end{proof}

\section{Smoothness Considerations}\label{sec:smoothness}

\noindent
We assume $r\in \N$ in this section.
We prove here as much smoothness of solutions of algebraic differential equations over Hardy fields as could be hoped for. 
In particular, the solutions in $\Car[\imag]^{\b}$ of the equation \eqref{eq:ADE} in Section~\ref{sec:split-normal over Hardy fields} actually have their germs in $\mathcal{C}^{<\infty}[\imag]$.
To make this precise, consider 
a ``differential'' polynomial $$P\  =\ P(Y,\dots,Y^{(r)})\ \in\ \c^n[\imag]\big[Y,\dots,Y^{(r)}\big].$$ We put {\em differential\/} in quotes since $\c^n[\imag]$  is not naturally 
a differential ring. Nevertheless, $P$ defines an obvious evaluation map
$$f  \mapsto P\big(f,\dots,f^{(r)}\big)\ :\  \Gr[\imag] \to \c[\imag].$$ We also have the ``separant'' of $P$:
$$S_P\ :=\ \frac{\partial P}{\partial Y^{(r)}}\ \in\ \c^n[\imag]\big[Y,\dots,Y^{(r)}\big].$$
  
\begin{prop}\label{hardysmooth} 
Assume $n\ge 1$. Let $f \in \Gr[\imag]$ be such that  
$$P\big(f,\dots,f^{(r)}\big) = 0\in \c[\imag]\quad\text{ and }\quad S_P\big(f,\dots,f^{(r)}\big)\in\c[\imag]^\times.$$ 
Then~$f \in \c^{n+r}[\imag]$. 
Thus if $P\in  \Gi[\imag]\big[Y,\dots,Y^{(r)}\big]$, then $f\in \Gi[\imag]$.
Moreover, if~$P\in  \Ginf[\imag]\big[Y,\dots,Y^{(r)}\big]$, then~$f\in\Ginf[\imag]$,  and likewise with $\Gom[\imag]$ in place of~$\Ginf[\imag]$.
\end{prop}

\noindent
We deduce this from the lemma below, which has a complex-analytic hypothesis. Let~${U\subseteq \R\times \C^{1+r}}$ be open. Let $t$ range over $\R$ and
$z_0,\dots, z_r$ over $\C$,  and set~$x_j:=\Re z_j$, $y_j:=\Im z_j$ for $j=0,\dots,r$, and
$$U(t,z_0,\dots,z_{r-1})\ :=\ \big\{z_r:(t,z_0,\dots,z_{r-1},z_r)\in U\big\},$$ 
an open subset of $\C$. Assume $\Phi\colon U \to \C$  and $n\ge 1$ are such that 
$\Re \Phi, \Im \Phi\colon U \to \R$
are $\c^n$-functions of
$(t, x_0,y_0,\dots,x_r,y_r)$, and for all~$t,z_0,\dots,z_{r-1}$ the function 
$$z_r\mapsto \Phi(t,z_0,\dots,z_{r-1},z_r)\ \colon\ U(t,z_0,\dots,z_{r-1})\to \C$$ is holomorphic (the complex-analytic hypothesis alluded to).

\begin{lemma}\label{lem:hardysmooth, complex}
{\samepage Let $I\subseteq \R$ be a nonempty open interval and suppose~$f\in\Gr(I)[\imag]$ is such that for all~$t\in I$, 
\begin{itemize}
\item $\big(t, f(t),\dots, f^{(r)}(t)\big)\in U$;
\item $\Phi\big(t, f(t),\dots, f^{(r)}(t)\big)=0$; and
\item $(\partial\Phi/\partial z_r)\big(t, f(t),\dots, f^{(r)}(t)\big)\ne 0$.
\end{itemize}
Then $f\in \c^{n+r}(I)[\imag]$.}
\end{lemma}
\begin{proof}
Set $A:= \Re \Phi,\ B:= \Im \Phi$ and $g:= \Re f,\ h:= \Im f$. Then for all $t\in I$, 
\begin{align*} A\big(t, g(t), h(t), g'(t),h'(t)\dots, g^{(r)}(t), h^{(r)}(t)\big)\ &=\ 0,\\
B\big(t, g(t), h(t), g'(t),h'(t)\dots, g^{(r)}(t), h^{(r)}(t)\big)\ &=\ 0.
\end{align*}
Consider the $\c^n$-map $(A,B)\colon U \to \R^2$, with $U$ identified in the usual way with an open subset of $\R^{1+2(1+r)}$. The Cauchy-Riemann equations give
$$\frac{\partial \Phi}{\partial z_r}\ =\ \frac{\partial A}{\partial x_r}+ \imag\frac{\partial B}{\partial x_r}, \qquad \frac{\partial A}{\partial x_r}\ =\ \frac{\partial B}{\partial y_r}, \qquad  \frac{\partial B}{\partial x_r}\ =\ -\frac{\partial A}{\partial y_r}.$$
Thus the Jacobian matrix of the map $(A,B)$ with respect to its last two variables~$x_r$ and $y_r$ has determinant
$$ D\ =\ \left(\frac{\partial A}{\partial x_r}\right)^2 +  \left(\frac{\partial B}{\partial x_r}\right)^2\ =\ \left|\frac{\partial \Phi}{\partial z_r}\right|^2\ :\ U\to \R.$$
Let $t_0\in I$. Then $$D\big(t_0, g(t_0), h(t_0),\dots, g^{(r)}(t_0), h^{(r)}(t_0)\big)\ne 0,$$ so by the Implicit Mapping Theorem \cite[(10.2.2), (10.2.3)]{Dieudonne} we have a connected open neighborhood~$V$ of the point
$$\big(t_0, g(t_0),h(t_0),\dots, g^{(r-1)}(t_0), h^{(r-1)}(t_0)\big)\in \R^{1+2r},$$ 
open intervals $J,K\subseteq\R$ containing $g^{(r)}(t_0)$, $h^{(r)}(t_0)$, respectively, and a $\c^n$-map
$$(G,H)\colon V \to J\times K$$
such that $V\times J\times K\subseteq U$ and the zero set of $\Phi$ on $V\times J\times K$ equals the graph of~$(G,H)$.
Take an open subinterval $I_0$ of $I$ with $t_0\in I_0$ such that for all $t\in I_0$, 
$$ \big(t, g(t), h(t),g'(t), h'(t),\dots, g^{(r-1)}(t), h^{(r-1)}(t), g^{(r)}(t), h^{(r)}(t)\big)\in V\times J\times K.$$
Then the above gives that for all $t\in I_0$ we have
\begin{align*} 
g^{(r)}(t)\ =\ G&\big(t, g(t), h(t),g'(t), h'(t),\dots, g^{(r-1)}(t), h^{(r-1)}(t)\big),\\
h^{(r)}(t)\ =\ H&\big(t, g(t), h(t),g'(t), h'(t),\dots, g^{(r-1)}(t), h^{(r-1)}(t)\big).
\end{align*} 
It follows easily from these two equalities that $g,h$ are of class
$\c^{n+r}$ on $I_0$.
\end{proof}

\noindent
Let $f$ continue to be as in Lemma~\ref{lem:hardysmooth, complex}.  If $\Re\Phi$, $\Im\Phi$ are $\Ginf$, then by taking $n$ arbitrarily high we conclude that $f\in \Ginf(I)[\imag]$. Moreover:

\begin{cor}\label{cor:hardysmooth, complex}
If  $\Re\Phi$, $\Im\Phi$ are  real-analytic, then~$f\in \Gom(I)[\imag]$. 
\end{cor}
\begin{proof}
Same as that of Lemma~\ref{lem:hardysmooth, complex}, with  the reference to~\cite[(10.2.3)]{Dieudonne} replaced  by~\cite[(10.2.4)]{Dieudonne} to obtain that $G$, $H$ are real-analytic, and  noting that then the last displayed relations for $t\in I_0$ force $g$, $h$ to be real-analytic on $I_0$ by~\cite[(10.5.3)]{Dieudonne}.
\end{proof}

\begin{lemma}\label{smo} 
Let $I\subseteq\R$ be a nonempty open interval, $n\ge 1$, 
and $$P\  =\ P\big(Y,\dots,Y^{(r)}\big)\ \in\ \Gn(I)[\imag]\big[Y,\dots,Y^{(r)}\big].$$ 
Let $f\in \Gr(I)[\imag]$ be such that 
$$P\big(f,\dots,f^{(r)}\big) = 0\in \c(I)[\imag]\quad\text{ and }\quad
(\partial P/\partial Y^{(r)})\big(f, \dots, f^{(r)}\big) \in \c(I)[\imag]^\times.$$ 
Then $f \in \c^{n+r}(I)[\imag]$. 
Moreover, if $P\in\Ginf(I)[\imag]\big[Y,\dots,Y^{(r)}\big]$, then $f\in \Ginf(I)[\imag]$, and likewise with
$\Gom(I)[\imag]$ in place of $\Ginf(I)[\imag]$.
\end{lemma} 
\begin{proof}
Let $P=\sum_{\i} P_{\i} Y^{\i}$ where all $P_{\i}\in \Gn(I)[\imag]$. Set
$U:=I\times\C^{1+r}$, and consider the map $\Phi\colon U\to\C$ given by 
$$\Phi(t,z_0,\dots,z_r):=\sum_{\i} P_{\i}(t) z^{\i}\quad\text{where $z^{\i}:=z_0^{i_0}\cdots z_r^{i_r}$ for $\i=(i_0,\dots,i_r)\in\N^{1+r}$.}$$
From Lemma~\ref{lem:hardysmooth, complex}  we obtain  $f\in\c^{n+r}(I)[\imag]$. In view of Corollary~\ref{cor:hardysmooth, complex} and the remark preceding it,  and replacing~$n$ by~$\infty$ respectively $\omega$, this argument also gives the second part of the lemma. 
\end{proof}

\noindent
Proposition~\ref{hardysmooth} follows from Lemma~\ref{smo} by taking suitable representatives of the germs involved. 
Let now $H$ be a Hardy field and $P\in H[\imag]\{Y\}$ of order~$r$. Then~$P\in \Gi[\imag]\big[Y,\dots,Y^{(r)}\big]$, and so $P(f):=P\big(f,\dots,f^{(r)}\big)\in\c[\imag]$ for $f\in\Gr[\imag]$
 as explained in the beginning of this section.


\medskip\noindent
For notational convenience we set
$$\c^n[\imag]^{\preceq} := \big\{ f\in \c^n[\imag]: f,f',\dots,f^{(n)}\preceq 1\big\},\qquad (\c^n)^{\preceq}:=\c^n[\imag]^{\preceq}\cap\c^n,$$
and likewise with $\prec$ instead of $\preceq$.
Then $\c^n[\imag]^{\preceq}$ is a $\C$-subalgebra of $\c^n[\imag]$ and $(\c^n)^{\preceq}$ is
an $\R$-subalgebra of $\c^n$. Also, $\c^n[\imag]^{\prec}$ is an ideal of $\c^n[\imag]^{\preceq}$,
and likewise with~$\c^n$ instead of $\c^n[\imag]$.
We have $\c^n[\imag]^{\preceq}\supseteq \c^{n+1}[\imag]^{\preceq}$ and 
$(\c^n)^{\preceq}\supseteq (\c^{n+1})^{\preceq}$, and likewise with $\prec$ instead of $\preceq$.
For $R\in \c^n[\imag]\big[Y,\dots,Y^{(r)}\big]$  we put $R_{\ge 1}:=R-R(0)$.

\begin{cor}\label{cor:ADE smooth}
Suppose  
$$P=Y^{(r)} + f_1Y^{(r-1)}+ \cdots + f_rY - R\quad\text{ with $f_1,\dots,f_r$ in~$H[\imag]$ and $R_{\ge 1}\prec 1$.}$$  
Let $f \in \mathcal{C}^r[\imag]^{\preceq}$ be such that $P(f) = 0$. Then $f \in \Gi[\imag]$. Moreover, if~$H\subseteq
\Ginf$, then $f\in \Ginf[\imag]$, and if~$H\subseteq \Gom$, then~$f\in\Gom[\imag]$.
\end{cor}
\begin{proof} We have $S_P=\frac{\partial P}{\partial Y^{(r)}}=1-S$ with $S:=\frac{\partial R_{\ge 1}}{\partial Y^{(r)}}\prec 1$
and thus $$S_P(f,\dots,f^{(r)})\ =\ 1-S(f,\dots,f^{(r)}), \qquad S(f,\dots,f^{(r)})\prec 1,$$   so
$S_P(f,\dots,f^{(r)})\in\c[\imag]^\times$.  Now appeal to Proposition~\ref{hardysmooth}.
\end{proof}

\noindent
Thus the germ of any solution on $[a,\infty)$ of the asymptotic equation \eqref{eq:ADE} of Section~\ref{sec:split-normal over Hardy fields} lies in $\Gi[\imag]$, and even in $\Ginf[\imag]$ (respectively $\Gom[\imag]$) if $H$ is in addition a
$\Ginf$-Hardy field (respectively, a $\Gom$-Hardy field).  

\medskip
\noindent
For the differential subfield $K:=H[\imag]$ of the differential ring~$\Calinf[\imag]$ we have: 

\begin{cor} 
Suppose $f\in\Calinf[\imag]$, $P(f)=0$, and $f$ generates a differential subfield $K\langle f\rangle$ of $\Calinf[\imag]$ over $K$.
If $H$ is a $\Ginf$-Hardy field, then~${K\langle f\rangle\subseteq\Ginf[\imag]}$, and likewise with
$\Gom$ in place of $\Ginf$.
\end{cor}
\begin{proof}
Suppose $H$ is a $\Ginf$-Hardy field; it suffices to show $f\in\Ginf[\imag]$.
We may assume that $P$ is a minimal annihilator of $f$ over $K$; then $S_P(f)\neq 0$ in $K\langle f\rangle$
and so $S_P(f)\in\c[\imag]^\times$. Hence the claim follows from Proposition~\ref{hardysmooth}.
\end{proof}

\noindent
With $H$ replacing $K$ in this proof we obtain the  ``real'' version:

\begin{cor} \label{realginfgom} 
Suppose $f\in\Calinf$ is hardian over $H$ and $P(f)=0$ for some $P$ in~$H\{Y\}^{\ne}$.   
Then $H\subseteq \Ginf\ \Rightarrow\ f\in \Ginf$, and  $H\subseteq \Gom\ \Rightarrow\ f\in \Gom$.
\end{cor}

\noindent
This leads to: 

\begin{cor}\label{cor:Hardy field ext smooth}
Suppose $H$ is a $\Ginf$-Hardy field. Then every $\d$-algebraic Hardy field extension of $H$ is a $\Ginf$-Hardy field; in particular, $\Dx(H)\subseteq\Ginf$. Likewise with~$\Ginf$ replaced by $\Gom$. 
\end{cor}

\noindent
In particular, $\Dx(\Q)\subseteq\Gom$~\cite[Theorems~14.3, 14.9]{Boshernitzan82}.
Let now $H$ be a $\Ginf$-Hardy field.  Then by Corollary~\ref{cor:Hardy field ext smooth},
$H$ is $\d$-maximal iff~$H$ has no proper $\d$-algebraic $\Ginf$-Hardy field extension; thus  every $\Ginf$-maximal Hardy field is $\d$-maximal, and $H$ has a $\d$-maximal $\d$-algebraic $\Ginf$-Hardy field extension.   The same remarks apply with $\omega$ in place of $\infty$.

\section{Application to Filling Holes in Hardy Fields}\label{secfhhf}

 \noindent
This section combines the analytic material above with results from~\cite{ADH4}, some of it summarized in
Lemma~\ref{lem:achieve strong splitting}.  {\em Throughout
$H$ is a Hardy field with $H\not\subseteq\R$, and~${r\in \N^{\ge 1}}$}.
Thus~$K:=H[\imag]\subseteq \Calinf[\imag]$ is an $H$-asymptotic extension of $H$.
(Later we impose extra assumptions on~$H$ like being real closed with asymptotic integration.) 
Note that~${v(H^\times)\ne \{0\}}$: take~$f\in H\setminus \R$; then $f'\ne 0$, and if $f\asymp 1$, then~$f'\prec 1$. 

\subsection*{Evaluating differential polynomials at germs} 
Any $Q\in K\{Y\}$ of order $\le r$ can be evaluated at any
germ $y\in \Calr[\imag]$ to give a germ $Q(y)\in \mathcal{C}[\imag]$,
with~$Q(y)\in \mathcal{C}$ for~$Q\in H\{Y\}$ of order $\le r$ and $y\in \Calr$. 
(See the beginning of Section~\ref{sec:smoothness}.)
Here is a variant that we shall need. Let~$\phi\in H^\times$; with~$\der$ denoting the derivation
of $K$, the derivation
of the differential field~$K^\phi$ is then~$\derdelta:= \phi^{-1}\der$. We also let $\derdelta$ denote its ex\-ten\-sion~$f\mapsto \phi^{-1}f'\colon \mathcal{C}^1[\imag] \to \mathcal{C}[\imag]$, which maps $\mathcal{C}^{n+1}[\imag]$ into $\c^n[\imag]$ and
$\c^{n+1}$ into~$\c^n$, for all~$n$. Thus for 
$j\le r$ we have the maps  
$$ \mathcal{C}^r[\imag]\ \xrightarrow{\ \  \derdelta\ \ }\ \mathcal{C}^{r-1}[\imag]\ \xrightarrow{\ \  \derdelta\ \ }\ \cdots\ \xrightarrow{\ \  \derdelta\ \ }\ \mathcal{C}^{r-j+1}[\imag]\ \xrightarrow{\ \  \derdelta\ \ }\  \mathcal{C}^{r-j}[\imag],$$
which by composition yield $\derdelta^j\colon \mathcal{C}^r[\imag]\to \mathcal{C}^{r-j}[\imag]$, mapping $\Calr$ into $\mathcal{C}^{r-j}$. This allows us to define for
$Q\in K^\phi\{Y\}$ of order $\le r$ and $y\in \mathcal{C}^r[\imag]$
the germ $Q(y)\in \mathcal{C}[\imag]$ by
$$Q(y) :=q\big(y, \derdelta(y),\dots, \derdelta^r(y)\big)\quad\text{ where $Q=q\big(Y,\dots, Y^{(r)})\in K^\phi\big[Y,\dots, Y^{(r)}\big]$.}$$ Note that  $H^\phi$ is a  differential subfield of $K^\phi$, and if $Q\in H^\phi\{Y\}$ is of order $\le r$ and~$y\in \Calr$, then $Q(y)\in \mathcal{C}$.

\begin{lemma}\label{lem:small derivatives of y}
Let $y\in\c^r[\imag]$ and $\fm\in K^\times$. Each of the following conditions implies~$y\in \c^r[\imag]^{\preceq}$:
\begin{enumerate}
\item[\textup{(i)}]   $\phi\preceq 1$ and 
$\derdelta^{0}(y),\dots,\derdelta^{r}(y)\preceq 1$; 
\item[\textup{(ii)}] $\fm\preceq 1$ and $y\in \fm\,\c^r[\imag]^{\preceq}$.
\end{enumerate}
Moreover, if $\fm\preceq 1$ and $(y/\fm)^{(0)},\dots,(y/\fm)^{(r)}\prec 1$, then  $y^{(0)},\dots,y^{(r)}\prec 1$.
\end{lemma}
\begin{proof}
For (i), use the smallness of the derivation of $H$ and the transformation formulas in [ADH, 5.7] expressing the iterates of $\der$ in terms of iterates of $\derdelta$. 
For (ii) and the ``moreover'' part, set $y=\fm z$ with
$z=y/\fm$ and use the Product Rule and the smallness of the derivation of $K$.  
\end{proof}

\subsection*{Equations over Hardy fields and over their complexifications}
Let~${\phi>0}$ be active in~$H$. 
For example, if $H$ contains the iterated logarithms $\ell_0,\ell_1,\dots$
defined recursively by $\ell_0=x$, $\ell_{n+1}=\log\ell_n$, then
 for each $n$, the element 
 $$ \ell_n^\dagger\ =\ \ell_{n+1}'\ =\ (\ell_0\cdots\ell_n)^{-1}$$ of $H^>$ is active in $H$.
 We recall here from \cite[Section~3]{ADH5} how the
 the asymptotic field~$K^\phi=H[\imag]^\phi$ (with derivation $\derdelta=\phi^{-1}\der$) 
 is isomorphic to the asymptotic field $K^\circ:=H^\circ[\imag]$ for a certain Hardy field $H^\circ$: Let
$\ell\in\mathcal C^1$ be such that $\ell'=\phi$; then~$\ell>\R$, $\ell\in \mathcal{C}^{<\infty}$, and $\ell^{\inv}\in \mathcal{C}^{<\infty}$ for the compositional inverse~$\ell^{\inv}$ of $\ell$.
The $\C$-algebra automorphism~$f\mapsto f^\circ:= f\circ \ell^{\inv}$ of $\mathcal{C}[\imag]$ (with inverse $g\mapsto g\circ\ell$) maps~$\Caln[\imag]$ and~$\Caln$
onto themselves, and hence restricts to a $\C$-algebra automorphism of $\Calinf[\imag]$ and~$\Calinf$ mapping
$\Calinf$ onto itself. 
Moreover,
\begin{equation}\label{eq:derc}\tag{$\der, \circ,\derdelta$} (f^\circ)'\ =\  (\phi^\circ)^{-1}(f')^\circ\ =\ \derdelta(f)^\circ \qquad\text{for $f\in \mathcal{C}^1[\imag]$.}
\end{equation}
Thus we have an isomorphism $f\mapsto f^\circ: (\Calinf[\imag])^\phi\to\Calinf[\imag]$ of differential rings, and likewise with $\Calinf$ in place of $\Calinf[\imag]$.
Then
$$H^\circ\ :=\ \{h^\circ:h\in H \}\  \subseteq\ \Calinf$$ is a Hardy field,
and $f\mapsto f^\circ$ restricts to  an isomorphism $H^\phi \to H^\circ$ of pre-$H$-fields, and to
an isomorphism $K^\phi\to K^\circ$ of asymptotic fields.
We extend the latter to the isomorphism
$$Q\mapsto Q^\circ\ \colon\ K^\phi\{Y\} \to K^\circ\{Y\}$$
of differential rings given by $Y^\circ=Y$, which restricts to a differential ring isomorphism $H^\phi\{Y\}\to H^\circ\{Y\}$.  Using the identity \eqref{eq:derc} it is routine to check that
for~$Q\in K^\phi\{Y\}$ of order $\le r$ and $y\in \Calr[\imag]$, we have
$Q(y)^\circ\  =\  (Q^\circ)(y^\circ)$.
This allows us to translate algebraic differential equations over $K$ 
 into algebraic differential equations over~$K^\circ$: Let $P\in K\{Y\}$ have order $\le r$ and let $y\in \Calr[\imag]$.

\begin{lemma}\label{lem:as equ comp} $P(y)^\circ=P^\phi(y)^\circ=P^{\phi\circ}(y^\circ)$  where $P^{\phi\circ}:=(P^\phi)^\circ\in K^\circ\{Y\}$, hence
$$P(y)=0\ \Longleftrightarrow\ P^{\phi\circ}(y^\circ)=0.$$
\end{lemma}

\noindent
Moreover, $y\prec \fm\ \Longleftrightarrow\ y^\circ \prec \fm^\circ$, for $\fm\in K^\times$, so asymptotic side conditions are automatically taken care of under this ``translation''. Also, if $\phi\preceq 1$ and $y^\circ\in\c^r[\imag]^{\preceq}$, then  $y\in\c^r[\imag]^{\preceq}$, by Lemma~\ref{lem:small derivatives of y}(i) and \eqref{eq:derc}.

\medskip\noindent
{\it In the rest of this section $H\supseteq\R$ is real closed with asymptotic integration.}\/
Then $H$ is an $H$-field, and $K=H[\imag]$ is the algebraic closure of $H$, a $\d$-valued field  with small derivation extending $H$,
constant field $\C$, and value group~$\Gamma:=v(K^\times)=v(H^\times)$.

\subsection*{Slots in Hardy fields and compositional conjugation}
{\it In this sub\-sec\-tion we let~${\phi>0}$ be active in $H$; as in the previous subsection we take $\ell\in\mathcal C^1$ such that~$\ell'=\phi$ and use the superscript $\circ$ accordingly: $f^\circ:= f\circ \ell^{\inv}$ for $f\in \c[\imag]$.}\/

Let  $(P,\fm,\hat a)$ be a slot in $K$ of order $r$, so $\hat a\notin K$ is an element of an immediate asymptotic extension $\hat K$ of~$K$ with~$P\in Z(K,\hat a)$ and $\hat a\prec\fm$.  
We associate to~$(P,\fm,\hat a)$ a slot in $K^\circ$ as follows:
choose an immediate asymptotic extension $\hat K^\circ$ of $K^\circ$ and an
isomorphism $\hat f\mapsto \hat f^\circ\colon\hat K^\phi\to \hat K^\circ$ of asymptotic fields
extending the isomorphism $f\mapsto f^\circ\colon K^\phi\to K^\circ$.
Then $(P^{\phi\circ},\fm^\circ,\hat a^\circ)$ is a slot in $K^\circ$ of the same complexity as~$(P,\fm,\hat a)$. The
equivalence class of   $(P^{\phi\circ},\fm^\circ,\hat a^\circ)$ does not depend on the choice of
$\hat K^\circ$ and the isomorphism $\hat K^\phi\to \hat K^\circ$. 
If $(P,\fm,\hat a)$ is a hole in $K$, then~$(P^{\phi\circ},\fm^\circ,\hat a^\circ)$ is a hole in~$K^\circ$, and likewise with ``minimal hole'' in place of ``hole''. By \cite[Lemmas~3.1.20, 3.3.20, 3.3.40]{ADH4} we have:

\begin{lemma} \label{lem:Pphicirc, 1}
If $(P,\fm,\hat a)$ is $Z$-minimal, then so is $(P^{\phi\circ},\fm^\circ,\hat a^\circ)$, and likewise with ``quasilinear'' and ``special'' in place of ``$Z$-minimal''. 
If $(P,\fm,\hat a)$ is steep and~${\phi\preceq 1}$, then $(P^{\phi\circ},\fm^\circ,\hat a^\circ)$ is steep,
and likewise with   ``deep'', ``normal'',  and ``strictly normal'' in place of ``steep''. 
\end{lemma}

\noindent
Next, let $(P,\fm,\hat a)$ be a slot in $H$ of order $r$, so $\hat a\notin H$ is an element of an immediate asymptotic extension $\hat H$ of~$H$ with~$P\in Z(H,\hat a)$ and $\hat a\prec\fm$. 
We associate to~$(P,\fm,\hat a)$ a slot in $H^\circ$ as follows:
choose an immediate asymptotic extension $\hat H^\circ$ of $H^\circ$ and an
isomorphism $\hat f\mapsto \hat f^\circ\colon\hat H^\phi\to \hat H^\circ$ of asymptotic fields
extending the isomorphism~$f\mapsto f^\circ\colon H^\phi\to H^\circ$.
Then $(P^{\phi\circ},\fm^\circ,\hat a^\circ)$ is a slot in $H^\circ$ of the same complexity as~$(P,\fm,\hat a)$. The
equivalence class of   $(P^{\phi\circ},\fm^\circ,\hat a^\circ)$ does not depend on the choice of
$\hat H^\circ$ and the isomorphism $\hat H^\phi\to \hat H^\circ$. 
If $(P,\fm,\hat a)$ is a hole in $H$, then~$(P^{\phi\circ},\fm^\circ,\hat a^\circ)$ is a hole in~$H^\circ$, and likewise with ``minimal hole'' in place of ``hole''.
Lemma~\ref{lem:Pphicirc, 1} goes through in this setting. Also, recalling \cite[Lemma~4.5]{ADH5}, if~$H$ is Liouville closed and~$(P,\fm,\hat a)$
is ultimate, then $(P^{\phi\circ},\fm^\circ,\hat a^\circ)$ is ultimate. 

\medskip
\noindent
Moreover, by \cite[Lemmas~4.3.5, 4.3.28 , Corollaries~4.5.23, 4.5.39]{ADH4}: 

{\sloppy\samepage
\begin{lemma} \label{lem:Pphicirc, 2}
\mbox{}
\begin{enumerate}
\item[\textup{(i)}]
If $\phi\preceq 1$ and $(P,\fm,\hat a)$ is split-normal, then  $(P^{\phi\circ},\fm^\circ,\hat a^\circ)$ is split-normal;
likewise with ``split-normal'' replaced by   ``strongly split-nor\-mal''.
\item[\textup{(ii)}]
If $\phi\prec 1$ and  $(P,\fm,\hat a)$  is $Z$-minimal, deep, and repulsive-normal, then~$(P^{\phi\circ},\fm^\circ,\hat a^\circ)$ is repulsive-normal; likewise with ``repulsive-normal'' replaced by
 ``strongly re\-pul\-sive-nor\-mal''. 
\end{enumerate}
\end{lemma} 
}

\subsection*{Reformulations}
We reformulate here some results of Sections~\ref{sec:split-normal over Hardy fields} and~\ref{sec:smoothness}   to facilitate their use.
For $\fv\in K^\times$ with $\fv\prec 1$ we set:
$$\Delta(\fv)\ :=\ \big\{ \gamma\in\Gamma: \gamma=o(v\fv)\big\},$$ 
a proper convex subgroup of $\Gamma$. 
In the next lemma, $P\in  K\{Y\}$ has order $r$ and~$P=Q-R$, where $Q,R\in K\{Y\}$ and $Q$ is homogeneous of degree~$1$ and order~$r$.
We set~$w:=\wt(P)$, so $w\ge r\ge 1$.

\begin{lemma}\label{prop:as equ 1}  
Suppose that $L:=L_Q$ splits strongly over $K$, $\fv(L)\prec^\flat 1$, and 
$$R\prec_\Delta \fv(L)^{w+1}Q, \quad \Delta := \Delta\big(\fv(L)\big).$$ Then $P(y) =0$ and $y',\dots, y^{(r)} \preceq 1$
for some $y\prec \fv(L)^w$ in~$\Calinf[\imag]$.  Moreover: \begin{enumerate}
\item[\textup{(i)}] if $P,Q\in H\{Y\}$, then there is such~$y$ in $\Calinf$;
\item[\textup{(ii)}] if $H\subseteq \Ginf$, then for any  $y\in\c^r[\imag]^{\preceq}$ with $P(y)=0$ we have $y\in \Ginf[\imag]$; likewise with $\c^\omega$ in place of $\c^{\infty}$.
\end{enumerate} 
\end{lemma}
\begin{proof}
Set $\fv:= \abs{\fv(L)}\in H^>$, so $\fv\asymp\fv(L)$. Take $f\in K^{\times}$  such that $A:=f^{-1}L$ is monic; then $\fv(A)=\fv(L)\asymp\fv$ and
$f^{-1}R \prec_\Delta f^{-1}\fv^{w+1}Q \asymp \fv^{w}$.
We have~$A=(\der-\phi_1)\cdots(\der-\phi_r)$ where $\phi_j\in K$ and
$\Re\phi_j\succeq\fv^\dagger\succeq 1$ for $j=1,\dots,r$ by the strong splitting assumption. Also $\phi_1,\dots,\phi_r\preceq\fv^{-1}$ by \eqref{bound on linear factors}. 
The claims now follow from various results in Section~\ref{sec:split-normal over Hardy fields}  applied to the equation~${A(y)=f^{-1}R(y)}$, $y\prec 1$ in the role of \eqref{eq:ADE}, using also Corollary~\ref{cor:ADE smooth}. 
\end{proof}

\noindent
{\it In the next two lemmas $\phi$ is active in~$H$ with ${0<\phi\preceq 1}$.}\/
 
\begin{lemma}\label{dentsolver} 
Let $(P,\fn,\hat h)$ be a slot in $H$ of order~$r$ such that
$(P^\phi,\fn,\hat h)$ is strongly split-normal. Then for some~$y$ in~$\Calinf$,
$$P(y)\ =\ 0,\quad y\ \prec\ \fn,\quad y\in \fn\,(\c^r)^{\preceq}.$$
If $H\subseteq\c^\infty$, then there exists such $y$ in~$\c^\infty$, and likewise with $\c^\omega$ in place of $\c^\infty$.  
\end{lemma}
\begin{proof}
First we consider the case $\phi=1$.
Replace $(P,\fn,\hat h)$ by $(P_{\times\fn},1,\hat h/\fn)$
to arrange $\fn=1$.  
Then $L_{P}$ has order~$r$, $\fv:=\fv(L_{P})\prec^{\flat} 1$, and $P =  Q-R$ where~$Q,R\in H\{Y\}$,
$Q$ is homogeneous of degree $1$ and order~$r$, $L_{Q}\in H[\der]$ splits strongly over~$K$, 
and~$R\prec_\Delta \fv^{w+1} P_1$, where~$\Delta:=\Delta(\fv)$ and $w:= \wt(P)$. 
Now $P_1=Q-R_1$, so~$\fv\sim \fv(L_{Q})$ by Lemma~\ref{lem:fv of perturbed op}(ii), and thus~$\Delta=\Delta\big(\fv(L_{Q})\big)$.
Lemma~\ref{prop:as equ 1} gives $y$ in~$\Calinf$ such that~$y \prec \fv^w\prec 1$, 
$P(y)=0$, and
$y^{(j)}\preceq 1$ for $j=1,\dots,r$.
Then $y$ has for $\fn=1$ the properties displayed in the lemma.  

Now suppose $\phi$ is  arbitrary. Employing $(\phantom{-})^\circ$ as explained earlier in this section, the
slot $(P^{\phi\circ},\fn^\circ,\hat h^\circ)$ in the Hardy field~$H^\circ$ is strongly split-normal, hence by the case $\phi=1$ we have $z\in\Calinf$ with 
$P^{\phi\circ}(z) = 0$, $z\prec\fn^\circ$, and $(z/\fn^\circ)^{(j)}\preceq 1$
for~${j=1,\dots,r}$.
Take $y\in\Calinf$ with $y^\circ=z$. Then
$P(y) = 0$, $y \prec \fn$, and $y\in \fn\,(\c^r)^{\preceq}$ by Lemma~\ref{lem:as equ comp} and a subsequent remark.  
Moreover, if $\phi,z\in\Ginf$, then $y\in\Ginf$, and likewise with $\Gom$ in place of~$\Ginf$.
\end{proof}

\noindent
In the next ``complex'' version, $(P, \fm, \hat a)$ is a slot in $K$ of order $r$ with $\fm\in H^\times$. 

\begin{lemma}\label{lemfillhole} 
Suppose
 the slot $(P^\phi,\fm,\hat a)$ in~$K^\phi$ is  strictly normal, and its linear part splits strongly
over $K^\phi$. Then for some~$y\in \Calinf[\imag]$ we have
$$P(y)\ =\ 0,\quad y\ \prec\ \fm,\quad y\in \fm\,\c^r[\imag]^{\preceq}.$$
If $H\subseteq\c^\infty$, then there is such $y$ in $\c^\infty[\imag]$.
If $H\subseteq\c^\omega$, then there is such $y$ in $\c^\omega[\imag]$. 
\end{lemma}
\begin{proof}
Consider first the case $\phi=1$. Replacing~$(P,\fm,\hat a)$  by $(P_{\times\fm},1,\hat a/\fm)$
we arrange $\fm=1$.
Set $L:=L_{P}\in K[\der]$, $Q:=P_1$, and $R:=P-Q$. 
Since~$(P,1,\hat a)$ is strictly normal, we have $\order(L)=r$,  $\fv:=\fv(L)\prec^\flat 1$,
and $R\prec_\Delta \fv^{w+1}Q$  where~$\Delta:=\Delta(\fv)$,  $w := \wt(P)$. As $L$ splits  strongly over~$K$,   Lemma~\ref{prop:as equ 1} gives~$y$ in~$\Calinf[\imag]$ such that
$P(y)=0$, $y \prec \fv^w\prec 1$, and $y^{(j)}\preceq 1$
for~$j=1,\dots,r$. For the last part of the lemma, use the last part of  Lemma~\ref{prop:as equ 1}.  
The general case reduces to this special case as in the proof of Lemma~\ref{dentsolver}. 
\end{proof}

\subsection*{Finding germs in holes} {\em In this subsection $\hat H$ is an immediate asymptotic extension of $H$.}\/
This fits into the setting of \cite[Section~4.3]{ADH4}  on split-normal slots: 
$K=H[\imag]$ and~$\hat H$ have $H$ as a common asymptotic subfield and~$\hat{K}:=\hat{H}[\imag]$ as a common asymptotic extension, $\hat H$ is an $H$-field, and  $\hat K$ is $\d$-valued.

{\em Assume also that $H$ is $\upo$-free.}\/  Then $K$ is $\upo$-free by [ADH, 11.7.23].
{\em Let $(P,\fm,\hat a)$ with~$\fm\in H^\times$ and~$\hat a\in 
\hat K\setminus K$ be a minimal hole in $K$ of order~${r\geq 1}$}. 

\begin{prop}\label{propdeg>1}
Suppose  $\deg P>1$. 
Then for some $y\in \Calinf[\imag]$ we have
$$P(y)\ =\ 0,\quad y\ \prec\ \fm,\quad y\in \fm\,\c^r[\imag]^{\preceq}.$$
If $\fm\preceq 1$, then $y\in \c^r[\imag]^{\preceq}$ for such~$y$.
Moreover, if   $H\subseteq\c^\infty$, then we can take such~$y$ in $\c^\infty[\imag]$, and 
if $H\subseteq\c^\omega$, then we can take such~$y$ in $\c^\omega[\imag]$. 
\end{prop}
\begin{proof} Lemma~\ref{lem:achieve strong splitting} gives
a refinement $(P_{+a},\fn,\hat a-a)$  of $(P,\fm,\hat a)$ with $\fn\in H^\times$ and an active $\phi$ in $H$ with $0<\phi\preceq 1$ such that 
the hole $(P^\phi_{+a},\fn,\hat a-a)$ in $K^\phi$ is  strictly normal  and its  linear part splits 
strongly over $K^\phi$. 
Lemma~\ref{lemfillhole} applied to $(P_{+a},\fn,\hat a-a)$ in place of $(P,\fm,\hat a)$
yields $z\in\Calinf[\imag]$ with $P_{+a}(z)=0$, $z\prec\fn$ and~$(z/\fn)^{(j)}\preceq 1$ for $j=1,\dots,r$.
Lem\-ma~\ref{lem:small derivatives of y}(ii) applied to~$z/\fm$,~$\fn/\fm$ in place of~$y$,~$\fm$, respectively, yields~$(z/\fm)^{(j)}\preceq 1$ for $j=0,\dots,r$. Also, $a\prec\fm$ (in $K$), hence~$(a/\fm)^{(j)}\prec 1$ for~$j=0,\dots,r$. Set $y:=a+z$; then $P(y)=0$, $y\prec\fm$, and~$(y/\fm)^{(j)}\preceq 1$ for~$j=1,\dots,r$.
For the rest use Lem\-ma~\ref{lem:small derivatives of y}(ii) and the last part of Lemma~\ref{lemfillhole}.
\end{proof}

\noindent
Next we treat the linear case:

\begin{prop} \label{prop:deg 1 analytic} 
Suppose $\deg P=1$. Then for some~$y\in \Calinf[\imag]$ we have
$$P(y)\ =\ 0,\quad y\ \prec\ \fm,\quad (y/\fm)'\ \preceq\ 1.$$
If  $\fm\preceq 1$, then $y\prec 1$ and $y'\preceq 1$ for each such~$y$. 
Moreover, if   $H\subseteq\c^\infty$, then we can take such $y$ in $\c^\infty[\imag]$, and 
if $H\subseteq\c^\omega$, then we can take such~$y$ in $\c^\omega[\imag]$.
\end{prop}
\begin{proof} 
We have $r=1$ by \cite[Corollary~3.2.8]{ADH4}.
If $\der(K)=K$ and $\I(K)\subseteq K^\dagger$, then Lemma~\ref{lem:achieve strong splitting, d=1} applies, and we can argue as in the proof of Proposition~\ref{propdeg>1},
using this lemma  instead of Lemma~\ref{lem:achieve strong splitting}.
We reduce the general case to this special case as follows: 
Set $H_1:=\operatorname{D}(H)$; then $H_1$ is an $\upo$-free Hardy field by \cite[Theorem~1.3.1]{ADH4}, and~$K_1:=H_1[\imag]$ satisfies $\der(K_1)=K_1$ and $\I(K_1)\subseteq K_1^\dagger$ by Propositions~\ref{prop:Li(H(R))} and~\ref{prop:I(K)}. Moreover, by Corollary~\ref{cor:Hardy field ext smooth}, if $H\subseteq\c^\infty$, then $H_1\subseteq\c^\infty$, and likewise with~$\c^\omega$ in place
of~$\c^\infty$.
The newtonization $\hat H_1$ of $H_1$ is an immediate asymptotic extension of~$H_1$ [ADH, 14.5.4], and~$\hat K_1:=\hat H_1[\imag]$ is   newtonian~[ADH, 14.5.7].  Lemma~\ref{find zero of P}  gives an embedding~$K\langle \hat a\rangle \to \hat K_1$ over $K$; let $\hat a_1$ be the image of $\hat a$ under this embedding.
If~$\hat a_1\in K_1$, then we are done by taking $y:=\hat a_1$, so we may assume~$\hat a_1\notin K_1$.   Then~$(P,\fm,\hat a_1)$ is a minimal hole in~$K_1$, and the above applies with~$H$,~$K$,~$\hat a$ replaced by $H_1$, $K_1$, $\hat a_1$, respectively.
\end{proof}

\noindent
We can improve on these results in a useful way:

\begin{cor}\label{improap} 
Suppose
$\hat a\sim a\in K$. Then  for some $y\in \Calinf[\imag]$ we have
$$P(y)\ =\ 0,\quad y\ \sim\ a,\quad (y/a)^{(j)}\ \prec\ 1\ 
\text{ for $j=1,\dots,r$.}$$
If   $H\subseteq\c^\infty$, then there is such $y$ in $\c^\infty[\imag]$.
If $H\subseteq\c^\omega$, then there is such~$y$ in $\c^\omega[\imag]$.
\end{cor}
\begin{proof} Take $a_1\in K$ and $\fn\in H^\times$ such that 
 $\fn\asymp \hat a-a \sim a_1$, and set~${b:=a+a_1}$.
Then~$(P_{+b}, \fn, \hat a-b)$ is a refinement of $(P,\fm,\hat a)$, and
Propositions~\ref{propdeg>1} and~\ref{prop:deg 1 analytic} give~${z\in \Calinf[\imag]}$
with $P(b+z)=0$, $z\prec \fn$ and~$(z/\fn)^{(j)}\preceq 1$ for~$j=1,\dots,r$. 
We have~$(a_1/a)^{(j)}\prec 1$ for~$j=0,\dots,r$, since $K$ has small derivation.
Likewise, $(\fn/a)^{(j)}\prec 1$ for $j=0,\dots,r$, and hence
$(z/a)^{(j)}\prec 1$ for $j=0,\dots,r$, by~$z/a=(z/\fn)\cdot (\fn/a)$ and the Product Rule.
So $y:= b+z$ has the desired property. The rest follows from the ``moreover'' parts of these propositions.
\end{proof}

\begin{remarkNumbered}\label{rem:improap}  If we replace our standing assumption that $H$ is $\upo$-free and~$(P, \fm, \hat a)$ is a minimal hole in $K$ by the assumption that $H$ is $\upl$-free, $\der(K)=K$, ${\I(K)\subseteq K^\dagger}$, and~$(P, \fm, \hat a)$ is a slot in $K$ of order and degree $1$,
then Proposition~\ref{prop:deg 1 analytic} and Corollary~\ref{improap} go   through by Remark~\ref{rem:achieve strong splitting, d=1}.
\end{remarkNumbered}

\noindent 
Can we choose $y$ in Corollary~\ref{improap} such that additionally $\Re y$, $\Im y$ are hardian over~$H$?  At this stage we cannot claim this.  
In the next section we introduce weights and their corresponding norms as a more refined tool. This will allow us to
obtain Corollary~\ref{cor:approx y} as a key approximation result for later use.

\section{Weights}\label{sec:weights}

\noindent
In this section we prove Proposition~\ref{prop:notorious 3.6} to  strengthen
Lemma~\ref{lem:close}. 
This uses the material on repulsive-normal slots from \cite[Section~4.5]{ADH4}, but
we also need more refined norms for differentiable functions, to which we turn now.

The final result, Corollary~\ref{cor:approx y},  is the only part of this section used  towards our main result, Theorem~\ref{thm:char d-max}. But this use, in the proof of Theorem~\ref{46}, is essential, and obtaining this corollary requires the entire section.

\subsection*{Weighted spaces of differentiable functions}
{\it In this subsection we fix $r\in \N$ and a {\em weight}
function $\w\in\c_a[\imag]^\times$.}\/ 
For $f\in \Car[\imag]$ we set \label{p:absawt}
$$\dabs{f}_{a;r}^\w\ :=\ 
\max\big\{ \dabs{\w^{-1}f}_a, \dabs{\w^{-1}f'}_a, \dots, \dabs{\w^{-1}f^{(r)}}_a\big\} \ \in\  [0,+\infty],$$
and $\dabs{f}_{a}^\w:=\dabs{f}_{a;0}^\w$ for $f\in \c_a[\imag]$. 
\noindent
Then
$$\Car[\imag]^\w\ :=\ \big\{ f\in \Car[\imag]:\, \dabs{f}_{a;r}^\w < +\infty \big\}$$
is a $\C$-linear subspace of 
$$\c_a[\imag]^\w\ :=\ \Caz[\imag]^\w\ =\ 
\w\,\c_a[\imag]^{\b}\ =\ \big\{ f\in \c_a[\imag]:\, f\preceq \w\big\}.$$ 
Below we consider the $\C$-linear space~$\Car[\imag]^\w$ to be equipped with the norm 
$$f\mapsto \dabs{f}_{a;r}^\w.$$  
Recall from Section~\ref{sec:IHF} the convention $b\cdot\infty=\infty\cdot b=\infty$ for $b\in [0,\infty]$. Note that 
\begin{equation}\label{eq:weighted norm, 0}
\dabs{fg}_{a;r}^\w\ \leq\ 2^r \dabs{f}_{a;r}\,\dabs{g}_{a;r}^\w\quad\text{ for $f,g\in \Car[\imag]$,}
\end{equation}
so $\Car[\imag]^\w$ is a $\Car[\imag]^{\b}$-submodule of $\Car[\imag]$. 
Note also that $\dabs{1}_{a;r}^\w=\dabs{\w^{-1}}_{a}$, hence
$$\dabs{f}_{a;r}^\w\ \leq\ 2^r \dabs{f}_{a;r}\,\dabs{\w^{-1}}_{a}\quad\text{ for $f\in \Car[\imag]$}$$
and
$$\w^{-1}\in\c_a[\imag]^{\b} \quad\Longleftrightarrow\quad 1\in \Car[\imag]^{\w} \quad\Longleftrightarrow\quad \Car[\imag]^{\b} \subseteq \Car[\imag]^{\w}.$$
We have
\begin{equation}\label{eq:weighted norm, 1}
\dabs{f}_{a;r}\ \leq\  \dabs{f}^\w_{a;r}\,\dabs{\w}_{a}\quad\text{ for $f\in \Car[\imag]$,}
\end{equation}
and thus
\begin{equation}\label{eq:weighted norm, 2}
\w\in\c_a[\imag]^{\b} \quad\Longleftrightarrow\quad \Car[\imag]^{\b} \subseteq \Car[\imag]^{\w^{-1}} \quad\Longrightarrow\quad \Car[\imag]^{\w}\subseteq \Car[\imag]^{\b}.
\end{equation}
Hence if $\w,\w^{-1}\in \c_a[\imag]^{\b}$, then $\Car[\imag]^{\b} = \Car[\imag]^{\w}$, and the norms
 $\dabs{\,\cdot\,}_{a;r}^\w$ and  $\dabs{\,\cdot\,}_{a;r}$ on this $\C$-linear space are equivalent. 
 (In later use, $\w\in \c_a[\imag]^{\b}$, $\w^{-1}\notin  \c_a[\imag]^{\b}$.)
If~$\w\in \c_a[\imag]^{\b}$, then $\Car[\imag]^{\w}$ is an ideal of the commutative ring~$\Car[\imag]^{\b}$. From \eqref{eq:weighted norm, 0} and~\eqref{eq:weighted norm, 1} we obtain
$$\dabs{fg}_{a;r}^\w\ \leq\ 2^r \,\dabs{\w}_{a} \, \dabs{f}_{a;r}^\w\,\dabs{g}_{a;r}^\w\quad\text{ for $f,g\in \Car[\imag]$.}$$
For $f\in \Carm[\imag]^\w$ we have $\dabs{f}_{a;r}^\w, \dabs{f'}_{a;r}^\w \leq \dabs{f}_{a;r+1}^\w$.
From \eqref{eq:weighted norm, 1} and  \eqref{eq:weighted norm, 2}:

\begin{lemma}\label{lem:w-conv => b-conv}
Suppose $\w\in\c_a[\imag]^{\b}$ \textup{(}so $\Car[\imag]^{\w}\subseteq \Car[\imag]^{\b}$\textup{)} and $f\in \Car[\imag]^{\w}$. If $(f_n)$ is a sequence in~$\Car[\imag]^\w$ and
$f_n\to f$ in $\Car[\imag]^\w$, then also $f_n\to f$ in $\Car[\imag]^{\b}$.
\end{lemma}

\noindent
This is used to show:

\begin{lemma}\label{lem:w-conv}
Suppose $\w\in\c_a[\imag]^{\b}$. Then
the $\C$-linear space $\Car[\imag]^\w$ equipped with the norm
$\dabs{\,\cdot\,}_{a;r}^\w$ is complete.
\end{lemma}
\begin{proof}
We proceed by induction on $r$.
Let $(f_n)$ be a cauchy sequence in the normed space~$\c_a[\imag]^\w$. Then  
the sequence $(\w^{-1}f_n)$ in the Banach space 
$\c_a[\imag]^{\b}$ is cauchy, hence has a limit $g\in \c_a[\imag]^{\b}$, so
with $f:=\w g\in\c_a[\imag]^\w$ we have
$\w^{-1}f_n\to \w^{-1}f$ in~$\c_a[\imag]^{\b}$ and hence $f_n\to f$ in~$\c_a[\imag]^\w$. 
Thus the lemma holds for $r=0$. Suppose the lemma holds for a certain value of
$r$, and let $(f_n)$ be a cauchy sequence in $\Carm[\imag]^\w$.
Then~$(f_n')$ is a cauchy sequence in $\Car[\imag]^\w$ and hence has a limit
$g\in\Car[\imag]^\w$, by inductive hypothesis.
By Lemma~\ref{lem:w-conv => b-conv},    $f_n'\to g$ in
$\c_a[\imag]^{\b}$.
Now~$(f_n)$ is also a cauchy sequence in $\c_a[\imag]^\w$, hence has a limit
$f\in\c_a[\imag]^\w$ (by the case $r=0$), and
by Lemma~\ref{lem:w-conv => b-conv} again, $f_n\to f$ in $\c_a[\imag]^{\b}$.
Thus 
$f$ is differentiable and $f'=g$ by~\cite[(8.6.4)]{Dieudonne}.
This yields~$f_n\to f$ in~$\Carm[\imag]^{\w}$.
\end{proof}

\begin{lemma}\label{lem:Qnk}
Suppose $\w\in \Car[\imag]^{\b}$. If $f\in\Car[\imag]$ and $f^{(k)}\preceq \w^{r-k+1}$ for $k=0,\dots,r$, then $f\w^{-1} \in \Car[\imag]^{\b}$. \textup{(}Thus $\Car[\imag]^{\w^{r+1}}\subseteq \w\,\Car[\imag]^{\b}$.\textup{)}
\end{lemma}
\begin{proof} 
An easy induction on $n\leq r$ shows that there are   $Q^n_k\in\Z[X_0,X_1,\dots,X_{n-k}]$  ($0\leq k\leq n$) 
such that for all $f\in \Car[\imag]$ and $n\le r$: 
$$(f\w^{-1})^{(n)}\ =\ \sum_{k=0}^n Q^n_k(\w,\w',\dots,\w^{(n-k)}) f^{(k)}\w^{k-n-1}.$$
Now use that $Q^n_k(\w,\w',\dots,\w^{(n-k)})\preceq 1$ for $n\le r$ and $k=0,\dots,n$. 
\end{proof}

\noindent
Next we generalize the inequality \eqref{eq:weighted norm, 0}:

\begin{lemma}\label{lem:weighted norm}
Let $f_1,\dots,f_{m-1},g\in \Car[\imag]$, $m\geq 1$; then
$$\dabs{f_1\cdots f_{m-1}g}_{a;r}^\w\ \leq\ m^r \,\dabs{f_1}_{a;r}\cdots\dabs{f_{m-1}}_{a;r} \,\dabs{g}_{a;r}^\w.$$
\end{lemma}
\begin{proof}
Use the generalized Product Rule [ADH, p.~199] and the well-known identity 
$\sum \frac{n!}{i_1!\cdots i_m!}=m^n$ with the sum over all $(i_1,\dots,i_m)\in\N^m$ with $i_1+\cdots+i_m=n$.
\end{proof}

\noindent
With $\i$ ranging over $\N^{1+r}$,  let
$P=\sum_{\i} P_{\i} Y^{\i}$ (all $P_{\i}\in \c_a[\imag]$) be a polynomial in~$\c_a[\imag]\big[Y,Y',\dots,Y^{(r)}\big]$;
for $f\in \Car[\imag]$ we have  $P(f)=\sum_{\i} P_{\i} f^{\i}\in\c_a[\imag]$.
(See also the beginning of Section~\ref{sec:split-normal over Hardy fields}.)
We set
$$\dabs{P}_a\ :=\ \max_{\i} \, \dabs{P_{\i}}_a \in [0,\infty].$$
{\em In the rest of this subsection we assume $\dabs{P}_a<\infty$}, that is,  $P\in\c_a[\imag]^{\b}\big[Y,\dots,Y^{(r)}\big]$. Hence if~$\w\in \c_a[\imag]^{\b}$, $P(0)\in \c_a[\imag]^\w$, and $f\in\Car[\imag]^\w$, then $P(f)\in \c_a[\imag]^\w$.  
Here are weighted versions of Lemma~\ref{lem:bound on P(f)}
and~\ref{cor:bound on P(f)}: 

\begin{lemma}\label{lem:weighted bd}
Suppose $P$ is homogeneous of degree $d\ge 1$, and let $\w\in \c_a[\imag]^{\b}$ and~$f\in\Car[\imag]^\w$. Then
$$\dabs{P(f)}_a^\w\ \leq\ {d+r \choose r} \cdot \dabs{P}_a\cdot  \dabs{f}_{a;r}^{d-1} \cdot  \dabs{f}_{a;r}^\w.$$
\end{lemma}
\begin{proof} For $j=0,\dots,r$ we have $\dabs{f^{(j)}}_a\le \dabs{f}_{a;r}$ and $\dabs{f^{(j)}}_a^\w\le \dabs{f}_{a;r}^\w$.
Now $f^{\i}$, where $\i=(i_0,\dots,i_r)\in \N^{r+1}$ and $i_0+\cdots + i_r=d$, is a product of $d$ such factors~$f^{(j)}$, so
Lemma~\ref{lem:weighted norm} with $m:=d$, $r:=0$, gives
$$ \dabs{f^{\i}}_{a}^\tau\  \le\  \dabs{f}_{a;r}^{d-1}\cdot \dabs{f}_{a;r}^{\w}.$$
It remains to note that by  \eqref{eq:weighted norm, 0} we have $\dabs{P_{\i}f^{\i}}_a^\tau\le \dabs{P_{\i}}_a\cdot\dabs{f^{\i}}_a^\tau$. 
\end{proof}

\begin{cor}\label{cor:weighted bd} Let $1\le d \le e$ in $\N$ be such that $P_{\i}=0$ if~$\abs{\i}<d$ or $\abs{\i}>e$.
Then for $f\in\Car[\imag]^\w$ and $\w\in \c_a[\imag]^{\b}$ we have
$$\dabs{P(f)}_a^\w\ \leq\  D \cdot \dabs{P}_a\cdot   \big( \dabs{f}_{a;r}^{d-1}+\cdots+\dabs{f}_{a;r}^{e-1} \big)\cdot \dabs{f}_{a;r}^\w$$
where
$D\ =\ D(d,e,r)\ :=\  {e+r+1 \choose r+1}-{d+r\choose r+1}\in\N^{\geq 1}$.
\end{cor}

\subsection*{Doubly-twisted integration}
In this subsection we adopt the setting in  {\it Twisted integration}\/ of  Section~\ref{sec:IHF}.
Thus $\phi\in\c_a[\imag]$ and $\Phi=\der_a^{-1}\phi$.
Let $\w\in \Cao$ satisfy~${\w(s)>0}$ for $s\geq a$, and set $\widetilde\phi:=\phi-\w^\dagger\in \c_a[\imag]$  and $\widetilde\Phi:=\der_a^{-1}\widetilde\phi$. Thus
$$\widetilde\Phi(t)\ =\ \int_a^t (\phi-\w^\dagger)(s)\, ds\ =\ \Phi(t) - \log \w(t) + \log \w(a)\qquad\text{for $t\geq a$.}$$
Consider the right inverses
$B, \widetilde B\colon  \c_a[\imag] \to \Cao[\imag]$ 
to, respectively, $\der-\phi\colon\Cao[\imag]\to\c_a[\imag]$ and $\der-\widetilde\phi\colon\Cao[\imag]\to\c_a[\imag]$, 
given by
$$B\ :=\ \ex^\Phi \circ\, {\der_a^{-1}} \circ \ex^{-\Phi}, \quad
  \widetilde B\ :=\ \ex^{\widetilde\Phi} \circ\, {\der_a^{-1}} \circ \ex^{-\widetilde\Phi}.$$
For $f\in\c_a[\imag]$ and $t\geq a$ we have
\begin{align*}
\widetilde{B}f(t)\	&=\ \ex^{\widetilde\Phi(t)} \int_a^t \ex^{-\widetilde\Phi(s)}f(s)\,ds \\
			&=\ \w(t)^{-1}\w(a)\ex^{\Phi(t)} \int_a^t  \ex^{-\Phi(s)} \w(s)\w(a)^{-1} f(s)\,ds \\
	&=\ \w(t)^{-1}\ex^{\Phi(t)} \int_a^t \ex^{-\Phi(s)}\w(s) f(s)\,ds\ =\ \w^{-1}(t)\big( B(\w f)\big)(t)
\end{align*}
and so $\widetilde B=\w^{-1}\circ B\circ \w$.
Hence if  $\widetilde\phi$ is attractive, then $B_{\ltimes \w}:=\w^{-1}\circ B\circ \w$ maps~$\c_a[\imag]^{\b}$ into $\c_a[\imag]^{\b}\cap\Cao[\imag]$, and 
the operator $B_{\ltimes \w}\colon \c_a[\imag]^{\b}\to\c_a[\imag]^{\b}$ is continuous with $\dabs{B_{\ltimes \w}}_a \leq \big\| 1/  \Re \widetilde\phi\big\|_a$; if in addition $\widetilde\phi\in\Car[\imag]$, then $B_{\ltimes \w}$ maps $\c_a[\imag]^{\b}\cap\Car[\imag]$ into
$\c_a[\imag]^{\b}\cap\Carm[\imag]$.
Note that if $\phi\in\Car[\imag]$ and $\w\in\Carm$, then $\widetilde\phi\in\Car[\imag]$.

\medskip
\noindent
Next, suppose that $\phi$, $\widetilde\phi$ are both repulsive. Then we have the $\C$-linear op\-er\-a\-tors~$B, \widetilde B\colon  \c_a[\imag]^{\b} \to \Cao[\imag]$ given, for $f\in \c_a[\imag]^{\b}$ and $t\geq a$, by
$$Bf(t)\ :=\ \ex^{\Phi(t)} \int_\infty^t \ex^{-\Phi(s)}f(s)\,ds, \qquad 
\widetilde Bf(t)\ :=\ \ex^{\widetilde\Phi(t)} \int_\infty^t \ex^{-\widetilde\Phi(s)}f(s)\,ds.$$
Now assume that our  $\w\in(\c_a^1)^\times$ is in $\c_a^{\b}$. Then
we have the $\C$-linear operator $$B_{\ltimes \w}\ :=\ {\w^{-1}\circ B\circ \w}\ \colon\ \c_a[\imag]^{\b} \to \Cao[\imag].$$
A computation as above shows $\widetilde B = B_{\ltimes \w}$; thus $B_{\ltimes \w}$ maps $\c_a[\imag]^{\b}$ in\-to~$\c_a[\imag]^{\b}\cap\Cao[\imag]$, and the operator $B_{\ltimes \w}\colon \c_a[\imag]^{\b}\to\c_a[\imag]^{\b}$ is continuous with $\dabs{B_{\ltimes \w}}_a \leq \big\| 1/ \Re \widetilde\phi\big\|_a$.
If~$\widetilde{\phi}\in\Car[\imag]$, then~$B_{\ltimes \w}$ maps $\c_a[\imag]^{\b}\cap\Car[\imag]$ into
$\c_a[\imag]^{\b}\cap\Carm[\imag]$. 

\subsection*{More on twists and right-inverses of linear operators over Hardy fields}
In this subsection we adopt the assumptions in force for Lemma~\ref{cri}, which we repeat here. 
Thus $H$ is a Hardy field, $K=H[\imag]$, $r\in\N^{\geq 1}$, and $f_1,\dots,f_r\in K$.
We fix $a_0\in\R$ and functions in $\c_{a_0}[\imag]$ representing the germs $f_1,\dots,f_r$,
denoted by the same symbols. We let~$a$ range over $[a_0,\infty)$, and we denote the restriction of each
$f\in\c_{a_0}[\imag]$ to  $[a,\infty)$ also by $f$.
For each $a$ we then have the $\C$-linear map~$A_a\colon \Car[\imag]\to\c_a[\imag]$ given by
$$  A_a(y)\ =\ y^{(r)}+f_1y^{(r-1)}+\cdots+f_ry .$$ 
We are in addition given a splitting $(\phi_1,\dots,\phi_r)$ of
the linear differential operator~$A=\der^r+f_1\der^{r-1}+\cdots+f_r\in K[\der]$ over $K$
with $\Re\phi_1,\dots,\Re\phi_r\succeq 1$, 
as well as functions in $\c_{a_0}^{r-1}[\imag]$ representing $\phi_1,\dots,\phi_r$, 
denoted by the same symbols and satisfying  
$\Re\phi_1,\dots,\Re\phi_r\in (\c_{a_0})^\times$.
This gives rise to the continuous $\C$-linear operators
$$B_j\ :=\ B_{\phi_j}\ \colon\ \c_a[\imag]^{\b}\to \c_a[\imag]^{\b} \qquad (j=1,\dots,r)$$
and the right-inverse
$$A_a^{-1} \ :=\  B_r\circ \cdots \circ B_1\ \colon\ \c_a[\imag]^{\b}\to \c_a[\imag]^{\b}$$
of $A_a$ with the properties stated in Lemma~\ref{cri}.

\medskip\noindent
Now let $\fm\in H^\times$ with $\fm\prec 1$, and set~$\widetilde{A}:=A_{\ltimes\fm}\in K[\der]$. Let $\w\in (\Cazr)^\times$ be
a representative of $\fm$. Then $\w\in (\Cazr)^{\b}$
and $\widetilde{\phi}_j:=\phi_j-\w^\dagger\in \Cazrl[\imag]$ for $j=1,\dots,r$.
We   have the $\C$-linear maps
$$\widetilde{A}_j\  :=\  \der-\widetilde{\phi}_j\ \colon\ \Caj[\imag]\to \Cajl[\imag] \qquad (j=1,\dots,r)$$
and for sufficiently large $a$ a factorization
$$\widetilde{A}_a\  =\  \widetilde{A}_1\circ\cdots\circ\widetilde{A}_r\ \colon\ \Car[\imag]\to \c_a[\imag].$$
Below we assume this holds for all $a$, as can be arranged by increasing $a_0$. 
We call~$f,g\in\c_a[\imag]$ 
{\bf alike} if  $f$, $g$ are both attractive or
both repulsive. In the same way we define when germs $f,g\in \c[\imag]$
are alike.\index{alike}\index{germ!alike}\index{function!alike} 
Suppose that  $\phi_j$, $\widetilde{\phi}_j$ are alike for~$j=1,\dots,r$.
Then we have   continuous $\C$-linear operators
$$\widetilde{B}_j\ :=\ B_{\widetilde{\phi}_j}\ \colon\ \c_a[\imag]^{\b}\to \c_a[\imag]^{\b} \qquad (j=1,\dots,r)$$
and the right-inverse 
$$\widetilde{A}_a^{-1} \ :=\  \widetilde{B}_r\circ \cdots \circ \widetilde{B}_1\ \colon\ \c_a[\imag]^{\b}\to \c_a[\imag]^{\b}$$
of $\widetilde{A}_a$, and the arguments in the previous subsection show that $\widetilde{B}_j=(B_j)_{\ltimes \w}=\w^{-1}\circ B_j\circ \w$ for $j=1,\dots,r$,  and hence
$\widetilde{A}_a^{-1}=\w^{-1}\circ A_a^{-1}\circ \w$. 
For $j=0,\dots,r$ we set, in analogy with~$A_j^\circ$ and $B_j^\circ$ from \eqref{eq:Ajcirc} and \eqref{eq:Bjcirc}, 
$$\widetilde{A}_j^{\circ}\ :=\ \widetilde{A}_1\circ \cdots \circ \widetilde{A}_j\ \colon\ \Caj[\imag]\to \c_a[\imag],\quad
\widetilde{B}_j^\circ\ :=\ \widetilde{B}_j\circ \cdots \circ \widetilde{B}_1\ \colon\ \c_a[\imag]^{\b}  \to\c_a[\imag]^{\b}.$$ 
Then $\widetilde{B}_j$ maps $\c_a[\imag]^{\b}$ into $\c_a[\imag]^{\b}\cap \Caj[\imag]$, $\widetilde{A}_j^\circ \circ \widetilde{B}_j^\circ$ is the identity on $\c_a[\imag]^{\b}$, and~$\widetilde{B}_j^\circ = \w^{-1}\circ B_j^\circ \circ \w$
by the above.

\subsection*{A weighted version of Proposition~\ref{uban}}
We adopt the setting of the subsection {\it Damping factors}\/ of Section~\ref{sec:IHF}, and
make the same assumptions as in the paragraph before Proposition~\ref{uban}.
Thus $H$, $K$, $A$, $f_1,\dots, f_r$, $\phi_1,\dots, \phi_r$, $a_0$ are as in the previous subsection,
$\fv\in \Cazr$ satisfies $\fv(t)>0$ for all $t\ge a_0$, and its germ $\fv$ is in $H$ with $\fv\prec 1$.  
As part of those assumptions we also have $\phi_1,\dots,\phi_r\preceq_\Delta \fv^{-1}$ in the asymptotic field $K$,  
for the convex subgroup 
 $$\Delta\ :=\ \big\{\gamma\in v(H^\times):\ \gamma=o(v\fv)\big\}$$ 
of $v(H^\times)=v(K^\times)$. Also $\nu\in\R^>$ and $u := \fv^\nu|_{[a,\infty)}\in (\Car)^\times$.

\medskip\noindent
To state a weighted version of Proposition~\ref{uban}, let $\fm\in H^\times$, $\fm\prec 1$, and let $\fm$ also denote a representative  in $(\Cazr)^\times$ of the germ $\fm$.  
Set $\w:=\fm|_{[a,\infty)}$, so we have~$\w\in (\Car)^\times\cap (\Car)^{\b}$ and thus $\Car[\imag]^\w\subseteq \Car[\imag]^{\b}$.
(Note that $\w$, like $u$, depends on~$a$, but we do not indicate this dependence notationally.)
With notations as in the previous subsection we assume that for all $a$ we have the factorization $$\widetilde{A}_a\  =\  \widetilde{A}_1\circ\cdots\circ\widetilde{A}_r\ \colon\ \Car[\imag]\to \c_a[\imag],$$ as can be arranged by increasing $a_0$ if necessary.


\begin{prop} \label{prop:1.5 weighted}
Assume $H$ is real closed, $\nu\in\Q$, $\nu>r$,
and the elements  $\phi_j$,  $\phi_j-\fm^\dagger$ of $\c_{a_0}[\imag]$ are alike for $j=1,\dots,r$.
Then:
\begin{enumerate}
\item[\rm{(i)}] the $\C$-linear operator $u A_a^{-1}\colon \c_a[\imag]^{\b} \to \c_a[\imag]^{\b}$ maps $\c_a[\imag]^{\w}$ into $\Car[\imag]^{\w}$; 
\item[\rm(ii)]  its restriction to a $\C$-linear map $\c_a[\imag]^{\w} \to \Car[\imag]^{\w}$ is continuous; and
\item[\rm(iii)] denoting this restriction also by $uA_a^{-1}$, we have $\|uA_a^{-1}\|^\w_{a;r}\to 0$ as $a\to \infty$.
\end{enumerate}
\end{prop}
\begin{proof}
Let $f\in \c_a[\imag]^{\w}$, so $g:=\w^{-1}f\in\c_a[\imag]^{\b}$. Let $i\in\{0,\dots,r\}$; then with $\widetilde{B}^\circ_j$ as in the previous subsection and $u_{i,j}$ as  in Lemma~\ref{teq}, that lemma gives 
$$\w^{-1}\cdot \big(uA_a^{-1}(f)\big){}^{(i)}\  =\  \sum_{j=r-i}^r u_{i,j} u\cdot \w^{-1}B^\circ_j(\w g)\  =\ 
\sum_{j=r-i}^r u_{i,j}u\widetilde{B}^\circ_j(g).$$
The proof of Proposition~\ref{uban} shows $u_{i,j} u\in\c_a[\imag]^{\b}$ with $\dabs{u_{i,j}u}_a\to 0$ as $a\to\infty$.
Set
$$\widetilde{c}_{i,a}\ :=\ \sum_{j=r-i}^r \|u\,u_{i,j}\|_a \cdot \|\widetilde{B}_j\|_a\cdots \|\widetilde{B}_1\|_a\in [0,\infty)  \qquad (i=0,\dots,r).$$
Then
$\big\|\w^{-1}\cdot\big(uA_a^{-1}(f)\big){}^{(i)}\big\|_a\le \widetilde{c}_{i,a}\|g\|_a=\widetilde{c}_{i,a}\|f\|^\w_a$ 
where $\widetilde{c}_{i,a}\to 0$ as $a\to\infty$. This yields (i)--(iii).
\end{proof}

\subsection*{Weighted variants of results in Section~\ref{sec:split-normal over Hardy fields}}
In this subsection we adopt the hypotheses in force for Lemma~\ref{bdua}. To summarize those,   
$H$, $K$, $A$, $f_1,\dots, f_r$, $\phi_1,\dots, \phi_r$, $a_0$, $\fv$, $\nu$, $u$, $\Delta$ are as in the previous subsection,
$d,r\in \N^{\ge 1}$, $H$ is real closed, $R\in K\{Y\}$ has order~$\leq d$ and weight~$\leq w\in\N^{\geq r}$.
Also $\nu\in \Q$, $\nu>w$,
$R\prec_{\Delta}\fv^\nu$, $\nu\fv^\dagger\not\sim \Re \phi_j$ and
$\Re\phi_j - \nu\fv^\dagger\in (\c_{a_0})^\times$
for $j=1,\dots,r$. Finally, $\tilde A:=A_{\ltimes\fv^\nu}\in K[\der]$ and $\tilde A_a(y)=u^{-1}A_a(uy)$ for $y\in\Car[\imag]$.
Next, let  $\fm$,~$\w$ be as in the previous subsection.  As in Lemma~\ref{lem:2.1 summary} we consider the
continuous operator
$$\Phi_a\colon \Car[\imag]^{\b}\times \Car[\imag]^{\b}\to \Car[\imag]^{\b}$$ given by
$$\Phi_a(f,y)\ :=\ \Xi_a(f+y)-\Xi_a(f)\ =\  u\widetilde{A}_a^{-1}\!\left( u^{-1}\big( R(f+y)-R(f)\big) \right).$$
Here is our weighted version of Lemma~\ref{lem:2.1 summary}:

{\samepage
\begin{lemma}\label{lem:2.1 weighted}
Suppose the elements $\phi_j-\nu\fv^\dagger$,~$\phi_j-\nu\fv^\dagger-\fm^\dagger$ of $\c_{a_0}[\imag]$ are alike, for~$j=1,\dots,r$, and let $f\in\Car[\imag]^{\b}$. Then the operator~$y\mapsto \Phi_a(f,y)$
maps~$\Car[\imag]^\w$ into itself. Moreover, there are~$E_a,E_a^+\in\R^{\geq}$ such that for all $g\in \Car[\imag]^{\b}$ and~$y\in \Car[\imag]^{\w}$, 
$$\dabs{\Phi_a(f,y)}^\w_{a;r}  
 \ \le\ E_a\cdot \max\!\big\{1, \|f\|_{a;r}^d\big\}\cdot  \big( 1+\dabs{y}_{a;r}+\cdots+\dabs{y}_{a;r}^{d-1} \big)\cdot\dabs{y}_{a;r}^\w,$$
\vskip-1.5em
\begin{multline*}
\|\Phi_a(f,g+y)-\Phi_a(f,g) \|_{a;r}^\w\ \le\ \\ E_a^+\cdot \max\!\big\{1, \|f\|_{a;r}^d\big\}\cdot \max\!\big\{1, \|g\|_{a;r}^d\big\}\cdot \big( 1+\dabs{y}_{a;r}+\cdots+\dabs{y}_{a;r}^{d-1} \big)\cdot\dabs{y}_{a;r}^\w.
\end{multline*}
We can take these $E_a$, $E_a^+$ such that $E_a,E_a^+ \to 0$ as $a\to \infty$, and do so below. 
\end{lemma}}
\begin{proof}
Let $y\in \Car[\imag]^{\w}$. 
By Taylor expansion we have
$$R(f+y)-R(f)\ =\ \sum_{\abs{\i}>0} \frac{1}{\i !}R^{(\i)}(f)y^{\i}\ =\ \sum_{\abs{\i}>0} S_{\i}(f)y^{\i}
\quad\text{where $S_{\i}(f):= \frac{1}{\i !}\sum_{\j}R_{\j}^{(\i)}f^\j$,}$$
and $u^{-1}S_{\i}(f)\in\c_a[\imag]^{\b}$. So
$h:=u^{-1}\big(R(f+y)-R(f)\big) \in \c_a[\imag]^\w$, since~$\c_a[\imag]^\w$ is an ideal of  $\c_a[\imag]^{\b}$. 
Applying Proposition~\ref{prop:1.5 weighted}(i) with $\phi_j-\nu\fv^\dagger$ in the role of $\phi_j$ yields~$\Phi_a(f,y)=u\widetilde{A}_a^{-1}(h) \in \Car[\imag]^\w$, establishing the first claim. Next, let $g\in \Car[\imag]^{\b}$. Then $\Phi_a(f,g+y)-\Phi_a(f,g) = \Phi_a(f+g,y)$ by \eqref{eq:Phia difference}. Therefore,  
\begin{align*} \Phi_a(f,g+y)-\Phi_a(f,g)\ =\  u\widetilde{A}_a^{-1}(h),&\quad h\ :=\ u^{-1}\big(R(f+g+y)-R(f+g)\big), \text{ so}\\
\dabs{\Phi_a(f,g+y)-\Phi_a(f,g)}_{a;r}^\w\ &=\ \dabs{ u\widetilde{A}_a^{-1}(h) }_{a;r}^\w\ \leq\ \dabs{u\widetilde{A}_a^{-1}}_{a;r}^\w\cdot
\dabs{h}^\w_a.
\end{align*}
By Corollary~\ref{cor:weighted bd} we have 
$$\dabs{h}_a^\w\ \leq\ D
\cdot\max_{\abs{\i}>0}\, \dabs{u^{-1}S_{\i}(f+g)}_a\cdot \big( 1+\dabs{y}_{a;r}+\cdots+\dabs{y}_{a;r}^{d-1} \big)\cdot\dabs{y}_{a;r}^\w$$
where $D=D(d,r):=\big( \textstyle{d+r+1 \choose r+1}-1\big)$.
Let $D_a$ be as in the proof of Lem\-ma~\ref{bdua, bds}. Then $D_a\to 0$ as $a\to \infty$, and  Lemma~\ref{lem:inequ power d} gives for $\abs{\i}>0$,
\begin{align*} \dabs{u^{-1}S_{\i}(f)}_a\ &\leq\ D_a\cdot\max\!\big\{1,\dabs{f}_{a;r}^d\big\}\\
\dabs{u^{-1}S_{\i}(f+g)}_a\ &\leq\ D_a\cdot\max\!\big\{1,\dabs{f+g}_{a;r}^d\big\}\\
&\le\ 2^dD_a\cdot \max\!\big\{1,\dabs{f}_{a;r}^d\big\}\cdot\max\!\big\{1,\dabs{g}_{a;r}^d\big\}.
\end{align*}
This gives the desired result for 
$E_a:=\dabs{u\widetilde{A}_a^{-1}}_{a;r}^\w\cdot D\cdot D_a$ and $E_a^+:=2^d E_a$, using also Proposition~\ref{prop:1.5 weighted}(iii)
with $\phi_j-\nu\fv^\dagger$ in the role of $\phi_j$.
\end{proof}

\noindent
Lemma~\ref{lem:2.1 weighted} allows us to refine Theorem~\ref{thm:fix} as follows:

\begin{cor}\label{cor:2.3 weighted}
Suppose the elements $\phi_j-\nu\fv^\dagger$,~$\phi_j-\nu\fv^\dagger-\fm^\dagger$ of $\c_{a_0}[\imag]$ are alike, for $j=1,\dots,r$, and $R(0)\preceq \fv^\nu\fm$. Then
for sufficiently large $a$ the operator~$\Xi_a$ maps the closed ball
$B_a := \big\{f\in \Car[\imag]:\, \|f\|_{a;r}\le 1/2\big\}$
of the normed space~$\Car[\imag]^{\b}$ into itself, has a unique fixed point in $B_a$, and this fixed point 
lies in $\Car[\imag]^\w$.
\end{cor}
\begin{proof}
Take $a$ such that $\dabs{\w}_a\leq 1$. Then by \eqref{eq:weighted norm, 1}, $B_a$ contains the closed ball
$$B_a^\w\ :=\ \big\{f\in \Car[\imag]:\, \|f\|^\w_{a;r}\le 1/2\big\}$$
of the normed space $\Car[\imag]^\w$. Let $f,g\in B_a^\w$. Then
$\Xi_a(g)-\Xi_a(f)=\Phi_a(f,g-f)$ lies in $\Car[\imag]^\w$ by Lemma~\ref{lem:2.1 weighted}, and
 with $E_a$ as in that lemma,
\begin{align*}
\dabs{\Xi_a(f)-\Xi_a(g)}_{a;r}^\w\	&=\ \dabs{\Phi_a(f,g-f)}_{a;r}^\w \\
									& \leq\  E_a\cdot
									\max\!\big\{1, \|f\|_{a;r}^d\big\}\cdot  \big( 1+\cdots+\dabs{g-f}_{a;r}^{d-1} \big)\cdot\dabs{g-f}_{a;r}^\w \\
									& \leq\  E_a\cdot d\cdot \dabs{g-f}^\w_{a;r}. 
\end{align*}
Taking $a$ so that moreover $E_ad\leq\frac{1}{2}$ we obtain 
\begin{equation}\label{eq:2.3 weighted}
\dabs{\Xi_a(f)-\Xi_a(g)}_{a;r}^\w\leq\textstyle\frac{1}{2}\dabs{f-g}^\w_{a;r}\quad\text{ for all
$f,g\in B_a^\w$.}
\end{equation} 
Next we consider the case $g=0$. 
Our hypothesis $R(0)\preceq \fv^\nu \fm$ gives $u^{-1}R(0)\in \c_a[\imag]^\w$. Proposition~\ref{prop:1.5 weighted}(i),(ii)
with $\phi_j-\nu\fv^\dagger$ in the role of~$\phi_j$ gives~$\Xi_a(0)\in \Car[\imag]^\w$ and $\dabs{\Xi_a(0)}_{a;r}^\w \leq \dabs{u\widetilde{A}_a^{-1}}_{a;r}^\w\dabs{u^{-1}R(0)}_a^\w$.
Using Proposition~\ref{prop:1.5 weighted}(iii) we now take~$a$ so large that
$\dabs{\Xi_a(0)}_{a;r}^\w \leq \frac{1}{4}$. Then \eqref{eq:2.3 weighted} for $g=0$ gives $\Xi_a(B_a^\w)\subseteq B_a^\w$. 
By Lemma~\ref{lem:w-conv} the normed space~$\Car[\imag]^\w$ is complete,
hence~$\Xi_a$ has a unique fixed point in $B_a^\w$.
\end{proof}

\noindent
Now suppose in addition that $A\in H[\der]$ and $R\in H\{Y\}$.  Set
$$(\Car)^\w\ :=\ \big\{f\in\Car:\dabs{f}^\w_{a;r}<\infty\}\ =\ \Car[\imag]^\w\cap\Car,$$
a real Banach space with respect to $\dabs{\,\cdot\,}_{a;r}^\w$.
Increase $a_0$ as at the beginning of the subsection {\it Preserving reality}\/ of Section~\ref{sec:split-normal over Hardy fields}. Then we have the map 
$$\Re\Phi_a\colon(\Car)^{\b}\times(\Car)^{\b}\to(\Car)^{\b},\qquad (f,y)\mapsto \Re\!\big(\Phi_a(f,y)\big).$$
Suppose the elements $\phi_j-\nu\fv^\dagger$,~$\phi_j-\nu\fv^\dagger-\fm^\dagger$ are alike for  $j=1,\dots,r$, and let~$a$ and $E_a, E_a^{+}$  be as in
Lemma~\ref{lem:2.1 weighted}. Then this lemma yields:

\begin{lemma}\label{lem:2.1 weighted, real}
Let $f,g\in (\Car)^{\b}$ and $y\in(\Car)^\w$. Then $(\Re\Phi_a)(f,y)\in  (\Car)^\w$ and 
$$\dabs{\Re(\Phi_a)(f,y)}^\w_{a;r}  
 \ \le\ E_a\cdot \max\!\big\{1, \|f\|_{a;r}^d\big\}\cdot  \big( 1+\dabs{y}_{a;r}+\cdots+\dabs{y}_{a;r}^{d-1} \big)\cdot\dabs{y}_{a;r}^\w,$$
\begin{multline*}
\|(\Re\Phi_a)(f,g+y)-(\Re\Phi_a)(f,g) \|_{a;r}^\w\ \le\ \\ E_a^{+}\cdot \max\!\big\{1, \|f\|_{a;r}^d\big\}\cdot \max\!\big\{1, \|g\|_{a;r}^d\big\}\cdot \big( 1+\dabs{y}_{a;r}+\cdots+\dabs{y}_{a;r}^{d-1} \big)\cdot\dabs{y}_{a;r}^\w.
\end{multline*}
\end{lemma}

\noindent
The same way we derived Corollary~\ref{cor:2.3 weighted} from Lemma~\ref{lem:2.1 weighted}, this leads to: 

\begin{cor}\label{cor:2.3 weighted, real}  If $R(0)\preceq \fv^\nu\fm$, then
for sufficiently large $a$ the operator~$\Re\Xi_a$ maps the closed ball
$B_a := \big\{f\in \Car:\, \|f\|_{a;r}\le 1/2\big\}$
of the normed space $(\Car)^{\b}$ into itself, has a unique fixed point in $B_a$, and this fixed point 
lies in $(\Car)^\w$.
\end{cor}

\subsection*{Revisiting Lemma~\ref{lem:Psin, b}}
We adopt the setting of the
previous subsection. As usual, $a$ ranges over~$[a_0,\infty)$.
We continue the investigation of the differences~${f-g}$ between solutions~$f$,~$g$
of the equation \eqref{eq:ADE} on $[a_0,\infty)$ that began in Lemma~\ref{lem:close}, and so we
take $f$, $g$, $E$, $\epsilon$,  $h_a$, $\theta_a$ as in that lemma.
Recall that in the remarks preceding Lemma~\ref{lem:Psin, b} we defined continuous operators~$\Phi_a, \Psi_a\colon \Car[\imag]^{\b}\to \Car[\imag]^{\b}$ by
$$\Phi_a(y)\ :=\ \Phi_a(g,y)\ =\ \Xi_a(g+y)-\Xi_a(g), \quad \Psi_a(y)\ :=\ \Phi_a(y)+h_a\qquad(y\in\Car[\imag]^{\b}).$$
As in those remarks, we set $\rho:=\dabs{f-g}_{a_0;r}$ and
$$B_a\ :=\ \big\{ y\in\Car[\imag]^{\b}:\ \dabs{y-h_a}_{a;r} \leq 1/2\big\},$$
and take $a_1\ge a_0$ so that $\theta_a\in B_a$ for all $a\ge a_1$. Then by \eqref{eq:dabs(y)} we have~$\dabs{y}_{a;r}\leq 1+\rho$ 
for $a\geq a_1$ and $y\in B_a$.
Next, take $a_2\geq a_1$ as in Lemma~\ref{lem:Psin, b}; thus 
for~$a\geq a_2$ and~$y,z\in B_a$ we have
$\Psi_a(y) \in B_a$ and $\dabs{{\Psi_a(y)-\Psi_a(z)}}_{a;r} \leq \textstyle\frac{1}{2} \|y-z\|_{a;r}$.
As in the previous subsection,  $\fm\in H^\times$, $\fm\prec 1$, $\fm$ denotes also  a representative  in~$(\Cazr)^\times$ of the germ $\fm$, and $\w:=\fm|_{[a,\infty)}\in (\Car)^\times\cap (\Car)^{\b}$, so $\Car[\imag]^\w\subseteq \Car[\imag]^{\b}$.

\medskip
\noindent
{\it In the rest of this subsection
$\phi_1-\nu\fv^\dagger,\dots,\phi_r-\nu\fv^\dagger\in K$ are $\gamma$-repulsive for $\gamma:=v\fm\in v(H^\times)^>$,
and $h_a\in\Car[\imag]^\w$ for all $a\ge a_2$.}\/
Then Lemma~\ref{lem:repulsive} 
 gives~$a_3\geq a_2$ such that for all~$a\geq a_3$ and~$j=1,\dots,r$, the functions~$\phi_j-u^\dagger, \phi_j-(u\w)^\dagger\in\c_a[\imag]$ are alike and hence  $\Psi_a\big(\Car[\imag]^\w\big)\subseteq  \Car[\imag]^\w$
by Lem\-ma~\ref{lem:2.1 weighted}. 
Thus $\Psi_a^n(h_a) \in \Car[\imag]^\w$ for all~$n$ and $a\geq a_3$.  

For $a\geq a_2$ we have  $\lim_{n\to\infty} \Psi_a^n(h_a)=\theta_a$ in $\Car[\imag]^{\b}$ by Corollary~\ref{cor:Psin, b};
we now aim to strengthen this to ``in  $\Car[\imag]^{\w}$'' (possibly for a larger $a_2$).
Towards this:


\begin{lemma}\label{lem:Psia contract}
There exists $a_4\geq a_3$ such that  
$\dabs{{\Psi_a(y)-\Psi_a(z)}}^\w_{a;r} \leq \frac{1}{2} \|y-z\|^\w_{a;r}$ for all $a\ge a_4$ and $y,z\in B_a\cap \Car[\imag]^\w$.
\end{lemma}
\begin{proof}
For $a\geq a_3$ and $y,z\in\Car[\imag]^\w$, and with $E_a^+$ as in Lemma~\ref{lem:2.1 weighted} we have
\begin{multline*}
\|\Psi_a(y)-\Psi_a(z) \|_{a;r}^\w\ \le\ \\ 
E_a^+\cdot \max\!\big\{1, \|g\|_{a;r}^d\big\}\cdot \max\!\big\{1, \|z\|_{a;r}^d\big\}\cdot \big( 1+\dabs{y-z}_{a;r}+\cdots+\dabs{y-z}_{a;r}^{d-1} \big)\cdot\dabs{y-z}_{a;r}^\w.
\end{multline*}
For each $a\geq a_1$ and $y,z\in B_a$  we then have
$$\max\!\big\{1, \|z\|_{a;r}^d\big\}\cdot \big( 1+\dabs{y-z}_{a;r}+\cdots+\dabs{y-z}_{a;r}^{d-1} \big)\ \leq\ (1+\rho)^d\cdot d,$$
so taking $a_4\geq a_3$ with 
$$E_a^+\max\!\big\{1, \|g\|_{a_0;r}^d\big\} (1+\rho)^dd\le 1/2\quad\text{ for all $a\geq a_4$,}$$ 
we have $\dabs{{\Psi_a(y)-\Psi_a(z)}}^\w_{a;r} \leq \frac{1}{2} \|y-z\|^\w_{a;r}$ for all $a\ge a_4$ and $y,z\in B_a\cap \Car[\imag]^\w$.
\end{proof}

\noindent
Let $a_4$ be as in the previous lemma.

\begin{cor}\label{cor:small theta, 1}
Suppose $a\geq a_4$. Then $\theta_a\in \Car[\imag]^\w$ and $\lim_{n\to \infty} \Psi_a^n(h_a)=\theta_a$ in the normed space $\Car[\imag]^\w$. In particular,~$f-g,(f-g)',\dots, (f-g)^{(r)}\preceq\fm$.
\end{cor}

\begin{proof} 
We have $\Phi_a(h_a)=\Psi_a(h_a)-h_a\in\Car[\imag]^\w$, so  $M:=\dabs{\Phi_a(h_a)}_{a;r}^\w<\infty$. Since~$\Psi_a(B_a)\subseteq B_a$, 
induction on $n$ using Lemma~\ref{lem:Psia contract} gives
$$\dabs{\Psi_a^{n+1}(h_a)-\Psi^n_a(h_a)}^\w_{a;r} \leq M/2^n\qquad\text{ for all $n$.}$$
Thus $\big( \Psi_a^n(h_a) \big)$ is a cauchy sequence in the normed space~$\Car[\imag]^\w$, and so converges in~$\Car[\imag]^\w$ by Lemma~\ref{lem:w-conv}.
In the normed space $\Car[\imag]^{\b}$ we have $\lim_{n\to\infty}\Psi_a^n(h_a)  =\theta_a$,
by Corollary~\ref{cor:Psin, b}.
Thus $\lim_{n\to \infty}\Psi_a^n(h_a)  =\theta_a$ in $\Car[\imag]^\w$ by Lem\-ma~\ref{lem:w-conv => b-conv}.
\end{proof}




\subsection*{An application to slots in $H$}
Here we adopt the setting of the subsection {\it An application to slots in $H$}\/ in Section~\ref{sec:ueeh}. Thus
$H\supseteq\R$ is a Liouville closed Hardy field,
$K:=H[\imag]$, 
 $\I(K)\subseteq K^\dagger$, and
$(P,1,\hat h)$ is a slot in $H$ of order~$r\geq 1$; we set~$w:= \wt(P)$, $d:= \deg P$. 
{\it Assume also that $K$ is  $1$-linearly surjective if $r\geq 3$.}\/

\begin{prop} \label{prop:notorious 3.6}  
Suppose $(P,1,\hat h)$ is special, ultimate, $Z$-minimal, deep,  and strongly repulsive-normal. 
Let $f,g\in \Calr[\imag]$ and $\fm\in H^\times$ be such that 
$$P(f)\ =\ P(g)\ =\ 0,\qquad f,g\ \prec\ 1, \qquad v\fm\in v(\hat h - H).$$
Then $(f-g)^{(j)}\preceq\fm$ for $j=0,\dots,r$.
\end{prop}
\begin{proof}
We arrange~${\fm\prec 1}$. Let $\fv:=\abs{\fv(L_P)}\in H^>$, so $\fv\prec^\flat 1$, and set $\Delta:=\Delta(\fv)$.
Take $Q,R\in H\{Y\}$ where $Q$ is homogeneous of degree $1$ and order $r$, 
$A:=L_Q\in H[\der]$ has a strong $\hat h$-repulsive splitting over~$K$,
 $P=Q-R$, and $R\prec_\Delta \fv^{w+1}P_1$, so~$\fv(A)\sim \fv(L_P)$ by Lemma~\ref{lem:fv of perturbed op}. Multiplying $P$, $Q$, $R$ by some $b\in H^\times$ we
arrange that $A$ is monic, so $A=\der^r+ f_1\der^{r-1}+\cdots + f_r$ with $f_1,\dots, f_r\in H$ and~$R\prec_\Delta\fv^w$.
Let $(\phi_1,\dots,\phi_r)\in K^r$ be a strong $\hat h$-repulsive splitting of $A$ over $K$, so $\phi_1,\dots,\phi_r$ are $\hat h$-repulsive and
$$A\ =\ (\der-\phi_1)\cdots (\der-\phi_r), \qquad \Re\phi_1,\dots,\Re \phi_r\ \succeq\ \fv^\dagger\ \succeq\ 1. $$
By \eqref{bound on linear factors} we have $\phi_1,\dots,\phi_r\preceq\fv^{-1}$. Thus we can take $a_0\in \R$ and functions on~$[a_0,\infty)$ representing the germs $\phi_1,\dots, \phi_r$, $f_1,\dots, f_r$, $f$, $g$, as well as the~$R_{\j}$ with~${\j\in \N^{1+r}}$, $|\j|\le d$, $\|\j\|\le w$ (using the same symbols for the germs mentioned as for their chosen representatives)
so as to be in the situation described in the beginning of Section~\ref{sec:split-normal over Hardy fields}, with $f$ and $g$ solutions on $[a_0,\infty)$ of the differential equation~\eqref{eq:ADE} there. 
As there, we take $\nu\in\Q$ with $\nu > w$ so that~$R \prec_\Delta \fv^\nu$ and~$\nu\fv^\dagger\not\sim \Re\phi_j$
for $j=1,\dots,r$, and then increase $a_0$ to satisfy all assumptions for Lemma~\ref{bdua}. 
Proposition~\ref{specialvariant} gives
$v(\fv^\nu) \in v(\hat h-H)$, so
$\phi_j-\nu\fv^\dagger=\phi_j-(\fv^\nu)^\dagger$ ($j=1,\dots,r$)
is $\hat h$-repulsive by \cite[Lemma~4.5.13(iv)]{ADH4}, so $\gamma$-repulsive for $\gamma:=v\fm>0$. Now $A$ splits over~$K$, and $K$ is $1$-linearly surjective if $r\ge 3$,  hence 
$\dim_{\C}\ker_{\Univ} A =r$ by~\eqref{eq:2.4.8} and Lemma~\ref{lemvar12}. 
Thus 
by Corollary~\ref{cor:8.8 refined} we have $y,y',\dots,y^{(r)}\prec\fm$ for all~$y\in\Calr[\imag]$ with $A(y)=0$, $y\prec 1$. 
In particular, $\fm^{-1}h_a, \fm^{-1}h_a',\dots, \fm^{-1}h_a^{(r)}\prec 1$ for all $a\ge a_0$. Thus the assumptions on $\fm$ and the $h_a$ made just before Lemma~\ref{lem:Psia contract} are satisfied for a suitable choice of $a_2$,
so we can appeal to Corollary~\ref{cor:small theta, 1}.
\end{proof}

\noindent
The assumption that $K$ is $1$-linearly surjective for $r\ge 3$ was only used in the proof above to obtain $\dim_{\C}\ker_{\Univ} A =r$. So if $A$ as in this proof satisfies $\dim_{\C}\ker_{\Univ} A =r$, then we can drop this assumption about $K$, also in the next corollary. 

\begin{cor}\label{cor:notorious 3.6}
Suppose $(P,1,\hat h)$, $f$, $g$, $\fm$  are as in Proposition~\ref{prop:notorious 3.6}. Then 
$$f-g\in \fm\, \c^r[\imag]^{\preceq}.$$
\end{cor}
\begin{proof}
If $\fm\succeq 1$,  then  Lemma~\ref{lem:small derivatives of y}(ii) applied with $y=(f-g)/\fm$ and $1/\fm$ in place of~$\fm$ gives what we want.
Now assume $\fm\prec 1$. 
Since $\hat h$ is special over $H$, Proposition~\ref{prop:notorious 3.6}
applies to $\fm^{r+1}$ in place of $\fm$, so $(f-g)^{(j)}\preceq\fm^{r+1}$ for $j=0,\dots,r$. Now apply Lemma~\ref{lem:Qnk} to suitable representatives of $f-g$ and~$\fm$.
\end{proof}

\noindent
Later in this section we use Proposition~\ref{prop:notorious 3.6} and its Corollary~\ref{cor:notorious 3.6} to
strengthen some results from Section~\ref{secfhhf}.

\subsection*{Weighted refinements of results in Section~\ref{secfhhf}}
We now adopt the setting of the subsection {\it Reformulations}\/ of Section~\ref{secfhhf}. Thus $H\supseteq\R$ is a real closed Hardy field
with asymptotic integration, and $K:=H[\imag]\subseteq \Calinf[\imag]$ is its algebraic closure,
with value group $\Gamma:=v(H^\times)=v(K^\times)$.
The next lemma and its corollary refine Lemma~\ref{prop:as equ 1}. Let~$P$,~$Q$,~$R$,~$L_Q$,~$w$ be as introduced before that lemma, set~$\fv:=\abs{\fv(L_Q)}\in H^>$,  and, in case $\fv\prec 1$,  $\Delta:=\Delta(\fv)$.

\begin{lemma}\label{abc}
Let $f\in K^\times$ and $\phi_1,\dots,\phi_r\in K$ be such that
$$ L_Q \ =\  f(\der-\phi_1)\cdots(\der-\phi_r),\qquad \Re\phi_1,\dots,\Re\phi_r\ \succeq\ 1.$$
Assume $\fv\prec 1$ and $R\prec_\Delta \fv^{w+1}Q$. Let $\fm\in H^\times$, 
$\fm\prec 1$, $P(0)\preceq \fv^{w+2}\fm Q$. Suppose that for $j=1,\dots,r$ and all~$\nu\in\Q$ with $w<\nu<w+1$, 
$\phi_j-(\fv^\nu)^\dagger$ and ${\phi_j-(\fv^\nu\fm)^\dagger}$  are alike.
Then $P(y)=0$ and $y,y',\dots,y^{(r)}\preceq\fm$ for some $y\prec\fv^w$ in $\Calinf[\imag]$.  If~$P,Q\in H\{Y\}$,
then there is such $y$ in $\Calinf$.
\end{lemma}
\begin{proof} Note that $\phi_1,\dots, \phi_r\preceq \fv^{-1}$ by \eqref{bound on linear factors}, and   $R\prec_{\Delta} \fv^{w+1}Q$ gives~$f^{-1}R\prec_{\Delta} \fv^w$.
Take~$\nu\in\Q$ such that $w<\nu<w+1$, $f^{-1}R
\prec_{\Delta} \fv^{\nu}$ and~$\nu\fv^\dagger\not\sim\Re\phi_j$
for~$j=1,\dots,r$. Set $A\ :=\ f^{-1}L_Q$. From $\nu < w+1$ and 
$$R(0)\ =\ -P(0)\ \prec_\Delta\ \fv^{w+2}\fm Q$$ 
we obtain $f^{-1}R(0) \prec_{\Delta} \fv^{\nu}\fm$.  Now apply successively Corollary~\ref{cor:2.3 weighted}, 
Lem\-ma~\ref{bdua}, and
Corollary~\ref{cor:ADE smooth} to the equation $A(y)=f^{-1}R(y)$, $y\prec 1$ in the role of \eqref{eq:ADE} in Section~\ref{sec:split-normal over Hardy fields} to obtain the first part.  For the real variant, use instead Corollary~\ref{cor:2.3 weighted, real} and
Lemma~\ref{realbdua}.
\end{proof}

\noindent
Lemma~\ref{abc} with $\fm^{r+1}$ for $\fm$ has the following consequence, using Lemma~\ref{lem:Qnk}: 

\begin{cor}\label{cor:abc}  Let  $f\in K^\times$ and $\phi_1,\dots,\phi_r\in K$ be such that
$$  L_Q \ =\  f(\der-\phi_1)\cdots(\der-\phi_r),\qquad \Re\phi_1,\dots,\Re\phi_r\ \succeq\ 1.$$
Assume $\fv\prec 1$ and $R\prec_\Delta \fv^{w+1}Q$. Let $\fm\in H^\times$, 
$\fm\prec 1$, $P(0)\preceq \fv^{w+2}\fm^{r+1} Q$. Suppose 
that  for~${j=1,\dots,r}$ and all $\nu\in\Q$ with $w<\nu<w+1$,
$\phi_j-(\fv^\nu)^\dagger$ and~${\phi_j-(\fv^\nu\fm^{r+1})^\dagger}$  are alike.
Then for some $y\prec\fv^w$ in $\Calinf[\imag]$ we have~${P(y)=0}$ and $y\in \fm\, \c^r[\imag]^{\preceq}$.
 If $P,Q\in H\{Y\}$,
then there is such $y$ in $\Calinf$.
\end{cor} 

\begin{remark}
If $H$ is a $\c^{\infty}$-Hardy field, then Lem\-ma~\ref{abc} and Corollary~\ref{cor:abc} go through with $\Calinf[\imag]$, $\Calinf$ replaced by~$\mathcal{C}^{\infty}[\imag]$, $\mathcal{C}^{\infty}$, respectively. Likewise with $\c^\omega$ in place of $\c^\infty$. (Use Corollary~\ref{cor:ADE smooth}.)  
\end{remark}

\noindent
Next a  variant of Lemma~\ref{dentsolver}. {\it In the rest of this subsection $(P,\fn,\hat h)$ is a deep, strongly repulsive-normal, $Z$-minimal slot in $H$ of order~$r\ge 1$ and weight~$w:=\wt(P)$.  We assume also that $(P,\fn, \hat h)$ is special \textup{(}as will be the case if $H$ is $r$-linearly newtonian,  and $\upo$-free if $r>1$, by Lemma~\ref{lem:special dents}\textup{)}}. 


\begin{lemma}\label{bcd} 
Let $\fm\in H^\times$ be such that $v\fm\in v(\hat h-H)$, $\fm\prec\fn$,  and $P(0)\preceq\fv(L_{P_{\times\fn}})^{w+2}\,(\fm/\fn)^{r+1}\, P_{\times\fn}$.  
Then for some~$y\in \Calinf$,
$$P(y)\ =\ 0,\quad  y\in \fm\,(\c^r)^{\preceq}.$$
If $H\subseteq \c^\infty$, then there is such $y$ in $\c^\infty$; likewise with $\c^\omega$ in place of $\c^\infty$.
\end{lemma}
\begin{proof} {\sloppy
Replace $(P,\fn,\hat h)$, $\fm$ by $(P_{\times\fn},1,\hat h/\fn)$, $\fm/\fn$, respectively, to arrange~${\fn=1}$. 
Then~$L_P$ has order $r$, $\fv(L_{P})\prec^{\flat} 1$, and $P =  Q-R$ where $Q,R\in H\{Y\}$,
$Q$ is homogeneous of degree~$1$ and order~$r$, $L_{Q}\in H[\der]$ has a  strong $\hat h$-repulsive splitting~${(\phi_1,\dots,\phi_r)\in K^r}$ over $K=H[\imag]$, 
and $R\prec_{\Delta^*} \fv(L_P)^{w+1} P_1$ with $\Delta^*:=\Delta\big(\fv(L_P)\big)$. 
By Lemma~\ref{lem:fv of perturbed op} we have $\fv(L_P)\sim \fv(L_Q)\asymp \fv$, so
$\Re \phi_j\succeq \fv^\dagger\succeq 1$ for~$j=1,\dots,r$, and~$\Delta = \Delta^*$.
Moreover, $P(0) \preceq \fv^{w+2}\fm^{r+1}Q$.
 Let $\nu\in\Q$, $\nu>w$, and~${j\in\{1,\dots,r\}}$. Then $0<v(\fv^\nu) \in v({\hat h-H})$ by Proposition~\ref{specialvariant}, so~$\phi_j$ is $\gamma$-repulsive for~${\gamma=v(\fv^\nu)}$, thus~$\phi_j$ and $\phi_j-(\fv^\nu)^\dagger$ are alike by Lemma~\ref{lem:repulsive}. 
 Likewise, $0 <v(\fv^\nu\fm^{r+1})\in v({\hat h-H})$ since $\hat h$ is special over $H$, so~$\phi_j$ and $\phi_j -(\fv^\nu \fm^{r+1})^\dagger$ are alike.
Therefore~${\phi_j-(\fv^\nu)^\dagger}$ and $\phi_j-(\fv^\nu\fm^{r+1})^\dagger$ are alike as well.   Corollary~\ref{cor:abc} now gives~${y\prec \fv^w}$ in~$\Calinf$ with~$P(y)=0$ and $y\in \fm\, (\c^r)^{\preceq}$. 
For the rest use the remark following that corollary.  }
\end{proof}

\begin{cor}\label{mfhc}
Suppose $\fn=1$,  and let  $\fm\in H^\times$ be such that $v\fm\in v(\hat h-H)$.  Then there are $h\in H$  and $y\in \Calinf$ such that:
$$  \hat h-h\ \preceq\ \fm,\qquad P(y)\ =\ 0,\qquad y\ \prec\ 1,\ y\in (\c^r)^{\preceq},\qquad
y-h\in \fm\,(\c^r)^{\preceq}.$$
If $H\subseteq \c^\infty$, then we have such $y\in\c^\infty$; likewise with $\c^\omega$ in place of $\c^\infty$. 
\end{cor}
\begin{proof}
Suppose first that $\fm\succeq 1$, and let $h:=0$ and $y$ be as in Lemma~\ref{dentsolver} for~$\phi=\fn=1$. Then $y\prec 1$, 
$y\in (\c^r)^{\preceq}$, so $y/\fm\prec 1$, $y/\fm\in (\c^r)^{\preceq}$ by the Product Rule. 
Next assume~$\fm\prec 1$. Set $\fv:=\abs{\fv(L_{P})}\in H^>$. Then $\fv \prec 1$, since $(P, 1, \hat h)$ is  deep. Proposition~\ref{specialvariant} gives $v\fv\in v(\hat h - H)$, and as $(P,1,\hat h)$ is special, this yields
$v(\fv\fm)\in v(\hat h-H)$. So we have
$h\in H$ with $\hat h-h\prec (\fv\fm)^{(w+3)(r+1)}\prec\fm$.
Proposition~\ref{small P(a)} for $\fn:= (\fv\fm)^{w+3}$  and $P\asymp P_{+h}$ by [ADH, 4.5.1(i)] now give
$$P_{+h}(0)\ =\ P(h)\ \prec\ (\fv\fm)^{w+3} P\ \preceq\ \fv^{w+3}\fm^{r+1}P_{+h}. $$
By \cite[Lemma~4.5.35]{ADH4}, $(P_{+h},1,\hat h-h)$ is strongly repulsive-normal, 
and by \cite[Corollary~3.3.8]{ADH4} it is   deep with $\fv(L_{P_{+h}}) \asymp_{\Delta(\fv)} \fv$. Hence Lemma~\ref{bcd} applies to the slot~${(P_{+h},1,\hat h-h)}$   in place of $(P,1,\hat h)$ to yield a $z\in\Calinf$ with
$P_{+h}(z)=0$ and $(z/\fm)^{(j)}\preceq 1$ for~$j=0,\dots,r$. 
  Lemma~\ref{lem:small derivatives of y}  gives $z^{(j)}\prec 1$ for~$j=0,\dots,r$. 
Set~$y:=h+z$; then $P(y)=0$, $y^{(j)} \prec 1$  and 
$\big( (y-h)/\fm\big){}^{(j)}\preceq 1$  for $j=0,\dots,r$. 
\end{proof}

\noindent
We now use the results above to approximate zeros of $P$ in $\Calinf$ by elements of~$H$:

\begin{cor}\label{cor:approx y} Suppose $H$ is Liouville closed, $\I(K)\subseteq K^\dagger$, $\fn=1$, and our slot~$(P,1,\hat h)$ in $H$ is ultimate. Assume also that $K$ is $1$-linearly surjective if~${r\geq 3}$. Let
$y\in\Calinf$ and $h\in H,\ \fm\in H^\times$ be such that~${P(y)=0}$, $y\prec 1$, and $\hat h-h\preceq\fm$.
Then  
$y-h\in \fm\,(\c^r)^{\preceq}$.
\end{cor}
\begin{proof}
Corollary~\ref{mfhc} gives  $h_1\in H$, $z\in\Calinf$ with $\hat h-h_1\preceq\fm$, $P(z)=0$, $z\prec 1$,  and $\big( {(z-h_1)/\fm}\big){}^{(j)} \preceq 1$ for $j=0,\dots,r$.  Now 
$$\frac{y-h}{\fm}\ =\ \frac{y-z}{\fm} + \frac{z-h_1}{\fm} + \frac{h_1-h}{\fm}$$
with 
${\big( (y-z)/\fm \big){}^{(j)} \preceq 1}$ for $j=0,\dots,r$ by Corollary~\ref{cor:notorious 3.6}.
Also $(h_1-h)/\fm\in H$ and $(h_1-h)/\fm\preceq 1$, so
$\big( (h_1-h)/\fm\big){}^{(j)}  \preceq 1$ for all $j\in\N$.
\end{proof}



\section{Asymptotic Similarity} \label{sec:asymptotic similarity}

\noindent 
Let $H$ be a Hausdorff field and $\hat{H}$ an immediate valued field extension of~$H$. 
Equip~$\hat H$ with the unique field ordering making it an ordered field extension of~$H$ such that $\mathcal O_{\hat H}$  is convex [ADH, 3.5.12]. 
Let~$f\in\c$ and~$\hat{f}\in \hat{H}$  be given. 

\begin{definition} \label{p:simH}
Call $f$  {\bf asymptotically similar\/} to $\hat{f}$ over $H$ (notation: $f\sim_{H}\hat{f}$) if $f\sim \phi$  in $\c$  and~$\phi\sim \hat{f}$ in $\hat{H}$ for some $\phi\in H$. 
(Note that then $f\in\c^\times$ and $\hat f\neq 0$.)\index{germ!asymptotically similar}  
\end{definition}

\noindent
Recall that the binary relations~$\sim$ on $\c^\times$ and~$\sim$ on $\hat H^\times$ are  equivalence relations which 
restrict to the same equivalence relation on $H^\times$. As a consequence, if
$f\sim_H\hat f$, then $f\sim \phi$ in~$\c$ for any~$\phi\in H$ with~$\phi\sim \hat f$ in  $\hat H$, and $\phi\sim\hat f$ in $\hat H$ for any~$\phi\in H$ with $f\sim\phi$ in $\c$. Moreover,
if $f\in H$, then $f \sim_H \hat f \Leftrightarrow f\sim\hat f$ in~$\hat H$, and if $\hat f\in H$, then $f \sim_H \hat f \Leftrightarrow f\sim\hat f$ in $\c$.

\begin{lemma}\label{lem:simH sim}
Let $f_1\in\c$, $f_1\sim f$, let $\hat f_1\in \hat H_1$ for an immediate valued field extension $\hat H_1$ of $H$, and suppose $\hat f\sim\theta$ in $\hat H$ and $\hat f_1\sim \theta$ in $\hat H_1$ for some $\theta\in H$. 
 Then:~$f\sim_H \hat f \Leftrightarrow f_1\sim_H\hat f_1$.
\end{lemma}

\noindent
For $\fn\in H^\times$ we have $f\sim_H \hat{f}\Leftrightarrow \fn f\sim_H \fn \hat{f}$. Moreover, by \cite[Lemma~2.1]{ADH5}:

\begin{lemma}\label{lem:simH}
Let $g\in\c$, $\hat g\in\hat H$, and suppose $f\sim_H\hat f$ and $g\sim_H\hat g$. Then ${1/f\sim_H 1/\hat f}$ and $fg\sim_H\hat f\,\hat g$. Moreover,
$$f\preceq g \text{ in $\c$} \quad\Longleftrightarrow\quad \hat f\preceq \hat g \text{ in $\hat H$,}$$
and likewise with $\prec$, $\asymp$, or $\sim$ in place of $\preceq$.
\end{lemma}

\noindent
Lemma~\ref{lem:simH} readily yields:

{\samepage 
\begin{cor}\label{cor:simHt}
Suppose
$\hat f$ is transcendental over $H$ and  $Q(f)\sim_H Q(\hat f)$ for all~$Q\in H[Y]^{\neq}$. 
Then we have:
\begin{enumerate}
\item[\textup{(i)}]  a subfield $H(f)\supseteq H$ of $\c$ generated by $f$ over~$H$;
\item[\textup{(ii)}] a field isomorphism $\iota\colon H(f)\to H(\hat f)$ over $H$ with~$\iota(f)=\hat f$;  
\item[\textup{(iii)}] with $H(f)$ and $\iota$ as in \textup{(i)} and \textup{(ii)} we have $g\sim_H\iota(g)$ for all $g\in H(f)^{\times}$, hence for all $g_1,g_2\in H(f)$: 
$g_1\preceq g_2$  in~$\c$~$\Leftrightarrow$ $\iota(g_1)\preceq \iota(g_2)$  in~$\hat{H}$.
\end{enumerate}
Also, $\iota$ in \textup{(ii)} is an ordered field isomorphism, where  the ordering on~$H(f)$ is its ordering as a Hausdorff field.
\end{cor}}
\begin{proof} To see that $\iota$ is order preserving, use that $\iota$ is a valued field isomorphism by~(iii), 
and apply [ADH, 3.5.12].
\end{proof}

\noindent
Here is the analogue when $\hat f$ is algebraic over $H$: 

\begin{cor}\label{cor:simHa}
Suppose  $\hat f$ is algebraic over $H$ with minimum polynomial~$P$ over~$H$ of degree $d$, and
$P(f)=0$, $Q(f)\sim_H Q(\hat f)$ for all~$Q\in H[Y]^{\neq}$ of degree~$<d$.
Then we have:
\begin{enumerate}
\item[\textup{(i)}]  a subfield $H[f]\supseteq H$ of $\c$ generated by $f$ over~$H$;
\item[\textup{(ii)}] a field isomorphism $\iota\colon H[f]\to H[\hat f]$ over $H$ with~$\iota(f)=\hat f$;  
\item[\textup{(iii)}] with $H[f]$ and $\iota$ as in \textup{(i)} and \textup{(ii)} we have $g\sim_H\iota(g)$ for all $g\in H[f]^{\times}$, hence for all $g_1,g_2\in H[f]$: 
$g_1\preceq g_2$  in~$\c$~$\Leftrightarrow$ $\iota(g_1)\preceq \iota(g_2)$  in~$\hat{H}$.
\end{enumerate}
Also, $\iota$ in \textup{(ii)} is an ordered field isomorphism, where  the ordering on $H[f]$ is its ordering as a Hausdorff field.
 \end{cor}

\noindent
If $\hat f\notin H$, then to show that $f-\phi\sim_H \hat f-\phi$ for all $\phi\in H$ it is enough to do this for $\phi$ arbitrarily close to $\hat f$: 

\begin{lemma}\label{lem:simH phi0}
Let $\phi_0\in H$ be such that $f-\phi_0\sim_H\hat f-\phi_0$. Then $f-\phi\sim_H \hat f-\phi$ for all $\phi\in H$ with
$\hat f-\phi_0 \prec \hat f-\phi$.
\end{lemma}
\begin{proof}
Let $\phi\in H$ with $\hat f-\phi_0 \prec \hat f-\phi$. Then $\phi_0-\phi\succ \hat f-\phi_0$, so $\hat f-\phi=(\hat f-\phi_0)+(\phi_0-\phi)\sim\phi_0-\phi$. By Lemma~\ref{lem:simH} we also have $\phi_0-\phi\succ f-\phi_0$, and hence likewise
$f-\phi \sim\phi_0-\phi$. 
\end{proof}

\noindent
We define: $f\approx_H\hat f: \Leftrightarrow f-\phi \sim_H \hat f-\phi$ for all $\phi\in H$. \label{p:approxH}
If $f\approx_H\hat f$, then~$f\sim_H\hat f$ as well as $f,\hat f\notin H$, and 
$\fn f\approx_H \fn\hat f$ for all $\fn\in H^\times$. Hence
$f\approx_H \hat f$ iff $f, \hat f\notin H$ and the isomorphism $\iota\colon H+H f\to H+H\hat f$ of $H$-linear spaces that is the identity on~$H$ and sends $f$ to $\hat f$   satisfies~$g\sim_H \iota(g)$ for all nonzero $g\in H+Hf$.

\medskip
\noindent
Here is an easy consequence of Lemma~\ref{lem:simH phi0}: 

\begin{cor}\label{cor:simH phi0}
Suppose $\hat f\notin H$ and $f-\phi_0 \sim_H \hat f-\phi_0$ for all  $\phi_0\in H$ such that~$\phi_0\sim\hat f$. Then~${f\approx_H \hat f}$.
\end{cor}
\begin{proof}
Take $\phi_0\in H$ with $\phi_0\sim\hat f$, and
let $\phi\in H$ be given. If~${\hat f-\phi\prec\hat f}$, then~${f-\phi \sim_H \hat f-\phi}$ by hypothesis;
otherwise we have~$\hat f-\phi\succeq\hat f\succ \hat f-\phi_0$, and then~${f-\phi} \sim_H \hat f-\phi$ by Lemma~\ref{lem:simH phi0}.
\end{proof}

\noindent
Lemma~\ref{lem:simH sim} yields an analogue  for $\approx_H$:

\begin{lemma}\label{lem:approxH sim}
Let $f_1\in\c$ be such that $f_1-\phi\sim f-\phi$ for all $\phi\in H$, 
and let $\hat f_1$ be an element of an immediate valued field extension of $H$  such that
$v(\hat f-\phi)=v(\hat f_1-\phi)$ for all $\phi\in H$. Then~$f\approx_H \hat f$ iff $f_1\approx_H\hat f_1$.
\end{lemma}

\noindent
Let $g\in\c$ be eventually strictly increasing with $g(t)\to+\infty$ as $t\to+\infty$; we then have the Hausdorff field
$H\circ g=\{h\circ g:h\in H\}$, with ordered valued field isomorphism~$h\mapsto h\circ g\colon H\to H\circ g$. (See \cite[Section~2]{ADH5}.) Suppose  
$$\hat h\mapsto \hat h\circ g\ \colon\  \hat H\to\hat H\circ g$$  extends this isomorphism to
a valued field isomorphism, where $\hat H\circ g$ is an immediate valued field extension of the Hausdorff field $H\circ g$.
Then
$$f\sim_H \hat f\quad\Longleftrightarrow\quad f\circ g\sim_{H\circ g} \hat f\circ g,\qquad
f\approx_H \hat f\quad\Longleftrightarrow\quad f\circ g\approx_{H\circ g} \hat f\circ g.$$

\subsection*{The complex version}
We now assume in addition that~$H$ is real closed, with algebraic closure 
$K:=H[\imag]\subseteq\c[\imag]$. We take $\imag$  with $\imag^2=-1$ also as an element of a field
$\hat K:=\hat H[\imag]$ extending both $\hat{H}$ and $K$, and equip $\hat K$ with the unique valuation ring of $\hat K$ lying over
$\mathcal O_{\hat H}$. (See \cite[remarks following Lemma~4.1.2]{ADH4}.) 
Then~$\hat K$ is an immediate valued field extension of~$K$.
Let~$f\in\c[\imag]$ and~$\hat f\in \hat K$ below.  \label{p:simK}

\medskip
\noindent
Call~$f$  {\bf asymptotically similar\/}\index{germ!asymptotically similar} to $\hat{f}$ over~$K$ (notation: $f\sim_{K}\hat{f}$) if for some~${\phi\in K}$ we have $f\sim \phi$  in $\c[\imag]$  and~$\phi\sim \hat{f}$ in~$\hat{K}$. 
Then~$f\in\c[\imag]^\times$ and~$\hat f\neq 0$. As before, 
if~$f\sim_K\hat f$, then $f\sim \phi$ in~$\c[\imag]$ for any~${\phi\in K}$ for which~$\phi\sim \hat f$ in~$\hat K$, and~$\phi\sim\hat f$ in $\hat K$ for any $\phi\in K$ for which~$f\sim\phi$ in~$\c[\imag]$.
Moreover, if~$f\in K$, then~$f \sim_K \hat f$ reduces to $f \sim \hat f$ in $\hat K^\times$. Likewise, if $\hat f\in K$, then
$f \sim_K \hat f$ reduces to $f \sim \hat f$ in~$\c[\imag]^\times$.

\begin{lemma}\label{lem:simH sim K}
Let $f_1\in\c[\imag]$ with $f_1\sim f$. Let  $\hat H_1$ be an immediate valued field extension of $H$, let $\hat K_1:=\hat H_1[\imag]$ be the corresponding immediate valued field extension of $K$ obtained from $\hat H_1$ as $\hat K$ was obtained from $\hat H$. Let  $\hat f_1\in \hat K_1$, and $\theta\in K$ be such that 
$\hat f\sim\theta$ in $\hat K$ and $\hat f_1\sim \theta$ in~$\hat K_1$. Then~$f\sim_K \hat f$ iff $f_1\sim_K\hat f_1$.
\end{lemma}

\noindent
For~$\fn\in K^\times$ we have $f\sim_K \hat{f}\Leftrightarrow \fn f\sim_K \fn \hat{f}$, and~$f\sim_K \hat f\Leftrightarrow\overline{f} \sim_K \overline{\hat f}$ (complex conjugation).
Here is a useful observation relating $\sim_K$ and $\sim_H$:

\begin{lemma}\label{lem:simK Re Im} Suppose $f\sim_K \hat{f}$ and 
$\Re \hat f\succeq \Im\hat f$; then $$\Re f\ \succeq\ \Im f,\qquad \Re f\ \sim_H\ \Re \hat f.$$
\end{lemma}
\begin{proof} Let $\phi\in K$ be such that $f\sim \phi$ in $\c[\imag]$ and
$\phi\sim \hat f$ in $\hat K$. The latter yields~$\Re\phi\succeq \Im\phi$ in $H$ and $\Re\phi\sim \Re\hat f$ in $\hat H$. Using
that $f=(1+\varepsilon)\phi$ with~$\varepsilon\prec 1$ in~$\c[\imag]$
it follows easily that $\Re f\succeq \Im f$ and $\Re f\sim \Re \phi$ in $\c$. 
\end{proof}

\begin{cor}\label{corsimas}
Suppose $f\in\c$ and $\hat f\in\hat H$. Then $f\sim_H \hat f$ iff $f\sim_K \hat f$.
\end{cor}

\noindent
Lemmas~\ref{lem:simH} and~\ref{lem:simH phi0} go through with~$\c[\imag]$,~$K$,~$\hat K$, and~$\sim_K$ in place of $\c$, $H$, $\hat H$,  and~$\sim_H$. 
We define:  $f\approx_K\hat f :\Leftrightarrow f-\phi \sim_K \hat f-\phi$ for all $\phi\in K$. 
Now Corollary~\ref{cor:simH phi0} goes through  with $K$, $\sim_K$, $\approx_K$ in place of~$H$,~$\sim_H$,~$\approx_H$. 


\begin{lemma}\label{lemsimimag}
Suppose $f\in\c$, $\hat f\in\hat H$, and $f\sim_H\hat f$.
Then $f+g\imag \sim_K \hat f+g\imag$ for all~$g\in H$.
\end{lemma}
\begin{proof}
Let $g\in H$, and take $\phi\in H$ with $f\sim \phi$ in $\c$ and $\phi \sim \hat f$ in $\hat H$. 
Suppose first that $g\prec\phi$. Then $g\imag\prec\phi$, and together with $f-\phi\prec\phi$ this yields~${(f+g\imag)-\phi\prec\phi}$, that is,
$f+g\imag \sim \phi$ in $\c[\imag]$.  
Using likewise analogous properties of $\prec$ on $\hat K$ we obtain~$\phi\sim\hat f+g\imag$ in $\hat K$. If $\phi\prec g$, then~$f\preceq \phi\prec g\imag$ and thus $f+g\imag \sim g\imag$ in~$\c[\imag]$, and
likewise $\hat f+g\imag \sim g\imag$ in $\hat K$.
Finally, suppose $g\asymp\phi$. Take $c\in\R^\times$ and~$\varepsilon\in H$ with $g = c\phi(1+\varepsilon)$ and~$\varepsilon\prec 1$.
We have
$f=\phi(1+\delta)$ where $\delta\in\c$, $\delta\prec 1$, so~$f+g\imag = \phi(1+c\imag)(1+\rho)$ where~$\rho=(1+c\imag)^{-1}(\delta+c\imag\varepsilon) \prec 1$ in $\c[\imag]$, so~$f+g\imag\sim \phi(1+c\imag)$ in $\c[\imag]$.  
Likewise, $\hat f+g\imag\sim \phi(1+c\imag)$ in $\hat K$. 
\end{proof}
 
\begin{cor}\label{cor:approxH vs approxK}
Suppose $f\in\c$ and $\hat f\in\hat H$. Then $f\approx_H \hat f$ iff $f\approx_K\hat f$.
\end{cor}
\begin{proof}
If $f\approx_K\hat f$, then for all $\phi\in H$ we have $f-\phi\sim_K \hat f-\phi$, so $f-\phi\sim_H \hat f-\phi$ by Corollary~\ref{corsimas},
hence $f\approx_H \hat f$. Conversely, suppose $f\approx_H \hat f$. Then for all $\phi\in K$
we have $f-\Re \phi \sim_H \hat f-\Re \phi$, so
$f-\phi \sim_K \hat f-\phi$ by Lemma~\ref{lemsimimag}.
\end{proof}

\noindent
Next we exploit that $K$ is algebraically closed:
 
\begin{lemma}
$f\approx_K \hat f \ \Longrightarrow\ Q(f) \sim_K Q(\hat f)$ for all $Q\in K[Y]^{\neq}$. 
\end{lemma} 
\begin{proof}
Factor $Q\in K[Y]^{\ne}$ as 
$$Q=a({Y-\phi_1})\cdots ({Y-\phi_n})\qquad\text{ where $a\in K^\times$ and~$\phi_1,\dots,\phi_n\in K$,}$$
and use $f-\phi_j \sim_K \hat f-\phi_j$  ($j=1,\dots,n$) and the complex version of Lem\-ma~\ref{lem:simH}.
\end{proof}

\noindent
This yields a more useful ``complex'' version of Corollary~\ref{cor:simHt}: 

\begin{cor}\label{complexsimHt} 
Suppose $f\approx_K\hat f$. Then $\hat f$ is transcendental over $K$, and:
\begin{enumerate}
\item[\textup{(i)}]  $f$ generates over $K$ a subfield $K(f)$ of $\c[\imag]$;
\item[\textup{(ii)}] we have a field isomorphism $\iota\colon K(f)\to K(\hat f)$ over $K$ with~$\iota(f)=\hat f$;  
\item[\textup{(iii)}]  $g\sim_K\iota(g)$ for all $g\in K(f)^{\times}$, hence  for all $g_1,g_2\in K(f)$:  $$g_1\preceq g_2 \text{  in }\c[\imag]\ \Longleftrightarrow\ \iota(g_1)\preceq \iota(g_2)  \text{ in }\hat{K}.$$
\end{enumerate}
\textup{(}Thus the restriction of the binary relation $\preceq$ on $\c[\imag]$ to~$K(f)$ is a dominance relation on the field $K(f)$ in the sense of
\textup{[ADH, 3.1.1].)}
\end{cor}

\noindent 
In the next lemma $f=g+h\imag$, $g,h\in \c$, and $\hat f=\hat g+\hat h \imag$, $\hat g, \hat h\in \hat H$. 
Recall from \cite[Lem\-ma~4.1.3]{ADH4} that if $\hat f\notin K$, then  $v(\hat g-H) \subseteq v(\hat h-H)$ or  $v(\hat h-H) \subseteq v(\hat g-H)$.

\begin{lemma}\label{lem:real part approx}
Suppose $f\approx_K \hat f$ and $v(\hat g-H) \subseteq v(\hat h-H)$. Then $g \approx_H  \hat g$.
\end{lemma}
\begin{proof}
Let $\rho\in H$ with $\rho\sim\hat g$; by Corollary~\ref{cor:simH phi0} it is enough to show that~${g-\rho} \sim_H \hat g-\rho$.
Take $\sigma\in H$ with $\hat g - \rho \succeq \hat h-\sigma$, and set $\phi:=\rho+\sigma\imag\in K$.
Then $$\Re(f-\phi)=g-\rho\quad\text{ and }\quad \Re(\hat f-\phi)=\hat g-\rho\succeq  \hat h-\sigma=\Im(\hat f-\phi),$$
and so by~$f-\phi\sim_K \hat f -\phi$ and Lemma~\ref{lem:simK Re Im} we have $g-\rho\sim_H \hat g-\rho$.
\end{proof}

\begin{cor}
If $f\approx_K \hat f$, then $\Re f \approx_H \Re \hat f$ or $\Im f\approx_H \Im\hat f$.
\end{cor}
 
\noindent
Let $g\in\c$ be eventually strictly increasing with $g(t)\to+\infty$ as $t\to+\infty$; we then have the subfield
$K\circ g=(H\circ g)[\imag]$ of $\c[\imag]$. Suppose the valued field isomorphism
$${h\mapsto h\circ g\colon H\to H\circ g}$$ is
extended to a valued field isomorphism $$\hat h\mapsto \hat h\circ g \colon  \hat H\to\hat H\circ g,$$ 
 where $\hat H\circ g$ is an immediate valued field extension of the Hausdorff field $H\circ g$. 
In the same way we took a common valued field extension $\hat K=\hat H[\imag]$ of $\hat H$ and~$K=H[\imag]$ we now take a common valued field extension $\hat K\circ g=(\hat H\circ g)[\imag]$ of $\hat H\circ g$ and~$K\circ g=(H\circ g)[\imag]$.
This makes $\hat K\circ g$ an immediate extension of $K\circ g$, and we have a 
unique valued field isomorphism~$y\mapsto y\circ g\colon \hat K\to \hat K\circ g$ extending the above
map~$\hat h\mapsto \hat h\circ g\colon \hat H\to\hat H\circ g$ and sending $\imag\in \hat{K}$ to $\imag\in \hat K\circ g$.
This map $\hat K\to \hat K\circ g$ also extends $f\mapsto f\circ g\colon K \to K\circ g$ and is the identity on $\C$. 
See the commutative diagram below, where the labeled arrows are valued field isomorphisms and all unlabeled arrows are natural inclusions.

\begin{equation*}
\xymatrix{%
&\hat{H}\circ g\ar@{->}[rr]\ar@{<-}[dl]_(0.6){\hat h\mapsto \hat h\circ g}  \ar@{<-}[dd]|!{[d];[d]}\hole && \hat{K}\circ g  \ar@{<-}[dd] \ar@{<-}[dl]_(0.6){y\mapsto y\circ g} \\
\hat{H}\ar@{->}[rr]\ar@{<-}[dd]&&\hat{K} \ar@{<-}[dd] \\ 
&H\circ g\ar@{->}[rr]|!{[r];[r]}\hole\ar@{<-}[dl]^(0.45){h\mapsto h\circ g} && K\circ g 
\ar@{<-}[dl]^(0.45){f\mapsto f\circ g} \\
H\ar@{->}[rr]&&K
}
\end{equation*}

\noindent
Now
we have
$$f\sim_K \hat f\quad\Longleftrightarrow\quad f\circ g\sim_{K\circ g} \hat f\circ g,\qquad
f\approx_K \hat f\quad\Longleftrightarrow\quad f\circ g\approx_{K\circ g} \hat f\circ g.$$
At various places in the next section we use this for a Hardy field $H$ and active~$\phi>0$ in $H$, with
$g=\ell^{\inv}$, $\ell\in \c^1,\ \ell'=\phi$. In that situation, $H^\circ:=H\circ g$, $\hat H^\circ:=\hat H\circ g$,
and $h^\circ:= h\circ g$, $\hat h^\circ := \hat h\circ g$ for $h\in H$ and
$\hat h\in \hat H$, and likewise with $K$ and $\hat K$ and their elements instead of $H$ and $\hat H$.

\section{Differentially Algebraic Hardy Field Extensions} \label{sec:d-alg extensions}

\noindent
In this section we are finally able to generate under reasonable conditions Hardy field extensions by solutions in $\Calinf$ of algebraic differential equations, culminating in the proof of our main theorem.
We begin with a generality about enlarging differential fields within an ambient differential ring.
Here, a {\em differential subfield\/} of a differential ring  $E$ is a differential subring of $E$ whose underlying ring is a field. 
 
\begin{lemma}\label{dsfe} Let $K$ be a differential field with irreducible $P\in K\{Y\}^{\ne}$ of order~$r\ge 1$, and $E$ a differential ring extension of $K$ with  $y\in E$ such that $P(y)=0$ and~$Q(y)\in E^\times$ for all $Q\in K\{Y\}^{\neq}$
of order $<r$. Then $y$ generates over $K$ a differential subfield $K\langle y \rangle\supseteq K$ of $E$. Moreover, $y$ has $P$ as a minimal annihilator over $K$ and~$K\langle y \rangle$ equals
$$\left\{\frac{A(y)}{B(y)}:\ A,B\in K\{Y\},\, \order A\leq r,\, \deg_{Y^{(r)}}A <\deg_{Y^{(r)}} P,\, B\ne 0,\, \order B < r\right\}.$$
\end{lemma}
\begin{proof} Let $p\in K[Y_0,\dots, Y_r]$ with distinct indeterminates $Y_0,\dots, Y_r$ be such that ${P(Y)=p(Y, Y',\dots, Y^{(r)})}$. The $K$-algebra morphism 
$K[Y_0,\dots, Y_r]\to E$ sending $Y_i$ to $y^{(i)}$ for $i=0,\dots,r$
extends to a $K$-algebra morphism $K(Y_0,\dots, Y_{r-1})[Y_r]\to E$  with $p$ in its kernel, and so induces a $K$-algebra morphism 
$$\iota\ :\ K(Y_0,\dots, Y_{r-1})[Y_r]/(p)\to E, \qquad (p)\ :=\ pK(Y_0,\dots, Y_{r-1})[Y_r].$$ 
Now $p$  as an element of $K(Y_0,\dots, Y_{r-1})[Y_r]$ remains irreducible \cite[Chapter~IV, \S{}2]{Lang}. Thus $K(Y_0,\dots, Y_{r-1})[Y_r]/(p)$ is a field, so $\iota$ is injective, and it is routine to check that the image of $\iota$ is $K\langle y \rangle$ as described; see also [ADH, 4.1.6]. 
\end{proof}

\noindent
{\it In the rest of this section 
$H$ is a real closed Hardy field with $H\supseteq \R$, 
and~$\hat{H}$ is an immediate $H$-field extension of $H$.}\/

\subsection*{Application to Hardy fields}  
 Let $f\in\Calinf$ and~$\hat f\in\hat H$.
Note that if~${Q\in H\{Y\}}$ and $Q(f)\sim_H Q(\hat f)$, then $Q(f)\in\c^\times$.
So by Lemma~\ref{dsfe}  with~$E=\Calinf$, $K=H$:

{\sloppy
\begin{lemma}\label{lem:dsfeh}
Suppose $\hat{f}$ is $\operatorname{d}$-algebraic over $H$ with minimal annihilator $P$ over~$H$ of order~$r\ge 1$,
and $P(f)=0$ and $Q(f)\sim_H Q(\hat f)$ for all $Q\in H\{Y\}\setminus H$ with~${\order Q < r}$. Then $f\notin H$ and: \begin{enumerate}
\item[\rm{(i)}] $f$ is hardian over $H$; 
\item[\rm{(ii)}] we have a \textup{(}necessarily unique\textup{)}  isomorphism $\iota\colon H\langle f\rangle \to H\langle\hat f\rangle$ of differential fields over $H$ such that $\iota(f)=\hat{f}$.
\end{enumerate}
\end{lemma}}

\noindent
With an extra assumption $\iota$ in Lemma~\ref{lem:dsfeh} is an isomorphism of $H$-fields:

{\sloppy
\begin{cor}\label{cor:dsfeh} 
Let $\hat f$, $f$, $P$, $r$, $\iota$ be as in Lem\-ma~\ref{lem:dsfeh}, and suppose also that~${Q(f)\sim_{H} Q(\hat{f})}$ for all~$Q\in H\{Y\}$ with $\order Q = r$ and $\deg_{Y^{(r)}} Q < \deg_{Y^{(r)}} P$.
Then  $g\sim_H \iota(g)$ for all~$g\in H\langle f\rangle^\times$,
hence for
${g_1,g_2\in H\langle f\rangle}$ we have $$g_1\preceq g_2 \text{  in }\c\ \Longleftrightarrow\
\iota(g_1)\preceq \iota(g_2)  \text{  in }\hat{H}.$$
Moreover, $\iota$ is an isomorphism of $H$-fields.
\end{cor}}
\begin{proof}
Most of this follows from Lemmas~\ref{lem:simH} and~\ref{lem:dsfeh} and the description of~$H\langle f\rangle$ in Lemma~\ref{dsfe}. For the last statement, use [ADH, 10.5.8]. 
\end{proof}

\subsection*{Analogues for $K$} We have the $\d$-valued extension~$K := H[\imag]\subseteq \Calinf[\imag]$ of $H$. As before we arrange that $\hat{K}=\hat{H}[\imag]$ is a $\d$-valued extension of $\hat{H}$ as well as  
an immediate extension of $K$. Let~${f\in\Calinf[\imag]}$ and~$\hat{f}\in \hat{K}$.
We now have the obvious ``complex'' analogues of Lemma~\ref{lem:dsfeh} and Corollary~\ref{cor:dsfeh}: 

\begin{lemma}\label{lem:dsfek}
Suppose $\hat{f}$ is $\d$-algebraic over $K$ with minimal annihilator $P$ over~$K$ of order~$r\ge 1$,
and $P(f)=0$ and $Q(f)\sim_K Q(\hat f)$ for all $Q\in K\{Y\}\setminus K$ with~${\order Q < r}$. Then \begin{enumerate}
\item[\rm{(i)}] $f$ generates over $K$ a differential subfield $K\langle f\rangle$ of $\Calinf[\imag]$;
\item[\rm{(ii)}] we have a \textup{(}necessarily unique\textup{)} isomorphism $\iota\colon K\langle f\rangle \to K\langle\hat f\rangle$ of differential fields over $K$ such that $\iota(f)=\hat{f}$.
\end{enumerate}
\end{lemma}

{\sloppy
\begin{cor}\label{cor:dsfek} 
Let $\hat f$, $f$, $P$, $r$, $\iota$ be as in Lem\-ma~\ref{lem:dsfek}, and   suppose also that~${Q(f)\sim_K Q(\hat{f})}$ for all $Q\in K\{Y\}$ with $\order Q = r$ and $\deg_{Y^{(r)}} Q < \deg_{Y^{(r)}} P$. 
Then $g\sim_K \iota(g)$ for all $g\in K\langle f\rangle^\times$, so for all $g_1,g_2\in K\langle f\rangle$ we have: 
$$g_1\preceq g_2  \text{ in }\c[\imag]\ \Longleftrightarrow\ \iota(g_1)\preceq \iota(g_2)  \text{ in }\hat{K}.$$
Thus the relation~$\preceq$ on $\c[\imag]$ restricts to a dominance relation on the field $K\langle f\rangle$.
\end{cor}}


\noindent
From $K$ being algebraically closed we obtain a useful variant of Corollary~\ref{cor:dsfek}: 

\begin{cor}\label{cor:dsfek, order 1}
Suppose $f\approx_K \hat f$, and $P\in K\{Y\}$ is irreducible with 
$$\order P\ =\ \deg_{Y'} P\  =\ 1, \quad P(f)=0\ \text{ in }\ \Calinf[\imag],\quad P(\hat f)=0\ \text{ in }\ \hat K.$$
Then $P$ is a minimal annihilator of $\hat f$ over~$K$, $f$ generates over $K$ a differential subfield~$K\<f\>=K(f)$ of~$\Calinf[\imag]$, and we have an isomorphism $\iota\colon K\<f\>\to K\<\hat f\>$
of differential fields over $K$ such that~${\iota(f)=\hat f}$ and~$g\sim_K\iota(g)$ for all $g\in K\langle f\rangle^\times$. Thus for all $g_1,g_2\in K\langle f\rangle$: 
$g_1\preceq g_2  \text{ in }\c[\imag]\ \Longleftrightarrow\ \iota(g_1)\preceq \iota(g_2)  \text{ in }\hat{K}.$
\end{cor}
\begin{proof} By Corollary~\ref{complexsimHt}, $\hat f$ is transcendental over $K$, so $P$ is a minimal annihilator
of $\hat f$ over $K$ by [ADH, 4.1.6]. Now use Lemma~\ref{dsfe} and Corollary~\ref{complexsimHt}. 
\end{proof}

\noindent
This corollary leaves open whether $\Re f$ or $\Im f$ is hardian over $H$. This issue is critical for us and we treat a special case in Proposition~\ref{prop:Re or Im} below.  The example following  Corollary~5.24 in \cite{ADH5}   shows that that there is a 
differential subfield of~$\Calinf[\imag]$ such that the
binary relation~$\preceq$ on~$\c[\imag]$ restricts to a dominance relation on it, but which is not contained in $F[\imag]$ for any
Hardy field $F$.

\subsection*{Sufficient conditions for asymptotic similarity}
Let~$\hat h$ be an element of our immediate $H$-field extension $\hat H$ of $H$.
Note that in the next variant of [ADH, 11.4.3] we use $\operatorname{ddeg}$ instead of $\ndeg$.

\begin{lemma}\label{1142} Let $Q\in H\{Y\}^{\neq}$, $r:=\order Q$, $h\in H$, and
$\fv\in H^\times$ be such that~${\hat h-h\prec \fv}$ and $\ddeg_{\prec \fv} Q_{+h}=0$.  
Assume $y\in \Calinf$ and $\fm\in H^\times$  satisfy
$$y-h\preceq\fm\prec\fv,\qquad \left(\frac{y-h}{\fm}\right)',\dots,\left(\frac{y-h}{\fm}\right)^{(r)} \preceq 1.$$ Then $Q(y)\sim Q(h)$ in $\Calinf$ and $Q(h)\sim Q(\hat h)$ in $\hat H$; in particular, $Q(y)\sim_H Q(\hat h)$.
\end{lemma} 
\begin{proof} We have $y=h+\fm u$ with $u=\frac{y-h}{\fm}\in \Calinf$ and
$u, u',\dots, u^{(r)}\preceq 1$.
Now 
$$Q_{+h,\times \fm}\ =\ Q(h) + R\qquad\text{ with }R\in H\{Y\},\ R(0)=0,$$ 
which in view of $\operatorname{ddeg} Q_{+h,\times \fm}=0$ gives $R\prec Q(h)$.
Thus 
$$Q(y)\ =\ Q_{+h,\times\fm}(u)\ =\ Q(h) + R(u),\qquad R(u)\ \preceq\ R\ \prec\ Q(h),$$ so $Q(y)\sim Q(h)$ in $\Calinf$. Increasing $|\fm|$ if necessary we arrange $\hat h - h\preceq \fm$, and then a similar computation with $\hat h$ instead of $y$ gives $Q(h)\sim Q(\hat h)$ in $\hat H$.  
\end{proof}

\noindent
{\it In the remainder of this subsection  we assume that $H$ is ungrounded and $H\ne \R$.} (See [ADH, p. 384] for {\em grounded\/} and {\em ungrounded}.) 
 
\begin{cor}  \label{cor:iso hat h, d-alg, weak}
Suppose $\hat h$ is $\d$-algebraic over $H$ with minimal annihilator $P$ over~$H$ of order $r\geq 1$, and let~${y\in\Calinf}$ satisfy $P(y)=0$. Suppose for all~$Q$ in~$H\{Y\}\setminus H$ of order $< r$  there are $h\in H$, $\fm,\fv\in H^\times$, and an active~$\phi>0$ in~$H$ such that~$\hat h-h\preceq\fm\prec\fv$,  $\ddeg_{\prec\fv} Q^\phi_{+h}=0$, and
$$ \derdelta^j\!\left(\frac{y-h}{\fm}\right)\preceq 1\qquad\text{ for }j=0,\dots,r-1\text{ and }\derdelta:=\phi^{-1}\der.$$
Then $y\notin H$ and $y$ is hardian over $H$.
\end{cor} 
\begin{proof}
Let $Q\in H\{Y\}\setminus H$ have order $< r$, and take $h$, $\fm$, $\fv$, $\phi$ as in the statement of the corollary. 
By Lemma~\ref{lem:dsfeh} it is enough to show that  then $Q(y) \sim_H Q(\hat h)$. We use $(\phantom{-})^\circ$ as explained at the beginning of Section~\ref{secfhhf}. Thus we have the  Hardy field~$H^\circ$ and the  $H$-field isomorphism $h\mapsto h^\circ\colon H^\phi\to H^\circ$, extended  to an $H$-field isomorphism $\hat f\mapsto \hat f^\circ\colon\hat H^\phi\to\hat H^\circ$, for an immediate $H$-field extension $\hat H^\circ$ of~$H^\circ$. 
Set $u:=(y-h)/\fm\in \Calinf$. We have $\ddeg_{\prec\fv^\circ} Q^{\phi\circ}_{+h^\circ}=0$ and  $(u^\circ)^{(j)}\preceq 1$ for~$j=0,\dots,r-1$;
hence $Q^{\phi\circ}(y^\circ)\sim_{H^\circ} Q^{\phi\circ}(\hat h^\circ)$ by Lemma~\ref{1142}. Now $Q^{\phi\circ}(y^\circ)=Q(y)^\circ$ in~$\Calinf$ and~$Q^{\phi\circ}(\hat h^\circ)=Q(\hat h)^\circ$ in $\hat H^\circ$, hence $Q(y) \sim_H Q(\hat h)$.
\end{proof}

\noindent
Using Corollary~\ref{cor:dsfeh} instead of Lemma~\ref{lem:dsfeh} we  show likewise:

{\sloppy
\begin{cor}\label{cor:iso hat h, d-alg}
Suppose $\hat h$ is $\d$-algebraic over $H$ with minimal annihilator $P$ over~$H$ of order $r\geq 1$, and let~${y\in\Calinf}$ satisfy $P(y)=0$. Suppose   that for all~$Q$ in~$H\{Y\}\setminus H$ with $\order Q\leq r$ and $\deg_{Y^{(r)}} Q<\deg_{Y^{(r)}} P$ there are~$h\in H$, ${\fm,\fv\in H^\times}$, and an active~$\phi>0$ in $H$ with $\hat h-h\preceq\fm\prec\fv$,  $\ddeg_{\prec\fv} Q^\phi_{+h}=0$, and
$$ \derdelta^j\!\left(\frac{y-h}{\fm}\right)\preceq 1\qquad\text{ for }j=0,\dots,r \text{ and }  \derdelta:=\phi^{-1}\der.$$
Then $y$ is hardian over $H$ and there is an isomorphism $H\langle y\rangle \to H\langle \hat h\rangle$ of $H$-fields over
$H$ sending~$y$ to $\hat h$.
\end{cor}}

\noindent
In the next subsection we use Corollary~\ref{cor:iso hat h, d-alg} to fill  in certain kinds of holes in Hardy fields.

\subsection*{Generating immediate $\d$-algebraic Hardy field extensions}
{\it In this subsection $H$ is Liouville closed, $(P, \fn, \hat h)$ is a special $Z$-minimal slot 
in $H$ of order~$r\ge 1$,  $K:=H[\imag]\subseteq\Calinf[\imag]$,  $\I(K)\subseteq K^\dagger$, and $K$ is $1$-linearly surjective  if~${r\ge 3}$.}\/ Reminder: $H$ being Liouville closed is stronger than $H$ having rational asymptotic integration. We first treat the case where~$(P, \fn, \hat h)$ is a hole in $H$ (not just a slot):

\begin{theorem} \label{46}
Assume $(P, \fn, \hat h)$ is a deep, ultimate, and strongly repulsive-normal hole in $H$, and $y\in \Calinf$, $P(y)=0$, $y\prec \fn$.
Then $y$ is hardian over~$H$, and there is an
isomorphism $H\langle y\rangle \to H\langle \hat h\rangle$ of $H$-fields over $H$ sending $y$ to $\hat h$.
\end{theorem}
\begin{proof} 
Replacing $(P, \fn, \hat h)$, $y$ by $(P_{\times \fn}, 1, \hat h/\fn)$, $y/\fn$ we arrange $\fn=1$. 
Let $Q$ in~$H\{Y\}\setminus H$, $\order Q\le r$, and
$\deg_{Y^{(r)}} Q< \deg_{Y^{(r)}} P$. Then $Q\notin Z(H,\hat h)$, so we have
$h\in H$  and $\fv\in H^\times$ such that $h-\hat h \prec \fv$ and 
$\ndeg_{\prec \fv} Q_{+h}=0$.  Take any $\fm\in H^\times$ with $\hat h - h\preceq \fm\prec \fv$. 
Take $\fw\in H^\times$ with $\fm\prec \fw\prec\fv$. Then~$\ndeg Q_{+h,\times \fw}=0$, so we have active $\phi$ in $H$, $0<\phi\prec 1$, with $\ddeg Q_{+h, \times \fw}^\phi =0$, and hence $\ddeg_{\prec \fw} Q_{+h}^\phi=0$. Thus renaming $\fw$ as $\fv$ we have arranged $\ddeg_{\prec \fv} Q_{+h}^\phi=0$. 

Set $\derdelta:=\phi^{-1}\der$;
by Corollary~\ref{cor:iso hat h, d-alg} it is enough to show that 
$\derdelta^j\big(\frac{y-h}{\fm}\big)\preceq 1$
for~$j=0,\dots, r$. 
Now using~$(\phantom{-})^\circ$ as before, the hole~$(P^{\phi\circ},1,\hat h^\circ)$ in~$H^\circ$ is special, $Z$-minimal, deep, ultimate, and strongly repulsive-normal,
 by Lemmas~\ref{lem:Pphicirc, 1} and~\ref{lem:Pphicirc, 2}.  It remains to apply
 Corollary~\ref{cor:approx y} to this hole in~$H^\circ$  
with $h^\circ$, $\fm^\circ$, $y^\circ$ in place of~$h$,~$\fm$,~$y$.
\end{proof}

\begin{cor}  \label{cor:46}
Let $\phi$ be active in $H$, $0<\phi\preceq 1$, and suppose the slot $(P^\phi,\fn,\hat h)$ in $H^\phi$ is deep, ultimate, and strongly split-normal.
Then $P(y)=0$ and $y\prec\fn$ for some $y\in\Calinf$. If 
$(P^\phi,\fn,\hat h)$ is strongly repulsive-normal, then any such $y$ is hardian over $H$ with $y\notin H$.  
\end{cor}
\begin{proof}
Lemma~\ref{dentsolver}  gives $y\in\Calinf$ with~$P(y)=0$, $y\prec\fn$.
Now suppose $(P^\phi,\fn,\hat h)$ is strongly repulsive-normal, and 
$y\in\Calinf$, $P(y)=0$, $y\prec\fn$.  Use \cite[Lemma 3.2.14]{ADH4} to arrange that~$(P,\fn,\hat h)$  is a hole in $H$. Then the hole~$(P^{\phi\circ},\fn^\circ,\hat h^\circ)$ in~$H^\circ$ is
special, $Z$-minimal, deep, ultimate, and
strongly repulsive-normal. Theorem~\ref{46} with~$H^\circ$,~$(P^{\phi\circ}, \fn^\circ, \hat h^\circ)$,~$y^\circ$
in place of~$H$,~$(P,\fn,\hat h)$,~$y$ now shows
 that $y^\circ$ is hardian over~$H^\circ$ with $y^\circ\notin H^\circ$.
Hence $y$ is hardian over $H$ and $y\notin H$. 
\end{proof}

\subsection*{Achieving $1$-linear newtonianity}  For the proof of our main theorem we need to show first that for any $\d$-maximal Hardy field $H$ the corresponding $K=H[\imag]$ is $1$-linearly newtonian, the latter being a key hypothesis in Lemma~\ref{53} below. In this subsection we take this vital step: Corollary~\ref{cor:dmax K 1-linnewt}. 

\begin{lemma}\label{prop:dmax 1-newt}
Every $\d$-maximal Hardy field is $1$-newtonian.
\end{lemma}
\begin{proof} Let $H$ be a $\d$-maximal Hardy field. Then $H$ is $\upo$-free by Theorem~\ref{upo}, and~$H$ is Liouville closed with $\I(K)\subseteq K^\dagger$ for $K:=H[\imag]$
by Proposition~\ref{prop:I(K)}.
Hence~$H$ is $1$-linearly newtonian by \cite[Corollary~1.7.29]{ADH4}. 
Towards a contradiction, assume that~$H$ is not $1$-newtonian. Then
\cite[Lemma~3.2.1]{ADH4} yields a minimal  hole
$(P,\fn,\hat h)$ in $H$ of order~$r=1$. Using~\cite[Lem\-ma~3.2.26]{ADH4}, replace
$(P,\fn,\hat h)$ by a refinement  to  arrange that~$(P,\fn,\hat h)$ is quasilinear.  Then $(P,\fn,\hat h)$  is special,
by Lemma~\ref{lem:special dents}.
Using~\cite[Corollary~4.5.42]{ADH4}, further refine 
$(P,\fn,\hat h)$  to arrange that for some active $\phi$ in $H$ with $0<\phi\preceq 1$, $(P^\phi,\fn,\hat h)$ is eventually deep, ultimate, and strongly repulsive-normal.
Now Corollary~\ref{cor:46} gives a proper $\d$-algebraic Hardy field extension of~$H$, contradicting $\d$-maximality of $H$.
\end{proof}

\noindent
{\it In the rest of this subsection $H$ has asymptotic integration.}\/  We have the $\d$-valued extension $K:=H[\imag]\subseteq \Calinf[\imag]$ of $H$ and as before we arrange that~$\hat K=\hat H[\imag]$ is a $\d$-valued extension of $\hat H$ as well as an immediate extension of $K$.

\begin{lemma}\label{lem:deg 1 order 1 approx}
Suppose $H$ is Liouville closed and $\I(K)\subseteq K^\dagger$.
Let $(P,\fn,\hat f)$ be an ultimate minimal hole in $K$ of order $r\geq 1$ with $\deg P=1$, where $\hat f\in\hat K$, such that~$\dim_{\C}\ker_{\Univ} L_P=r$.
Assume also that $H$ is $\upo$-free if $r\ge 2$. Let $f\in\Calinf[\imag]$ be such that $P(f)=0$, $f\prec\fn$.
Then~$f\approx_K \hat f$.
\end{lemma}
\begin{proof}
Replacing $(P,\fn,\hat f)$, $f$ by $(P_{\times\fn},1,\hat f/\fn)$, $f/\fn$ we arrange $\fn=1$.
Let~${\theta\in K^\times}$ be such that $\theta\sim\hat f$; we claim that $f\sim\theta$ in $\c[\imag]$ (and so $f\sim_K \hat f$).
Note that~${(P_{+\theta},\theta,\hat f-\theta)}$ is a  linear minimal hole in~$K$. If $r\ge 2$, we apply Proposition~\ref{prop:deg 1 analytic} and $H$ being $\upo$-free to $(P_{+\theta},\theta,\hat f-\theta)$ to obtain $g\in\Gi[\imag]$
such that~$P_{+\theta}(g)=0$ and~$g\prec\theta$. If $r=1$, we use instead Remark~\ref{rem:improap} and $H$ being Liouville closed (hence $\upl$-free) to obtain such $g$. 
In both cases, ${P(\theta+g)=0}$ and~${\theta+g \prec 1}$, thus 
$L_P(y)=0$ and $y\prec 1$ for~$y:=f-(\theta+g)\in\Calinf[\imag]$.
Hence~$y\prec\theta$ by the  version of Lemma~\ref{lem3.9 linear} for slots in~$K$; see the remark following Corollary~\ref{cor:8.8 refined}. There\-fore~$f-\theta=y+g\prec\theta$  and so~$f\sim\theta$, as claimed. 

The refinement $(P_{+\theta},1,\hat f-\theta)$ of $(P,1,\hat f)$ is ultimate thanks to the remarks before Proposition~\ref{prop:achieve ultimate}, and $L_{P_{+\theta}}=L_P$, so we can apply the claim to $(P_{+\theta},1,\hat f-\theta)$ instead of~$(P,1,\hat f)$
and $f-\theta$ instead of $f$ to give $f-\theta \sim_K \hat f-\theta$.
Since this holds for all~$\theta\in K$ with $\theta\sim\hat f$, the sentence preceding Lemma~\ref{lemsimimag} 
then yields~$f\approx_K \hat f$.
\end{proof}

\begin{cor}\label{cor:embed K<hatf>}
Let $(P,\fn,\hat f)$ be a hole of order $1$ in $K$ with $\deg P =1$. \textup{(}We do not assume here that $\hat f\in \hat{K}$.\textup{)} Then there is an embedding $\iota\colon K\langle\hat f\rangle \to \Calinf[\imag]$ of differential $K$-algebras such that $\iota(g)\sim_K g$ for all~$g\in  K\langle\hat f\rangle^\times$. 
\end{cor}
\begin{proof}
 Note that $(P,\fn,\hat f)$ is minimal. We first show how to arrange that $H$ is Liouville closed and $\upo$-free with $\I(K)\subseteq K^\dagger$ and $\hat f\in\hat K$. 
Let $H_1$ be a Liouville closed and $\upo$-free Hardy field extension of~$H$ with~$\I(K_1)\subseteq K_1^\dagger$, where $K_1:=H_1[\imag]$. (Any maximal Hardy field extension of $H$ can serve as $H_1$.) 
A remark after~[ADH, Theorem~14.0.1] gives a newtonian $\upo$-free immediate $H$-field extension~$\hat{H_1}$ of $H_1$. 
Then the immediate $H$-asymptotic extension $\hat K_1:=\hat H_1[\imag]$ of $K_1$ is  newtonian by~[ADH, 14.5.7].
The remark after Lemma~\ref{find zero of P} gives an embedding $j\colon K\langle \hat f\rangle \to \hat K_1$ of valued differential fields over $K$; set $\hat f_1:=j(\hat{f})$. If~$\hat f_1\in K_1$, then we are done by~$K_1\subseteq \Calinf[\imag]$, so assume~$\hat f_1\notin K_1$.   Then~$(P,\fn,\hat f_1)$ is a  hole in~$K_1$.  Replacing $H$, $K$, $\hat K$, $(P,\fn,\hat f)$ by
 $H_1$, $K_1$, $\hat K_1$, $(P,\fn,\hat f_1)$  we arrange that~$H$ is Liouville closed and $\upo$-free with $\I(K)\subseteq K^\dagger$
 and~$\hat f\in \hat K$. 
 
Replacing $(P,\fn,\hat f)$ by a refinement we also arrange that~$(P,\fn,\hat f)$ is ultimate and~$\fn\in H^\times$, by Proposition~\ref{prop:achieve ultimate}. Then
Proposition~\ref{prop:deg 1 analytic} yields an $f\in\Calinf[\imag]$ with~$P(f)=0$, $f\prec \fn$.
Now Lemma~\ref{lem:deg 1 order 1 approx} gives~$f\approx_K\hat f$, and it remains to appeal to Corollary~\ref{cor:dsfek, order 1}.
\end{proof}

\begin{prop}\label{prop:Re or Im}
Suppose $H$ is $\upo$-free and $1$-newtonian.
Let   $(P,\fn,\hat f)$ be a hole in $K$ of order $1$ with $\deg P=1$, $\hat f\in \hat K$. 
Let $f\in\Calinf[\imag]$, $P(f)=0$,
and~$f\approx_K\hat f$.
Then~$\Re f$ or $\Im f$ generates a proper $\d$-algebraic Hardy field extension of~$H$.
\end{prop}

\begin{proof}
Let $\hat g:=\Re\hat f$ and $\hat h:=\Im\hat f$.
By \cite[Lemma~4.1.3]{ADH4} we have~${v(\hat g-H)} \subseteq {v(\hat h-H)}$ or $v(\hat h-H)\subseteq v(\hat g-H)$.
Below we assume $v(\hat g-H)\subseteq v(\hat h-H)$ (so~$\hat g\in \hat H\setminus H$) and show that then $g:=\Re f$ generates a proper $\d$-algebraic Hardy field extension of~$H$. (If $v(\hat h-H)\subseteq v(\hat g-H)$ one shows likewise that $\Im f$ generates a proper $\d$-algebraic Hardy field extension of $H$.)
The hole  $(P,\fn,\hat f)$ in~$K$ is minimal, and by arranging~$\fn\in H^\times$ we see that $\hat g$ is $\d$-algebraic over~$H$, by a remark preceding~\cite[Lemma~4.3.7]{ADH4}. 
Every element of $Z(H,\hat g)$ has order~$\geq 2$, by~\cite[Corollary~3.2.16]{ADH4} and $1$-newtonianity of~$H$. 
We have $P=P(0)-aY+bY'$ with~$a,b\in K$, $b\ne 0$, and replacing $P$ by $P/b$ we arrange $b=1$, so
the linear part~$A=\der-a$ of~$P$ is monic, and  $A(\hat f)=-P(0)$, $A(f)=-P(0)$.
Then \cite[Example~1.1.7 and Remark~1.1.9]{ADH4}  applied to
$F=\Calinf$ yields $Q\in H\{Y\}$ with~$1\le \order Q\le 2$ and~$\deg Q=1$ such that~$Q(\hat g)=0$ and $Q(g)=0$. 
Hence~$\order Q=2$ and
$Q$ is a minimal annihilator of~$\hat g$ over $H$.

Towards applying Corollary~\ref{cor:iso hat h, d-alg, weak} to $Q$,~$\hat g$,~$g$ in the role of $P$,~$\hat h$,~$y$ there, let $R$ in~$H\{Y\}\setminus H$ have order $< 2$. Then $R\notin Z(H,\hat g)$, so we have $h\in H$ and $\fv\in H^\times$
such that~$\hat g-h\prec\fv$ and $\ndeg_{\prec\fv} R_{+h}=0$. Take any~$\fm\in H^\times$ with~$\hat g-h\preceq\fm\prec\fv$.
By Lemma~\ref{lem:real part approx} we have $g\approx_H\hat g$ and thus~$g-h\preceq\fm$.
After changing $\fv$ as in the proof of Theorem~\ref{46} we obtain an active $\phi$ in $H$, $0<\phi\preceq 1$, such that~$\ddeg_{\prec\fv} R^\phi_{+h}=0$. Set $\derdelta:=\phi^{-1}\der$;
 by Corollary~\ref{cor:iso hat h, d-alg, weak} it is now enough to show that~${\derdelta\big( (g-h)/\fm \big) \preceq 1}$. 
 
 Towards this and using~$(\phantom{-})^\circ$ as before, we have $f^\circ \approx_{K^\circ}\hat f^\circ$,
 and $g^\circ\approx_{H^\circ} \hat g^\circ$
 by the facts about composition in Section~\ref{sec:asymptotic similarity}.
 Moreover, $(g - h)^\circ\preceq \fm^\circ$, and
 $H^\circ$ is $\upo$-free and $1$-newtonian, hence closed under integration by [ADH, 14.2.2].
We now apply Corollary~\ref{cor:dsfek, order 1} with $H^\circ$,~$K^\circ$,~$P^{\phi\circ}$,~$f^\circ$,~$\hat f^\circ$ in the role of
$H$,~$K$,~$P$,~$f$,~$\hat f$ to give~$\big(f^\circ/\fm^\circ\big)'   \approx_{K^\circ}   \big(\hat f^\circ/\fm^\circ\big)'$,   
hence
$(g^\circ/\fm^\circ)' \approx_{H^\circ} (\hat g^\circ/\fm^\circ)'$ by \cite[Lemma~4.1.4]{ADH4} and Lemma~\ref{lem:real part approx}. Therefore,
$$\big((g-h)^\circ/\fm^\circ\big)'\ =\ (g^\circ/\fm^\circ)'-(h^\circ/\fm^\circ)' \ \sim_{H^\circ}\ (\hat g^\circ/\fm^\circ)'-(h^\circ/\fm^\circ)' \ =\  \big((\hat g-h)^\circ/\fm^\circ\big)'.$$
Now $(\hat g-h)^\circ/\fm^\circ \preceq 1$, so $ \big((\hat g-h)^\circ/\fm^\circ\big)' \prec 1$, hence $\big((g-h)^\circ/\fm^\circ\big)' \prec 1$ by the last display, and thus ${\derdelta\big( (g-h)/\fm \big) \prec 1}$, which is more than enough.
\end{proof}

\noindent
If $K$ has a hole of order and degree $1$, then $K$ has a proper $\d$-algebraic
differential field extension inside  $\Calinf[\imag]$, by Corollary~\ref{cor:embed K<hatf>}. Here is a Hardy version:

\begin{lemma}\label{linhole1}
Suppose $K$ has a hole of order and degree $1$.
Then~$H$ has a proper $\d$-algebraic Hardy field extension.
\end{lemma}
\begin{proof} If $H$ is not $\d$-maximal, then $H$ has indeed a proper $\d$-algebraic Hardy field extension, and if $H$ is $\d$-maximal, then 
$H$ is  Liouville closed,  $\upo$-free, $1$-newtonian, and $\I(K)\subseteq K^\dagger$, 
by Propositions~\ref{prop:Li(H(R))} and~\ref{prop:I(K)},    Theorem~\ref{upo}, and Lemma~\ref{prop:dmax 1-newt}.  So assume below that $H$ is  Liouville closed,  $\upo$-free, $1$-newtonian, 
and~${\I(K)\subseteq K^\dagger}$, and~$(P,\fn,\hat f)$ is a hole of order and degree $1$  in $K$. 
As in the proof of Corollary~\ref{cor:embed K<hatf>} with $H=H_1$, we arrange an immediate $\upo$-free newtonian $H$-field extension $\hat H$ of~$H$ with $\hat f\in \hat K:= \hat H[\imag]$ (which may involve changing the original $\hat{f}$). 
Then $\hat K$ is also newtonian by~[ADH, 14.5.7]. 
Using Proposition~\ref{prop:achieve ultimate} we refine $(P,\fn,\hat f)$ to arrange that $(P,\fn,\hat f)$ is ultimate; we further arrange~$\fn\in H^\times$. 
Proposition~\ref{prop:deg 1 analytic} now yields $f\in\Gi[\imag]$ with~$P(f)=0$ and~$f\prec \fn$.
Then~$f\approx_K\hat f$ by Lemma~\ref{lem:deg 1 order 1 approx}, and so 
$\Re f$ or $\Im f$ generates a proper $\d$-algebraic Hardy field extension of $H$, by Proposition~\ref{prop:Re or Im}.
\end{proof}

\begin{cor}\label{cor:dmax K 1-linnewt}
If $H$ is $\d$-maximal, then $K$ is $1$-linearly newtonian.
\end{cor}
\begin{proof} Assume $H$ is $\d$-maximal. Then $K$ is $\upo$-free
by Theorem~\ref{upo} and [ADH, 11.7.23].
If $K$ is not $1$-linearly newtonian, 
then it has a hole of order and degree $1$, by \cite[Lemma~3.2.5]{ADH4}, and so 
$H$ has a proper $\d$-algebraic Hardy field extension, by Lemma~\ref{linhole1}, contradicting $\d$-maximality of~$H$.
\end{proof}

 \subsection*{Finishing the story} With one more lemma we will be done.


\begin{lemma}\label{53} Suppose $H$ is Liouville closed, $\upo$-free, not newtonian, and $K:=H[\imag]$ is $1$-linearly newtonian. Then $H$ has a proper $\d$-algebraic Hardy field extension.
\end{lemma}
\begin{proof} 
Theorem~\ref{c4543}  yields a $Z$-minimal special hole $(Q,  1, \hat b)$ in $H$ of order $r\ge 1$ and an active $\phi$ in $H$ 
with $0<\phi\preccurlyeq 1$ such that  the hole $(Q^{\phi}, 1, \hat b)$ in~$H^\phi$
is deep, strongly repulsive-normal, and ultimate.  By \cite[Proposition 1.7.28]{ADH4}, $K$ is $1$-linearly surjective and $\I(K)\subseteq K^\dagger$.
Then Corollary~\ref{cor:46} applied to~$(Q,1,\hat b)$ in the role of~$(P,\fn,\hat h)$ gives a $y\in \Calinf\setminus H$   
that is hardian over $H$ with $Q(y)=0$.  Hence $H\langle y\rangle$ is a proper $\d$-algebraic Hardy field extension of $H$. 
\end{proof}

\noindent
Recall from the introduction that an {\it $H$-closed field}\/ is an $\upo$-free newtonian Liouville closed $H$-field. Recall also that
Hardy fields containing $\R$ are $H$-fields.  The main result of this paper now follows in a few lines: 

\begin{theorem}\label{thm:char d-max}
A Hardy field  is $\d$-maximal iff it contains $\R$ and is $H$-closed.
\end{theorem}
\begin{proof}
The ``if'' part is a special case of [ADH, 16.0.3]. Suppose $H$ is $\d$-maximal.
By Proposition~\ref{prop:Li(H(R))}  and Theorem~\ref{upo},  $H\supseteq\R$ and $H$ is
Liouville closed and $\upo$-free. By Corollary~\ref{cor:dmax K 1-linnewt},
$K$ is $1$-linearly newtonian, so~$H$ is newtonian by Lemma~\ref{53}.
\end{proof}

\noindent
Theorem~\ref{thm:char d-max} and Corollary~\ref{cor:Hardy field ext smooth} yield Theorem~B  in a refined form:

\begin{cor}\label{thm:extend to H-closed}
Any Hardy field $F$ has a $\d$-algebraic $H$-closed Hardy field extension.
If $F$ is a $\Ginf$-Hardy field, then so is any such extension, and likewise with~$\Gom$ in place of $\Ginf$.
\end{cor}

\section{Transfer Theorems}\label{sec:transfer}

\noindent
From [ADH, 16.3] we recall the notion of a {\it pre-$\HLO$-field~$\boldsymbol  H=(H,\I,\Lambda,\Omega)$}\/: this is a pre-$H$-field $H$ equipped with a {\it $\HLO$-cut $(\I,\Lambda,\Omega)$}\/ of $H$. 
A {\it $\HLO$-field\/} is a pre-$\HLO$-field~$\boldsymbol  H=(H;\dots)$ where $H$ is an $H$-field. 
 If~$\boldsymbol  M=(M;\dots)$ is a pre-$\HLO$-field and~$H$ is a pre-$H$-subfield of~$M$, then $H$ has a unique expansion to a  pre-$\HLO$-field~$\boldsymbol  H$ such that $\boldsymbol  H\subseteq\boldsymbol  M$.   
 By~[ADH, 16.3.19 and remarks before it], a pre-$H$-field~$H$  has exactly one or exactly two $\HLO$-cuts, and $H$
has a unique $\HLO$-cut iff  
\begin{enumerate}
\item $H$ is grounded; or
\item there exists $b\asymp 1$ in $H$ such that $v(b')$ is a gap in $H$;  or
\item $H$ is $\upo$-free.
\end{enumerate}
In particular, each  $\d$-maximal Hardy field $M$ (being $\upo$-free) has a  unique expansion to a pre-$\HLO$-field~$\boldsymbol  M$,
namely $\boldsymbol  M=\big(M;\I(M),\Upl(M),\omega(M)\big)$,
and then $\boldsymbol  M$ is a $\HLO$-field with constant field $\R$. Below we always view any $\d$-maximal Hardy field
as a $\HLO$-field in this way.

\begin{lemma}\label{lem:canonical HLO}
Let $H$ be a Hardy field. Then $H$ has an expansion to a pre-$\HLO$-field~$\boldsymbol H$ such that $\boldsymbol H\subseteq\boldsymbol M$ for every $\d$-maximal Hardy field $M\supseteq H$.
\end{lemma}
\begin{proof}
Since every $\d$-maximal Hardy field containing $H$ also contains $\operatorname{D}(H)$, it suffices to show this
for $\operatorname{D}(H)$ in place of $H$. So we assume $H$ is $\d$-perfect,
and thus a  Liouville closed $H$-field.
For each $\d$-maximal Hardy field~$M\supseteq H$ we now have
$\I(H)=\I(M)\cap H$ by [ADH, 11.8.2], $\Upl(H)=\Upl(M)\cap H$ by [ADH, 11.8.6], and~$\omega(H)=\bar{\omega}(H)=\bar{\omega}(M)\cap H=\omega(M)\cap H$ by \cite[Corollary~6.2]{ADH5}, as required.
\end{proof}

\noindent
Given a  Hardy field $H$, we call the unique expansion $\boldsymbol H$ of $H$ to a pre-$\HLO$-field  with the property stated in the previous lemma the {\bf canonical $\HLO$-expansion} of~$H$.\index{Hardy field!canonical $\HLO$-expansion}

\begin{cor}\label{cor:canonical HLO}
Let $H$, $H^*$ be Hardy fields, with their canonical $\HLO$-ex\-pan\-sions~$\boldsymbol H$ and~$\boldsymbol H^*$, respectively, such that $H\subseteq H^*$. Then $\boldsymbol H\subseteq\boldsymbol H^*$.
\end{cor}
\begin{proof}
Let $M^*$ be any $\d$-maximal Hardy field extension of $H^*$. Then $\boldsymbol H\subseteq\boldsymbol M^*$ as well as $\boldsymbol H^*\subseteq\boldsymbol M^*$, hence $\boldsymbol H\subseteq\boldsymbol H^*$.
\end{proof}

\noindent
{\it In the rest of this section $\mathcal L=\{0,1,{-},{+},{\,\cdot\,},{\der},{\leq},{\preceq}\}$ is the language of ordered valued differential rings}\/ [ADH, p.~678]. We view 
each ordered valued differential field as an $\mathcal L$-structure in the natural way. Given an ordered valued differential field~$H$
and a subset~$A$ of~$H$ we let $\mathcal L_A$ be $\mathcal L$ augmented by names for the elements of~$A$, and expand the $\mathcal L$-structure $H$ 
to an $\mathcal L_A$-structure by interpreting the name of any $a\in A$ as the element $a$ of $H$; cf.~[ADH, B.3].
Let~$H$ be a Hardy field and $\sigma$ be an $\mathcal L_H$-sentence.
We now have our Hardy field analogue of the ``Tarski principle''~[ADH, B.12.14] in real algebraic geometry promised in the introduction:

\begin{theorem}\label{thm:transfer}
The following are equivalent:
\begin{enumerate}
\item[\textup{(i)}] $M\models\sigma$ for some $\d$-maximal Hardy field $M\supseteq H$;
\item[\textup{(ii)}] $M\models\sigma$ for every $\d$-maximal Hardy field $M\supseteq H$;
\item[\textup{(iii)}] $M\models\sigma$ for every maximal Hardy field $M\supseteq H$;
\item[\textup{(iv)}] $M\models\sigma$ for some maximal Hardy field $M\supseteq H$.
\end{enumerate}
\end{theorem}
\begin{proof}
The implications (ii)~$\Rightarrow$~(iii)~$\Rightarrow$~(iv)~$\Rightarrow$~(i) are obvious, since ``maximal~$\Rightarrow$~$\d$-maximal'';
  it remains to show (i)~$\Rightarrow$~(ii).
Let $M$, $M^*$ be $\d$-maximal Hardy field extensions of $H$.
By Lemma~\ref{lem:canonical HLO} and Corollary~\ref{cor:canonical HLO}
expand $M$, $M^*$, $H$ to pre-$\HLO$-fields $\boldsymbol M$, $\boldsymbol M^*$, $\boldsymbol H$,   such that
$\boldsymbol H\subseteq\boldsymbol M$ and $\boldsymbol H\subseteq \boldsymbol M^*$. 
In~[ADH, introduction to Chapter~16] we extended $\mathcal{L}$ to a language $\mathcal L^\iota_{\HLO}$, and explained in~[ADH, 16.5] how each
pre-$\HLO$-field~$\boldsymbol K$ is construed as an  $\mathcal L^\iota_{\HLO}$-structure in such a way that
every extension~${\boldsymbol K\subseteq\boldsymbol L}$  of pre-$\HLO$-fields corresponds to an extension of the associated~$\mathcal L^\iota_{\HLO}$-structures.
By~[ADH, 16.0.1], the  $\mathcal L^\iota_{\HLO}$-theory~$T^{\operatorname{nl},\iota}_{\HLO}$ of $H$-closed
$\HLO$-fields  eliminates quantifiers, and  $\mathbf M,\mathbf M^*\models T^{\operatorname{nl},\iota}_{\HLO}$ by
Theorem~\ref{thm:char d-max}.
Hence~$\boldsymbol M \equiv_H \boldsymbol M^*$~[ADH, B.11.6], so if~$\boldsymbol M  \models \sigma$, then~$\boldsymbol M^*  \models \sigma$.   
\end{proof}

\noindent
Corollaries~\ref{cor:elem equiv} and~\ref{cor:systems, 1} from the introduction are special cases of
Theorem~\ref{thm:transfer}.
By Corollary~\ref{realginfgom},
$\Ginf$-maximal and $\Gom$-maximal Hardy fields are $\d$-maximal,  so Theorem~\ref{thm:transfer} also yields Corollary~\ref{cor:systems, 2} from the introduction in a stronger form:

\begin{cor}
If $H\subseteq\Ginf$ and
$M\models\sigma$ for some $\d$-maximal Har\-dy field extension~$M$ of $H$, then
$M\models\sigma$ for every $\Ginf$-maximal Hardy field~${M\supseteq H}$. 
Likewise with $\Gom$ in place of $\Ginf$.
\end{cor}

\subsection*{The structure induced on $\R$}
In the next corollary $H$ is a Hardy field and~$\varphi(x)$ is an $\mathcal L_H$-formula where $x=(x_1,\dots,x_n)$ and
$x_1,\dots, x_n$ are distinct variables. 
Also, $\mathcal L_{\operatorname{OR}}=\{0,1,{-},{+},{\,\cdot\,},{\leq}\}$ is the language of ordered rings, and the ordered field $\R$ of real numbers is interpreted as an $\mathcal L_{\operatorname{OR}}$-structure in the obvious way. 
By Theorem~\ref{thm:char d-max}, $\d$-maximal Hardy fields are $H$-closed fields, so from [ADH, 16.6.7, B.12.13] 
in combination with Theorem~\ref{thm:transfer} we  obtain:

\begin{cor}\label{cor:sa}
There is a quantifier-free
 $\mathcal L_{\operatorname{OR}}$-formula $\varphi_{\operatorname{OR}}(x)$ such that for all
$\d$-maximal Hardy fields $M\supseteq H$ and $c\in\R^n$:
$\ M \models \varphi(c)\ \Leftrightarrow\ \R\models \varphi_{\operatorname{OR}}(c)$.
\end{cor}

\noindent
This yields Corollary~\ref{cor:parametric systems} from the introduction. 
The first of the examples after that corollary is   covered by~[ADH, 5.1.18, 11.8.25, 11.8.26];
for the details of the second example we refer to \cite[Section~7.1]{ADHmax}.

\subsection*{Uniform finiteness}
We now let $H$ be a Hardy field and $\varphi(x,y)$ and $\theta(x)$ be   $\mathcal L_H$-for\-mu\-las, where $x=(x_1,\dots,x_m)$ and $y=(y_1,\dots,y_n)$.

\begin{lemma}
There is a $B=B(\varphi)\in\N$ such that for all $f\in H^m$: if 
for some $\d$-maximal Hardy field extension $M$ of $H$ there are more than $B$ tuples $g\in M^n$ with $M\models\varphi(f,g)$, then 
for every $\d$-maximal Hardy field extension $M$ of $H$ there are infinitely many $g\in M^n$ with~$M\models\varphi(f,g)$. 
\end{lemma}
\begin{proof}
Fix a $\d$-maximal Hardy field extension $M^*$ of $H$. By \cite[Proposition~6.4]{ADHdim} we have $B=B(\varphi)\in\N$ such that
for all $f\in (M^*)^m$: if $M^*\models\varphi(f,g)$ for more than~$B$ many $g\in (M^*)^n$, then
$M^*\models\varphi(f,g)$ for infinitely many $g\in (M^*)^n$. Now use Theorem~\ref{thm:transfer}. 
\end{proof}

\noindent
In the proof of the next lemma we use that $\mathcal C$ has the cardinality $\mathfrak c=2^{\aleph_0}$ of the continuum, hence
$\abs{H}=\mathfrak c$ if $H\supseteq\R$.

\begin{lemma}
Suppose $H$ is $\d$-maximal and $S:=\big\{f\in H^m:H\models\theta(f)\big\}$ is infinite. Then $\abs{S}=\mathfrak c$.
\end{lemma}
\begin{proof}
Let $d:=\dim(S)$ be the dimension of the definable set $S\subseteq H^m$ as introduced in \cite{ADHdim}.
If $d=0$, then $\abs{S}=\abs{\R}=\mathfrak c$ by remarks following  \cite[Pro\-po\-si\-tion~6.4]{ADHdim}.
Suppose $d>0$, and for $g=(g_1,\dots,g_m)\in H^m$ and $i\in\{1,\dots,m\}$, let~$\pi_i(g):=g_i$. 
Then for some $i\in\{1,\dots,m\}$, the subset $\pi_i(S)$ of $H$ has nonempty interior, by~\cite[Corollary~3.2]{ADHdim}, and hence 
  $\abs{S}=\abs{H}=\mathfrak c$.
\end{proof}

\noindent
The two lemmas above together now yield Corollary~\ref{cor:uniform finiteness} from the introduction. 

\subsection*{Transfer between maximal Hardy fields and transseries}
Let   $\boldsymbol T$ be the unique expansion of  $\mathbb T$ to a pre-$\HLO$-field, so~$\boldsymbol T$ is an $H$-closed 
$\HLO$-field with small derivation and constant field $\R$.

\begin{lemma}\label{lem:unique HLO-expansion, 1}
Let $H$ be a pre-$H$-subfield of $\mathbb T$ with $H\not\subseteq \R$.
Then $H$ has a unique expansion to a pre-$\HLO$-field.
\end{lemma}
\begin{proof} If $H$ is grounded, this follows from [ADH, 16.3.19].
Suppose $H$ is not grounded. Then $H$ has asymptotic integration by the proof of [ADH, 10.6.19] applied to $\Delta:= v(H^\times)$. Starting with an $h_0\succ 1$ in $H$
with $h_0'\asymp 1$ we construct a logarithmic sequence $(h_n)$ in $H$ as in [ADH, 11.5], so $h_n\asymp\ell_n$ for all $n$.
Hence $\Gamma^<$ is cofinal in $\Gamma_{\mathbb T}^<$, so $H$ is $\upo$-free by [ADH, remark before 11.7.20].
Now use [ADH, 16.3.19] again. 
\end{proof}


\noindent
{\em In the rest of this section $H$ is a Hardy field with canonical $\HLO$-ex\-pan\-sion~$\boldsymbol H$, and~$\iota\colon H\to \mathbb T$ is an embedding of ordered differential fields, and thus of pre-$H$-fields}.  

\begin{cor}\label{lem:unique HLO-expansion, 2}
The map $\iota$ is an embedding~$\boldsymbol H\to\boldsymbol T$ of pre-$\HLO$-fields.
\end{cor}
\begin{proof}
If $H\not\subseteq\R$, then this follows from Lemma~\ref{lem:unique HLO-expansion, 1}.
Suppose $H\subseteq\R$. Then $\iota$ is the identity on $H$, so extends to the embedding
$\R(x)\to \mathbb T$ that is the identity on~$\R$ and sends the germ $x$ to $x\in \T$. Now use that $\R(x)\not\subseteq \R$ and Corollary~\ref{cor:canonical HLO}. 
\end{proof}


\noindent
Recall from~[ADH, B.4] that for any $\mathcal L_H$-sentence $\sigma$ we obtain an $\mathcal L_{\mathbb T}$-sentence $\iota(\sigma)$ by replacing the name of each~$h\in H$ occurring in $\sigma$ with the name of $\iota(h)$.

\begin{cor}\label{cor:transfer T, 1}
Let  $\sigma$ be an $\mathcal L_H$-sentence. Then \textup{(i)}--\textup{(iv)} in Theorem~\ref{thm:transfer} are also equivalent to:
\begin{enumerate}
\item[\textup{(v)}] $\mathbb T\models\iota(\sigma)$.
\end{enumerate}
\end{cor}
\begin{proof}
Let $M$ be a $\d$-maximal Hardy field extension of $H$; it suffices to show that~$M\models\sigma$ iff $\mathbb T\models\iota(\sigma)$. For this, mimic the proof of (i)~$\Rightarrow$~(ii) in Theorem~\ref{thm:transfer}, using Corollary~\ref{lem:unique HLO-expansion, 2}. 
\end{proof}

\noindent
Corollary~\ref{cor:transfer T, 1} yields the first part of Corollary~\ref{cor:systems, 3} from the introduction, even in a stronger
form.  
We prove the second part of that corollary in Section~\ref{sec:embeddings into T}: Corollary~\ref{cor:transfer T, 2}. There we also use:

\begin{lemma}\label{lemhr} $\iota$ extends uniquely to an embedding $H(\R)\to \mathbb{T}$ of pre-$H$-fields. 
\end{lemma}
\begin{proof}  Let $\hat{H}$ be the $H$-field hull of $H$ in $H(\R)$. Then $\iota$ extends uniquely to an $H$-field embedding $\hat{\iota}: \hat{H}\to \T$ by [ADH, 10.5.13]. By [ADH, remark before 4.6.21] and [ADH, 10.5.16] $\hat{\iota}$ extends uniquely to an embedding $H(\R)\to \T$ of $H$-fields. 
\end{proof} 

\noindent
We now derive Theorem~A from the introduction (in stronger form):

\begin{cor}\label{divpcor}
If $P\in H\{Y\}$, $f<g$ in $H$, and $P(f)<0<P(g)$, then each $\d$-maximal Hardy field extension of $H$ contains a
$y$ with $f<y<g$ and $P(y)=0$.
\end{cor}
\begin{proof}
By  \cite{JvdH}, 
the ordered differential field $\mathbb T_{\text{g}}$ of grid-based transseries is $H$-closed with
small derivation and has the differential intermediate value property  (DIVP).
Hence $\mathbb T$ also has  DIVP, by completeness of $T_H$ (see the introduction). Now use Corollary~\ref{cor:transfer T, 1}.
\end{proof}

\begin{cor}\label{cor:odd degree}
For every $P\in H\{Y\}$ of odd degree there is an $H$-hardian germ~$y$   with $P(y)=0$.
\end{cor}
\begin{proof}
This follows from Theorem~\ref{thm:extend to H-closed} and [ADH, 14.5.3]. Alternatively, use Corollary~\ref{divpcor}: 
with Proposition~\ref{prop:Li(H(R))}, arrange $H\supseteq\R$ and $H$ is Liouville closed, and appeal to
 the example following Corollary~1.9 in \cite{ADH5}.
\end{proof}

\noindent
Note that if $H\subseteq\Ginf$, then in the previous two corollaries we have $H\langle y\rangle\subseteq\Ginf$, by Corollary~\ref{cor:Hardy field ext smooth}; likewise with $\Gom$ in place if $\Ginf$. 

\medskip
\noindent
The following contains Corollaries~\ref{cor:zeros in complexified Hardy field extensions} and~\ref{cor:factorization intro} from the introduction.
In~[ADH] we defined a differential field $F$  to be {\it weakly $\d$-closed}\/ if every $P\in F\{Y\}\setminus F$ has a zero in $F$. 
If $F$ is  weakly $\d$-closed, then $F$ is clearly linearly surjective, and also linearly closed  by \textup{[ADH, 5.8.9]}.

\begin{cor}\label{cor:d-max weakly d-closed}
If $H$ is $\d$-maximal, then $K:=H[\imag]$ is weakly $\d$-closed.
\end{cor}
\begin{proof}
By our main Theorem~\ref{thm:char d-max}, if $H$ is $\d$-maximal, then $H$ is newtonian, and thus
 $K$ is  weakly $\d$-closed by~[ADH, 14.5.7, 14.5.3].
\end{proof}

\noindent
The remarks after Corollary~\ref{cor:factorization intro} in the introduction concerning fundamental systems of solutions to scalar linear differential equations over $K$  follow from
Lemma~\ref{lem:basis of kerUA} and Corollary~\ref{cor:d-max weakly d-closed} in combination with the equivalence \eqref{eq:2.4.8}.

\section{Embeddings into Transseries and Maximal Hardy Fields}\label{sec:embeddings into T}

\noindent
We first derive a fact about ``Newton-Liouville closure''  (as defined in~[ADH]). 
Let $\boldsymbol H$ be a   $\HLO$-field with underlying $H$-field $H$. 
Then $\mathbf H$ has an $\upo$-free newtonian Liouville closed $\HLO$-field extension, called a {\it Newton-Liouville closure}\/ of~$\mathbf H$,   that embeds over~$\mathbf H$ into any $\upo$-free newtonian Liouville closed $\HLO$-field extension of~$\mathbf H$~[ADH, 16.4.8]. Any two   Newton-Liouville closures of $\mathbf H$ are isomorphic over~$\mathbf H$ [ADH, 16.4.9], and this permits us to speak of {\em the\/}
 Newton-Liouville closure of $\mathbf H$.    
By~[ADH, 14.5.10, 16.4.1, 16.4.8], the constant field of the Newton-Liouville closure of $\boldsymbol H$ is the 
real closure  of $C:= C_{H}$.
Let $\boldsymbol M=(M;\dots)$ be an $H$-closed $\HLO$-field extension of $\boldsymbol H$ and
$\boldsymbol H^{\operatorname{da}}:=(H^{\operatorname{da}};\dots)$ be the $\HLO$-subfield of~$\mathbf M$ with~$H^{\operatorname{da}}:=\{f\in M:\text{$f$ is $\d$-algebraic over $H$}\}$.

\begin{prop}\label{prop:embed into H-closed}   
Let $\boldsymbol H^*$ be a $\d$-algebraic $\HLO$-field extension of $\boldsymbol H$ such that the constant field of $\boldsymbol H^*$ is algebraic over $C$. Then there is an embedding~$\boldsymbol H^*\to\boldsymbol M$
over $\boldsymbol H$, and the image of any such embedding is contained in~$\boldsymbol H^{\operatorname{da}}$.
\end{prop}
\begin{proof}  The image of any embedding $\boldsymbol H^*\to\boldsymbol M$
over $\boldsymbol H$ is $\d$-algebraic over $H$ and thus contained in $\boldsymbol H^{\operatorname{da}}$. 
For existence, take a Newton-Liouville closure $\boldsymbol M^*$ of $\boldsymbol H^*$. Then $\boldsymbol M^*$ is also a
Newton-Liouville closure of $\boldsymbol H$, by [ADH, 16.0.3],  and thus embeds into $\boldsymbol M$ over $\boldsymbol H$.
\end{proof}

\noindent
Let  $\mathcal L$  be the language of ordered valued differential rings, as  in Section~\ref{sec:transfer}.
The second part of Corollary~\ref{cor:systems, 3} in the introduction now follows from the next result: 
 
\begin{cor}\label{cor:transfer T, 2}
Let $H$ be a Hardy field, $\iota\colon H\to\mathbb T$ an ordered differential field embedding, and
$H^*$ a $\d$-maximal $\d$-algebraic Hardy field extension of $H$. Then $\iota$ extends to an ordered valued differential field embedding $H^*\to\mathbb T$, and so for any  $\mathcal L_H$-sentence $\sigma$,  $H^*\models \sigma$ iff $\mathbb T\models \iota(\sigma)$.
\end{cor}
\begin{proof} We have $H(\R)\subseteq H^*$, and so by Lemma~\ref{lemhr} we arrange that $H\supseteq\R$. 
Let~$\boldsymbol H$,~$\boldsymbol H^*$ be the canonical $\HLO$-expansions of $H$, $H^*$, respectively, and let $\boldsymbol T$ be the expansion of~$\mathbb T$ to a $\HLO$-field.  Then $\boldsymbol H\subseteq\boldsymbol H^*$, and by Lemma~\ref{lem:unique HLO-expansion, 2}, $\iota$ is an
embedding $\boldsymbol H\to\boldsymbol T$.  
By Proposition~\ref{prop:embed into H-closed}, $\iota$ extends to an embedding $\boldsymbol H^*\to\boldsymbol T$. 
\end{proof}

\noindent
Consider the Hardy field $H:=\R(\ell_0, \ell_1, \ell_2,\dots)\subseteq\Gom$ where   $\ell_0=x$ and $\ell_{n+1}=\log\ell_n$ for each $n$, 
and  mimic this
in $\mathbb T$ by setting $\ell_0:=x$ and $\ell_{n+1}:=\log\ell_n$ in~$\mathbb T$. This yields the unique ordered differential
field embedding $H\to \T$ over $\R$ sending~${\ell_n\in H}$ to~$\ell_n \in \T$ for all $n$. Its image is the $H$-subfield
$\R(\ell_0,\ell_1,\dots)$ of $\mathbb T$. 
Since the sequence~$(\ell_n)$ in~$\T$ is coinitial in $\T^{>\R}$, each ordered differential subfield of $\mathbb T$
containing~$\R(\ell_0,\ell_1,\dots)$ is an $\upo$-free $H$-field, by the remark preceding [ADH, 11.7.20].
 From Lemma~\ref{lem:unique HLO-expansion, 1} and Proposition~\ref{prop:embed into H-closed} we obtain:

\begin{cor}\label{cor:embed into H-closed, 1}
If $H\supseteq \R$ is an $\upo$-free $H$-subfield of $\mathbb T$ and $H^*$ is 
a $\d$-algebraic $H$-field extension of $H$ with constant field $\R$, then there exists an $H$-field embedding~$H^*\to\mathbb T$ over $H$. 
\end{cor}

\noindent
Corollary~\ref{cor:embed into H-closed, 1} goes through
with $\mathbb T$ replaced by its  $H$-subfield 
$$\mathbb T^{\operatorname{da}}\ :=\  \big\{ f\in\mathbb T : \text{$f$ is $\d$-algebraic (over $\Q$)}\big\}.$$
We now apply this observation to o-minimal structures. The {\it Pfaffian closure}\/ of an expansion of the ordered field of real numbers is its smallest expansion  that is closed under taking Rolle leaves of definable $1$-forms of class~$\c^1$. See Speissegger~\cite{Speissegger} for complete definitions, and the proof that 
the Pfaffian closure of an o-minimal expansion of  the ordered field of reals remains o-minimal.  

\begin{cor}\label{cor:embed into H-closed, 2}
The Hardy field $H$ of the Pfaffian closure of the ordered field of real numbers embeds as an $H$-field over $\R$  into $\mathbb T^{\operatorname{da}}$.
\end{cor}
\begin{proof}
Let $f\colon\R\to\R$ be definable in the Pfaffian closure of the ordered field of real numbers.
The proof of \cite[Theorem~3]{LMS} gives  $r\in\N$,  semialgebraic $g\colon\R^{r+2}\to\R$, and  $a\in\R$ 
such that $f|_{(a,\infty)}$ is $\c^{r+1}$ and $f^{(r+1)}(t)=g\big(t,f(t),\dots,f^{(r)}(t)\big)$ for all~$t>a$. Take $P\in \R[Y_1,\dots,Y_{r+3}]^{\neq}$  vanishing identically on the graph of $g$; see~[ADH, B.12.18].
Then $P\big(t,f(t),\dots,f^{(r+1)}(t)\big)=0$ for $t>a$. Hence the germ of $f$ is $\d$-algebraic over $\R$, and so
 $H$ is $\d$-algebraic over $\R$. As $H$ contains the $\upo$-free Hardy field $\R(\ell_0, \ell_1,\dots)$,  we can use the remark following Corollary~\ref{cor:embed into H-closed, 1}.
\end{proof}

\begin{question}
Let $H$ be the Hardy field of an o-minimal expansion of the ordered field of reals, and 
let $H^*\supseteq H$ be the Hardy field of the Pfaffian closure of this expansion.
Does every embedding $H\to\mathbb T$ extend to an embedding $H^*\to\mathbb T$?
\end{question}

\noindent
We mentioned in the introduction that  an embedding $H\to\mathbb T$ as in Corollaries~\ref{cor:transfer T, 2} and~\ref{cor:embed into H-closed, 2}  can be viewed as an {\it expansion operator}\/ for the Hardy field $H$
and  its inverse as a {\it summation operator.}\/
The corollaries above concern the existence of expansion operators; this relied on the $H$-closedness of $\mathbb T$. Likewise, Theorem~\ref{thm:char d-max} and Proposition~\ref{prop:embed into H-closed}  also give rise to summation operators: 
\begin{cor}\label{cor:embed into H-closed, 3}
Let $H$ be an $\upo$-free $H$-field and let $H^*$ be
a $\d$-algebraic $H$-field extension of $H$ with $C_{H^*}$ algebraic over $C_H$.
Then any $H$-field embedding~$H\to M$ into a $\d$-maximal Hardy field  extends to an $H$-field embedding~$H^*\to M$.
\end{cor}

\noindent
Hence any $\mathcal{L}$-isomorphism between  an ordered differential subfield $H\supseteq\R(\ell_0,\ell_1,\dots)$ of $\mathbb T$  and
a Hardy field $F$ extends to an $\mathcal{L}$-isomorphism between the ordered differential subfield~$H^*:=\{f\in\mathbb T:\text{$f$ is $\d$-algebraic over $H$}\}$ of $\mathbb T$ and a Hardy field extension of $F$. For~$H=\R(\ell_0,\ell_1,\dots)\subseteq \mathbb T$ (so $H^*=\mathbb T^{\operatorname{da}}$) we recover the main result of~\cite{vdH:hfsol}:

\begin{cor} \label{cor:transserial}
The $H$-field $\mathbb T^{\operatorname{da}}$ is $\mathcal{L}$-isomorphic to a Hardy field $\supseteq \R(\ell_0, \ell_1,\dots)$. 
\end{cor}

\noindent
Any Hardy field that is $\mathcal{L}$-isomorphic to~$\mathbb T^{\operatorname{da}}$ is $H$-closed  by [ADH,~16.6] and thus $\d$-maximal by Theorem~\ref{thm:char d-max}, so contains the Hardy field $\operatorname{D}(\Q)$.
Thus we have an $\mathcal{L}$-embedding $e\colon \Dx(\Q)\to \mathbb T^{\operatorname{da}}$, which we can view as an expansion operator for~$\Dx(\Q)$.
We suspect that  $e(\Dx(\Q))$ is independent of the choice of $e$.

\medskip
\noindent
In the remainder of this section we use the results above to determine the universal theory of Hardy fields.
First, some generalities on valued differential fields.

\subsection*{Valued differential fields with very small derivation}
Let $K$ be  a valued differential field. Let $\mathcal O$ be its valuation ring, with maximal ideal $\smallo$, and let
$\der$ be its derivation. Recall that if $K$ has small derivation (that is, $\der(\smallo)\subseteq \smallo$), then also~${\der(\mathcal O)\subseteq\mathcal O}$ by [ADH, 4.4.2],  so we have a unique derivation on the residue field~$\k:=\mathcal O/\smallo$ that makes the residue morphism $\mathcal O\to \k$ into a morphism of differential rings (and we call $\k$ with this induced derivation the differential residue field of $K$). We say that $\der$ is {\bf very small}\index{very small derivation}\index{valued differential field!very small derivation} if $\der(\mathcal O)\subseteq\smallo$. So $K$ has  very small derivation iff~$K$ has small derivation and the induced derivation on $\k$ is trivial. If  $K$ has small derivation and~$\mathcal O=C+\smallo$, then $K$ has very small derivation. If $K$ has very small derivation, then so does every valued differential subfield of $K$, and if~$L$ is a valued differential field extension of~$K$ with small derivation and $\k_L=\k$,  then $L$ has very small derivation. Moreover:

\begin{lemma}\label{lem:very small der alg ext}
Let $L$ be a valued differential field extension of $K$, algebraic over~$K$, and suppose $K$ has very small derivation. Then $L$ also has very small derivation.
\end{lemma}
\begin{proof}
By [ADH, 6.2.1], $L$ has small derivation. The derivation of $\k$ is trivial and~$\k_L$ is algebraic over $\k$ [ADH, 3.1.9], so the derivation of $\k_L$ is also trivial.
\end{proof}

\noindent
Next we focus on  pre-$\d$-valued fields with very small derivation. First an easy observation about asymptotic couples:

\begin{lemma}\label{lem:small deriv char}
Let $(\Gamma,\psi)$ be an asymptotic couple;  then  
$$\text{$(\Gamma,\psi)$ has gap $0$}\quad\Longleftrightarrow\quad  \text{$(\Gamma,\psi)$ has small derivation and $\Psi\subseteq\Gamma^{<}$.}$$
In particular, if $(\Gamma,\psi)$ has small derivation and does not have gap $0$, then each asymptotic couple extending $(\Gamma,\psi)$ has small derivation.
\end{lemma}

\begin{cor}
Suppose $K$ is pre-$\d$-valued   with   small derivation, and suppose~$0$ is not a gap in $K$. Then $K$ has very small derivation.
\end{cor}
\begin{proof}
The  previous lemma gives $g\in K^\times$ with $g\not\asymp 1$ and $g^\dagger\preceq 1$. Then for each~$f\in K$ with  $f\preceq 1$ we have $f'\prec g^\dagger\preceq 1$.
\end{proof}

\begin{cor}\label{cor:dv(K) very small der}
Suppose $K$ is pre-$\d$-valued of $H$-type with very small derivation. Then the $\d$-valued hull $\operatorname{dv}(K)$ of $K$ has small derivation.
\end{cor}
\begin{proof}
By Lemma~\ref{lem:small deriv char}, if $0$ is not a gap in $K$, then  every pre-$\d$-valued field extension of $K$ has small derivation. If $0$ is a gap in $K$, then no $b\asymp 1$ in~$K$  satisfies~$b'\asymp 1$, since $K$ has  very small derivation. Thus $\Gamma_{\operatorname{dv}(K)}=\Gamma$ by [ADH,~10.3.2(ii)], so $0$ remains a gap in $\operatorname{dv}(K)$. In both cases, $\operatorname{dv}(K)$ has small derivation.
\end{proof}

\noindent
If $K$ is pre-$\d$-valued and ungrounded, then for each $\phi\in K$ which is active  in $K$,   the pre-$\d$-valued field $K^\phi$ (with derivation $\derdelta=\phi^{-1}\der$) has very small derivation.

\subsection*{The universal theory of Hardy fields}
 Let
$\mathcal L^\iota$ be the language $\mathcal L$ of
ordered valued differential rings from above augmented by a new unary function symbol~$\iota$.
We view each pre-$H$-field~$H$ as an $\mathcal L^\iota$-structure by interpreting the symbols
from~$\mathcal L$ in the natural way and
$\iota$ by the function $\iota\colon H\to H$ given by $\iota(a):=a^{-1}$ for $a\in H^\times$ and~$\iota(0):=0$. 
Since every Hardy field extends to a maximal one, each universal 
 $\mathcal L^\iota$-sentence which holds in every maximal Hardy field also  holds in every   Hardy field.  
We now use  Corollary~\ref{cor:transserial} to show:

\begin{prop}\label{prop:univ theory}
Let $\Sigma$ be the set of universal $\mathcal L^\iota$-sentences true in all   Hardy fields. Then the models of $\Sigma$ are the  pre-$H$-fields with very small derivation.
\end{prop}

\noindent
For this we need  a refinement of [ADH, 14.5.11]:

\begin{lemma}\label{lem:very small der extend}
Let $H$ be a pre-$H$-field with very small derivation. Then $H$ extends to an $H$-closed field with small derivation.
\end{lemma}
\begin{proof}
By Corollary~\ref{cor:dv(K) very small der}, replacing $H$ by its $H$-field hull, we first arrange that $H$ is an $H$-field.
Let   $(\Gamma,\psi)$ be the asymptotic couple of $H$. Then 
$\Psi^{\geq 0}\neq\emptyset$ or~$(\Gamma,\psi)$ has gap~$0$.
Suppose $(\Gamma,\psi)$ has gap $0$. Let $H(y)$ be the $H$-field extension from [ADH, 10.5.11] for $K:=H$, $s:=1$.
Then $y\succ 1$ and $y^\dagger=1/y\prec 1$, so replacing $H(y)$ by~$H$
we can arrange that $\Psi^{\geq 0}\neq\emptyset$.
Then
every pre-$H$-field extension of $H$ has small derivation, and so we are done by [ADH, 14.5.11].
\end{proof}

\begin{proof}[Proof of Proposition~\ref{prop:univ theory}]
The natural axioms for pre-$H$-fields with very small derivation formulated in $\mathcal L^\iota$ are universal, so all models of $\Sigma$ are
 pre-$H$-fields with very small derivation. Conversely,   any pre-$H$-field $H$  with very small derivation   is a model of $\Sigma$: use Lemma~\ref{lem:very small der extend}  to extend $H$ to an $H$-closed field
 with small derivation, and note that the $\mathcal L^\iota$-theory of $H$-closed fields with small derivation is complete by [ADH, 16.6.3] and has a Hardy field model by Corollary~\ref{cor:transserial}.
\end{proof}

\noindent
Similar arguments 
allow us to settle a conjecture from \cite{AvdD2}, in slightly strengthened form.
For this, let $\mathcal L_x^\iota$ be $\mathcal L^\iota$ augmented by a constant symbol $x$.
We view each Hardy field containing the germ of the identity function on $\R$ as an
$\mathcal L_x^\iota$-structure by interpreting the symbols from~$\mathcal L^\iota$ as described
at the beginning of this subsection and the symbol $x$ by the germ  of the identity function
on $\R$, which we also denote by $x$ as usual. Each universal $\mathcal L_x^\iota$-sentence which holds in every maximal Hardy field also holds in every  Hardy field containing $x$.

\begin{prop}\label{prop:univ theory x}
Let $\Sigma_x$ be the set of universal $\mathcal L_x^\iota$-sentences true in all  
Hardy fields that contain $x$. Then the models of $\Sigma_x$ are the pre-$H$-fields with
distinguished element $x$  satisfying $x'=1$ and $x\succ 1$.
\end{prop}

\noindent
This follows from [ADH, 14.5.11] and the next lemma just like Proposition~\ref{prop:univ theory} followed from
Lemma~\ref{lem:very small der extend}  and [ADH, 16.6.3].

\begin{lemma}\label{lem:complete x}
The  $\mathcal L_x^\iota$-theory of $H$-closed fields  with distinguished element $x$
satisfying $x'=1$ and $x\succ 1$ is complete.
\end{lemma}
\begin{proof}
Let $K_1$, $K_2$ be  models of this theory, and let $x_1, x_2$ be the interpretations of $x$ in $K_1$, $K_2$.
Then [ADH, 10.2.2, 10.5.11] gives an isomorphism $\Q(x_1)\to\Q(x_2)$ of valued ordered differential fields sending $x_1$ to $x_2$.
To show that $K_1\equiv K_2$ as $\mathcal L_x^\iota$-structures, identify $\Q(x_1)$ with $\Q(x_2)$ via this isomorphism. 
View  $\HLO$-fields as $\mathcal L^\iota_{\HLO}$-structures where $\mathcal L^\iota_{\HLO}$ extends $\mathcal L^\iota$ as specified
in~[ADH, Chapter~16]. (See also the proof of Theorem~\ref{thm:transfer}.)
The $\upo$-free $H$-fields $K_1$, $K_2$ uniquely expand to  $\HLO$-fields $\boldsymbol{K}_1$, $\boldsymbol{K}_2$. The $H$-subfield $\Q(x_1)$ of~$K_1$ is grounded, so expands also uniquely to an $\HLO$-field,
and this $\HLO$-field is a common substructure of both $\boldsymbol{K}_1$ and~$\boldsymbol{K}_2$. 
Hence~$\boldsymbol{K} _1\equiv_{\Q(x_1)} \boldsymbol{K}_2$ by [ADH, 16.0.1, B.11.6]. 
This yields the claim.
\end{proof}

\noindent
A. Robinson~\cite{Robinson73} raised the issue of axiomatizing $\Sigma_x$. Proposition~\ref{prop:univ theory x} provides a finite axiomatization.  The completeness of the $\mathcal L^\iota$-theory of $H$-closed fields with small derivation together with Lemma~\ref{lem:complete x} and Theorem~\ref{thm:char d-max} yield:

\begin{cor}
The set $\Sigma$ of   universal $\mathcal L^\iota$-sentences true in all   Hardy fields is decidable, and so is
the set $\Sigma_x$ of universal $\mathcal L^\iota_x$-sentences true in all   Hardy fields
containing $x$. 
\end{cor}


\newlength\templinewidth
\setlength{\templinewidth}{\textwidth}
\addtolength{\templinewidth}{-2.25em}

\patchcmd{\thebibliography}{\list}{\printremarkbeforebib\list}{}{}

\let\oldaddcontentsline\addcontentsline
\renewcommand{\addcontentsline}[3]{\oldaddcontentsline{toc}{section}{References}}

\def\printremarkbeforebib{\bigskip\hskip1em The citation [ADH] refers to our book \\

\hskip1em\parbox{\templinewidth}{
M. Aschenbrenner, L. van den Dries, J. van der Hoeven,
\textit{Asymptotic Differential Algebra and Model Theory of Transseries,} Annals of Mathematics Studies, vol.~195, Princeton University Press, Princeton, NJ, 2017.
}

\bigskip

} 

\bibliographystyle{amsplain}

\end{document}